\documentclass[oneside,final, letterpaper]{ucr}

\usepackage{amsmath}
\usepackage{amsfonts}
\usepackage{amssymb}
\usepackage{enumitem}
\usepackage{hyperref}
\usepackage{tikz} % TikZ 3.0.0 or higher.  See TikZ node type declarations below.
\usetikzlibrary{shapes.geometric,arrows,decorations.markings,decorations.pathreplacing}

\definecolor{myurlcolor}{rgb}{0.6,0,0}
\definecolor{mycitecolor}{rgb}{0,0,0.8}
\definecolor{myrefcolor}{rgb}{0,0,0.8}
\hypersetup{colorlinks, linkcolor=myrefcolor, citecolor=mycitecolor, urlcolor=myurlcolor}

\newcounter{invisibleink}
\newenvironment{invisiblelabel}[1][]{\refstepcounter{invisibleink}\noindent%
   {}}{}

\newcommand*\pgfdeclareanchoralias[3]{%
  \expandafter\def\csname pgf@anchor@#1@#3\expandafter\endcsname
     \expandafter{\csname pgf@anchor@#1@#2\endcsname}}

\tikzset{
integral/.style={
  regular polygon, regular polygon sides=3, shape border rotate=180, draw=black, very thick,
  outer sep=0.025em, inner sep=0, minimum size=2em, fill=blue!5, text centered},
multiply/.style={
  regular polygon, regular polygon sides=3, shape border rotate=180, draw=black, very thick,
  outer sep=0.025em, inner sep=0, minimum size=2em, fill=blue!5, text centered},
upmultiply/.style={
  regular polygon, regular polygon sides=3, draw=black, very thick,
  outer sep=0.025em, inner sep=0, minimum size=2em, fill=blue!5, text centered},
zero/.style={
  circle, draw=black, very thick, minimum size=0.15cm, fill=black,
  inner sep=0, outer sep=0},
hole/.style={
  circle, draw=white, very thick, minimum size=0.25cm, fill=white,
  inner sep=0, outer sep=0},
bang/.style={
  circle, draw=black, very thick, minimum size=0.15cm, fill=green!10,
  inner sep=0, outer sep=0},
delta/.style={
  regular polygon, regular polygon sides=3, minimum size=0.4cm, inner
  sep=0, outer sep=0.025em, draw=black, very thick, fill=green!10},
codelta/.style={
  regular polygon, regular polygon sides=3, shape border rotate=180, minimum size=0.4cm,
  inner sep=0, outer sep=0.025em, draw=black, very thick, fill=green!10},
plus/.style={
  regular polygon, regular polygon sides=3, shape border rotate=180, minimum size=0.4cm,
  inner sep = 0, outer sep=0.025em, draw=black, very thick, fill=black},
coplus/.style={
  regular polygon, regular polygon sides=3, minimum size=0.4cm,
  inner sep = 0, outer sep=0.025em, draw=black, very thick, fill=black},
oao/.style={
  circle, draw=black, very thick, minimum size=0.15cm, fill=black!60!yellow!50,
  inner sep=0, outer sep=0},
ab/.style={
  regular polygon, regular polygon sides=3, minimum size=0.4cm, inner
  sep=0, outer sep=0.025em, draw=black, very thick, fill=black!60!yellow!50},
ba/.style={
  regular polygon, regular polygon sides=3, shape border rotate=180, minimum size=0.4cm,
  inner sep=0, outer sep=0.025em, draw=black, very thick, fill=black!60!yellow!50},
sqnode/.style={
  regular polygon, regular polygon sides=4, minimum size=2.6em,
  draw=black, very thick, inner sep=0.2em, outer sep=0.025em,
  fill=yellow!10, text centered},
blackbox/.style={
  regular polygon, regular polygon sides=4, minimum size=2.6em,
  draw=black, very thick, inner sep=0.2em, outer sep=0.025em, fill=black},
 }

\pgfdeclareanchoralias{regular polygon}{corner 1}{io}
\pgfdeclareanchoralias{regular polygon}{corner 2}{left out}
\pgfdeclareanchoralias{regular polygon}{corner 3}{right out}
\pgfdeclareanchoralias{regular polygon}{corner 3}{left in}
\pgfdeclareanchoralias{regular polygon}{corner 2}{right in}

\newcommand{\Cob}{\mathrm{Cob}}
\newcommand{\Hilb}{\mathrm{Hilb}}
\newcommand{\Fin}{\mathrm{Fin}}
\newcommand{\Set}{\mathrm{Set}}

\newcommand{\R}{{\mathbb R}}
\newcommand{\C}{{\mathbb C}}
\newcommand{\N}{{\mathbb N}}
\newcommand{\B}{{\mathbb B}}
\newcommand{\Z}{{\mathbb Z}}

\newcommand{\maps}{\colon} 
\newcommand{\iso}{\cong}
\renewcommand{\hom}{{\rm hom}}
\newcommand{\tensor}{\otimes}
\newcommand{\asrelto}{\nrightarrow}
\newcommand{\of}{\circ}
\newcommand{\To}{\Rightarrow}

\newcommand{\End}{\mathrm{End}}
\newcommand{\bbox}{\blacksquare}
\newcommand{\dBox}{\blacklozenge}
\newcommand{\dbox}{\lozenge}

\renewcommand{\P}{\mathtt{P}}
\newcommand{\F}{\mathtt{F}}
\newcommand{\U}{\mathtt{U}}
\newcommand{\sige}{\mathcal{E}}

\newcommand{\relks}{\mathtt{FinRel}_{k(s)}}
\newcommand{\relk}{\mathtt{FinRel}_k}
\newcommand{\relS}{\mathtt{FinRel}_S}
\newcommand{\Relk}{\mathrm{FinRel}_k}
\newcommand{\Relks}{\mathrm{FinRel}_{k(s)}}
\newcommand{\finrel}{\mathtt{FinRel}}

\newcommand{\prop}{{\mathsf{PROP}}}
\newcommand{\porp}{\mathtt{StFlow}_k}
\newcommand{\propQ}{\mathtt{Q}}
\newcommand{\catC}{\mathcal{C}}
\newcommand{\D}{\mathcal{D}}
\newcommand{\boxful}{\Box\catC}

\newcommand{\vectks}{\mathtt{FinVect}_{k(s)}}
\newcommand{\Vectks}{\mathrm{FinVect}_{k(s)}}
\newcommand{\vectk}{\mathtt{FinVect}_k}
\newcommand{\Vectk}{\mathrm{FinVect}_k}

\newcommand{\boxv}{\Box(\vectk)}
\newcommand{\boxvs}{\Box(\vectks)}
\newcommand{\boxr}{\Box(\relk)}

\newcommand{\sigflow}{\mathtt{SigFlow}}
\newcommand{\goodflow}{\mathtt{ContFlow}}
\newcommand{\looseflow}{\mathtt{CtrlFlow}}
\newcommand{\st}{\mathtt{Stateful}}
\newcommand{\St}{\mathrm{Stateful}}

\newcommand{\Mat}{\mathtt{Mat}}
\newcommand{\SV}{{\mathbb {SV}}}
\newcommand{\Ecirc}{\mathrm{Circ}}
\newcommand{\ecirc}{\mathtt{Circ}}
\newcommand{\Lagr}{\mathrm{LagrRel}_{k(s)}}
\newcommand{\lagr}{\mathtt{LagrRel}_{k(s)}}

\newcommand{\eval}{\mathrm{eval}}
\newcommand{\feed}{\mathrm{feed}}
\newcommand{\G}{\mathrm{Box}}
\newcommand{\Obj}{\mathrm{Ob}}
\newcommand{\Mor}{\mathrm{Mor}}

\newcommand{\Define}[1]{{\bf \boldmath{#1}}}

\def\smc{symmetric monoidal category}
\def\smcs{symmetric monoidal categories}
\def\smf{symmetric monoidal functor}
\def\smdf{symmetric monoidal dagger functor}
\def\smt{symmetric monoidal theory}
\def\smts{symmetric monoidal theories}
\def\zero{0}

\begin{document}

% Declarations for Front Matter

\title{Categories in Control:  Applied PROPs}
\author{Jason Michael Erbele}
\degreemonth{December}
\degreeyear{2016}
\degree{Doctor of Philosophy}
\chair{Dr. John C. Baez}
\othermembers{Dr. Julie Bergner\\
Dr. Kevin Costello}
\numberofmembers{3}
\field{Mathematics}
\campus{Riverside}

\maketitle
\copyrightpage{}
\approvalpage{}

\degreesemester{Fall}

\begin{frontmatter}

\begin{acknowledgements}
Large parts of Chapters 1 and 3 of the current work appeared in 2015 in \textsl{Theories and
Applications of Categories}, Volume 30.  This prior incarnation of \emph{Categories in control} has
been expanded to the present corpus.  I am grateful for the direction and supervision of John Baez,
whose clear exposition and advice have been incredibly useful.  I could not have completed this
dissertation without his guidance, prodding, and helpful meddling.  Also valuable were the many
conversations with Brendan Fong, which helped to crystallize several ideas.

Most of all, I owe an enormous debt of gratitude to the late Dr. Gene Scott.  The scope of that debt
is too large to fit in this section.  The least of which, his leadership and tenacity inspired me to
pursue higher education and to persevere when the path looked impossible.
\end{acknowledgements}

\begin{dedication}
\null\vfil
{\large
\vfill
\begin{flushright}
To the young at heart,
\end{flushright}
\vfill
\begin{center}
to the curious in mind,
\end{center}
\vfill
to the kindred soul.
\vfill
}
\vfil\null
\end{dedication}

\begin{abstract}
Control theory uses `signal-flow diagrams' to describe processes where real-valued functions of 
time are added, multiplied by scalars, differentiated and integrated, duplicated and deleted.  
These diagrams can be seen as string diagrams for the PROP \(\relk\), the strict version of the 
category of finite-dimensional vector spaces over the field of rational functions \(k = \R(s)\) and 
linear relations, where the variable \(s\) acts as differentiation and the monoidal structure is 
direct sum rather than the usual tensor product of vector spaces.  Control processes are also 
described by controllability and observability---whether the input can drive the process to any 
state, and whether any state can be determined from later outputs.  For any field \(k\) we give a 
presentation of \(\relk\) in terms of generators of the free PROP of signal-flow diagrams together 
with the equations that give \(\relk\) its structure.  The `cap' and `cup' generators, missing when 
the morphisms are linear maps, make it possible to model feedback.  The relations say, among other 
things, that the 1-dimensional vector space \(k\) has two special commutative \(\dagger\)-Frobenius 
structures, such that the multiplication and unit of either one and the comultiplication and counit 
of the other fit together to form a bimonoid.  This sort of structure, but with tensor product 
replacing direct sum, is familiar from the `ZX-calculus' obeyed by a finite-dimensional Hilbert 
space with two mutually unbiased bases.  In order to address controllability and observability, we 
construct the PROP \(\st_k\) and relate it back to the PROP of signal-flow diagrams.  This provides 
a way to graphically express controllability and observability for linear time-invariant processes.
\end{abstract}

\tableofcontents
\listoffigures
% \listoftables
\end{frontmatter}

\chapter{Introduction\label{intro}}
\section{Outline}
\label{intro:outline}
Control theory is the branch of engineering that focuses on manipulating `open systems'---systems
with inputs and outputs---to achieve desired goals.  In control theory, several graphical
models---\emph{e.g.} `signal-flow graphs' and `box diagrams'---have been used to describe linear
ways of manipulating signals, which we will take here to be smooth real-valued functions of time
\cite{Friedland}.  For a category theorist, at least, it is natural to treat these graphical models
as string diagrams in a symmetric monoidal category \cite{JS1,JS2}.  Here we use the term
\Define{signal-flow diagram} to refer to these string diagrams.  This forces some small changes
of perspective, which we discuss below, but more important is the question: \emph{which symmetric
monoidal category?}

We shall argue that a first approximation to the answer is: the category \(\Relk\) of
finite-dimensional vector spaces over a certain field \(k\), but with \emph{linear relations} rather
than linear maps as morphisms, and \emph{direct sum} rather than tensor product providing the
symmetric monoidal structure.  We use the field \(k = \R(s)\) consisting of rational functions in
one real variable \(s\).  This variable has the meaning of differentation.  A linear relation from
\(k^m\) to \(k^n\) is thus a system of linear constant-coefficient ordinary differential equations
relating \(m\) `input' signals and \(n\) `output' signals.

A second approximation to the answer is: the category \(\St_k\) of finite-dimensional vector spaces
over a certain field \(k(s)\) with `stateful' morphisms which, roughly speaking, distinguish the
paths that involve \(s\) from those that do not involve \(s\).  Now there are \(m\) `inputs', \(n\)
`states', and \(p\) `outputs'.  When \(k = \R\), we are again back to the situation of rational
functions in one real variable \(s\).  This category is developed and discussed in
Chapter~\ref{stateful}.  The key advantage to \(\St_k\) over \(\Relks\) is the ability to extract
the control theoretic concepts of controllability and observability from a stateful morphism.  The
key disadvantage is stateful morphisms evaluate to linear maps rather than linear relations.  So
while every signal-flow diagram has a linear relation associated to it, not every signal-flow
diagram has a stateful morphism associated to it.

Our main goal for the first approximation is to provide a complete `generators and equations'
picture of this symmetric monoidal category, with the generators being familiar components of
the graphical models used by control theorists.  It turns out that the answer has an intriguing but
mysterious connection to ideas that are familiar in the diagrammatic approach to quantum theory.
Quantum theory also involves linear algebra, but it uses linear maps between Hilbert spaces as
morphisms, and the tensor product of Hilbert spaces provides the symmetric monoidal structure.

For the second approximation, our main goal is to identify which signal-flow diagrams describe
controllable (\emph{resp.} observable) systems.  It turns out that not all signal-flow diagrams
admit as `stateful' description, so part of this goal is the question, for which signal-flow
diagrams can we ask about controllability and observability?  We hope that the category-theoretic
viewpoint on signal-flow diagrams will shed new light on control theory.  However, in this
dissertation we only lay the groundwork.  

Briefly, the plan is as follows:  Chapter~\ref{PROPs} introduces the the machinery of PROPs,
explaining how to describe a PROP using generators and equations and how to work with PROPs using
this description.  PROPs form a particularly simple class of symmetric monoidal categories that
includes \(\relk\) and \(\st_k\).  By Mac~Lane's coherence theorem \cite{MacLane}, the PROPs
\(\st_k\) and \(\relk\) are equivalent to the categories \(\St_k\) and \(\Relk\) described above.
This leads to Chapter~\ref{vectrel}, which gives a presentation of \(\relk\), introduces signal-flow
diagrams, and summarizes the main results of our first approximation.  To get to the second
approximation, Chapter~\ref{stateful} introduce a new PROP, \(\st_k\) and describes how it relates
to \(\relks\).  Chapter~\ref{stateful} also describes how to determine controllability and
observability from a stateful morphism.  The main result of the second approximation appears in
Chapter~\ref{goodflow}, where we consider signal-flow diagrams as mathematical entities in their own
right and describe the subcategory of signal-flow diagrams that admit a stateful description.  For
a signal-flow diagram that admit such a description, the description provides a path to determining
controllability and observability for the signal-flow diagram.  Finally, Chapter~\ref{conclusions}
deals with future work: we describe ways in which the stateful description can be extended to larger
subcategories of signal-flow diagrams and how the category \(\ecirc\) of open passive electric
circuits with linear circuit elements can be viewed as a category of signal-flow diagrams.  This
second direction for future work would connect the present work with that of Baez and Fong
\cite{BF}.  In the following sections we sketch some of the main ideas of this plan.
%%%%%
%%%%%
\section{PROPs, linear relations, and signal-flow diagrams}
\label{sigflow}

In his famous thesis, Lawvere \cite{Lawvere} introduced `functorial semantics'.  In this idea, a
functor \(F \maps \catC \to \D\) sends formal expressions, which are morphisms in \(\catC\), to
their `meanings', which are morphisms in \(\D\).  One says that \(\catC\) provides the `syntax' and
\(\D\) the `semantics'.  Here we apply this idea to control theory.  For example, we may take
\(\catC\) to be a category where morphisms are signal-flow diagrams, and \(\D\) to be \(\relk\):
then we shall construct a `black-boxing functor' sending any signal-flow diagram to the linear
relation it stands for.

To apply Lawvere's ideas one wants categories equipped with extra structure: in our work we use
\Define{PROPs}, which are strict symmetric monoidal categories whose objects are natural numbers,
the tensor product of objects being given by addition.  In Chapter~\ref{PROPs} we explain how to
describe PROPs using generators and equations.  This follows the work of Baez, Coya and Rebro
\cite{BCR}, which is based on the work of Trimble \cite{Trimble}.  Chapter~\ref{PROPs} also has
parallels in Zanasi's Ph.D. dissertation \cite[Chap. 2.2]{Za}.

Of key importance in the present work is the existence and uniqueness of a functor from the free
PROP on some generators to a PROP presented by those generators and some equations.  We continue in
Chapter~\ref{vectrel} with a generators and equations description of \(\relk\).  This chapter can
also be found in \cite{BE}, with some minor changes made in the present work to streamline its
connections to the other chapters.  This begins the formalization into signal-flow diagrams of what
control theorists do with their graphical models.  When \(k = \R(s)\), a morphisms in \(\relk\)
describes the relation between some input signals and output signals, corresponding to what control
theorists call `transfer functions'.\footnote{Control theorists generally only deal with linear
\emph{maps} rather than linear \emph{relations} in this context, so a pedant may argue for the
invention of a new jargon term here, `transfer relation'.}  The generators of \(\relk\) correspond
to some of the most basic operations one might want to perform when manipulating signals.  The
simplest operation is amplification, or `scaling': multiplying a signal by a scalar.  A signal can
be scaled by a constant factor:
\[  f \mapsto cf,  \]
where \(c \in \R\).  We can write this as a signal-flow diagram:
  
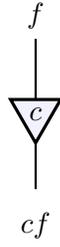
\begin{center}
% scalar multiplication
   \begin{tikzpicture}[thick]
   \node[coordinate] (in) [label=\(f\)] {};
   \node [multiply] (mult) [below of=in] {\(c\)};
   \node[coordinate] (out) [below of=mult, label={[shift={(0,-0.8)}]\(cf\)}] {};

   \draw (in) -- (mult) -- (out);
   \end{tikzpicture}.
\end{center}
Here the labels \(f\) and \(cf\) on top and bottom are just for explanatory purposes and not really
part of the diagram.  Control theorists often draw arrows on the wires, but this is unnecessary from
the string diagram perspective.  Arrows on wires are useful to distinguish objects from their
duals, but ultimately we will obtain a compact closed category where each object is its own dual, so
the arrows can be dropped.  What we really need is for the box denoting scalar multiplication to
have a clearly defined input and output.  This is why we draw it as a triangle.  Control theorists
often use a rectangle or circle, using arrows on wires to indicate which carries the input \(f\) and
which the output \(c f\).

A signal can also be integrated with respect to the time variable:
\[ f \mapsto \int f .\]
Mathematicians typically take differentiation as fundamental, but engineers sometimes prefer 
integration, because it is more robust against small perturbations.  In the end it will not matter 
much here.  We can again draw integration as a signal-flow diagram:
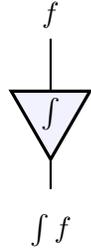
\begin{center}
  % integration
   \begin{tikzpicture}[thick]
   \node[coordinate] (in) [label=\(f\)] {};
   \node [integral] (int) [below of=in] {\(\int\)};
   \node[coordinate] (out) [below of=int, label={[shift={(0.0,-0.9)}]\(\int f\)}] {};

   \draw (in) -- (int) -- (out);
   \end{tikzpicture}.
  \end{center}
Since this looks like the diagram for scaling, it is natural to extend \(\R\) to \(\R(s)\), the
field of rational functions of a variable \(s\) which stands for differentiation.  Then
differentiation becomes a special case of scalar multiplication, namely multiplication by \(s\), and
integration becomes multiplication by \(1/s\).  Engineers accomplish the same effect with Laplace
transforms, since differentiating a signal \(f\) is equivalent to multiplying its Laplace transform 
\[   (\mathcal{L}f)(s) = \int_0^\infty f(t) e^{-st} \,dt  \]
by the variable \(s\).  Another option is to use the Fourier transform: differentiating \(f\) is
equivalent to multiplying its Fourier transform 
\[   (\mathcal{F}f)(\omega) = \int_{-\infty}^\infty f(t) e^{-i\omega t}\, dt  \]
by \(-i\omega\).   Of course, the function \(f\) needs to be sufficiently well-behaved to justify
calculations involving its Laplace or Fourier transform.  At a more basic level, it also requires
some work to treat integration as the two-sided inverse of differentiation.  Engineers do this by
considering signals that vanish for \(t < 0\), and choosing the antiderivative that vanishes under
the same condition.  Luckily all these issues can be side-stepped in a formal treatment of
signal-flow diagrams: we can simply treat signals as living in an unspecified vector space over the
field \(\R(s)\).  The field \(\C(s)\) would work just as well, and control theory relies heavily on
complex analysis.  In most of this paper we work over an arbitrary field \(k\).

The simplest possible signal processor is a rock, which takes the `input' given by the force \(F\)
on the rock and produces as `output' the rock's position \(q\).  Thanks to Newton's second law
\(F=ma\), we can describe this using a signal-flow diagram:
% F=ma
  
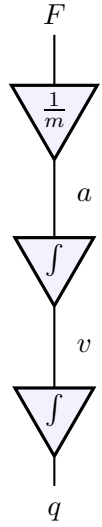
\begin{center}
   \begin{tikzpicture}[thick]
   \node[coordinate] (q) [label={[shift={(0,-0.6)}]\(q\)}] {};
   \node [integral] (diff) [above of=q] {\(\int\)};
   \node (v) [above of=diff, label={[shift={(0.4,-0.5)}]\(v\)}] {};
   \node [integral] (dot) [above of=v] {\(\int\)};
   \node (a) [above of=dot, label={[shift={(0.4,-0.5)}]\(a\)}] {};
   \node [multiply] (m) [above of=a] {\(\frac{1}{m}\)};
   \node[coordinate] (F) [above of=m, label={[shift={(0,0)}]\(F\)}] {};

   \draw (F) -- (m) -- (dot) -- (diff) -- (q);
   \end{tikzpicture}.
  \end{center}
Here composition of morphisms is drawn in the usual way, by attaching the output wire of one
morphism to the input wire of the next.

To build more interesting machines we need more building blocks, such as addition:
\[ + \maps (f,g) \mapsto f + g   \]
and duplication:
\[ \Delta \maps  f \mapsto (f,f).  \]
When these linear maps are written as matrices, their matrices are transposes of each other.  This
is reflected in the signal-flow diagrams for addition and duplication:
  
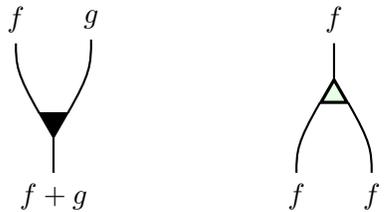
\begin{center}
% Sum generating morphism
   \begin{tikzpicture}[thick]
   \node[plus] (adder) {};
   \node (f) at (-0.5,1.35) {\(f\)};
   \node (g) at (0.5,1.35) {\(g\)};
   \node (out) [below of=adder] {\(f+g\)};

   \draw (f) .. controls +(-90:0.6) and +(120:0.6) .. (adder.left in);
   \draw (g) .. controls +(-90:0.6) and +(60:0.6) .. (adder.right in);
   \draw (adder) -- (out);
   \end{tikzpicture}
     \hspace{2cm}
% Duplication generating morphism
   \begin{tikzpicture}[thick]
   \node[delta] (dupe){};
   \node (o1) at (-0.5,-1.35) {\(f\)};
   \node (o2) at (0.5,-1.35) {\(f\)};
   \node (in) [above of=dupe] {\(f\)};

   \draw (o1) .. controls +(90:0.6) and +(-120:0.6) .. (dupe.left out);
   \draw (o2) .. controls +(90:0.6) and +(-60:0.6) .. (dupe.right out);
   \draw (in) -- (dupe);
   \end{tikzpicture}.
  \end{center}
The second is essentially an upside-down version of the first.  However, we draw addition as a dark
triangle and duplication as a light one because we will later want another way to `turn addition
upside-down' that does \emph{not} give duplication.  As an added bonus, a light upside-down triangle
resembles the Greek letter \(\Delta\), the usual symbol for duplication.  

While they are typically not considered worthy of mention in control theory, for completeness we
must include two other building blocks.  One is the zero map from \(\{0\}\) to our field \(k\),
which we denote as \(0\) and draw its signal-flow diagram as follows:
  \begin{center}
% zero
  \begin{tikzpicture}[thick]
   \node (out1) {\(0\)};
   \node [zero] (ins1) at (0,1) {};

   \draw (out1) -- (ins1);
   \end{tikzpicture}.
  \end{center}
The other is the zero map from \(k\) to \(\{0\}\), sometimes called `deletion', which we denote as
\(!\) and draw thus:
  \begin{center}
% deletion
   \begin{tikzpicture}[thick]
   \node (in1) {\(f\)};
   \node [bang] (del1) at (0,-1) {};

   \draw (in1) -- (del1);
   \end{tikzpicture}.
  \end{center}

Just as the matrices for addition and duplication are transposes of each other, so are the matrices
for zero and deletion, though they are rather degenerate, being \(1 \times 0\) and \(0 \times 1\)
matrices, respectively.  Addition and zero make \(k\) into a commutative monoid, meaning that the
following equations hold:
 \begin{center}
    \scalebox{0.80}
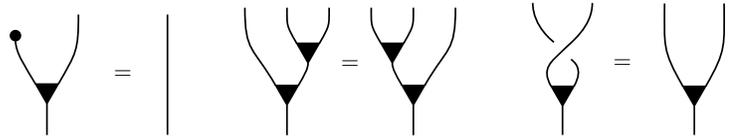
{
% Unitality
   \begin{tikzpicture}[-, thick, node distance=0.74cm]
   \node [plus] (summer) {};
   \node [coordinate] (sum) [below of=summer] {};
   \node [coordinate] (Lsum) [above left of=summer] {};
   \node [zero] (insert) [above of=Lsum, shift={(0,-0.35)}] {};
   \node [coordinate] (Rsum) [above right of=summer] {};
   \node [coordinate] (sumin) [above of=Rsum] {};
   \node (equal) [right of=Rsum, shift={(0,-0.26)}] {\(=\)};
   \node [coordinate] (in) [right of=equal, shift={(0,1)}] {};
   \node [coordinate] (out) [right of=equal, shift={(0,-1)}] {};

   \draw (insert) .. controls +(270:0.3) and +(120:0.3) .. (summer.left in)
         (summer.right in) .. controls +(60:0.6) and +(270:0.6) .. (sumin)
         (summer) -- (sum)    (in) -- (out);
   \end{tikzpicture}
        \hspace{1.0cm}
% Associativity
   \begin{tikzpicture}[-, thick, node distance=0.7cm]
   \node [plus] (uradder) {};
   \node [plus] (adder) [below of=uradder, shift={(-0.35,0)}] {};
   \node [coordinate] (urm) [above of=uradder, shift={(-0.35,0)}] {};
   \node [coordinate] (urr) [above of=uradder, shift={(0.35,0)}] {};
   \node [coordinate] (left) [left of=urm] {};

   \draw (adder.right in) .. controls +(60:0.2) and +(270:0.1) .. (uradder.io)
         (uradder.right in) .. controls +(60:0.35) and +(270:0.3) .. (urr)
         (uradder.left in) .. controls +(120:0.35) and +(270:0.3) .. (urm)
         (adder.left in) .. controls +(120:0.75) and +(270:0.75) .. (left)
         (adder.io) -- +(270:0.5);

   \node (eq) [right of=uradder, shift={(0,-0.25)}] {\(=\)};

   \node [plus] (ulsummer) [right of=eq, shift={(0,0.25)}] {};
   \node [plus] (summer) [below of=ulsummer, shift={(0.35,0)}] {};
   \node [coordinate] (ulm) [above of=ulsummer, shift={(0.35,0)}] {};
   \node [coordinate] (ull) [above of=ulsummer, shift={(-0.35,0)}] {};
   \node [coordinate] (right) [right of=ulm] {};

   \draw (summer.left in) .. controls +(120:0.2) and +(270:0.1) .. (ulsummer.io)
         (ulsummer.left in) .. controls +(120:0.35) and +(270:0.3) .. (ull)
         (ulsummer.right in) .. controls +(60:0.35) and +(270:0.3) .. (ulm)
         (summer.right in) .. controls +(60:0.75) and +(270:0.75) .. (right)
         (summer.io) -- +(270:0.5);
   \end{tikzpicture}
        \hspace{1.0cm}
% Commutativity
   \begin{tikzpicture}[-, thick, node distance=0.7cm]
   \node [plus] (twadder) {};
   \node [coordinate] (twout) [below of=twadder] {};
   \node [coordinate] (twR) [above right of=twadder, shift={(-0.2,0)}] {};
   \node (cross) [above of=twadder] {};
   \node [coordinate] (twRIn) [above left of=cross, shift={(0,0.3)}] {};
   \node [coordinate] (twLIn) [above right of=cross, shift={(0,0.3)}] {};

   \draw (twadder.right in) .. controls +(60:0.35) and +(-45:0.25) .. (cross)
                            .. controls +(135:0.2) and +(270:0.4) .. (twRIn);
   \draw (twadder.left in) .. controls +(120:0.35) and +(-135:0.25) .. (cross.center)
                           .. controls +(45:0.2) and +(270:0.4) .. (twLIn);
   \draw (twout) -- (twadder);

   \node (eq) [right of=twR] {\(=\)};

   \node [coordinate] (L) [right of=eq] {};
   \node [plus] (adder) [below right of=L] {};
   \node [coordinate] (out) [below of=adder] {};
   \node [coordinate] (R) [above right of=adder] {};
   \node (cross) [above left of=R] {};
   \node [coordinate] (LIn) [above left of=cross] {};
   \node [coordinate] (RIn) [above right of=cross] {};

   \draw (adder.left in) .. controls +(120:0.7) and +(270:0.7) .. (LIn)
         (adder.right in) .. controls +(60:0.7) and +(270:0.7) .. (RIn)
         (out) -- (adder);
   \end{tikzpicture}
    }.
\end{center}
The equation at right is the commutative law, and the crossing of strands is the `braiding'
\[       B \maps (f,g) \mapsto (g,f)  \]
by which we switch two signals.   In fact this braiding is a `symmetry', so it does not matter which
strand goes over which:
 
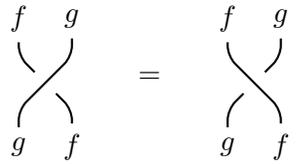
\begin{center}
% Symmetry
   \begin{tikzpicture}[thick, node distance=0.5cm]
   \node (fstart) {\(f\)};
   \node [coordinate] (ftop) [below of=fstart] {};
   \node (center) [below right of=ftop] {};
   \node [coordinate] (fout) [below right of=center] {};
   \node (fend) [below of=fout] {\(f\)};
   \node [coordinate] (gtop) [above right of=center] {};
   \node (gstart) [above of=gtop] {\(g\)};
   \node [coordinate] (gout) [below left of=center] {};
   \node (gend) [below of=gout] {\(g\)};

   \draw [rounded corners] (fstart) -- (ftop) -- (center) --
   (fout) -- (fend) (gstart) -- (gtop) -- (gout) -- (gend);
   \end{tikzpicture}
\hspace{1em} 
\raisebox{3em}{=} 
\hspace{1em}
   \begin{tikzpicture}[thick, node distance=0.5cm]
   \node (fstart) {\(f\)};
   \node [coordinate] (ftop) [below of=fstart] {};
   \node (center) [below right of=ftop] {};
   \node [coordinate] (fout) [below right of=center] {};
   \node (fend) [below of=fout] {\(f\)};
   \node [coordinate] (gtop) [above right of=center] {};
   \node (gstart) [above of=gtop] {\(g\)};
   \node [coordinate] (gout) [below left of=center] {};
   \node (gend) [below of=gout] {\(g\)};

   \draw [rounded corners] (fstart) -- (ftop) -- (fout) -- (fend)
   (gstart) -- (gtop) -- (center) -- (gout) -- (gend);
   \end{tikzpicture}.
  \end{center}

Dually, duplication and deletion make \(k\) into a cocommutative comonoid.  This means that if we
reflect the equations obeyed by addition and zero across the horizontal axis and turn dark
operations into light ones, we obtain another set of valid equations:
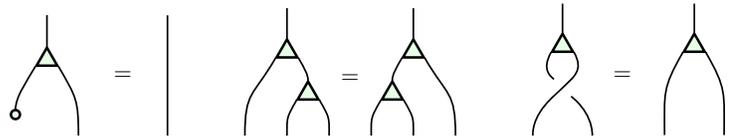
\begin{center}
    \scalebox{0.80}{
% Co-unitality
   \begin{tikzpicture}[-, thick, node distance=0.74cm]
   \node [delta] (dupe) {};
   \node [coordinate] (top) [above of=dupe] {};
   \node [coordinate] (Ldub) [below left of=dupe] {};
   \node [bang] (delete) [below of=Ldub, shift={(0,0.35)}] {};
   \node [coordinate] (Rdub) [below right of=dupe] {};
   \node [coordinate] (dubout) [below of=Rdub] {};
   \node (equal) [right of=Rdub, shift={(0,0.26)}] {\(=\)};
   \node [coordinate] (in) [right of=equal, shift={(0,1)}] {};
   \node [coordinate] (out) [right of=equal, shift={(0,-1)}] {};

   \draw (delete) .. controls +(90:0.3) and +(240:0.3) .. (dupe.left out)
         (dupe.right out) .. controls +(300:0.6) and +(90:0.6) .. (dubout)
         (dupe) -- (top)    (in) -- (out);
   \end{tikzpicture}
       \hspace{1.0cm}
% Co-associativity
   \begin{tikzpicture}[-, thick, node distance=0.7cm]
   \node [delta] (lrduper) {};
   \node [delta] (duper) [above of=lrduper, shift={(-0.35,0)}] {};
   \node [coordinate](lrm) [below of=lrduper, shift={(-0.35,0)}] {};
   \node [coordinate](lrr) [below of=lrduper, shift={(0.35,0)}] {};
   \node [coordinate](left) [left of=lrm] {};

   \draw (duper.right out) .. controls +(300:0.2) and +(90:0.1) .. (lrduper.io)
         (lrduper.right out) .. controls +(300:0.35) and +(90:0.3) .. (lrr)
         (lrduper.left out) .. controls +(240:0.35) and +(90:0.3) .. (lrm)
         (duper.left out) .. controls +(240:0.75) and +(90:0.75) .. (left)
         (duper.io) -- +(90:0.5);

   \node (eq) [right of=lrduper, shift={(0,0.25)}] {\(=\)};

   \node [delta] (lldubber) [right of=eq, shift={(0,-0.25)}] {};
   \node [delta] (dubber) [above of=lldubber, shift={(0.35,0)}] {};
   \node [coordinate] (llm) [below of=lldubber, shift={(0.35,0)}] {};
   \node [coordinate] (lll) [below of=lldubber, shift={(-0.35,0)}] {};
   \node [coordinate] (right) [right of=llm] {};

   \draw (dubber.left out) .. controls +(240:0.2) and +(90:0.1) .. (lldubber.io)
         (lldubber.left out) .. controls +(240:0.35) and +(90:0.3) .. (lll)
         (lldubber.right out) .. controls +(300:0.35) and +(90:0.3) .. (llm)
         (dubber.right out) .. controls +(300:0.75) and +(90:0.75) .. (right)
         (dubber.io) -- +(90:0.5);
   \end{tikzpicture}
       \hspace{1.0cm}
% Co-commutativity
   \begin{tikzpicture}[-, thick, node distance=0.7cm]
   \node [coordinate] (twtop) {};
   \node [delta] (twdupe) [below of=twtop] {};
   \node [coordinate] (twR) [below right of=twdupe, shift={(-0.2,0)}] {};
   \node (cross) [below of=twdupe] {};
   \node [coordinate] (twROut) [below left of=cross, shift={(0,-0.3)}] {};
   \node [coordinate] (twLOut) [below right of=cross, shift={(0,-0.3)}] {};

   \draw (twdupe.left out) .. controls +(240:0.35) and +(135:0.25) .. (cross)
                           .. controls +(-45:0.2) and +(90:0.4) .. (twLOut)
         (twdupe.right out) .. controls +(300:0.35) and +(45:0.25) .. (cross.center)
                            .. controls +(-135:0.2) and +(90:0.4) .. (twROut)
         (twtop) -- (twdupe);

   \node (eq) [right of=twR] {\(=\)};

   \node [coordinate] (L) [right of=eq] {};
   \node [delta] (dupe) [above right of=L] {};
   \node [coordinate] (top) [above of=dupe] {};
   \node [coordinate] (R) [below right of=dupe] {};
   \node (uncross) [below left of=R] {};
   \node [coordinate] (LOut) [below left of=uncross] {};
   \node [coordinate] (ROut) [below right of=uncross] {};

   \draw (dupe.left out) .. controls +(240:0.7) and +(90:0.7) .. (LOut)
         (dupe.right out) .. controls +(300:0.7) and +(90:0.7) .. (ROut)
         (top) -- (dupe);
   \end{tikzpicture}
    }.
\end{center}
There are also equations between the monoid and comonoid operations.  For example, adding two
signals and then duplicating the result gives the same output as duplicating each signal and then
adding the results:
% Sumdupe bimonoid law
  
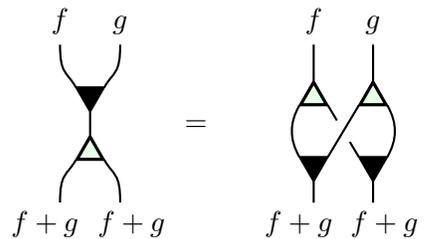
\begin{center}
   \begin{tikzpicture}[thick]
   \node[plus] (adder) {};
   \node [coordinate] (f) [above of=adder, shift={(-0.4,-0.325)}, label={\(f\)}] {\(f\)};
   \node [coordinate] (g) [above of=adder, shift={(0.4,-0.325)}, label={\(g\)}] {\(g\)};
   \node[delta] (dupe) [below of=adder, shift={(0,0.25)}] {};
   \node [coordinate] (outL) [below of=dupe, shift={(-0.4,0.325)}, label={[shift={(-0.2,-0.6)}]\(f+g\)}] {};
   \node [coordinate] (outR) [below of=dupe, shift={(0.4,0.325)}, label={[shift={(0.15,-0.6)}]\(f+g\)}] {};

   \draw (adder.io) -- (dupe.io)
         (f) .. controls +(270:0.4) and +(120:0.25) .. (adder.left in)
         (adder.right in) .. controls +(60:0.25) and +(270:0.4) .. (g)
         (dupe.left out) .. controls +(240:0.25) and +(90:0.4) .. (outL)
         (dupe.right out) .. controls +(300:0.25) and +(90:0.4) .. (outR);
   \end{tikzpicture}
      \raisebox{4em}{=}
      \hspace{1em}
   \begin{tikzpicture}[-, thick, node distance=0.7cm]
   \node [plus] (addL) {};
   \node (cross) [above right of=addL, shift={(-0.1,-0.0435)}] {};
   \node [plus] (addR) [below right of=cross, shift={(-0.1,0.0435)}] {};
   \node [delta] (dupeL) [above left of=cross, shift={(0.1,-0.0435)}] {};
   \node [delta] (dupeR) [above right of=cross, shift={(-0.1,-0.0435)}] {};
   \node [coordinate] (f) [above of=dupeL, label={\(f\)}] {};
   \node [coordinate] (g) [above of=dupeR, label={\(g\)}] {};
   \node [coordinate] (sum1) [below of=addL, shift={(0,0.2)}, label={[shift={(-0.2,-0.6)}]\(f+g\)}] {};
   \node [coordinate] (sum2) [below of=addR, shift={(0,0.2)}, label={[shift={(0.15,-0.6)}]\(f+g\)}] {};

   \path
   (addL) edge (sum1) (addL.right in) edge (dupeR.left out) (addL.left in) edge [bend left=30] (dupeL.left out)
   (addR) edge (sum2) (addR.left in) edge (cross) (addR.right in) edge [bend right=30] (dupeR.right out)
   (dupeL) edge (f)
   (dupeL.right out) edge (cross)
   (dupeR) edge (g);
   \end{tikzpicture}.
  \end{center} 
This diagram is familiar in the theory of Hopf algebras, or more generally bialgebras.  Here it is
an example of the fact that the monoid operations on \(k\) are comonoid homomorphisms---or
equivalently, the comonoid operations are monoid homomorphisms.  We summarize this situation by
saying that \(k\) is a \Define{bimonoid}.

So far all our string diagrams denote linear maps.  We can treat these as morphisms in the category
\( \Vectk \), where objects are finite-dimensional vector spaces over a field \(k\) and morphisms
are linear maps.  This category is equivalent to a skeleton where the only objects are vector spaces
\(k^n\) for \(n \ge 0\), and then morphisms can be seen as \(n \times m\) matrices.  This skeleton
is actually a PROP.  The space of signals is a vector space \(V\) over \(k\) which may not be
finite-dimensional, but this does not cause a problem: an \(n \times m\) matrix with entries in
\(k\) still defines a linear map from \(V^n\) to \(V^m\) in a functorial way.

In applications of string diagrams to quantum theory \cite{BS,CP}, we make \(\Vectk\) into a
symmetric monoidal category using the tensor product of vector spaces.  In control theory, we
instead make \(\Vectk\) into a symmetric monoidal category using the \emph{direct sum} of vector
spaces.  In Lemma~\ref{gensvk} we prove that for any field \(k\), \(\Vectk\) with direct sum
is generated as a symmetric monoidal category by the one object \(k\) together with these morphisms:
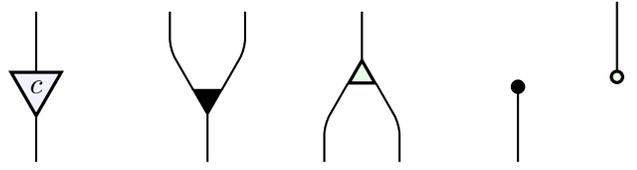
\begin{center}
% List of generating morphims for \(\Vectk\)
\scalebox{1}{
% scalar multiplication
 \begin{tikzpicture}[thick]
   \node[coordinate] (in) at (0,2) {};
   \node [multiply] (mult) at (0,1) {\(c\)};
   \node[coordinate] (out) at (0,0) {};
   \draw (in) -- (mult) -- (out);
   \end{tikzpicture}
\hspace{3 em}
% addition
\begin{tikzpicture}[thick]
   \node[plus] (adder) at (0,0.85) {};
   \node[coordinate] (f) at (-0.5,1.5) {}; 
   \node[coordinate] (g) at (0.5,1.5) {};
   \node[coordinate] (out) at (0,0) {};
   \node [coordinate] (pref) at (-0.5,2) {};
   \node [coordinate] (preg) at (0.5,2) {};
   \draw[rounded corners] (pref) -- (f) -- (adder.left in);
   \draw[rounded corners] (preg) -- (g) -- (adder.right in);
   \draw (adder) -- (out);
   \end{tikzpicture} 
\hspace{2em}
% duplication
  \begin{tikzpicture}[thick]
   \node[delta] (dupe) at (0,1.15) {};
   \node[coordinate] (o1) at (-0.5,0.5) {};
   \node[coordinate] (o2) at (0.5,0.5) {};
   \node[coordinate] (in) at (0,2) {};
   \node [coordinate] (posto1) at (-0.5,0) {};
   \node [coordinate] (posto2) at (0.5,0) {};

   \draw[rounded corners] (posto1) -- (o1) -- (dupe.left out);
   \draw[rounded corners] (posto2) -- (o2) -- (dupe.right out);
   \draw (in) -- (dupe);
\end{tikzpicture}
\hspace{3em}
% zero
   \begin{tikzpicture}[thick]
   \node[hole] (heightHolder) at (0,2) {};
   \node [coordinate] (out) at (0,0) {};
   \node [zero] (del) at (0,1) {};
   \draw (del) -- (out);
   \end{tikzpicture}
\hspace{2em}
% deletion
  \begin{tikzpicture}[thick]
   \node[coordinate] (in) at (0,2) {};
   \node [bang] (mult) at (0,1) {};
   \node [hole] (heightHolder) at (0,0) {};
   \draw (in) -- (mult);
   \end{tikzpicture}
},
\end{center}
where \(c \in k\) is arbitrary.  

However, these generating morphisms obey some unexpected equations!  For example, we have:
  
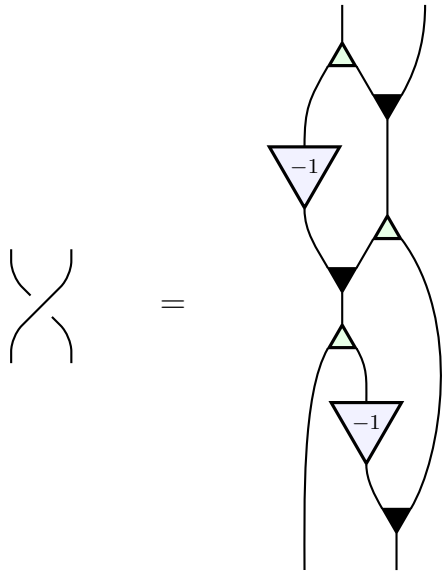
\begin{center}
% Funky transposition
   \begin{tikzpicture}[-, thick, node distance=0.85cm]
   \node (UpUpLeft) at (-0.4,-0.1) {};
   \node [coordinate] (UpLeft) at (-0.4,-0.6) {};
   \node (mid) at (0,-1) {};
   \node [coordinate] (DownRight) at (0.4,-1.4) {};
   \node (DownDownRight) at (0.4,-1.9) {};
   \node [coordinate] (UpRight) at (0.4,-0.6) {};
   \node (UpUpRight) at (0.4,-0.1) {};
   \node [coordinate] (DownLeft) at (-0.4,-1.4) {};
   \node (DownDownLeft) at (-0.4,-1.9) {};

   \draw [rounded corners=2mm] (UpUpLeft) -- (UpLeft) -- (mid) --
   (DownRight) -- (DownDownRight) (UpUpRight) -- (UpRight) -- (DownLeft) -- (DownDownLeft);

\begin{scope}[font=\fontsize{20}{20}\selectfont]
   \node (equals) at (1.75,-1) {\scalebox{0.65}{\(=\)}};
\end{scope}

   \node [coordinate] (sum2L) at (3.5,0) {};
   \node [multiply] (neg) [above of=sum2L] {\(\scriptstyle{-1}\)};
   \node [coordinate] (dupe1L) [above of=neg] {};
   \node [delta] (dupe1) [above right of=dupe1L, shift={(-0.1,0)}] {};
   \node (in1) [above of=dupe1] {};
   \node [plus] (sum1) [below right of=dupe1] {};
   \node [coordinate] (sum1R) [above right of=sum1, shift={(-0.1,0)}] {};
   \node (in2) [above of=sum1R] {};
   \node [delta] (dupe2) [below of=sum1, shift={(0,-0.85)}] {};
   \node [plus] (sum2) [below right of=sum2L, shift={(-0.1,0)}] {};
   \node [coordinate] (dupe2R) [below right of=dupe2, shift={(0.4,-0.6)}] {};
   \node [delta] (dupe3) [below of=sum2] {};
   \node [coordinate] (dupe3L) [below left of=dupe3, shift={(0.1,0)}] {};
   \node [multiply] (neg1) [below right of=dupe3, shift={(-0.28,-0.5)}] {\(\scriptstyle{-1}\)};
   \node [plus] (sum3) [below right of=neg1, shift={(-0.2,-0.65)}] {};
   \node [coordinate] (sum3R) [above right of=sum3, shift={(0,0.2)}] {};
   \node (out2) [below of=sum3] {};
   \node (out1) [below of=dupe3L, shift={(0,-1.75)}] {};

   \draw (in1) -- (dupe1);
   \draw (dupe1.left out) .. controls +(240:0.5) and +(90:0.5) .. (neg.90);
   \draw (dupe1.right out) -- (sum1.left in);
   \draw (in2) .. controls +(270:0.5) and +(60:0.5) .. (sum1.right in);
   \draw (sum1) -- (dupe2);
   \draw (dupe2.right out) .. controls (dupe2R) and (sum3R) .. (sum3.right in);
   \draw (dupe2.left out) -- (sum2.right in);
   \draw (neg.io) .. controls +(270:0.3) and +(120:0.3) .. (sum2.left in);
   \draw (sum2) -- (dupe3);
   \draw (dupe3.left out) .. controls +(240:0.7) and +(90:1) .. (out1);
   \draw (dupe3.right out) .. controls +(300:0.3) and +(90:0.3) .. (neg1.90);
   \draw (neg1.io) .. controls +(270:0.2) and +(120:0.2) .. (sum3.left in);
   \draw (sum3) -- (out2);
   \end{tikzpicture}.
  \end{center}
Thus, it is important to find a complete set of equations obeyed by these generating morphisms, thus
obtaining a presentation of \(\vectk\) as a PROP.  We do this in Theorem~\ref{presvk}.  In brief,
these equations say:
\begin{enumerate}
\item \( (k, +, 0, \Delta, !) \) is a bicommutative bimonoid;
\item the rig operations of \(k\) can be recovered from the generating morphisms;
\item all the generating morphisms commute with scaling.
\end{enumerate}
Here item (2) means that \(+\), \(\cdot\), \(0\) and \(1\) in the field \(k\) can be expressed in 
terms of signal-flow diagrams as follows:
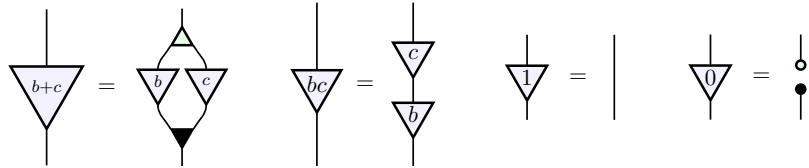
\begin{center}
    \scalebox{0.80}{
% sum of products / product of sum
   \begin{tikzpicture}[-, thick, node distance=0.85cm]
   \node (bctop) {};
   \node [multiply] (bc) [below of=bctop, shift={(0,-0.59)}] {\(\scriptstyle{b+c}\)};
   \node (bcbottom) [below of=bc, shift={(0,-0.59)}] {};

   \draw (bctop) -- (bc) -- (bcbottom);

   \node (eq) [right of=bc, shift={(0.15,0)}] {\(=\)};

   \node [multiply] (b) [right of=eq, shift={(0,0.1)}] {\(\scriptstyle{b}\)};
   \node [delta] (dupe) [above right of=b, shift={(-0.2,0.1)}] {};
   \node (top) [above of=dupe, shift={(0,-0.2)}] {};
   \node [multiply] (c) [below right of=dupe, shift={(-0.2,-0.1)}] {\(\scriptstyle{c}\)};
   \node [plus] (adder) [below right of=b, shift={(-0.2,-0.3)}] {};
   \node (out) [below of=adder, shift={(0,0.2)}] {};

   \draw
   (dupe.left out) .. controls +(240:0.15) and +(90:0.15) .. (b.90)
   (dupe.right out) .. controls +(300:0.15) and +(90:0.15) .. (c.90)
   (top) -- (dupe.io)
   (adder.io) -- (out)
   (adder.left in) .. controls +(120:0.15) and +(270:0.15) .. (b.io)
   (adder.right in) .. controls +(60:0.15) and +(270:0.15) .. (c.io);
   \end{tikzpicture}
        \hspace{0.8cm}
% multiplying twice
   \begin{tikzpicture}[-, thick]
   \node (top) {};
   \node [multiply] (c) [below of=top] {\(c\)};
   \node [multiply] (b) [below of=c] {\(b\)};
   \node (bottom) [below of=b] {};

   \draw (top) -- (c) -- (b) -- (bottom);

   \node (eq) [left of=b, shift={(0.2,0.5)}] {\(=\)};

   \node (bctop) [left of=top, shift={(-0.6,0)}] {};
   \node [multiply] (bc) [left of=eq, shift={(0.2,0)}] {\(bc\)};
   \node (bcbottom) [left of=bottom, shift={(-0.6,0)}] {};

   \draw (bctop) -- (bc) -- (bcbottom);
   \end{tikzpicture}
        \hspace{0.8cm}
% multiplication by 1
\raisebox{2em}{
   \begin{tikzpicture}[-, thick, node distance=0.85cm]
   \node (top) {};
   \node [multiply] (one) [below of=top] {1};
   \node (bottom) [below of=one] {};

   \draw (top) -- (one) -- (bottom);

   \node (eq) [right of=one] {\(=\)};
   \node (topid) [right of=top, shift={(0.6,0)}] {};
   \node (botid) [right of=bottom, shift={(0.6,0)}] {};

   \draw (topid) -- (botid);
   \end{tikzpicture}
}
        \hspace{0.6cm}
% multiplication by 0
\raisebox{2em}{
   \begin{tikzpicture}[-, thick, node distance=0.85cm]
   \node [multiply] (prod) {\(0\)};
   \node (in0) [above of=prod] {};
   \node (out0) [below of=prod] {};
   \node (eq) [right of=prod] {\(=\)};
   \node [bang] (del) [right of=eq, shift={(-0.2,0.2)}] {};
   \node [zero] (ins) [right of=eq, shift={(-0.2,-0.2)}] {};
   \node (in1) [above of=del, shift={(0,-0.2)}] {};
   \node (out1) [below of=ins, shift={(0,0.2)}] {};

   \draw (in0) -- (prod) -- (out0);
   \draw (in1) -- (del);
   \draw (ins) -- (out1);
   \end{tikzpicture}
    }
}.
\end{center}
Multiplicative inverses cannot be so expressed, so our signal-flow diagrams so far do not know that
\(k\) is a field.  Additive inverses also cannot be expressed in this way.  And indeed, a version of
Theorem~\ref{presvk} holds whenever \(k\) is a commutative rig: that is, a commutative `ring without
negatives', such as \(\N\).  The case of a commutative rig \(k\) was examined by Wadsley and Woods
\cite{WW}: see Section \ref{vectrelated} for details.  The idea of finding a presentation for the
category \(\Vectk\) is not new.  Indeed, Lafont \cite{Lafont} gave a presentation of \(\Vectk\) as a
monoidal category, with especial interest in the field of two elements, using generators and
equations similar to the ones given here.

While Theorem~\ref{presvk} is a step towards understanding the category-theoretic underpinnings of
control theory, it does not treat signal-flow diagrams that include `feedback'.  Feedback is one of
the most fundamental concepts in control theory because a control system without feedback may be
highly sensitive to disturbances or unmodeled behavior.  Feedback allows these disturbances to be
mollified (or exacerbated!).  As an annotated string diagram, a basic feedback system might look
like this:
  \begin{center}
    \scalebox{0.80}
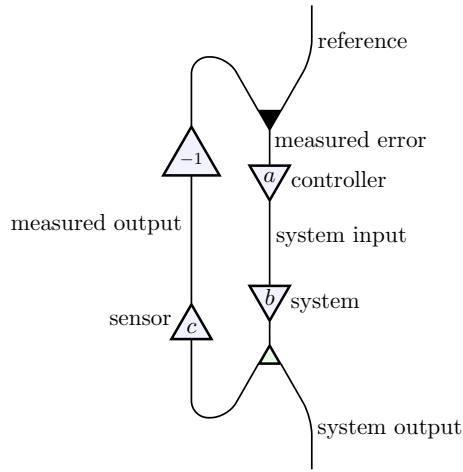
{
%Feedback
   \begin{tikzpicture}[thick]
   \node (in) {};
   \node [coordinate] (inplus) [below of=in, label={[shift={(2.2em,0)}]reference}] {};
   \node [plus] (plus) [below of=inplus, shift={(-0.7,0)}]{};
   \node [multiply] (controller) [below of=plus, label={[shift={(3.0em,-0.5)}]controller},
   label={[shift={(3.5em,0.15)}]measured error}, label={[shift={(3.1em,-1.5)}]system input}] {\(a\)};
   \node [multiply] (system) [below of=controller, shift={(0,-1)},
   label={[shift={(2.4em,-0.6)}]system}] {\(b\)};
   \node [delta] (split) [below of=system, label={[shift={(5.2em,-1.7)}]system output}] {};
   \node [coordinate] (outsplit) [below of=split, shift={(0.7,0)}] {};
   \node (out) [below of=outsplit] {};
   \node [coordinate] (rcup) [below of=split, shift={(-0.7,0)}] {};
   \node [coordinate] (lcup) [left of=rcup, shift={(0.4,0)}] {};
   \node [upmultiply] (sensor) [above of=lcup, shift={(0,0.5)}, label={[shift={(-2.2em,-0.5)}]sensor},
   label={[shift={(-4.1em,1)}]measured output}] {\(c\)};
   \node [upmultiply] (minus) [above of=sensor, shift={(0,1.8)}] {\(\scriptstyle{-1}\)};
   \node [coordinate] (rcap) [above of=plus, shift={(-0.7,0)}] {};
   \node [coordinate] (lcap) [left of=rcap, shift={(0.4,0)}] {};

   \draw[rounded corners=8pt] (in) -- (inplus) -- (plus.right in) (plus) --
   (controller) -- (system) -- (split) (split.right out) --
   (outsplit) -- (out) (split.left out) -- (rcup) -- (lcup) --
   (sensor) -- (minus) -- (lcap) -- (rcap) -- (plus.left in);
   \end{tikzpicture}}.
  \end{center}
The user inputs a `reference' signal, which is fed into a controller, whose output is fed into a
system, or `plant', which in turn produces its own output.  But then the system's output is
duplicated, and one copy is fed into a sensor, whose output is added\footnote{More typically this
output is subtracted in controlled systems, since disturbances are frequently unwanted.} to the
reference signal.

In string diagrams---unlike in the usual thinking on control theory---it is essential to be 
able to read any diagram from top to bottom as a composite of tensor products of generating 
morphisms.  Thus, to incorporate the idea of feedback, we need two more generating morphisms.  
These are the `cup':
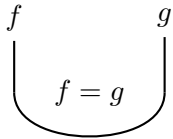
\begin{center}
% cup generating morphism
   \begin{tikzpicture}[thick]
   \node (0) [label={[shift={(0,-1.6)}]}] {\(f=g\)};
   \node[coordinate] (3) [left of=0] {};
   \node[coordinate] (4) [right of=0] {};
   \node (1) [above of=3] {\(f\)};
   \node (2) [above of=4] {\(g\)};

   \path
   (1) edge (3)
   (2) edge (4)
   (3) edge [-, bend right=90] (4);
   \end{tikzpicture}
\end{center}
and `cap':
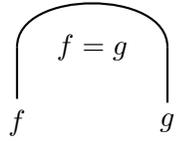
\begin{center}
% cap generating morphism
   \begin{tikzpicture}[thick]
   \node (0) {\(f=g\)};
   \node[coordinate] (3) [left of=0] {};
   \node[coordinate] (4) [right of=0] {};
   \node (1) [below of=3] {\(f\)};
   \node (2) [below of=4] {\(g\)};

   \path
   (3) edge (1)
   (4) edge (2)
   (3) edge [bend left=90] (4);
   \end{tikzpicture}.
\end{center}
These are not maps; they are relations.  The cup imposes the relation that its two inputs be equal,
while the cap does the same for its two outputs.  This is a way of describing how a signal flows
around a bend in a wire.

To make this precise, we use a category called \(\Relk\).  An object of this category is a
finite-dimensional vector space over \(k\), while a morphism from \(U\) to \(V\), denoted \(L \maps
U \asrelto V\), is a \Define{linear relation}, meaning a linear subspace
\[         L \subseteq U \oplus V  .\]
In particular, when \(k = \R(s)\), a linear relation \(L \maps k^m \to k^n\) is just an arbitrary 
system of constant-coefficient linear ordinary differential equations relating \(m\) input 
variables and \(n\) output variables.

Since the direct sum \(U \oplus V\) is also the cartesian product of \(U\) and \(V\), a linear 
relation is indeed a relation in the usual sense, but with the property that if \(u \in U\) is 
related to \(v \in V\) and \(u' \in U\) is related to \(v' \in V\) then \(cu + c'u'\) is related to
\(cv + c'v'\) whenever \(c,c' \in k\).  We compose linear relations \(L \maps U \asrelto V\) and 
\(L' \maps V \asrelto W\) as follows:
\[         L'L = \{(u,w) \colon \; \exists\; v \in V \;\; (u,v) \in L \textrm{ and } 
(v,w) \in L'\} .\]
Any linear map \(f \maps U \to V\) gives a linear relation \(F \maps U \asrelto V\), namely the
graph of that map:
\[                  F = \{ (u,f(u)) : u \in U \}. \]
Composing linear maps thus becomes a special case of composing linear relations, so \(\Vectk\)
becomes a subcategory of \(\Relk\).  Furthermore, we can make \(\Relk\) into a monoidal category
using direct sums, and it becomes symmetric monoidal using the braiding already present in
\(\Vectk\).

In these terms, the \Define{cup} is the linear relation
\[                 \cup \maps k^2 \asrelto \{0\}   \]
given by
\[            \cup \; = \; \{ (x,x,0) : x \in k   \} \; \subseteq \; k^2 \oplus \{0\},   \]
while the \Define{cap} is the linear relation 
\[                 \cap \maps \{0\} \asrelto k^2  \]
given by
\[            \cap \; = \; \{ (0,x,x) : x \in k   \} \; \subseteq \; \{0\} \oplus k^2  .\]
These obey the \Define{zigzag equations}:
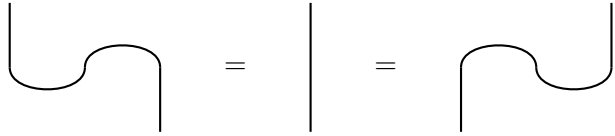
\begin{center}
% Zigzag
   \begin{tikzpicture}[-, thick, node distance=1cm]
% Zig
   \node (zigtop) {};
   \node [coordinate] (zigincup) [below of=zigtop] {};
   \node [coordinate] (zigcupcap) [right of=zigincup] {};
   \node [coordinate] (zigoutcap) [right of=zigcupcap] {};
   \node (zigbot) [below of=zigoutcap] {};
   \node (equal) [right of=zigoutcap] {\(=\)};
% Vertical
   \node (mid) [right of=equal] {};
   \node (vtop) [above of=mid] {};
   \node (vbot) [below of=mid] {};
   \node (equals) [right of=mid] {\(=\)};
% Zag
   \node [coordinate] (zagoutcap) [right of=equals] {};
   \node (zagbot) [below of=zagoutcap] {};
   \node [coordinate] (zagcupcap) [right of=zagoutcap] {};
   \node [coordinate] (zagincup) [right of=zagcupcap] {};
   \node (zagtop) [above of=zagincup] {};
% Zigpath
   \path
   (zigincup) edge (zigtop) edge [bend right=90] (zigcupcap)
   (zigoutcap) edge (zigbot) edge [bend right=90] (zigcupcap)
% Verticalpath
   (vtop) edge (vbot)
% Zagpath
   (zagincup) edge (zagtop) edge [bend left=90] (zagcupcap)
   (zagoutcap) edge (zagbot) edge [bend left=90] (zagcupcap);
   \end{tikzpicture}.
    \end{center}
Thus, they make \(\Relk\) into a compact closed category where \(k\), and thus every object, is its
own dual.  As with \(\Vectk\), we will focus on a skeleton \(\relk\) of \(\Relk\), which is a PROP.

Besides feedback, one of the things that make the cap and cup useful is that they allow any 
morphism \(L \maps U \asrelto V \) to be `plugged in backwards' and thus `turned around'.  For 
instance, turning around integration:
  
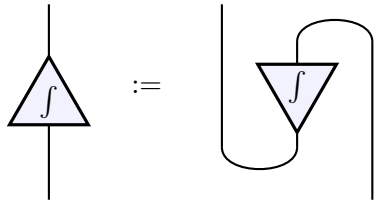
\begin{center}
% Differentiation
   \begin{tikzpicture}[thick]
   \node [integral] (dot) {\(\int\)};
   \node [coordinate] (cupout) [below of=dot, shift={(0,0.2)}] {};
   \node [coordinate] (cupin) [left of=cupout] {};
   \node [coordinate] (capin) [above of=dot, shift={(0,-0.4)}] {};
   \node [coordinate] (in) [left of=capin, shift={(0,0.5)}] {};
   \node [coordinate] (capout) [right of=capin] {};
   \node [coordinate] (out) [right of=cupout, shift={(0,-0.7)}] {};

   \draw (capin) -- (dot) -- (cupout);
   \path
   (in) edge (cupin)
   (capout) edge (out)
   (cupin) edge [bend right=90] (cupout)
   (capin) edge [bend left=90] (capout);

   \node (eq) [left of=dot, shift={(-1,0)}] {\(:=\)};
   \node [upmultiply] (diff) [left of=eq, shift={(-0.3,-0.2)}] {\(\int\)};
   \node [coordinate] (diffin) [above of=diff, shift={(0,0.3)}] {};
   \node [coordinate] (diffout) [below of=diff, shift={(0,-0.3)}] {};

   \draw (diffin) -- (diff) -- (diffout);
   \end{tikzpicture}
  \end{center}
we (essentially) obtain differentiation.  In general, using caps and cups we can turn around any
linear relation \(L \maps U \asrelto V\) and obtain a linear relation \(L^\dagger \maps V \asrelto
U\), called the \Define{adjoint} of \(L\), which turns out to given by
\[            L^\dagger = \{(v,u) : (u,v) \in L \}  .\]
For example, if \(c \in k\) is nonzero, the adjoint of scalar multiplication by \(c\) is
multiplication by \(c^{-1}\):
 
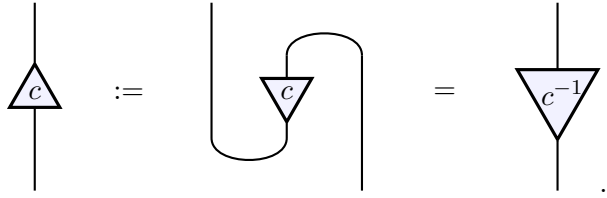
\begin{center}
% Multiplicative inverses in terms of cap and cup
   \begin{tikzpicture}[thick]
   \node [multiply] (c) {\(c\)};
   \node [coordinate] (cupout) [below of=c, shift={(0,0.4)}] {};
   \node [coordinate] (cupin) [left of=cupout] {};
   \node [coordinate] (capin) [above of=c, shift={(0,-0.5)}] {};
   \node [coordinate] (in) [left of=capin, shift={(0,0.7)}] {};
   \node [coordinate] (capout) [right of=capin] {};
   \node [coordinate] (out) [right of=cupout, shift={(0,-0.7)}] {};

   \draw (capin) -- (c) -- (cupout);
   \path
   (in) edge (cupin)
   (capout) edge (out)
   (cupin) edge [bend right=90] (cupout)
   (capin) edge [bend left=90] (capout);

   \node (eq) [right of=c, shift={(1.1,0)}] {\(=\)};

   \node [multiply] (mult) [right of=eq, shift={(0.5,0)}] {\(c^{-1}\!\!\)};
   \node [coordinate] (min) [above of=mult, shift={(0,0.2)}] {};
   \node [coordinate] (mout) [below of=mult, shift={(0,-0.3)}] {};

   \draw (min) -- (mult) -- (mout);

   \node (colon) [left of=c, shift={(-1.1,0)}] {\(:=\)};

   \node [upmultiply] (adj) [left of=colon, shift={(-0.25,0)}] {\(c\)};
   \node [coordinate] (adin) [above of=adj, shift={(0,0.2)}] {};
   \node [coordinate] (adout) [below of=adj, shift={(0,-0.3)}] {};

   \draw (adin) -- (adj) -- (adout);
   \end{tikzpicture}.
  \end{center}
Thus, caps and cups allow us to express multiplicative inverses in terms of signal-flow diagrams!
One might think that a problem arises when when \(c = 0\), but no: the adjoint of scaling by \(0\)
is the linear relation
\[          \{(0,x) : x \in k \} \subseteq k \oplus k .\]

In Lemma~\ref{gensrk} we show that \(\relk\) is generated, as a symmetric monoidal category, by
these morphisms:
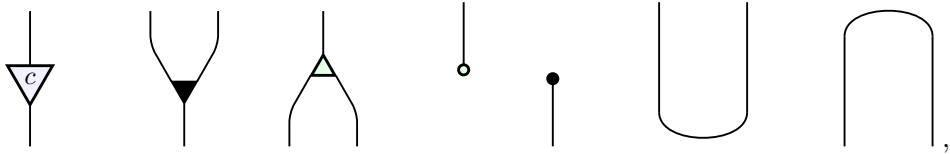
\begin{center}
\scalebox{0.9}{
 \begin{tikzpicture}[thick]
% scalar multiplication
   \node [coordinate] (in) at (0,2) {};
   \node [multiply] (mult) at (0,1) {\(c\)};
   \node [coordinate] (out) at (0,0) {};
   \draw (in) -- (mult) -- (out);
   \end{tikzpicture}
\hspace{3 em}
\begin{tikzpicture}[thick]
% addition
   \node [plus] (adder) at (0,0.85) {};
   \node [coordinate] (f) at (-0.5,1.5) {}; 
   \node [coordinate] (g) at (0.5,1.5) {};
   \node [coordinate] (out) at (0,0) {};
   \node [coordinate] (pref) at (-0.5,2) {};
   \node [coordinate] (preg) at (0.5,2) {};

   \draw [rounded corners] (pref) -- (f) -- (adder.left in);
   \draw [rounded corners] (preg) -- (g) -- (adder.right in);
   \draw (adder) -- (out);
   \end{tikzpicture} 
\hspace{2em}
  \begin{tikzpicture}[thick]
% duplication
   \node[delta] (dupe) at (0,1.15) {};
   \node[coordinate] (o1) at (-0.5,0.5) {};
   \node[coordinate] (o2) at (0.5,0.5) {};
   \node[coordinate] (in) at (0,2) {};
   \node [coordinate] (posto1) at (-0.5,0) {};
   \node [coordinate] (posto2) at (0.5,0) {};

   \draw[rounded corners] (posto1) -- (o1) -- (dupe.left out);
   \draw[rounded corners] (posto2) -- (o2) -- (dupe.right out);
   \draw (in) -- (dupe);
\end{tikzpicture}
\hspace{3em}
\begin{tikzpicture}[thick]
% deletion
   \node[coordinate] (in) at (0,2) {};
   \node [bang] (mult) at (0,1) {};
   \node [hole] (heightHolder) at (0,0) {};
   \draw (in) -- (mult);
   \end{tikzpicture}
\hspace{2em}
  \begin{tikzpicture}[thick]
% zero
   \node[hole] (heightHolder) at (0,2) {};
   \node [coordinate] (out) at (0,0) {};
   \node [zero] (del) at (0,1) {};
   \draw (del) -- (out);
\end{tikzpicture}
\hspace{3em}
 \begin{tikzpicture}[thick]
% cup
   \node [coordinate] (3) at (0,0.375) {};
   \node [coordinate] (4) at (1.3,0.375) {};
   \node [coordinate] (1) at (0,2) {};
   \node [coordinate] (2) at (1.3,2) {};
   \path
   (1) edge (3)
   (2) edge (4)
   (3) edge [-, bend right=90] (4);
   \end{tikzpicture}
        \hspace{3em}
   \begin{tikzpicture}[thick]
% cap 
   \node [coordinate] (3) at (0,1.625) {};
   \node [coordinate] (4) at (1.3,1.625) {};
   \node [coordinate] (1) at (0,0) {};
   \node [coordinate] (2) at (1.3,0) {};
   \path
   (3) edge (1)
   (4) edge (2)
   (3) edge [bend left=90] (4);
   \end{tikzpicture}
},
\end{center}
where \(c \in k\) is arbitrary.  

In Theorem~\ref{presrk} we find a complete set of equations obeyed by these generating morphisms,
thus giving a presentation of \(\relk\) as a PROP.  To describe these equations, it is useful to
work with adjoints of the generating morphisms.  We have already seen that the adjoint of scaling by
\(c\) is scaling by \(c^{-1}\), except when \(c = 0\).  Taking adjoints of the other four generating
morphisms of \(\vectk\), we obtain four important but perhaps unfamiliar linear relations.  We draw
these as `turned around' versions of the original generating morphisms:

\begin{itemize}
\item \Define{Coaddition} is a linear relation from \(k\) to \(k^2\) that holds when the two 
outputs sum to the input:
\[           +^\dagger \maps k \asrelto k^2 \]
\[           +^\dagger = \{(x,y,z)  : \; x = y + z  \} \subseteq k \oplus k^2 \]
  
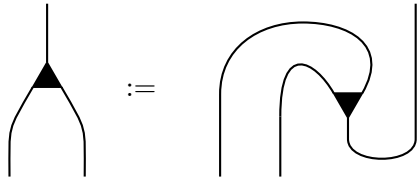
\begin{center}
   \begin{tikzpicture}[thick]
% Coaddition
   \node [plus]       (adder)     at (0,0)       {};
   \node [coordinate] (sum)       at (0,-0.5)    {};
   \node [coordinate] (sumup)     at (0.9,-0.5)  {};
   \node [coordinate] (input)     at (0.9,1.3)   {};
   \node [coordinate] (highcap)   at (-0.6,1.05) {};
   \node [coordinate] (outerloop) at (-1.7,0.1)  {};
   \node [coordinate] (outerout)  at (-1.7,-1)   {};
   \node [coordinate] (innerloop) at (-0.9,-0.2) {};
   \node [coordinate] (innerout)  at (-0.9,-1)   {};

   \draw (innerloop) -- (innerout);
   \path (sum) edge [bend right=90] (sumup);
   \draw (adder.io) -- (sum)
   (sumup) -- (input)
   (outerloop) -- (outerout)
   (adder.left in) .. controls +(120:0.5) and +(0,1) .. (innerloop)
   (adder.right in) .. controls +(60:0.75) and +(0.6,0) .. (highcap)
   (highcap) .. controls +(-0.6,0) and +(0,0.6) .. (outerloop);

   \node              (eq)       at (-2.75,0.15) {\(:=\)};
   \node [coplus]     (coadder)  at (-4,0.3)     {};
   \node [coordinate] (topco)    at (-4,1.3)     {};
   \node [coordinate] (leftout)  at (-4.5,-1)    {};
   \node [coordinate] (rightout) at (-3.5,-1)    {};

   \draw
   (coadder.left out) .. controls +(240:0.7) and +(90:0.7) .. (leftout)
   (coadder.right out) .. controls +(300:0.7) and +(90:0.7) .. (rightout)
   (coadder.io) -- (topco);
   \end{tikzpicture}
  \end{center}
\item \Define{Cozero} is a linear relation from \(k\) to \(\{0\}\) that holds 
when the input is zero:
\[           0^\dagger \maps k \asrelto \{0\}   \]
\[           0^\dagger = \{ (0,0)\} \subseteq k \oplus \{0\}   \]
\begin{center}
\begin{tikzpicture}[thick]
% Cozero
   \node [zero] (Ze) at (0,0)    {};
   \node        (eq) at (-1,0)    {\(:=\)};
   \node [zero] (Ro) at (-1.8,-0.5) {};

   \draw[rounded corners=7pt]
   (Ze) -- (0,-0.5) -- (0.5,-0.5) -- (0.5,0.5);
   \draw (Ro) -- (-1.8,0.5);
\end{tikzpicture}
\end{center}
\item \Define{Coduplication} is a linear relation from \(k^2\) to \(k\) that holds when the two
inputs both equal the output:
\[           \Delta^\dagger \maps k^2 \asrelto k \]
\[           \Delta^\dagger = \{(x,y,z)  : \; x = y = z \} \subseteq k^2 \oplus k \]
  \begin{center}
   \begin{tikzpicture}[thick]
% Coduplication
   \node [delta]      (copier)    at (0,0)        {};
   \node [coordinate] (original)  at (0,0.5)      {};
   \node [coordinate] (origdown)  at (0.9,0.5)    {};
   \node [coordinate] (output)    at (0.9,-1.3)   {};
   \node [coordinate] (lowcup)    at (-0.6,-1.05) {};
   \node [coordinate] (outerloop) at (-1.7,-0.1)  {};
   \node [coordinate] (outerin)   at (-1.7,1)     {};
   \node [coordinate] (innerloop) at (-0.9,0.2)   {};
   \node [coordinate] (innerout)  at (-0.9,1)     {};

   \draw (innerloop) -- (innerout);
   \path (original) edge [bend left=90] (origdown);
   \draw (copier.io) -- (original)
   (origdown) -- (output)
   (outerloop) -- (outerin)
   (copier.left out) .. controls +(240:0.5) and +(0,-1) .. (innerloop)
   (copier.right out) .. controls +(300:0.75) and +(0.6,0) .. (lowcup)
   (lowcup) .. controls +(-0.6,0) and +(0,-0.6) .. (outerloop);

   \node              (eq)       at (-2.75,-0.15) {\(:=\)};
   \node [codelta]    (pier)     at (-4,-0.3)     {};
   \node [coordinate] (bottomco) at (-4,-1.3)     {};
   \node [coordinate] (leftin)  at (-4.5,1)       {};
   \node [coordinate] (rightin) at (-3.5,1)       {};

   \draw
   (pier.left in) .. controls +(120:0.7) and +(270:0.7) .. (leftin)
   (pier.right in) .. controls +(60:0.7) and +(270:0.7) .. (rightin)
   (pier.io) -- (bottomco);
   \end{tikzpicture}
  \end{center}
\item \Define{Codeletion} is a linear relation from \(\{0\}\) to \(k\) that holds always:
\[          !^\dagger \maps \{0\} \asrelto k \]
\[          !^\dagger = \{(0,x) \} \subseteq \{0\} \oplus k \]
\begin{center}
\begin{tikzpicture}[thick]
% codeletion
   \node [bang] (Ba) at (0,0)    {};
   \node        (eq) at (-1,0)    {\(:=\)};
   \node [bang] (ng) at (-1.8,0.5) {};

   \draw[rounded corners=7pt]
   (Ba) -- (0,0.5) -- (0.5,0.5) -- (0.5,-0.5);
   \draw (ng) -- (-1.8,-0.5);
\end{tikzpicture}.
\end{center}
\end{itemize}
Since \(+^\dagger,0^\dagger,\Delta^\dagger\) and \(!^\dagger\) automatically obey turned-around 
versions of the equations obeyed  by \(+,0,\Delta\) and \(!\), we see that \(k\) acquires a 
\emph{second} bicommutative bimonoid structure when considered as an object in \(\Relk\).  

Moreover, the four dark operations make \(k\) into a \Define{Frobenius monoid}.  This means that
\((k,+,0)\) is a monoid, \((k,+^\dagger, 0^\dagger)\) is a comonoid, and the \Define{Frobenius 
relation} holds:
\begin{center}
 \scalebox{1}{
% Sum Frobenius Law
   \begin{tikzpicture}[thick]
   \node [plus] (sum1) at (0.5,-0.216) {};
   \node [coplus] (cosum1) at (1,0.216) {};
   \node [coordinate] (sum1corner) at (0,0.434) {};
   \node [coordinate] (cosum1corner) at (1.5,-0.434) {};
   \node [coordinate] (sum1out) at (0.5,-0.975) {};
   \node [coordinate] (cosum1in) at (1,0.975) {};
   \node [coordinate] (1cornerin) at (0,0.975) {};
   \node [coordinate] (1cornerout) at (1.5,-0.975) {};

   \draw[rounded corners] (1cornerin) -- (sum1corner) -- (sum1.left in)
   (1cornerout) -- (cosum1corner) -- (cosum1.right out);
   \draw (sum1.right in) -- (cosum1.left out)
   (sum1.io) -- (sum1out)
   (cosum1.io) -- (cosum1in);

   \node (eq1) at (2,0) {\(=\)};
   \node [plus] (sum2) at (3,0.325) {};
   \node [coplus] (cosum2) at (3,-0.325) {};
   \node [coordinate] (sum2inleft) at (2.5,0.975) {};
   \node [coordinate] (sum2inright) at (3.5,0.975) {};
   \node [coordinate] (cosum2outleft) at (2.5,-0.975) {};
   \node [coordinate] (cosum2outright) at (3.5,-0.975) {};

   \draw (sum2inleft) .. controls +(270:0.3) and +(120:0.15) .. (sum2.left in)
   (sum2inright) .. controls +(270:0.3) and +(60:0.15) .. (sum2.right in)
   (cosum2outleft) .. controls +(90:0.3) and +(240:0.15) .. (cosum2.left out)
   (cosum2outright) .. controls +(90:0.3) and +(300:0.15) .. (cosum2.right out)
   (sum2.io) -- (cosum2.io);

   \node (eq2) at (4,0) {\(=\)};
   \node [plus] (sum3) at (5.5,-0.216) {};
   \node [coplus] (cosum3) at (5,0.216) {};
   \node [coordinate] (sum3corner) at (6,0.434) {};
   \node [coordinate] (cosum3corner) at (4.5,-0.434) {};
   \node [coordinate] (sum3out) at (5.5,-0.975) {};
   \node [coordinate] (cosum3in) at (5,0.975) {};
   \node [coordinate] (3cornerin) at (6,0.975) {};
   \node [coordinate] (3cornerout) at (4.5,-0.975) {};

   \draw[rounded corners] (3cornerin) -- (sum3corner) -- (sum3.right in)
   (3cornerout) -- (cosum3corner) -- (cosum3.left out);
   \draw (sum3.left in) -- (cosum3.right out)
   (sum3.io) -- (sum3out)
   (cosum3.io) -- (cosum3in);
   \end{tikzpicture}

}.
\end{center}
All three expressions in this equation are linear relations saying that the sum of the two inputs
equal the sum of the two outputs.  

The operation sending each linear relation to its adjoint extends to a contravariant functor 
\[ \dagger \maps \Relk\ \to \Relk ,\]
which obeys a list of properties that are summarized by saying that \(\Relk\) is a
`\(\dagger\)-compact' category \cite{AC,Selinger}.  Because two of the operations in the Frobenius
monoid \((k, +,0,+^\dagger,0^\dagger)\) are adjoints of the other two, it is a
\Define{\(\dagger\)-Frobenius monoid}.  This Frobenius monoid is also \Define{special}, meaning that
comultiplication (in this case \(+^\dagger\)) followed by multiplication (in this case \(+\)) equals
the identity on \(k\):
\begin{center}
% The Frobenius monoid given by addition is special
   \begin{tikzpicture}[thick]
   \node [plus] (sum) at (0.4,-0.5) {};
   \node [coplus] (cosum) at (0.4,0.5) {};
   \node [coordinate] (in) at (0.4,1) {};
   \node [coordinate] (out) at (0.4,-1) {};
   \node (eq) at (1.3,0) {\(=\)};
   \node [coordinate] (top) at (2,1) {};
   \node [coordinate] (bottom) at (2,-1) {};

   \path (sum.left in) edge[bend left=30] (cosum.left out)
   (sum.right in) edge[bend right=30] (cosum.right out);
   \draw (top) -- (bottom)
   (sum.io) -- (out)
   (cosum.io) -- (in);
   \end{tikzpicture}.
  \end{center}
This Frobenius monoid is also commutative---and cocommutative, but for Frobenius monoids this
follows from commutativity.

Starting around 2008, commutative special \(\dagger\)-Frobenius monoids have become important in the
categorical foundations of quantum theory, where they can be understood as `classical structures'
for quantum systems \cite{CPV,Vicary}.  The category \(\Fin\Hilb\) of finite-dimensional Hilbert
spaces and linear maps is a \(\dagger\)-compact category, where any linear map \(f \maps H \to K\)
has an adjoint \(f^\dagger \maps K \to H\) given by
\[         \langle f^\dagger \phi, \psi \rangle = \langle \phi, f \psi \rangle \]
for all \(\psi \in H, \phi \in K \).  A commutative special \(\dagger\)-Frobenius monoid in
\(\Fin\Hilb\) is then the same as a Hilbert space with a chosen orthonormal basis.  The reason is
that given an orthonormal basis \( \psi_i \) for a finite-dimensional Hilbert space \(H\), we can
make \(H\) into a  commutative special \(\dagger\)-Frobenius monoid with multiplication \(m \maps H
\otimes H \to H\) given by
\[     m (\psi_i \otimes \psi_j ) = \left\{ \begin{array}{cl}  \psi_i & i = j \\
                                                                                0 & i \ne j  
\end{array}\right.  \]
and unit \(i \maps \C \to H\) given by
\[   i(1) = \sum_i \psi_i . \]
The comultiplication \(m^\dagger\) duplicates basis states:
\[        m^\dagger(\psi_i) = \psi_i \otimes \psi_i  . \]
Conversely, any commutative special \(\dagger\)-Frobenius monoid in \(\Fin\Hilb\) arises this way.  

Considerably earlier, around 1995, commutative Frobenius monoids were recognized as important in
topological quantum field theory.  The reason, ultimately, is that the free symmetric monoidal
category on a commutative Frobenius monoid is \(2\Cob\), the category with 2-dimensional oriented
cobordisms as morphisms: see Kock's textbook \cite{Kock} and the many references therein.  But the
free symmetric monoidal category on a commutative \emph{special} Frobenius monoid was worked out 
even earlier \cite{CW,Kock2,RSW}: it is the category with finite sets as objects, where a morphism 
\(f \maps X \to Y\) is an isomorphism class of cospans
\[        X \longrightarrow S \longleftarrow Y  .\]
This category can be made into a \(\dagger\)-compact category in an obvious way, and then the
1-element set becomes a commutative special \(\dagger\)-Frobenius monoid.  

For all these reasons, it is interesting to find a commutative special \(\dagger\)-Frobenius monoid
lurking at the heart of control theory!  However, the Frobenius monoid here has yet another
property, which is more unusual.  Namely, the unit \(0 \maps \{0\} \asrelto k\) followed by the
counit \(0^\dagger \maps k \asrelto \{0\} \) is the identity on \(\{0\}\):
\begin{center}
\scalebox{1}{
% The Frobenius monoid given by addition is extra-special
\begin{tikzpicture}[-, thick, node distance=0.7cm]
   \node [zero] (Bins) {};
   \node [zero] (Tins) [above of=Bins] {};
   \path
   (Tins) edge (Bins);
   \end{tikzpicture}
\quad
\raisebox{1em}{=}
}\qquad.
\end{center}
We call a special Frobenius monoid that also obeys this `extra' law \Define{extra-special}.  One can
check that the free symmetric monoidal category on a commutative extra-special Frobenius monoid is
the category with finite sets as objects, where a morphism \(f \maps X \to Y\) is an equivalence
relation on the disjoint union \(X \sqcup Y\), and we compose \(f \maps X \to Y\) and \(g \maps Y
\to Z\) by letting \(f\) and \(g\) generate an equivalence relation on \(X \sqcup Y \sqcup Z\) and
then restricting this to \(X \sqcup Z\).

As if this were not enough, the light operations share many properties with the dark ones.  In
particular, these operations make \(k\) into a commutative extra-special \(\dagger\)-Frobenius
monoid in a second way.  In summary:
\begin{itemize}
\item \((k, +, 0, \Delta, !)\) is a bicommutative bimonoid;
\item \((k, \Delta^\dagger, !^\dagger, +^\dagger, 0^\dagger)\) is a bicommutative bimonoid;
\item \((k, +, 0, +^\dagger, 0^\dagger)\) is a commutative extra-special 
\(\dagger\)-Frobenius monoid;
\item \((k, \Delta^\dagger, !^\dagger, \Delta, !)\) is a commutative extra-special 
\(\dagger\)-Frobenius monoid.
\end{itemize}

It should be no surprise that with all these structures built in, signal-flow diagrams are a
powerful method of designing processes.  However, it is surprising that most of these structures 
are present in a seemingly very different context: the so-called `ZX calculus', a diagrammatic 
formalism for working with complementary observables in quantum theory \cite{CD}.  This arises 
naturally when one has an \(n\)-dimensional Hilbert space \(H\) with two orthonormal bases 
\(\psi_i, \phi_i \) that are `mutually unbiased', meaning that
\[           |\langle \psi_i, \phi_j\rangle|^2 = \displaystyle{\frac{1}{n}}  \]
for all \(1 \le i, j \le n\).  Each orthonormal basis makes \(H\) into commutative special
\(\dagger\)-Frobenius monoid in \(\Fin\Hilb\).  Moreover, the multiplication and unit of either one
of these Frobenius monoids fits together with the comultiplication and counit of the other to form 
a bicommutative bimonoid.  So, we have all the structure present in the list above---except
that these Frobenius monoids are only \emph{extra}-special if \(H\) is 1-dimensional.  

The field \(k\) is also a 1-dimensional vector space, but this is a red herring: in \(\relk\)
\emph{every} finite-dimensional vector space naturally acquires all four structures listed above,
since addition, zero, duplication and deletion are well-defined and obey all the equations we have
discussed.  We focus on \(k\) in this paper simply because it generates all the objects \(\relk\)
via direct sum.

Finally, in \(\relk\) the cap and cup are related to the light and dark operations as follows:
\begin{center}
\scalebox{1}{
% cap in terms of duplication and codeletion
   \begin{tikzpicture}[thick]
   \node (eq) at (0.2,-0.1) {\(=\)};
   \node [coordinate] (lcap) at (-1.5,0.5) {};
   \node [coordinate] (rcap) at (-0.5,0.5) {};
   \node [coordinate] (lcapbot) at (-1.5,-1) {};
   \node [coordinate] (rcapbot) at (-0.5,-1) {};
   \node [delta] (dub) at (1.25,0) {};
   \node [bang] (boom) at (1.25,0.65) {};
   \node [coordinate] (Leftout) at (0.75,-1) {};
   \node [coordinate] (Rightout) at (1.75,-1) {};

   \draw (dub.left out) .. controls +(240:0.5) and +(90:0.5) .. (Leftout)
      (dub.right out) .. controls +(300:0.5) and +(90:0.5) .. (Rightout);
   \draw (boom) -- (dub) (lcapbot) -- (lcap) (rcap) -- (rcapbot);
   \path (lcap) edge[bend left=90] (rcap);
   \end{tikzpicture}
\qquad \qquad
% cup in terms of sum, deletion, and -1
   \begin{tikzpicture}[thick]
   \node [multiply] (neg) at (0,0.1) {\(\scriptstyle{-1}\)};
   \node [coordinate] (cupInLeft) at (0,1) {};
   \node [coordinate] (Lcup) at (0,-0.5) {};
   \node [coordinate] (Rcup) at (1,-0.5) {};
   \node [coordinate] (cupInRight) at (1,1) {};
   \node (eq) at (1.7,0.1) {\(=\)};
   \node [coordinate] (SumLeftIn) at (2.25,1) {};
   \node [coordinate] (SumRightIn) at (3.25,1) {};
   \node [plus] (Sum) at (2.75,0) {};
   \node [zero] (coZero) at (2.75,-0.65) {};

   \draw (SumRightIn) .. controls +(270:0.5) and +(60:0.5) .. (Sum.right in)
      (SumLeftIn) .. controls +(270:0.5) and +(120:0.5) .. (Sum.left in);
   \draw (cupInLeft) -- (neg) -- (Lcup)
      (Rcup) -- (cupInRight)
      (Sum) -- (coZero);
   \path (Lcup) edge[bend right=90] (Rcup);
   \end{tikzpicture}
}.
\end{center}
Note the curious factor of \(-1\) in the second equation, which breaks some of the symmetry we have
seen so far.  This equation says that two elements \(x, y \in k\) sum to zero if and only if \(-x =
y\).  Using the zigzag equations, the two equations above give the antipode
\begin{center}
 \begin{tikzpicture}[thick]
   \node (eq) {\(=\)};
   \node[delta] (Lup) at (-1,0.216) {};
   \node[plus] (Ldn) at (-1.5,-0.216) {};
   \node[coordinate] (Lupo) at (-0.5,-0.434) {};
   \node[coordinate] (Ldni) at (-2,0.434) {};
   \node[bang] (Lupi) at (-1,0.866) {};
   \node[zero] (Ldno) at (-1.5,-0.866) {};
   \node (Lo) at (-0.5,-1.082) {};
   \node (Li) at (-2,1.082) {};
   \node [multiply] (neg1) [right of=eq] {\(\scriptstyle{-1}\)};
   \node[coordinate] (inR) [above of=neg1] {};
   \node[coordinate] (outR) [below of=neg1] {};

   \draw[rounded corners] (Li) -- (Ldni) -- (Ldn.left in) (Lo) -- (Lupo) -- (Lup.right out);
   \draw (Ldn) -- (Ldno) (Lup) -- (Lupi) (Ldn.right in) -- (Lup.left out);
   \path (neg1) edge (inR) edge (outR);
   \end{tikzpicture}.
\end{center}
We thus see that in \(\relk\), both additive and multiplicative inverses can be expressed in terms
of the generating morphisms used in signal-flow diagrams.

The break in symmetry at this point can be explained by yet another second way of doing something.
We have seen one contravariant functor on \(\relk\), \(\dagger\), but there is a second
contravariant functor on \(\relk\), \(*\).  This one extends a contravariant functor on \(\vectk\)
that was already lurking in the background.  The functor 
\[  * \maps \relk \to \relk  \]
extends the notion of \emph{transposition} of linear maps, and these two equations relating the cap
and cup to light and dark operations show how to consistently extend transposition to cap and cup,
and thus to linear relations.  Thus we have
\begin{itemize}
\item \(+^*=\Delta\),
\item \(\Delta^*=+\),
\item \(\zero^*=!\),
\item \(!^*=\zero\),
\item \(\cap^*=\cup\of(1 \oplus s_{-1})\),
\item \(\cup^*=(1 \oplus s_{-1})\of\cap\).
\end{itemize}
Graphically,

% Addition and duplication are transpose
 
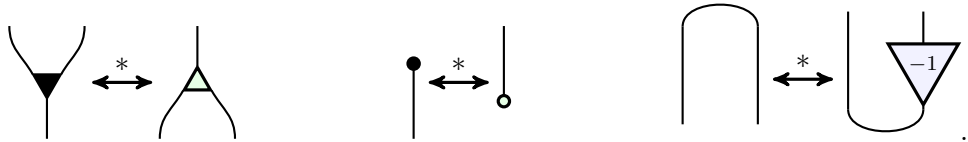
\begin{tikzpicture}[thick,>=stealth']
   \node[plus] (plus) at (-1,0) {};
   \node[delta] (delta) at (1,0) {};

   \draw (plus.io) -- (-1,-0.75) (delta.io) -- (1,0.75)
         (plus.left in) .. controls +(120:0.3) and +(270:0.3) .. (-1.5,0.75)
         (plus.right in) .. controls +(60:0.3) and +(270:0.3) .. (-0.5,0.75)
         (delta.left out) .. controls +(240:0.3) and +(90:0.3) .. (0.5,-0.75)
         (delta.right out) .. controls +(300:0.3) and +(90:0.3) .. (1.5,-0.75);
   \draw[very thick,<->] (-0.4,0) to node [above] {\(*\)} (0.4,0);
 \end{tikzpicture}
\hfill
% Zero and deletion are transpose
 \begin{tikzpicture}[thick,>=stealth']
   \node[zero] (zero) at (-0.6,0.25) {};
   \node[bang] (bang) at (0.6,-0.25) {};

   \draw (zero) -- +(0,-1) (bang) -- +(0,1);
   \draw[very thick,<->] (-0.4,0) to node [above] {\(*\)} (0.4,0);
 \end{tikzpicture}
\hfill
% Cup and cap are transpose when one has a negative 1 attached
 \begin{tikzpicture}[thick,>=stealth']
   \node [coordinate] (lcap) at (-1.6,0.7) {};
   \node [coordinate] (rcap) at (-0.6,0.7) {};
   \node [coordinate] (lcapbot) at (-1.6,-0.6) {};
   \node [coordinate] (rcapbot) at (-0.6,-0.6) {};
   \node [multiply] (neg) at (1.6,0.2) {\(\scriptstyle{-1}\)};
   \node [coordinate] (cupInLeft) at (1.6,0.9) {};
   \node [coordinate] (Lcup) at (1.6,-0.4) {};
   \node [coordinate] (Rcup) at (0.6,-0.4) {};
   \node [coordinate] (cupInRight) at (0.6,0.9) {};

   \draw (cupInLeft) -- (neg) -- (Lcup)
         (Rcup) -- (cupInRight);
   \path (Lcup) edge[bend left=90] (Rcup);
   \draw (lcapbot) -- (lcap) (rcap) -- (rcapbot);
   \path (lcap) edge[bend left=90] (rcap);
   \draw[very thick,<->] (-0.4,0) to node [above] {\(*\)} (0.4,0);
 \end{tikzpicture}.
\qquad \qquad

Theorem~\ref{presrk} gives a presentation of \(\relk\) based on some of the ideas just discussed.
Briefly, it says that \(\relk\) is the PROP generated by these morphisms:
\begin{enumerate}
\item addition \(+\maps k^2 \asrelto k\)
\item zero \(0 \maps \{0\} \asrelto k \)
\item duplication \(\Delta\maps k\asrelto k^2 \)
\item deletion \(! \maps k \asrelto 0\)
\item scaling \(s_c\maps k\asrelto k\) for any \(c\in k\)
\item cup \(\cup \maps k^2 \asrelto \{0\} \)
\item cap \(\cap \maps \{0\} \asrelto k^2 \)
\end{enumerate}
obeying these equations:
\begin{enumerate}
\item \((k, +, 0, \Delta, !)\) is a bicommutative bimonoid;
\item \(\cap\) and \(\cup\) obey the zigzag equations;
\item \((k, +, 0, +^\dagger, 0^\dagger)\) is a commutative extra-special 
\(\dagger\)-Frobenius monoid;
\item \((k, \Delta^\dagger, !^\dagger, \Delta, !)\) is a commutative extra-special 
\(\dagger\)-Frobenius monoid;
\item the field operations of \(k\) can be recovered from the generating morphisms; 
\item the generating morphisms (1)--(4) commute with scaling.
\end{enumerate}
Note that item (2) makes \(\relk\) into a \(\dagger\)-compact category, allowing us to mention the
adjoints of generating morphisms in the subsequent equations.  Item (5) means that \(+,\cdot, 0,
1\) and also additive and multiplicative inverses in the field \(k\) can be expressed in terms of 
signal-flow diagrams in the manner we have explained.
%%%%%
%%%%%
\section{State space}
\label{intro:statespace}
Control theory underwent a paradigm shift in the 1960s with the advent of the state-space approach.
Chapter~\ref{stateful} introduces the basic ideas of this approach and builds up to the PROP
\(\st_k\), which we designed to describe this approach more closely than \(\relk\) can.

The state-space approach to control theory was born around 1960 with Kalman's paper \cite{Kalman60}
that introduced to the world the concepts of controllability and observability.  This approach
addresses some of the limitations of the frequency analysis approach, which had enjoyed significant
early success.  Kalman noticed any linear time-invariant (LTI) control system can be partitioned into
four subsystems\footnote{This partitioning can also be done for nonlinear or time-varying systems,
but the four parts are no longer necessarily control systems in their own right.}, only one of which
is accounted for by the transfer function of the frequency analysis approach.  The other three
subsystems lack inputs, lack outputs, or lack both, thus are best studied by looking at the internal
states of a system.  The continuous time version of the state-space approach uses matrix
differential equations that involve the input and output of a system, mediated by the internal state
of the system.  In a linear time-invariant system, which is the only kind we consider, these equations are 
\begin{equation}\label{stateeq}
\dot{x}(t) = A x(t) + B u(t)
\end{equation}
\begin{equation}\label{outputeq}
y(t) = C x(t) + D u(t),
\end{equation}
where \(u(t)\) is the input vector, \(y(t)\) is the output vector, and \(x(t)\) is the state
vector.  These equations can also be discretized to matrix difference equations for a discrete time
approach.  Unless otherwise stated, we will use the convention that \(\mathrm{dim}(u) = m\), 
\(\mathrm{dim}(x) = n\), and \(\mathrm{dim}(y) = p\).

Equations \ref{stateeq} and \ref{outputeq} can be found lurking in the following signal-flow
diagram:
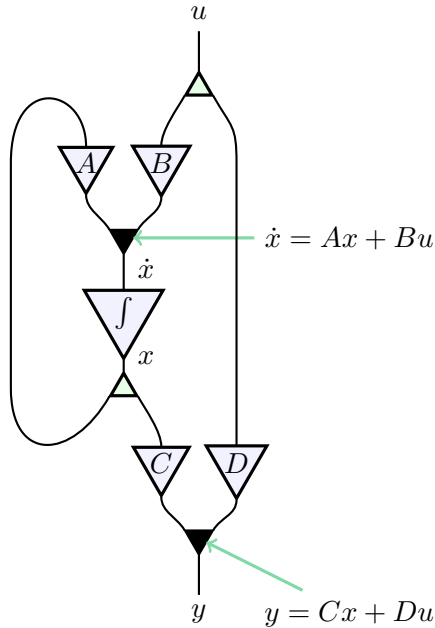
\begin{invisiblelabel}
\label{signalflowstateeqns}
\end{invisiblelabel}
\begin{center}
\scalebox{1}{
 \begin{tikzpicture}[thick]
% main nodes
   \node [delta] (usplit) at (2.5,3) {};
   \node [multiply] (A) at (1,2) {\(A\)};
   \node [multiply] (B) at (2,2) {\(B\)};
   \node [plus] (xdotsum) at (1.5,1) {};
   \node [multiply] (int) at (1.5,0) {\(\int\)};
   \node [delta] (xsplit) at (1.5,-1) {};
   \node [multiply] (C) at (2,-2) {\(C\)};
   \node [multiply] (D) at (3,-2) {\(D\)};
   \node [plus] (ysum) at (2.5,-3) {};

% auxiliary nodes
   \node [coordinate] (capend) [left of=A] {};
   \node [coordinate] (cupend) [below of=capend, shift={(0,-2)}] {};
   \node [coordinate] (ubend) [right of=B] {};

% label nodes
   \node at (1.8,-0.6) {\(x\)};
   \node at (1.8,0.6) {\(\dot{x}\)};
   \node at (2.5,-4) {\(y\)};
   \node at (2.5,4) {\(u\)};

% arrows to point out the state equations
   \node (yeqn) at (4.5,-4) {\(y = Cx + Du\)};
   \node (xeqn) at (4.5,1) {\(\dot{x} = Ax + Bu\)};
   \draw[very thick,green!70!blue!50,->] (xeqn) -- (xdotsum);
   \draw[very thick,green!70!blue!50,->] (yeqn) -- (ysum);

% wires
   \draw (usplit) -- +(0,0.75)
         (ysum) -- +(0,-0.75)
         (xsplit.left out) .. controls +(240:1) and +(270:1) .. (cupend)
         (A.90) .. controls +(90:0.8) and +(90:1.2) .. (capend)
         (capend) -- (cupend)
         (xsplit.right out) .. controls +(300:0.2) and +(90:0.2) .. (C.90)
         (C.270) .. controls +(270:0.2) and +(120:0.2) .. (ysum.left in)
         (D.270) .. controls +(270:0.2) and +(60:0.2) .. (ysum.right in)
         (D.90) -- (ubend)
         (usplit.right out) .. controls +(300:0.5) and +(90:0.5) .. (ubend)
         (usplit.left out) .. controls +(240:0.2) and +(90:0.2) .. (B.90)
         (A.270) .. controls +(270:0.2) and +(120:0.2) .. (xdotsum.left in)
         (B.270) .. controls +(270:0.2) and +(60:0.2) .. (xdotsum.right in)
         (xdotsum) -- (int) -- (xsplit)
;
 \end{tikzpicture}
}
\end{center}
where we have used the shorthand of drawing a single generating morphism where there are zero or
more parallel generating morphisms of the same kind and scaling representing matrix multiplication.
Note that taking integration to be scaling by \(\frac{1}{s}\), as when taking Laplace transforms,
the linear relation this signal-flow diagram depicts is the linear map \(D+C(sI-A)^{-1}B\).

A system is \Define{controllable} if for each state \(x\) and time \(t_0\) there is an input
function \(u(t)\) such that the state can be set to the equilibrium state, \emph{i.e.}\ the zero
vector, in a finite amount of time.  For the linear time-invariant systems we are interested in,
there is a simple characterization of controllability involving the row rank of the block matrix
\(M_c = [B, AB, \dotsc, A^{n-1}B]\).  This controllability matrix \(M_c\) is an \(n \times mn\)
matrix, and a system is controllable when its row rank is \(n\):
\[\mathrm{rank}(M_c) = n.\]

A system is \Define{observable} if for each state \(x\) and time \(t_0\), and with the input
function \(u(t)\) identically zero, measurements of the output function \(y(t)\) over a finite
duration can be used to determine the state \(x(t_0)\).  For the systems we are concerned with,
there is a characterization of observability in terms of the column rank of the block matrix \(M_o =
[C, CA, \dotsc, CA^{n-1}]^\top\).  This observability matrix \(M_o\) is an \(np \times n\) matrix,
and a linear time-invariant system is observable when its column rank is \(n\):
\[\mathrm{rank}(M_o) = n.\]

There are clear parallels in these descriptions of controllability and observability, but there is a
seeming fly in the ointment with observability depending on the input signal being zero and
controllability being independent of the output signal.  Despite this oddity, it is not difficult
to guess there might be some kind of duality relating controllability and observability.  Indeed,
Kalman defined observability in \cite{Kalman60} as a dual notion to controllability, and only
defined it as a separate concept later.  The action of Kalman's duality reverses the direction of
time, swaps the roles of the matrices \(B\) and \(C\), and transposes all the matrices \(A\), \(B\),
\(C\), and \(D\).  Even in the time-varying case, this process transforms a controllable system into
an observable system, and an observable system into a controllable system.

It is curious to see what happens when Kalman's duality is applied to the signal-flow diagram
{\hyperref[signalflowstateeqns]{above}} that encodes the state-space equations.
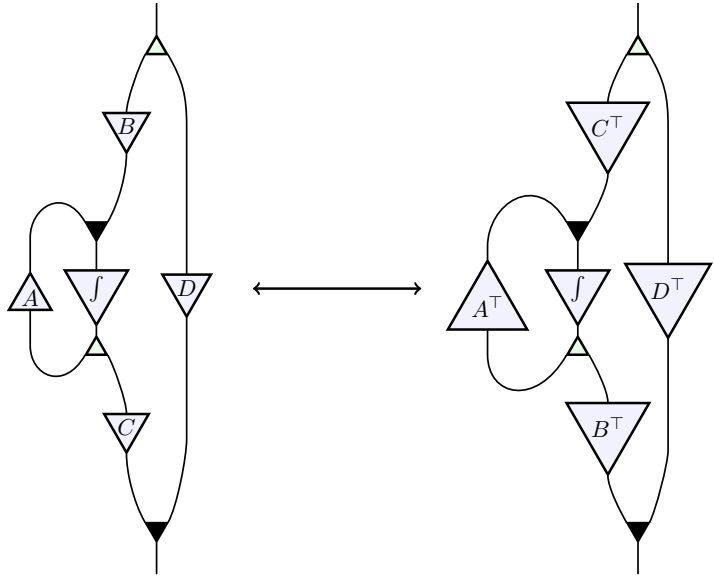
\begin{center}
% Kalman's duality on a signal-flow diagram
\scalebox{0.80}{
 \begin{tikzpicture}[thick]
% main nodes
   \node [delta] (usplit) at (-0.5,4) {};
   \node [upmultiply] (A) at (-2.6,-0.15) {\(A\)};
   \node [multiply] (B) at (-1,2.7) {\(B\)};
   \node [plus] (xdotsum) at (-1.5,1) {};
   \node [multiply] (int) at (-1.5,0) {\(\int\)};
   \node [delta] (xsplit) at (-1.5,-1) {};
   \node [multiply] (C) at (-1,-2.3) {\(C\)};
   \node [multiply] (D) at (0,0) {\(D\)};
   \node [plus] (ysum) at (-0.5,-4) {};

% auxiliary nodes
   \node [coordinate] (capend) [above of=A] {};
   \node [coordinate] (cupend) [below of=A, shift={(0,0.2)}] {};
   \node [coordinate] (ubend) [right of=B] {};
   \node [coordinate] (ybend) [right of=C, shift={(0,-0.2)}] {};

% wires
   \draw (usplit) -- +(0,0.75)
         (ysum) -- +(0,-0.75)
         (xsplit.left out) .. controls +(240:0.7) and +(270:0.5) .. (cupend)
         (xdotsum.left in) .. controls +(120:0.7) and +(90:0.5) .. (capend)
         (capend) -- (A) -- (cupend)
         (xsplit.right out) .. controls +(300:0.2) and +(90:0.2) .. (C.90)
         (C.270) .. controls +(270:0.3) and +(120:0.4) .. (ysum.left in)
         (ybend) .. controls +(270:0.3) and +(60:0.3) .. (ysum.right in)
         (ybend) -- (D) -- (ubend)
         (usplit.right out) .. controls +(300:0.5) and +(90:0.5) .. (ubend)
         (usplit.left out) .. controls +(240:0.2) and +(90:0.2) .. (B.90)
         (B.270) .. controls +(270:0.3) and +(60:0.4) .. (xdotsum.right in)
         (xdotsum) -- (int) -- (xsplit)
   ;

% Dual main nodes
   \node [delta] (usplit) at (7.5,4) {};
   \node [upmultiply] (A) at (5,-0.3) {\(A^{\top}\)};
   \node [multiply] (B) at (7,2.7) {\(C^{\top}\)};
   \node [plus] (xdotsum) at (6.5,1) {};
   \node [multiply] (int) at (6.5,0) {\(\int\)};
   \node [delta] (xsplit) at (6.5,-1) {};
   \node [multiply] (C) at (7,-2.3) {\(B^{\top}\)};
   \node [multiply] (D) at (8,0) {\(D^{\top}\)};
   \node [plus] (ysum) at (7.5,-4) {};

% Dual auxiliary nodes
   \node [coordinate] (capend) [above of=A] {};
   \node [coordinate] (cupend) [below of=A, shift={(0,0.2)}] {};
   \node [coordinate] (ubend) [right of=B] {};
   \node [coordinate] (ybend) [right of=C, shift={(0,-0.4)}] {};

% Dual wires
   \draw (usplit) -- +(0,0.75)
         (ysum) -- +(0,-0.75)
         (xsplit.left out) .. controls +(240:1) and +(270:0.7) .. (cupend)
         (xdotsum.left in) .. controls +(120:1) and +(90:0.7) .. (capend)
         (capend) -- (A) -- (cupend)
         (xsplit.right out) .. controls +(300:0.2) and +(90:0.2) .. (C.90)
         (C.270) .. controls +(270:0.2) and +(120:0.2) .. (ysum.left in)
         (ybend) .. controls +(270:0.2) and +(60:0.2) .. (ysum.right in)
         (ybend) -- (D) -- (ubend)
         (usplit.right out) .. controls +(300:0.5) and +(90:0.5) .. (ubend)
         (usplit.left out) .. controls +(240:0.2) and +(90:0.2) .. (B.90)
         (B.270) .. controls +(270:0.2) and +(60:0.2) .. (xdotsum.right in)
         (xdotsum) -- (int) -- (xsplit)
   ;

   \draw[very thick,<->] (1.1,0) -- (3.9,0);
 \end{tikzpicture}
}
\end{center}
Recalling that the transposition duality \(* \maps \relk \to \relk\) vertically flips signal-flow
diagrams and reverses the colors of the generators, Kalman's duality bears remarkable resemblance to
the transposition duality.  The similarity to the transposition duality can even be used to explain
the oddity of controllability ignoring (deleting) the output signal and observability setting the
input signal to zero:  \(!^* = \zero\).

While it is clear something connects Kalman's work on controllability and observability to the PROP
\(\relk\), taking the signal-flow diagrams above to be linear relations hides the evidence of the
connection: it is impossible to reconstruct \(A\), \(B\), \(C\), and \(D\) from a given linear
relation.  To deal with this shortcoming, we form a new PROP, \(\boxv\), as a stepping stone towards
finding the PROP \(\st_k\).  The objects of \(\boxv\) are the vector spaces \(k^n\) just as with
\(\vectk\), but the morphisms from \(V_1\) to \(V_2\) are now 4-tuples of linear maps, which can be
conveniently organized as non-commutative squares:
\begin{center}
    \begin{tikzpicture}[->]
     \node (A) at (0,0) {\(V_1\)};
     \node (S) at (0,1.5) {\(S\)};
     \node (T) at (1.5,1.5) {\(T\)};
     \node (B) at (1.5,0) {\(V_2\)};

     \path
      (A) edge node[below] {\(d\)} (B)
          edge node[left] {\(b\)} (S)
      (S) edge node[above] {\(a\)} (T)
      (T) edge node[right] {\(c\)} (B)
;
    \end{tikzpicture}.
\end{center}
For compactness of notation, this square can also be written \((d,c,a,b)\).

In Theorem~\ref{boxvectfunctors} we show there is an evaluation functor \(\eval \maps \boxv \to
\vectk\) that takes \((d,c,a,b)\) to \(d+cab\).  Even better, \(\eval\) is a \(\prop\) morphism.  As
noted above, the signal-flow diagram that encodes the state-space equations (Equations~\ref{stateeq}
and \ref{outputeq}) gives a linear map, \(D+C(sI-A)^{-1}B\).  The maps \(D\), \(C\), \(A\), and \(B\)
are all morphisms in \(\vectk\) in the linear time-invariant case, so this looks very similar to the
evaluation of a \(\boxv\) morphism.

To get them to match, we define \(\st_k\) as a subPROP of \(\boxvs\), where \(d=D\), \(c=C\),
\(a=(sI-A)^{-1}\), and \(b=B\) for some linear maps \(A\), \(B\), \(C\), and \(D\).  In
Proposition~\ref{st_k:def} we show \(\st_k\) is a PROP.  Given a stateful morphism \((d,c,a,b)\), it
is possible to find the linear maps \(A\), \(B\), \(C\), and \(D\) used in the state-space
equations.  Because \(\eval(d,c,a,b) = D+C(sI-A)^{-1}B\) for stateful morphisms, it is reasonable to
allow the signal-flow diagram
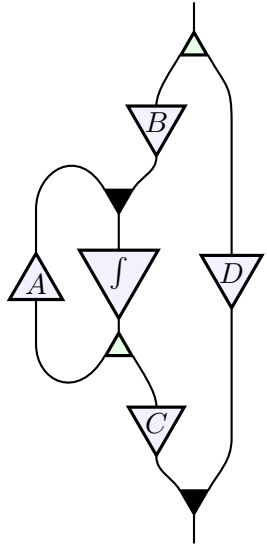
\begin{center}
% SFD for stateful morphisms
 \begin{tikzpicture}[thick]
% main nodes
   \node [delta] (usplit) at (-0.5,3) {};
   \node [upmultiply] (A) at (-2.6,-0.15) {\(A\)};
   \node [multiply] (B) at (-1,2) {\(B\)};
   \node [plus] (xdotsum) at (-1.5,1) {};
   \node [multiply] (int) at (-1.5,0) {\(\int\)};
   \node [delta] (xsplit) at (-1.5,-1) {};
   \node [multiply] (C) at (-1,-2) {\(C\)};
   \node [multiply] (D) at (0,0) {\(D\)};
   \node [plus] (ysum) at (-0.5,-3) {};

% auxiliary nodes
   \node [coordinate] (capend) [above of=A] {};
   \node [coordinate] (cupend) [below of=A, shift={(0,0.2)}] {};
   \node [coordinate] (ubend) [right of=B] {};
   \node [coordinate] (ybend) [right of=C, shift={(0,-0.2)}] {};

% wires
   \draw (usplit) -- +(0,0.6)
         (ysum) -- +(0,-0.6)
         (xsplit.left out) .. controls +(240:0.7) and +(270:0.5) .. (cupend)
         (xdotsum.left in) .. controls +(120:0.7) and +(90:0.5) .. (capend)
         (capend) -- (A) -- (cupend)
         (xsplit.right out) .. controls +(300:0.2) and +(90:0.2) .. (C.90)
         (C.270) .. controls +(270:0.2) and +(120:0.2) .. (ysum.left in)
         (ybend) .. controls +(270:0.3) and +(60:0.3) .. (ysum.right in)
         (ybend) -- (D) -- (ubend)
         (usplit.right out) .. controls +(300:0.5) and +(90:0.5) .. (ubend)
         (usplit.left out) .. controls +(240:0.2) and +(90:0.2) .. (B.90)
         (B.270) .. controls +(270:0.2) and +(60:0.2) .. (xdotsum.right in)
         (xdotsum) -- (int) -- (xsplit)
   ;
 \end{tikzpicture}
\end{center}
to depict a stateful morphism, not just a linear relation.  Furthermore, because it is possible to
find the linear maps \(A\), \(B\), \(C\), and \(D\) used in the state-space equations,
controllability and observability are well-defined for stateful morphisms.  This gives a sense in
which \(\st_k\) is a more detailed picture of a signal processing apparatus which captures not only
the linear relation between inputs and outputs, but how the apparatus implements this relation.

In category theoretic terms, a linear map having full row rank means it is an epimorphism, and
having full column rank means it is a monomorphism.  We can therefore translate the linear
time-invariant conditions for controllability and observability into signal-flow diagram form as follows:

A stateful morphism \((D,C,(sI-A)^{-1},B)\) is controllable when
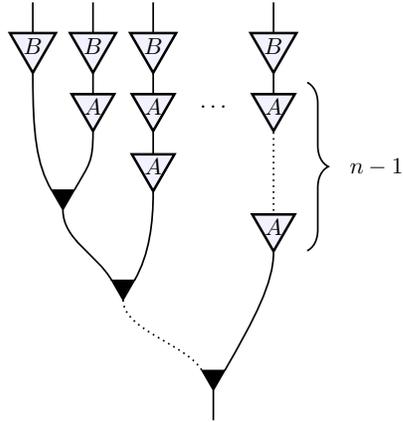
\begin{center}
\scalebox{0.8}{ \begin{tikzpicture}[thick]
   \node [multiply] (B1) {\(B\)};
   \node [multiply, right of=B1] (B2) {\(B\)};
   \node [multiply, right of=B2] (B3) {\(B\)};
   \node [multiply, below of=B2] (A21) {\(A\)};
   \node [multiply, below of=B3] (A31) {\(A\)};
   \node [multiply, below of=A31] (A32) {\(A\)};
   \node [right of=A31] (dots) {\(\dots\)};
   \node [multiply, right of=dots] (An1) {\(A\)};
   \node [multiply, above of=An1] (Bn) {\(B\)};
   \node [multiply, below of=An1, shift={(0,-1)}] (An) {\(A\)};

   \node [plus, below of=A21, shift={(-0.5,-0.5)}] (12sum) {};
   \node [plus, below of=A32, shift={(-0.5,-1)}] (123sum) {};
   \node [plus, below of=A32, shift={(1,-2.5)}] (12nsum) {};

   \draw (B1.90) -- +(90:0.5);
   \draw (B2.90) -- +(90:0.5);
   \draw (B3.90) -- +(90:0.5);
   \draw (Bn.90) -- +(90:0.5);
   \draw (12nsum.io) -- +(270:0.5);
   \draw (B2) -- (A21);
   \draw (B3) -- (A31) -- (A32);
   \draw (Bn) -- (An1);
   \draw[dotted] (An1) -- (An)
         (123sum.io) .. controls +(270:0.5) and +(120:0.5) .. (12nsum.left in);
   \draw (A21.270) .. controls +(270:0.5) and +(60:0.5) .. (12sum.right in)
         (A32.270) .. controls +(270:0.5) and +(60:0.5) .. (123sum.right in)
         (An.270) .. controls +(270:0.5) and +(60:1) .. (12nsum.right in)
         (B1.270) .. controls +(270:1) and +(120:0.5) .. (12sum.left in)
         (12sum.io) .. controls +(270:0.5) and +(120:0.5) .. (123sum.left in);

\draw [decorate,decoration={brace,amplitude=10pt},xshift=-4pt,yshift=0pt]
(4.7,-0.6) -- (4.7,-3.4)node [midway,xshift=3em] {\(n-1\)};

\end{tikzpicture} }
\end{center}
is an epimorphism in \(\vectk\), and it is observable when
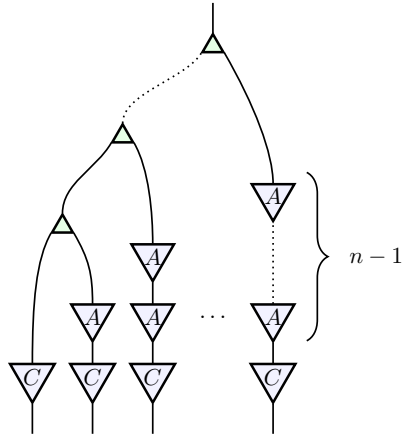
\begin{center}
\scalebox{0.8}{ \begin{tikzpicture}[thick]
   \node [multiply] (B1) {\(C\)};
   \node [multiply, right of=B1] (B2) {\(C\)};
   \node [multiply, right of=B2] (B3) {\(C\)};
   \node [multiply, above of=B2] (A21) {\(A\)};
   \node [multiply, above of=B3] (A31) {\(A\)};
   \node [multiply, above of=A31] (A32) {\(A\)};
   \node [right of=A31] (dots) {\(\dots\)};
   \node [multiply, right of=dots] (An1) {\(A\)};
   \node [multiply, below of=An1] (Bn) {\(C\)};
   \node [multiply, above of=An1, shift={(0,1)}] (An) {\(A\)};

   \node [delta, above of=A21, shift={(-0.5,0.5)}] (12sum) {};
   \node [delta, above of=A32, shift={(-0.5,1)}] (123sum) {};
   \node [delta, above of=A32, shift={(1,2.5)}] (12nsum) {};

   \draw (B1.270) -- +(270:0.5);
   \draw (B2.270) -- +(270:0.5);
   \draw (B3.270) -- +(270:0.5);
   \draw (Bn.270) -- +(270:0.5);
   \draw (12nsum.io) -- +(90:0.5);
   \draw (B2) -- (A21);
   \draw (B3) -- (A31) -- (A32);
   \draw (Bn) -- (An1);
   \draw[dotted] (An1) -- (An)
         (123sum.io) .. controls +(90:0.5) and +(240:0.5) .. (12nsum.left out);
   \draw (A21.90) .. controls +(90:0.5) and +(300:0.5) .. (12sum.right out)
         (A32.90) .. controls +(90:0.5) and +(300:0.5) .. (123sum.right out)
         (An.90) .. controls +(90:0.5) and +(300:1) .. (12nsum.right out)
         (B1.90) .. controls +(90:1) and +(240:0.5) .. (12sum.left out)
         (12sum.io) .. controls +(90:0.5) and +(240:0.5) .. (123sum.left out);

\draw [decorate,decoration={brace,amplitude=10pt},xshift=-4pt,yshift=0pt]
(4.7,3.4) -- (4.7,0.6)node [midway,xshift=3em] {\(n-1\)};

\end{tikzpicture} }
\end{center}
is a monomorphism in \(\vectk\).

Much of what has been discussed to this point has parallels in other contemporary work.  Bonchi,
Soboci\'nski and Zanasi \cite{BSZ1,BSZ2} built up a similar generators and equations picture of
\(\SV_k\), a PROP which is identical to our \(\relk\), using Lack's idea \cite{Lack} of composing
PROPs.  Soboci\'nski also continued by considering controllability, but again from a different
perspective:  About 30 years after Kalman gave his definitions of controllability and observability,
Willems \cite{Wi} proposed alternative definitions for controllability and observability that are
based on the behavior of a system.  However, the duality between controllability and observability
is less apparent in Willems' definition than in Kalman's definition.  Nevertheless, Willems'
behavioral approach is very fruitful, and Fong, Rapisarda and Soboci\'nski \cite{FRS} use this
alternative definition to give a categorical characterization of behavioral controllability.
%%%%%
%%%%%
\section{Controllability and observability in signal-flow diagrams}
\label{intro:goodflow}
Signal-flow diagrams can do much more than depict linear relations.  In Chapter~\ref{goodflow} our
goal is to define a PROP where the morphisms are the signal-flow diagrams used by control theorists,
for which the all-important notions of controllability and observability, which we saw in the
previous section, can be defined.  We begin by defining a preliminary free PROP \(\sigflow_k\),
where morphisms are all diagrams that can be built up by these generators:

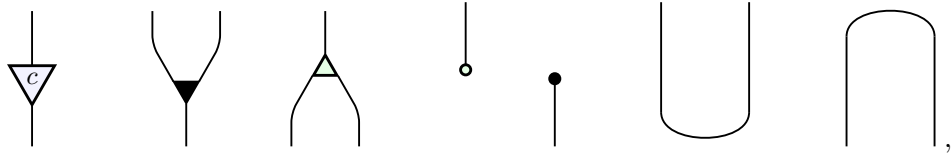
\begin{center}
\scalebox{0.9}{
 \begin{tikzpicture}[thick]
% scalar multiplication
   \node [coordinate] (in) at (0,2) {};
   \node [multiply] (mult) at (0,1) {\(c\)};
   \node [coordinate] (out) at (0,0) {};
   \draw (in) -- (mult) -- (out);
   \end{tikzpicture}
\hspace{3 em}
\begin{tikzpicture}[thick]
% addition
   \node [plus] (adder) at (0,0.85) {};
   \node [coordinate] (f) at (-0.5,1.5) {}; 
   \node [coordinate] (g) at (0.5,1.5) {};
   \node [coordinate] (out) at (0,0) {};
   \node [coordinate] (pref) at (-0.5,2) {};
   \node [coordinate] (preg) at (0.5,2) {};

   \draw [rounded corners] (pref) -- (f) -- (adder.left in);
   \draw [rounded corners] (preg) -- (g) -- (adder.right in);
   \draw (adder) -- (out);
   \end{tikzpicture} 
\hspace{2em}
  \begin{tikzpicture}[thick]
% duplication
   \node[delta] (dupe) at (0,1.15) {};
   \node[coordinate] (o1) at (-0.5,0.5) {};
   \node[coordinate] (o2) at (0.5,0.5) {};
   \node[coordinate] (in) at (0,2) {};
   \node [coordinate] (posto1) at (-0.5,0) {};
   \node [coordinate] (posto2) at (0.5,0) {};

   \draw[rounded corners] (posto1) -- (o1) -- (dupe.left out);
   \draw[rounded corners] (posto2) -- (o2) -- (dupe.right out);
   \draw (in) -- (dupe);
\end{tikzpicture}
\hspace{3em}
\begin{tikzpicture}[thick]
% deletion
   \node[coordinate] (in) at (0,2) {};
   \node [bang] (mult) at (0,1) {};
   \node [hole] (heightHolder) at (0,0) {};
   \draw (in) -- (mult);
   \end{tikzpicture}
\hspace{2em}
  \begin{tikzpicture}[thick]
% zero
   \node[hole] (heightHolder) at (0,2) {};
   \node [coordinate] (out) at (0,0) {};
   \node [zero] (del) at (0,1) {};
   \draw (del) -- (out);
\end{tikzpicture}
\hspace{3em}
 \begin{tikzpicture}[thick]
% cup
   \node [coordinate] (3) at (0,0.375) {};
   \node [coordinate] (4) at (1.3,0.375) {};
   \node [coordinate] (1) at (0,2) {};
   \node [coordinate] (2) at (1.3,2) {};
   \path
   (1) edge (3)
   (2) edge (4)
   (3) edge [-, bend right=90] (4);
   \end{tikzpicture}
        \hspace{3em}
   \begin{tikzpicture}[thick]
% cap 
   \node [coordinate] (3) at (0,1.625) {};
   \node [coordinate] (4) at (1.3,1.625) {};
   \node [coordinate] (1) at (0,0) {};
   \node [coordinate] (2) at (1.3,0) {};
   \path
   (3) edge (1)
   (4) edge (2)
   (3) edge [bend left=90] (4);
   \end{tikzpicture}
},
\end{center}%
where \(c \in k\) is arbitrary.  Appending one more generator for integration
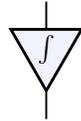
\begin{center}
\scalebox{0.9}{
\begin{tikzpicture}[thick]
 \node [integral] (int) {\(\int\)};
 \draw (int.90) -- +(0,0.4)
       (int.270) -- +(0,-0.4);
\end{tikzpicture}}
\end{center}
extends \(\sigflow_k\) to a larger free PROP, \(\sigflow_{k,s}\).  All together, these are the
generators of \(\relks\) with the element \(s\) treated separately, since integrators play a special
role in in control theory.  \(\sigflow_{k,s}\) is simply the free prop on these generators.  There
is thus a morphism of props
\[        \bbox \maps \sigflow_{k,s} \to \relks\]
sending each signal flow diagram to the linear relation between inputs and outputs that it
determines.  We call this the `black-boxing' functor.

However, many morphisms in \(\sigflow_{k,s}\) are not signal-flow diagrams of the sort used in
control theory; for example, one never sees the `cup' or `cap' above all by itself in a textbook
on control theory.   The challenge, then, is to pick out a subPROP \(\goodflow_k\) which consist of
`reasonable' signal-flow diagrams, for which controllability and observability can be defined.

We already have a category \(\st_k\) for which controllability and observability of morphisms can be
defined, and in Section~\ref{stsection} we constructed a functor \(\eval \maps \st_k \to \vectks\).
Composing with the inclusion \(i \maps \vectks \to \relks\) gives us a \(\prop\) morphism
\[        i \of \eval \maps \st_k \to \relks.\]
We would thus like \(\goodflow_k\) to be a PROP equipped with an inclusion \(j \maps \goodflow_k \to
\sigflow_{k,s}\) making this square commute:
\begin{center}
    \begin{tikzpicture}
   \node (SF) {\(\sigflow_{k,s}\)};
   \node (P) [above of=SF, shift={(0,1)}] {\(\goodflow_k\vphantom{\vectks}\)};
   \node (FR) [right of=SF, shift={(2,0)}] {\(\relks\)};
   \node (ST) [above of=FR, shift={(0,1)}] {\(\st_k\vphantom{\vectks}\)};

   \draw [->] (P) to node [above] {\(\dbox\)} (ST);
   \draw [right hook->] (P) to node [left] {\(j\)} (SF);
   \draw [->] (ST) to node [right] {\(i \of \eval\)} (FR);
   \draw [->] (SF) to node [below] {\(\bbox\)} (FR);
\end{tikzpicture}.
\end{center}
In fact, this desire will lead us directly to the definition of the PROP \(\goodflow_k\) in
Definition~\ref{gooddef}.  We conclude by showing some of the duality properties of \(\goodflow_k\)
and how they are related to Kalman's duality between controllability and observability.
%%%%%
%%%%%
\section{The `Box' construction}
\label{intro:appendix}
In Appendix~\ref{derivedeqns} we offer diagrammatic proofs of some derived equations used in the
proof of Theorem~\ref{presrk}.  Some other diagrammatic proofs with miscellaneous connections are
also included to indicate a portion of the richness of the connection between Frobenius bimonoids
and bicommutative bimonoids.  In Appendix~\ref{generalbox} we expand on the `Box' construction that
led us to \(\st_k\) in Chapter~\ref{stateful}.

When we first examine the Box construction in Chapter~\ref{stateful}, we only apply it to the PROP
\(\vectk\).  The idea behind the Box construction of breaking up a morphism into the direct and
indirect influences of the input on the output generalizes to a broader class of categories.  It is
straightforward to extend the Box construction to apply to the category \(\Vectk\), or any category
that has biproducts.  What is exciting for the purposes of future work is that the Box construction
can also be extended to apply to \(\relk\) and \(\Relk\).  The key property of \(\Relk\) that makes
it work is that \(\Vectk\) is an \Define{essentially wide} subcategory of \(\Relk\).  That is,
\(\Vectk\) `essentially' contains all the objects of \(\Relk\).  More precisely, the inclusion
functor \(i \maps \Vectk \to \Relk\) is essentially surjective.

Since \(\Vectk\) has biproducts, every object in \(\Vectk\) is a bicommutative bimonoid and every
morphism is a bimonoid homomorphism.  Thus every object in \(\Relk\) is a bicommutative bimonoid as
well.  In the Box construction in Chapter~\ref{stateful} we took advantage of the other fact, that
all morphisms of \(\vectk\) are bimonoid homomorphisms.  This is no longer the case in \(\Relk\),
but not all the arrows in the Box of a category need to be bimonoid homomorphisms.  This opens up
the possibility for a more general version of \(\st_k\), where a stateful morphism \((d,c,a,b)\) is
made up of linear relations \(a\), \(b\), \(c\), and \(d\), instead of simply linear maps.  Using
the same string diagram criteria for controllability and observability on the more general version
of \(\st_k\) could potentially generalize the notions of controllability and observability in a way
that has not been capitalized on in control theory.

\chapter{Generators and equations for PROPs}
\label{PROPs}

The formalism developed in this chapter gives us a way to present PROPs in an analogous way
to the presentation of groups, where elements in a group are the analog to morphisms in a PROP.  The
signal-flow diagrams of control theory that appear throughout this dissertation fit into the
convenient framework formed by PROPs for formalizing such diagrammatic techniques.  Whereas a group
is presented by a set of generators and a set of relations, a PROP is presented by a
\emph{distinguished object}, together with a \emph{signature} which can be thought of as a
collection of morphisms that generate the homsets, and a set of \emph{equations} between elements of
the same homset.  Stated slightly differently, a PROP is presented by a distinguished object, a
collection of generating morphisms, and a collection of equations.  Before we do anything with
PROPs, it would be good to say what a PROP is.

\begin{definition}
A \Define{PROP} is a strict \smc{} for which objects are natural numbers and the monoidal product is
addition.  A \Define{\(\prop\) morphism} is a strict \smf{} that maps the object 1 to the object 1.
\end{definition}
Stated this way, the distinguished object is the natural number 1.  We note by Mac~Lane's coherence
theorem \cite{MacLane} that any \smc{} is equivalent to a strict \smc{}.  We make the convention of
using roman typeface (as in \(\Relk\)) for names of \smcs{} that may not be strict and typewriter
typeface (as in \(\relk\)) for names of strict \smcs{}.  Other typefaces are used for categories
which, for the purposes of discussion, need not be \smcs{}.
In subsequent chapters it will be convenient to think of the objects of a PROP as tensor powers of a
distinguished object, \(X\), using the one-to-one correspondence \(X^{\tensor n} \mapsto n \in \N\).
For example, we will be concerned with PROPs that have the vector spaces \(k^n\) over some field
\(k\) as their objects, direct sum as tensor, and the one-dimensional vector space \(k\) as the
distinguished object.  There is a category \(\prop\) of PROPs and \(\prop\) morphisms.  

\begin{definition}
A \Define{\smt} \(T=(\Sigma,E)\) is a signature \(\Sigma\) together with a set \(E\) of equations.
A \Define{signature} is a set of formal symbols \(\sigma \maps m \to n\), where \(m, n \in \N\).
From a signature \(\Sigma\) we may formally construct the set of \(\Sigma\)-terms.  Defined
inductively, a \Define{\(\Sigma\)-term} takes one of the following forms:
\begin{itemize}
\item the unit \(\mathrm{id} \maps 1 \to 1\), the braiding \(\mathrm{b} \maps 2 \to 2\), or the
  formal symbols \(\sigma \maps m \to n\) in \(\Sigma\);
\item \(\beta \of \alpha \maps m \to p\), where \(\alpha \maps m \to n\) and \(\beta \maps n \to p\)
  are \(\Sigma\)-terms; or
\item \(\alpha + \gamma \maps m+p \to n+q\), where \(\alpha \maps m \to n\) and \(\gamma \maps p \to
  q\) are \(\Sigma\)-terms.
\end{itemize}
We call \((m,n)\) the \Define{type} of a \(\Sigma\)-term \(\alpha \maps m \to n\).  An
\Define{equation} is an ordered pair of \(\Sigma\)-terms with the same type.
\end{definition}
We can think of the type as an object in the discrete category \(\N \times \N\).  Then a signature
is a functor from \(\N \times \N\) to \(\Set\); to each type \((m,n) \in \N \times \N\), a
signature assigns a set of formal symbols of that type.  Note that each PROP \(\P\)
has an \Define{underlying signature}, given by the functor \(\hom_{\P}(\cdot, \cdot) \maps \N \times
\N \to \Set\).  The following result of Baez, Coya, and Rebro \cite{BCR}, building on the work of
Trimble \cite{Trimble}, allows us to understand the category \(\prop\).

\begin{proposition}
The underlying signature functor \(\U \maps \prop \to \Set^{\N \times \N}\) is monadic.
\end{proposition}
By saying \(\U\) is \Define{monadic} we mean \(\U\) has a right adjoint \(\F \maps \Set^{\N \times
\N} \to \prop\), and the resulting functor from \(\prop\) to the category of algebras of the monad
\(\F\U\) is an equivalence of categories \cite{Borc}.  We call \(\F\Sigma\) the free PROP on
the signature \(\Sigma\).  In fact, any \(\Sigma\)-term determines a morphism in \(\F\Sigma\), and
all morphisms in \(\F\Sigma\) arise this way.  For a \(\Sigma\)-term \(\alpha \maps m \to n\), we
abuse notation and refer to the corresponding morphism as \(\alpha \in \hom(X^{\tensor m},X^{\tensor
n})\).  For each formal symbol \(\sigma \maps m \to n\) in \(\Sigma\), we refer to its corresponding
morphism \(\sigma\) as a \Define{generator} for the free PROP on \(\Sigma\).

Another important consequence of this proposition is that \(\prop\) is cocomplete.  This guarantees
the existence of coequalizers, which we use to construct a PROP for a \smt.

Let \((\Sigma,E)\) be a \smt.  Then \(E\) determines a signature \(\sige\), where each ordered pair
in \(E\) determines a formal symbol in \(\sige\) whose type is the same as the type of the pair.  We
can define \(\prop\) morphisms \(\lambda, \rho \maps \F\sige \to \F\Sigma\) mapping the \(\F\)-image
of each equation to the \(\F\)-images of the first element and second element of the pair,
respectively.
\begin{definition}
The PROP presented by a \smt{} \((\Sigma,E)\), denoted \(\P(\Sigma,E)\), is the coequalizer of the
diagram \[\F E \mathrel{\mathop{\rightrightarrows}^{\lambda}_{\rho}} \F\Sigma.\]
\end{definition}
The intuition is that the coequalizer is the freest PROP subject to the constraints that the
`left-hand side' of each equation \((\alpha,\beta)\), given by \(\lambda\), is equal to the
`right-hand side', given by \(\rho\).

\begin{definition}
A \Define{subPROP} \(\P'\) of a given PROP \(\P\) is the source of a monomorphism in \(\prop\),
\(i \maps \P' \to \P\).  A \Define{quotient PROP} \(\propQ\) of a given PROP \(\P\) is the target of 
a regular epimorphism in \(\prop\), \(\phi \maps \P \to \propQ\).
\end{definition}

We are often interested in comparing PROPs that have similar generators and equations.  The next
proposition can be phrased as the slogan, ``Adding generators and removing equations both result in
bigger PROPs.''  Once again, a proof of this proposition will appear in \cite{BCR}.

\begin{proposition}
Given a \smt{} \((\Sigma, E)\), a signature \(\Sigma'\) such that \(\Sigma \subseteq \Sigma'\), and
equations \(E' \subseteq E\), the following are true:
 \begin{itemize}
  \item \(\P(\Sigma, E)\) is a subPROP of \(\P(\Sigma', E)\), and
  \item \(\P(\Sigma, E)\) is a quotient PROP of \(\P(\Sigma, E')\).
 \end{itemize}
\end{proposition}
It immediately follows that \(\P(\Sigma', E)\) is a quotient PROP of, and \(\P(\Sigma, E')\) is a
subPROP of \(\P(\Sigma', E')\).  Another immediate corollary is that \(\P(\Sigma, E)\) is a quotient
PROP of the free PROP \(\F\Sigma\).

In later chapters we will show a PROP is the PROP for a \smt{} \((\Sigma, E)\) by finding `standard
forms' for the morphisms.

\begin{definition}
Given a \smt{} \((\Sigma, E)\) and a PROP \(\P\) such that \(\phi \maps \F\Sigma \to \P\) is an
epimorphism in \(\prop\), a \Define{standard form} for a morphism \(p\) in \(\P\) is a particular
morphism \(\tilde p\) in \(\F\Sigma\) such that \(\phi \tilde p = p\).
\end{definition}

The requirement that \(\phi\) is an epimorphism in \(\prop\) means \(\phi\) is surjective on
morphisms.  Thus every morphism in \(\P\) has a standard form.  There is no requirement that
standard forms respect composition, so we do not get a functor \(\P \to \F\Sigma\) that satisfies
\(p \mapsto \tilde p\).  However, there is a functor \(\nu \maps \U\P \to \U\F\Sigma\) satisfying
\(\nu(\U p) = \U \tilde p\) since signatures are discrete.

\begin{proposition} \label{normform}
Given a PROP \(\P\) and a \smt{} \((\Sigma, E)\), let \(\pi \maps \F\Sigma \to \propQ\) be the
coequalizer of \(\F\sige \rightrightarrows \F\Sigma\).  Given any epimorphism \(\phi \maps \F\Sigma \to
\P\), such that \(\phi \lambda = \phi \rho\), and \(\pi (f) = \pi \widetilde{\phi(f)}\) for all
morphisms \(f\) in \(\F\Sigma\), then \(\P\) and \(\propQ\) are isomorphic PROPs.
\end{proposition}

This theorem says that if \(\P\) `respects the equations' of the \smt{} and every morphism in the
free PROP can be connected to a standard form using the equations in \(E\), then \(\P\) is
\(\P(\Sigma, E)\).

\begin{proof}
As noted above, there is a functor \(\nu \maps \U\P \to \U\F\Sigma\) that sends morphisms in \(\P\)
to their standard forms, on the level of signatures.  That is, \(\nu f = \tilde f\).  The condition
\(\phi \lambda = \phi \rho\) means there is a unique morphism \(\alpha \maps \propQ \to \P\) such
that \(\phi = \alpha \pi\).  We will show \(\alpha\) is an isomorphism by showing \(\U\alpha\) is an
isomorphism and lifting this isomorphism of signatures to an isomorphism of PROPs.  It is
immediately evident that \(\U\phi \of \nu = 1_{\U\P}\), since the image of the standard form of a
morphism is the same as the original morphism.  Thus \(\U\alpha \of \U\pi \of \nu = 1\).  It remains
to show \(\U\pi \of \nu\) is a two-sided inverse.

Since \(\pi (f) = \pi \widetilde{\phi(f)}\), applying the functor \(\U\) gives \(\U\pi (f) = \U\pi
\of \widetilde{\U\phi(f)} = \U\pi \of \nu(\U\phi(f))\).  Thus \(\U\pi = \U\pi \of \nu \of \U\alpha
\of \U\pi\).  Now \(\pi\) is a regular epimorphism and \(\U\) is a monadic functor over
\(\Set^{\N\times\N}\), a topos in which epimorphisms split (\emph{i.e.}\ the Axiom of Choice holds),
so \(\U\pi\) is an epimorphism \cite[Thm. 4.4.4]{Borc}.  This means \(\U\alpha\) can be cancelled on
the right, giving \(1 = \U\pi \of \nu \of \U\alpha\).  This shows \(\U\pi \of \nu\) is a two-sided
inverse to \(\U\alpha\).  Because \(\U\) is monadic, \(\U\) reflects isomorphisms
\cite[\emph{loc. cit.}]{Borc}, which means \(\alpha\) is an isomorphism, so \(\P \cong \propQ\).
\end{proof}

\chapter{Generators and equations description of $\relk$}
\label{vectrel}
Now that we have the proper tools for presenting PROPs in terms of generators and equations, we turn
our attention to the PROP \(\relk\), which we will use as the target for several \(\prop\)
morphisms.  In what follows we fix a field \(k\), and all vector spaces will be over this field.

\begin{definition}
Given vector spaces \(U\) and \(V\), a \Define{linear relation} \(L \maps U \asrelto V\), is a
linear subspace
\[ L \subseteq U \oplus V.\]
\end{definition}
In particular, a linear relation \(L \maps k^m \asrelto k^n\) is just an arbitrary system of linear
equations relating \(m\) input variables to \(n\) output variables.  This is why linear relations
are fundamental to control theory.

Since the direct sum \(U \oplus V\) is also the cartesian product of \(U\) and \(V\), a linear
relation is indeed a relation in the usual sense, but with the property that if \(u \in U\) is
related to \(v \in V\) and \(u' \in U\) is related to \(v' \in V\) then \(cu + c'u'\) is related to
\(cv + c'v'\) whenever \(c, c' \in k\).  We compose linear relations \(L \maps U \asrelto V\) and
\(L' \maps V \asrelto W\) in the usual way of composing relations:
\[ L'L = \{(u,w) \colon \; \exists \; v \in V \;\; (u,v) \in L \textrm{ and }
(v,w) \in L'\} .\]
There is thus a category \(\Relk\) whose objects are finite-dimensional vector spaces over \(k\),
and whose morphisms are linear relations.   

Moreover, \(\Relk\) becomes symmetric monoidal, with the direct sum of vector spaces providing the 
symmetric monoidal structure.  In particular, given linear relations \(L \maps U \asrelto V\) and
\(L' \maps U' \asrelto V'\), the linear relation \(L \oplus L' \maps U \oplus U' \asrelto V \oplus
V'\), is given by
\[  L \oplus L' = \{ (u,u',v,v') \colon \; (u,v) \in L \textrm{ and } (u',v') \in L' \}. \]

Any linear map \(f \maps U \to V\) gives a linear relation \(F \maps U \asrelto V\), namely the
graph of that map:
\[ F = \{(u,f(u)) \colon u \in U\}.\]
Composing linear maps thus becomes a special case of composing linear relations.  Thus, the category
\(\Vectk\) of finite-dimensional vector spaces and linear maps is a subcategory of \(\Relk\).  If we
make \(\Vectk\) into a symmetric monoidal category using direct sum, the inclusion of \(\Vectk\) in
\(\Relk\) is a symmetric monoidal functor.

To work with \(\Relk\) using the machinery of PROPs, we make the following definition:

\begin{definition}
\label{linearrelation}
For any field \(k\), let \(\relk\) be the PROP where a morphism from \(m\) to \(n\) is a linear
relation from \(k^m\) to \(k^n\), with the usual composition of relations, with direct sum providing
the tensor product.
\end{definition}

One can check that \(\relk\) is equivalent, as a symmetric monoidal category, to \(\Relk\).   It is
a skeleton of \(\Relk\), so it is clearly equivalent as a category.  However, note that \(\relk\)
has trivial associators and unitors (being a PROP), while \(\Relk\) does not, so the inclusion of
\(\relk\) in \(\Relk\) is not a \emph{strict} symmetric monoidal functor.

Our generators for \(\relk\) are logically organized into three pairs together with one `scaling'
morphism for each element of \(k\).  We make use of string diagrams to elucidate various
compositions.

The first pair is duplication and deletion:

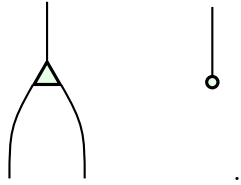
\begin{center}
     \begin{tikzpicture}[thick]
   \node[delta] (dupe){};
   \node[coordinate] (o1) at (-0.5,-1.35) {};
   \node[coordinate] (o2) at (0.5,-1.35) {};
   \node[coordinate] (in) [above of=dupe] {};

   \draw (o1) .. controls +(90:0.6) and +(-120:0.6) .. (dupe.left out);
   \draw (o2) .. controls +(90:0.6) and +(-60:0.6) .. (dupe.right out);
   \draw (in) -- (dupe);
   \end{tikzpicture}

     \hspace{3 em}
     \begin{tikzpicture}[thick]
   \node[coordinate] (in) at (0,2.15) {};
   \node [bang] (mult) at (0,1.15) {};
   \node [hole] (heightHolder) at (0,0) {};
   \draw (in) -- (mult);
   \end{tikzpicture}
.
\end{center}

\noindent
Duplication is the linear relation \(\Delta \maps k \asrelto k^2\) given by
\[                    \Delta = \{(x,x,x) : x \in k\} \subseteq k \oplus k^2.\]
That is, \(\Delta\) outputs two copies of its input.  Deletion is the linear relation \(!{} \maps k
\asrelto \{0\}\) given by
\[                    !{} = \{(x,0) : x \in k\} \subseteq k \oplus \{0\},\]
where \(\{0\}\) is the zero-dimensional vector space.  Thus \(!\) `eats up' its
input, yielding no output.  Both of these linear relations are also linear maps:  \(\Delta\) is the
diagonal map, while \(!\) is the unique (linear) map to \(\{0\}\).

Our next pair is addition and zero:

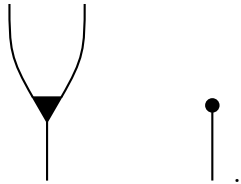
\begin{center}
     \begin{tikzpicture}[thick]
   \node[plus] (adder) {};
   \node[coordinate] (x) at (-0.5,1.35) {};
   \node[coordinate] (y) at (0.5,1.35) {};
   \node[coordinate] (out) [below of=adder] {};

   \draw (x) .. controls +(-90:0.6) and +(120:0.6) .. (adder.left in);
   \draw (y) .. controls +(-90:0.6) and +(60:0.6) .. (adder.right in);
   \draw (adder) -- (out);
   \end{tikzpicture}

     \hspace{3em}
     \begin{tikzpicture}[thick]
   \node[hole] (heightHolder) at (0,2) {};
   \node [coordinate] (out) at (0,0) {};
   \node [zero] (del) at (0,1) {};
   \draw (del) -- (out);
   \end{tikzpicture}
.
\end{center}

\noindent
Addition is the linear relation \(+{} \maps k^2 \asrelto k\) given by
\[                    +{} = \{(x,y,x+y) : x,y \in k\} \subseteq k^2{} \oplus k.\]
That is, its output is the sum of its two inputs.  Zero is the linear relation \(\zero \maps \{0\}
\asrelto k\) given by
\[                    \zero = \{(0,0)\} \subseteq \{0\} \oplus k.\]
Thus \(\zero\) takes no input and outputs the \emph{number} \(0\).  As with the first pair, these linear
relations are also linear maps, where \(\zero\) is the unique linear map from \(\{0\}\).

For any \(c \in k\), the scaling morphism \(s_c\), depicted

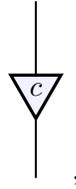
\begin{center}
   \begin{tikzpicture}[thick]
   \node[coordinate] (in) at (0,2.35) {};
   \node [multiply] (mult) at (0,1.175) {\(c\)};
   \node[coordinate] (out) at (0,0) {};
   \draw (in) -- (mult) -- (out);
   \end{tikzpicture}
,
\end{center}

\noindent
is the linear relation \(s_c \maps k \asrelto k\) given by
\[                    s_c = \{(x,cx) : x \in k\} \subseteq k \oplus k.\]
Thus \(s_c\) scales its input by a factor of \(c\).  That each \(s_c\) is a linear map is a direct
consequence of the closure of multiplication in any field.

The final pair is cup and cap:

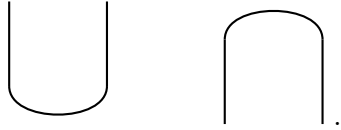
\begin{center}
     \begin{tikzpicture}[thick]
   \node[coordinate] (3) at (0,0.375) {};
   \node[coordinate] (4) at (1.3,0.375) {};
   \node[coordinate] (1) at (0,1.5) {};
   \node[coordinate] (2) at (1.3,1.5) {};
   \path
   (1) edge (3)
   (2) edge (4)
   (3) edge [-, bend right=90] (4);
   \end{tikzpicture}

     \hspace{3em}
     \begin{tikzpicture}[thick]
   \node[coordinate] (3) at (0,1.625) {};
   \node[coordinate] (4) at (1.3,1.625) {};
   \node [coordinate] (1) at (0,0.5) {};
   \node [coordinate] (2) at (1.3,0.5) {};
   \path
   (3) edge (1)
   (4) edge (2)
   (3) edge [bend left=90] (4);
   \end{tikzpicture}
.
\end{center}

\noindent
Cup is the linear relation \(\cup \maps k^2 \asrelto \{0\}\) given by
\[                    \cup = \{(x,x,0) : x \in k\} \subseteq k^2{} \oplus \{0\}.\]
Thus \(\cup\) is a partial function and not a linear map.  Cap is the linear relation \(\cap \maps
\{0\} \asrelto k^2\) given by
\[                    \cap = \{(0,x,x) : x \in k\} \subseteq \{0\} \oplus k^2.\]
Thus \(\cap\) is a multi-valued function and not a linear map.  Informally, both \(\cup\) and
\(\cap\) can be thought of as `bent identity morphisms', where the two inputs (\emph{resp.} outputs)
are identified.  Bending a string twice allows the output of a morphism to affect its own input,
which gives us a way to model feedback in a control system.

While other choices can be made for the generators, this choice has the advantage that all the
generators are linear maps, with the exception of \(\cup\) and \(\cap\).  Omitting these generators
and the equations that include them leaves us with a presentation for the subPROP \(\vectk \subseteq
\relk\) of finite-dimensional vector spaces over \(k\) and linear maps, which is another important
category.

While the list of equations in our presentation of \(\relk\) is lengthy, they can be summarized as
those necessary for several nice properties to hold:
\begin{enumerate}
\item \((k, +, 0, \Delta, !{})\) is a bicommutative bimonoid;
\item \(+,\, 0,\, \Delta\) and \(!\) commute with scaling;
\item \(\cup\) and \(\cap\) obey the zigzag equations;
\item \((k, +, 0, +^\dagger, 0^\dagger)\) is a commutative extra-special \(\dagger\)-Frobenius monoid;
\item \((k, \Delta^\dagger, !^\dagger, \Delta, !{})\) is a commutative extra-special \(\dagger\)-Frobenius monoid;
\item the field operations of \(k\) can be recovered from the generators.
\end{enumerate}
Note that item (3) makes \(\relk\) into a \(\dagger\)-compact category, allowing us to mention the
adjoints of generating morphisms in the subsequent properties.
% We will also note that the equations for the presentation of \(\vectk\) replace items (3)--(6) with
% \begin{itemize}
% \item the rig operations of \(k\) can be recovered from the generators.
% \end{itemize}
% Listing each equation individually, we have 18 equations for \(\vectk\) and 31 equations for \(\relk\).
 
The word `separable' is sometimes used as a synonym for `special' here \cite{RSW}.  A Frobenius
monoid is \Define{special} if the comultiplication followed by the multiplication is equal to the
identity.  An \Define{extra-special} Frobenius monoid has an additional, less common property: the
unit followed by the counit of the monoid is also equal to the identity (but now the identity on the
unit object for the tensor product).  This `extra' equation is one of two from the four bimonoid
equations that can be added to a \(\dagger\)-Frobenius monoid without making the monoid trivial
\cite{HV}.  The other bimonoid equation that can be added is a consequence of the `special'
equation.  See Appendix~\ref{A:connections} for a demonstration.  Because of its graphical
depiction, the extra equation has been called the `bone' equation by others \cite{FRS,GLA}.

We have placed some emphasis on the fact that cup and cap obey the zigzag equations, which allows
for a duality functor, \(\dagger\maps \relk \to \relk\), which `turns morphisms around'.  There is
another duality on \(\relk\) that is somewhat subtler.  The functor \(*\maps \relk \to \relk\)
`turns morphisms around' \emph{and} `swaps the color' of morphisms.  To wit, \(+^* = \Delta\), \(0^*
={} !\), \(\cap^* = (-1 \oplus 1) \of \cup\), and \(s_c^* = s_c\):
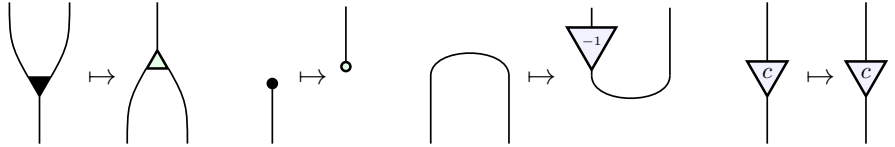
\begin{center}
\scalebox{0.8}{   \begin{tikzpicture}[thick]
   \node[plus] (adder) {};
   \node[coordinate] (x) at (-0.5,1.35) {};
   \node[coordinate] (y) at (0.5,1.35) {};
   \node[coordinate] (out) [below of=adder] {};

   \draw (x) .. controls +(-90:0.6) and +(120:0.6) .. (adder.left in);
   \draw (y) .. controls +(-90:0.6) and +(60:0.6) .. (adder.right in);
   \draw (adder) -- (out);
   \end{tikzpicture}
} \raisebox{8mm}{\(\mapsto\)} \scalebox{0.8}{   \begin{tikzpicture}[thick]
   \node[delta] (dupe){};
   \node[coordinate] (o1) at (-0.5,-1.35) {};
   \node[coordinate] (o2) at (0.5,-1.35) {};
   \node[coordinate] (in) [above of=dupe] {};

   \draw (o1) .. controls +(90:0.6) and +(-120:0.6) .. (dupe.left out);
   \draw (o2) .. controls +(90:0.6) and +(-60:0.6) .. (dupe.right out);
   \draw (in) -- (dupe);
   \end{tikzpicture}
} \qquad
\scalebox{0.8}{   \begin{tikzpicture}[thick]
   \node[hole] (heightHolder) at (0,2) {};
   \node [coordinate] (out) at (0,0) {};
   \node [zero] (del) at (0,1) {};
   \draw (del) -- (out);
   \end{tikzpicture}
} \raisebox{8mm}{\(\mapsto\)} \scalebox{0.8}{   \begin{tikzpicture}[thick]
   \node[coordinate] (in) at (0,2.15) {};
   \node [bang] (mult) at (0,1.15) {};
   \node [hole] (heightHolder) at (0,0) {};
   \draw (in) -- (mult);
   \end{tikzpicture}
} \qquad
\scalebox{0.8}{   \begin{tikzpicture}[thick]
   \node[coordinate] (3) at (0,1.625) {};
   \node[coordinate] (4) at (1.3,1.625) {};
   \node [coordinate] (1) at (0,0.5) {};
   \node [coordinate] (2) at (1.3,0.5) {};
   \path
   (3) edge (1)
   (4) edge (2)
   (3) edge [bend left=90] (4);
   \end{tikzpicture}
} \raisebox{8mm}{\(\mapsto\)} \raisebox{5mm}{\scalebox{0.8}{   \begin{tikzpicture}[thick]
   \node[coordinate] (3) at (0,0.375) {};
   \node[coordinate] (4) at (1.3,0.375) {};
   \node[coordinate] (1) at (0,1.5) {};
   \node[coordinate] (2) at (1.3,1.5) {};
   \node[multiply] (pode) at (0,0.9375) {\tiny \(-1\)};
   \path
   (1) edge (pode.90)
   (pode.270) edge (3)
   (2) edge (4)
   (3) edge [-, bend right=90] (4);
   \end{tikzpicture}
}} \qquad
\scalebox{0.8}{   \begin{tikzpicture}[thick]
   \node[coordinate] (in) at (0,2.35) {};
   \node [multiply] (mult) at (0,1.175) {\(c\)};
   \node[coordinate] (out) at (0,0) {};
   \draw (in) -- (mult) -- (out);
   \end{tikzpicture}
} \raisebox{8mm}{\(\mapsto\)} \scalebox{0.8}{   \begin{tikzpicture}[thick]
   \node[coordinate] (in) at (0,2.35) {};
   \node [multiply] (mult) at (0,1.175) {\(c\)};
   \node[coordinate] (out) at (0,0) {};
   \draw (in) -- (mult) -- (out);
   \end{tikzpicture}
}.
\end{center}
The extra factor of \(-1\) in \(\cap^*\) may seem surprising, given that cap and cup do not appear
to have any colors to swap, and turning the morphism around just alternates between cap and cup.  We
shall later see equations {\hyperref[eqn29]{\textbf{(29)}}} and {\hyperref[eqn30]{\textbf{(30)}}},
which show the cap and cup do have an implicit color that is swapped here.  As with \(\dagger\),
\(f^{**} = f\) for any morphism \(f\).  Other authors, such as Soboci\'nski \cite{GLA-bizarro}, have
compared this second duality to the Bizarro World in the Superman universe, where `good' and `evil'
are swapped, leading them to refer to this duality as `bizarro' duality.  Unlike the \(-^{\dagger}\)
duality, the \(-^*\) duality can be restricted to a duality on \(\vectk\), in which case it is
identifiable with transposition.\footnote{It is also possible to encode complex numbers so that
\(f^*\) is the \emph{conjugate} transpose of \(f\), as in \cite{GLA-i}.}

\section{Presenting $\vectk$}
\label{finvect}
As a warmup for our presentation of \(\relk\), in this section we give a presentation of a simpler PROP called
\(\vectk\), in which the morphisms are linear \emph{maps}, rather than fully general linear \emph{relations}.  Our generators for \(\vectk\) are a subset of our generators for \(\relk\): we simply leave out the cup and cap, and keep the rest.  The equations amount to saying:
\begin{enumerate}
\item \((k, +, 0, \Delta, !{})\) is a bicommutative bimonoid;
\item \(+,\, 0,\, \Delta\) and \(!\) commute with scaling;
\item the rig operations of \(k\) can be recovered from the generators.
\end{enumerate}
Here a \Define{rig} is a `ring without negatives', so the rig operations of \(k\) are the binary operations of addition and multiplication, together with the nullary operations (or constants) \(0\) and \(1\).  

In Definition~\ref{linearrelation}, we said that in the PROP \(\relk\) the morphisms from \(m\) to \(n\) are linear relations \(L \maps k^m \asrelto k^n\).  We have seen that linear maps are a special case of linear relations.  Thus we make the following definition:
\begin{definition}
Let \(\vectk\) be the subPROP of \(\relk\) whose morphisms are linear maps.
\end{definition}

One can check that \(\vectk\) is equivalent as a symmetric monoidal category to \(\Vectk\), where the objects are \emph{all} finite-dimensional vector spaces over \(k\) and where the morphisms are linear maps between these.  

\begin{lemma} \label{gensvk} 
For any field \(k\), the PROP \(\vectk\) is generated by these morphisms:
\begin{enumerate}
\item scaling \(s_c \maps k \to k\) for any \(c \in k\)
\item addition \(+ \maps k \oplus k \to k\)
\item zero \(0 \maps \{0\} \to k\)
\item duplication \(\Delta \maps k \to k \oplus k\)
\item deletion \(! \maps k \to \{0\}\)
\end{enumerate}
\end{lemma}

\begin{proof}
By this we mean that every morphism in \(\vectk\) can be obtained from these morphisms using
composition, tensor product, identity morphisms and the braiding.  A linear map in \(\vectk\), \(T
\maps k^m \to k^n\) can be expressed as \(n\) \(k\)-linear combinations of \(m\) elements of \(k\).
That is, \(T(k_1, \ldots, k_m) = (\sum_j{a_{1j} k_j}, \ldots, \sum_j{a_{nj} k_j})\), \(a_{ij} \in
k\).  Any \(k\)-linear combination of \(r\) elements can be constructed with only addition,
multiplication, and zero, with zero only necessary when providing the unique \(k\)-linear
combination for \(r=0\).  When \(r=1\), \(a_1 (k_1)\) is an arbitrary \(k\)-linear combination.  For
\(r>1\), \(+ (S_{r-1}, a_r (k_r))\) yields an arbitrary \(k\)-linear combination on \(r\) elements,
where \(S_{r-1}\) is an arbitrary \(k\)-linear combination of \(r-1\) elements.  The inclusion of
duplication allows the process of forming \(k\)-linear combinations to be repeated an arbitrary
(finite) positive number of times, and deletion allows the process to be repeated zero times.  When
\(n\) \(k\)-linear combinations are needed, each input may be duplicated \(n-1\) times.  Because
\(\vectk\) is being generated as a PROP, the \(mn\) outputs can then be permuted into \(n\)
collections of \(m\) outputs: one output from each input for each collection.  Each collection can
then form a \(k\)-linear combination, as above.  The following diagrams illustrate the pieces that
form this inductive argument.
  
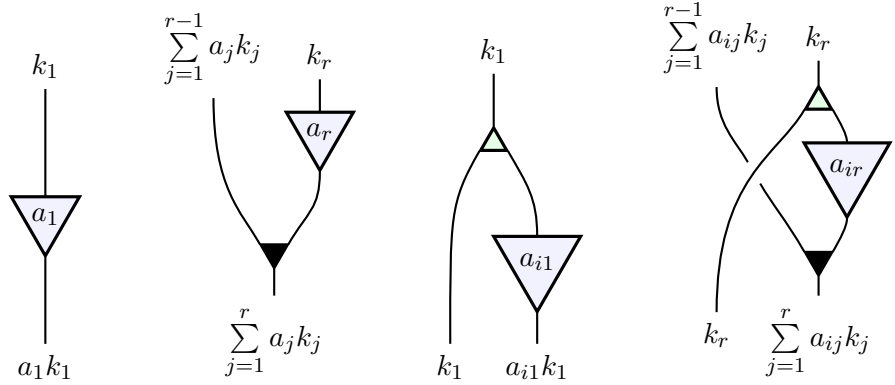
\begin{center}
% Starting the k-linear combination with one term
   \begin{tikzpicture}[thick]
   \node (top) at (0,4) {\(k_1\)};
   \node [multiply] (times) at (0,2) {\(a_1\)};
   \node (bottom) at (0,0) {\(a_1 k_1\)};

   \draw (top) -- (times) -- (bottom);
   \end{tikzpicture}
       \hspace{0.7cm}
% Making a k-linear combination
   \begin{tikzpicture}[thick]
   \node (sum) {\(\sum\limits_{j=1}^{r-1} a_j k_j\)};
   \node [coordinate] (UL) [below of=sum, shift={(0.1,-0.9)}] {};
   \node [plus] (adder) [below right of=UL, shift={(0,-0.1)}] {};
   \node [coordinate] (UR) [above right of=adder] {};
   \node [multiply] (mult) [above of=UR, shift={(-0.1,-0.1)}] {\(a_r\)};
   \node (next) [above of=mult] {\(k_r\)};
   \node (combo) [below of=adder, shift={(0,-0.2)}] {\(\sum\limits_{j=1}^{r} a_j k_j\)};

   \draw (next) -- (mult) (mult.io) .. controls +(270:0.5) and +(60:0.5) .. (adder.right in);
   \draw (sum.270) .. controls +(270:1.2) and +(120:0.5) .. (adder.left in);
   \draw (adder) -- (combo);
   \end{tikzpicture}
        \hspace{0.7cm}
% Starting multiple k-linear combinations
   \begin{tikzpicture}[thick, node distance=1.1cm]
   \node (top) {\(k_1\)};
   \node [delta] (dupe) [below of=top, shift={(0,-0.125)}] {};
   \node [coordinate] (L) [below left of=dupe] {};
   \node [coordinate] (R) [below right of=dupe] {};
   \node [multiply] (mult) [below of=R, shift={(-0.2,0.3)}] {\(a_{i1}\)};
   \node (prod) [below of=mult, shift={(0,-0.3)}] {\(a_{i1} k_1\)};
   \node (id) [below of=L, shift={(0.2,-1.1)}] {\(k_1\)};

   \draw (dupe.right out) .. controls +(300:0.5) and +(90:0.5) .. (mult.90) (mult.io) -- (prod);
   \draw (dupe.left out) .. controls +(240:0.8) and +(90:2) .. (id);
   \draw (top) -- (dupe);
   \end{tikzpicture}
        \hspace{0.7cm}
% Making multiple k-linear combinations
   \begin{tikzpicture}[thick, node distance=1.1cm]
   \node (core) {\(\sum\limits_{j=1}^{r-1} a_{ij} k_j\)};
   \node [coordinate] (subcore) [below of=core, shift={(-0.2,0)}] {};
   \node (cross) [below right of=subcore, shift={(-0.08,0.08)}] {};
   \node [coordinate] (out) [below left of=cross, shift={(0.08,-0.08)}] {};
   \node (id) [below of=out, shift={(0.2,-0.2)}] {\(k_r\)};
   \node [plus] (adder) [below right of=cross, shift={(0.08,-0.38)}] {};
   \node (sum) [below of=adder] {\(\sum\limits_{j=1}^{r} a_{ij} k_j\)};
   \node [multiply] (mult) [above right of=adder, shift={(-0.4,0.45)}] {\(a_{ir}\)};
   \node [delta] (dupe) [above left of=mult, shift={(0.4,0.1)}] {};
   \node (in) [above of=dupe, shift={(0,-0.25)}] {\(k_r\)};

   \draw (core.270) .. controls +(270:0.5) and +(120:0.5) .. (cross)
   (cross) .. controls +(300:0.5) and +(120:0.5) .. (adder.left in)
   (dupe.left out) .. controls +(240:0.7) and +(90:1.8) .. (id.90);
   \draw (dupe.right out) .. controls +(300:0.2) and +(90:0.1) .. (mult.90)
   (mult.io) .. controls +(270:0.1) and +(60:0.2) .. (adder.right in)
   (adder) -- (sum) (in) -- (dupe);
   \end{tikzpicture}
  \end{center}
Since scaling provides the map \(k_1 \mapsto a_1 k_1\), as in the far left diagram, the
middle-left diagram can be used inductively to form a \(k\)-linear combination of any number of
inputs.  In particular, we have any linear map \(S_r \maps k^m \to k\) given by \((k_1, \ldots, k_m)
\mapsto (\sum_j a_{rj} k_j)\).  Using duplication as in the middle-right diagram, one can produce
the map \(k_1 \mapsto (k_1, a_{i1} k_1)\), to which the right diagram can be inductively applied.
Thus we can build any linear map, \(T_j \in \vectk\), \(T_j \maps k^m \to k^{m+1}\) given by \((k_1,
\ldots, k_m) \mapsto (k_1, \ldots, k_m, \sum_j a_{ij} k_j)\).  If we represent the identity map on
\(k^r\) as \(1^r\), the \(r\)-fold monoidal product of the identity map on \(k\), any linear map \(T
\maps k^m \to k^n\) can be given by \((k_1, \ldots, k_m) \mapsto (\sum_j a_{1j} k_j, \ldots, \sum_j
a_{nj} k_j)\), which can be expressed as \(T = (S_1 \oplus 1^{n-1}) (T_2 \oplus 1^{n-2}) \cdots
(T_{n-1} \oplus 1^1) T_n\).  The above works as long as \(m,n \neq 0\).  Otherwise, \(f \maps k^m
\to \{0\}\) can be written as an \(m\)-fold tensor product of deletion, \(!^m\), and \(f \maps \{0\}
\to k^n\) can be written as an \(n\)-fold tensor product of zero, \(0^n\).  Since \(f \maps \{0\}
\to \{0\}\) is the monoidal unit, this has an empty diagram for its string diagram. 
\end{proof}

It is easy to see that the morphisms given in Lemma~\ref{gensvk} obey the following 18 equations:

\vskip 1em \noindent
\textbf{(1)--(3)}  Addition and zero make \(k\) into a commutative monoid:
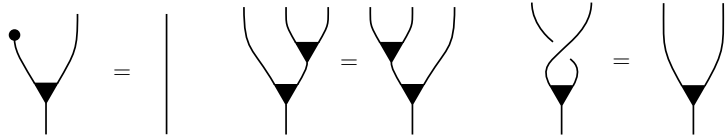
\begin{invisiblelabel}
\label{eqn123}
\end{invisiblelabel}
  \begin{center}
    \scalebox{0.80}{
% adding nothing
   \begin{tikzpicture}[-, thick, node distance=0.74cm]
   \node [plus] (summer) {};
   \node [coordinate] (sum) [below of=summer] {};
   \node [coordinate] (Lsum) [above left of=summer] {};
   \node [zero] (insert) [above of=Lsum, shift={(0,-0.35)}] {};
   \node [coordinate] (Rsum) [above right of=summer] {};
   \node [coordinate] (sumin) [above of=Rsum] {};
   \node (equal) [right of=Rsum, shift={(0,-0.26)}] {\(=\)};
   \node [coordinate] (in) [right of=equal, shift={(0,1)}] {};
   \node [coordinate] (out) [right of=equal, shift={(0,-1)}] {};

   \draw (insert) .. controls +(270:0.3) and +(120:0.3) .. (summer.left in)
         (summer.right in) .. controls +(60:0.6) and +(270:0.6) .. (sumin)
         (summer) -- (sum)    (in) -- (out);
   \end{tikzpicture}
        \hspace{1.0cm}
% Associativity
   \begin{tikzpicture}[-, thick, node distance=0.7cm]
   \node [plus] (uradder) {};
   \node [plus] (adder) [below of=uradder, shift={(-0.35,0)}] {};
   \node [coordinate] (urm) [above of=uradder, shift={(-0.35,0)}] {};
   \node [coordinate] (urr) [above of=uradder, shift={(0.35,0)}] {};
   \node [coordinate] (left) [left of=urm] {};

   \draw (adder.right in) .. controls +(60:0.2) and +(270:0.1) .. (uradder.io)
         (uradder.right in) .. controls +(60:0.35) and +(270:0.3) .. (urr)
         (uradder.left in) .. controls +(120:0.35) and +(270:0.3) .. (urm)
         (adder.left in) .. controls +(120:0.75) and +(270:0.75) .. (left)
         (adder.io) -- +(270:0.5);

   \node (eq) [right of=uradder, shift={(0,-0.25)}] {\(=\)};

   \node [plus] (ulsummer) [right of=eq, shift={(0,0.25)}] {};
   \node [plus] (summer) [below of=ulsummer, shift={(0.35,0)}] {};
   \node [coordinate] (ulm) [above of=ulsummer, shift={(0.35,0)}] {};
   \node [coordinate] (ull) [above of=ulsummer, shift={(-0.35,0)}] {};
   \node [coordinate] (right) [right of=ulm] {};

   \draw (summer.left in) .. controls +(120:0.2) and +(270:0.1) .. (ulsummer.io)
         (ulsummer.left in) .. controls +(120:0.35) and +(270:0.3) .. (ull)
         (ulsummer.right in) .. controls +(60:0.35) and +(270:0.3) .. (ulm)
         (summer.right in) .. controls +(60:0.75) and +(270:0.75) .. (right)
         (summer.io) -- +(270:0.5);
   \end{tikzpicture}
        \hspace{1.0cm}
% Commutativity
   \begin{tikzpicture}[-, thick, node distance=0.7cm]
   \node [plus] (twadder) {};
   \node [coordinate] (twout) [below of=twadder] {};
   \node [coordinate] (twR) [above right of=twadder, shift={(-0.2,0)}] {};
   \node (cross) [above of=twadder] {};
   \node [coordinate] (twRIn) [above left of=cross, shift={(0,0.3)}] {};
   \node [coordinate] (twLIn) [above right of=cross, shift={(0,0.3)}] {};

   \draw (twadder.right in) .. controls +(60:0.35) and +(-45:0.25) .. (cross)
                            .. controls +(135:0.2) and +(270:0.4) .. (twRIn);
   \draw (twadder.left in) .. controls +(120:0.35) and +(-135:0.25) .. (cross.center)
                           .. controls +(45:0.2) and +(270:0.4) .. (twLIn);
   \draw (twout) -- (twadder);

   \node (eq) [right of=twR] {\(=\)};

   \node [coordinate] (L) [right of=eq] {};
   \node [plus] (adder) [below right of=L] {};
   \node [coordinate] (out) [below of=adder] {};
   \node [coordinate] (R) [above right of=adder] {};
   \node (cross) [above left of=R] {};
   \node [coordinate] (LIn) [above left of=cross] {};
   \node [coordinate] (RIn) [above right of=cross] {};

   \draw (adder.left in) .. controls +(120:0.7) and +(270:0.7) .. (LIn)
         (adder.right in) .. controls +(60:0.7) and +(270:0.7) .. (RIn)
         (out) -- (adder);
   \end{tikzpicture}
    }
\end{center}

\vskip 1em \noindent
\textbf{(4)--(6)}  Duplication and deletion make \(k\) into a cocommutative comonoid:
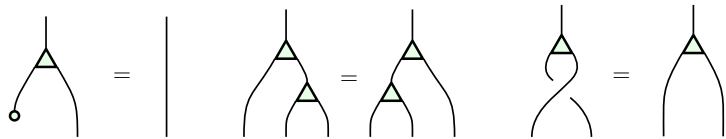
\begin{invisiblelabel}
\label{eqn456}
\end{invisiblelabel}

\begin{center}
    \scalebox{0.80}{
% dupe delete
   \begin{tikzpicture}[-, thick, node distance=0.74cm]
   \node [delta] (dupe) {};
   \node [coordinate] (top) [above of=dupe] {};
   \node [coordinate] (Ldub) [below left of=dupe] {};
   \node [bang] (delete) [below of=Ldub, shift={(0,0.35)}] {};
   \node [coordinate] (Rdub) [below right of=dupe] {};
   \node [coordinate] (dubout) [below of=Rdub] {};
   \node (equal) [right of=Rdub, shift={(0,0.26)}] {\(=\)};
   \node [coordinate] (in) [right of=equal, shift={(0,1)}] {};
   \node [coordinate] (out) [right of=equal, shift={(0,-1)}] {};

   \draw (delete) .. controls +(90:0.3) and +(240:0.3) .. (dupe.left out)
         (dupe.right out) .. controls +(300:0.6) and +(90:0.6) .. (dubout)
         (dupe) -- (top)    (in) -- (out);
   \end{tikzpicture}
       \hspace{1.0cm}
% Co-associativity
   \begin{tikzpicture}[-, thick, node distance=0.7cm]
   \node [delta] (lrduper) {};
   \node [delta] (duper) [above of=lrduper, shift={(-0.35,0)}] {};
   \node [coordinate](lrm) [below of=lrduper, shift={(-0.35,0)}] {};
   \node [coordinate](lrr) [below of=lrduper, shift={(0.35,0)}] {};
   \node [coordinate](left) [left of=lrm] {};

   \draw (duper.right out) .. controls +(300:0.2) and +(90:0.1) .. (lrduper.io)
         (lrduper.right out) .. controls +(300:0.35) and +(90:0.3) .. (lrr)
         (lrduper.left out) .. controls +(240:0.35) and +(90:0.3) .. (lrm)
         (duper.left out) .. controls +(240:0.75) and +(90:0.75) .. (left)
         (duper.io) -- +(90:0.5);

   \node (eq) [right of=lrduper, shift={(0,0.25)}] {\(=\)};

   \node [delta] (lldubber) [right of=eq, shift={(0,-0.25)}] {};
   \node [delta] (dubber) [above of=lldubber, shift={(0.35,0)}] {};
   \node [coordinate] (llm) [below of=lldubber, shift={(0.35,0)}] {};
   \node [coordinate] (lll) [below of=lldubber, shift={(-0.35,0)}] {};
   \node [coordinate] (right) [right of=llm] {};

   \draw (dubber.left out) .. controls +(240:0.2) and +(90:0.1) .. (lldubber.io)
         (lldubber.left out) .. controls +(240:0.35) and +(90:0.3) .. (lll)
         (lldubber.right out) .. controls +(300:0.35) and +(90:0.3) .. (llm)
         (dubber.right out) .. controls +(300:0.75) and +(90:0.75) .. (right)
         (dubber.io) -- +(90:0.5);
   \end{tikzpicture}
       \hspace{1.0cm}
% Co-commutativity
   \begin{tikzpicture}[-, thick, node distance=0.7cm]
   \node [coordinate] (twtop) {};
   \node [delta] (twdupe) [below of=twtop] {};
   \node [coordinate] (twR) [below right of=twdupe, shift={(-0.2,0)}] {};
   \node (cross) [below of=twdupe] {};
   \node [coordinate] (twROut) [below left of=cross, shift={(0,-0.3)}] {};
   \node [coordinate] (twLOut) [below right of=cross, shift={(0,-0.3)}] {};

   \draw (twdupe.left out) .. controls +(240:0.35) and +(135:0.25) .. (cross)
                           .. controls +(-45:0.2) and +(90:0.4) .. (twLOut)
         (twdupe.right out) .. controls +(300:0.35) and +(45:0.25) .. (cross.center)
                            .. controls +(-135:0.2) and +(90:0.4) .. (twROut)
         (twtop) -- (twdupe);

   \node (eq) [right of=twR] {\(=\)};

   \node [coordinate] (L) [right of=eq] {};
   \node [delta] (dupe) [above right of=L] {};
   \node [coordinate] (top) [above of=dupe] {};
   \node [coordinate] (R) [below right of=dupe] {};
   \node (uncross) [below left of=R] {};
   \node [coordinate] (LOut) [below left of=uncross] {};
   \node [coordinate] (ROut) [below right of=uncross] {};

   \draw (dupe.left out) .. controls +(240:0.7) and +(90:0.7) .. (LOut)
         (dupe.right out) .. controls +(300:0.7) and +(90:0.7) .. (ROut)
         (top) -- (dupe);
   \end{tikzpicture}
    }
\end{center}

\vskip 1em \noindent
\textbf{(7)--(10)}  The monoid and comonoid structures on \(k\) fit together to form a bimonoid:
\begin{invisiblelabel}
\label{eqn78910}
\end{invisiblelabel}

\begin{center}
    \scalebox{0.80}{
   \begin{tikzpicture}[thick]
% Sumdupe equation
   \node [plus] (adder) {};
   \node [coordinate] (f) [above of=adder, shift={(-0.4,-0.325)}] {};
   \node [coordinate] (g) [above of=adder, shift={(0.4,-0.325)}] {};
   \node [delta] (dupe) [below of=adder, shift={(0,0.25)}] {};
   \node [coordinate] (outL) [below of=dupe, shift={(-0.4,0.325)}] {};
   \node [coordinate] (outR) [below of=dupe, shift={(0.4,0.325)}] {};

   \draw (adder.io) -- (dupe.io)
         (f) .. controls +(270:0.4) and +(120:0.25) .. (adder.left in)
         (adder.right in) .. controls +(60:0.25) and +(270:0.4) .. (g)
         (dupe.left out) .. controls +(240:0.25) and +(90:0.4) .. (outL)
         (dupe.right out) .. controls +(300:0.25) and +(90:0.4) .. (outR);
   \end{tikzpicture}
      \raisebox{2.49em}{=}
      \hspace{1em}
   \begin{tikzpicture}[-, thick, node distance=0.7cm]
   \node [plus] (addL) {};
   \node (cross) [above right of=addL, shift={(-0.1,-0.0435)}] {};
   \node [plus] (addR) [below right of=cross, shift={(-0.1,0.0435)}] {};
   \node [delta] (dupeL) [above left of=cross, shift={(0.1,-0.0435)}] {};
   \node [delta] (dupeR) [above right of=cross, shift={(-0.1,-0.0435)}] {};
   \node [coordinate] (f) [above of=dupeL] {};
   \node [coordinate] (g) [above of=dupeR] {};
   \node [coordinate] (sum1) [below of=addL, shift={(0,0.2)}] {};
   \node [coordinate] (sum2) [below of=addR, shift={(0,0.2)}] {};

   \path
   (addL) edge (sum1) (addL.right in) edge (dupeR.left out) (addL.left in) edge [bend left=30] (dupeL.left out)
   (addR) edge (sum2) (addR.left in) edge (cross) (addR.right in) edge [bend right=30] (dupeR.right out)
   (dupeL) edge (f)
   (dupeL.right out) edge (cross)
   (dupeR) edge (g);
   \end{tikzpicture}
        \hspace{1.0cm}
   \begin{tikzpicture}[thick]
% Zerdupe equation
   \node [zero] (z) at (0,1) {};
   \node [delta] (dub) at (0,0.2) {};
   \node [coordinate] (oL) at (-0.35,-0.6) {};
   \node [coordinate] (oR) at (0.35,-0.6) {};

   \node (eq) at (1,0.42) {=};

   \node [zero] (zleft) at (2,1) {};
   \node [zero] (zright) at (2.7,1) {};
   \node [coordinate] (Lo) at (2,-0.6) {};
   \node [coordinate] (Ro) at (2.7,-0.6) {};

   \draw (z) -- (dub)
         (dub.left out) .. controls +(240:0.3) and +(90:0.5) .. (oL)
         (dub.right out) .. controls +(300:0.3) and +(90:0.5) .. (oR)
         (zleft) -- (Lo)
         (zright) -- (Ro);
   \end{tikzpicture}
        \hspace{1.0cm}
   \begin{tikzpicture}[thick]
% Delplus equation
   \node [bang] (b) at (0,-1) {};
   \node [plus] (sum) at (0,-0.2) {};
   \node [coordinate] (oL) at (-0.35,0.6) {};
   \node [coordinate] (oR) at (0.35,0.6) {};
   \node (spacemaker) at (0,-1.38) {};

   \node (eq) at (1,-0.47) {=};

   \node [bang] (bleft) at (2,-1) {};
   \node [bang] (bright) at (2.7,-1) {};
   \node [coordinate] (Lo) at (2,0.6) {};
   \node [coordinate] (Ro) at (2.7,0.6) {};

   \draw (b) -- (sum)
         (sum.left in) .. controls +(120:0.3) and +(270:0.5) .. (oL)
         (sum.right in) .. controls +(60:0.3) and +(270:0.5) .. (oR)
         (bleft) -- (Lo)
         (bright) -- (Ro);
   \end{tikzpicture}
       \hspace{1.0cm}
   \begin{tikzpicture}[thick]
% Delzero equation
   \node [zero] (z) at (0,0.11) {};
   \node [bang] (b) at (0,-1) {};
   \node (spacemaker) at (0,-1.38) {};

   \node (eq) at (0.7,-0.47) {=};

   \draw (z) -- (b);
   \end{tikzpicture}
    }
\end{center}

\vskip 1em \noindent
\textbf{(11)--(14)}  The rig structure of \(k\) can be recovered from the generating morphisms:
\begin{invisiblelabel}
\label{eqn11121314}
\end{invisiblelabel}

\begin{center}
    \scalebox{0.80}{
% multiplying twice
   \begin{tikzpicture}[-, thick]
   \node (top) {};
   \node [multiply] (c) [below of=top] {\(c\)};
   \node [multiply] (b) [below of=c] {\(b\)};
   \node (bottom) [below of=b] {};

   \draw (top) -- (c) -- (b) -- (bottom);

   \node (eq) [left of=b, shift={(0.2,0.5)}] {\(=\)};

   \node (bctop) [left of=top, shift={(-0.6,0)}] {};
   \node [multiply] (bc) [left of=eq, shift={(0.2,0)}] {\(bc\)};
   \node (bcbottom) [left of=bottom, shift={(-0.6,0)}] {};

   \draw (bctop) -- (bc) -- (bcbottom);
   \end{tikzpicture}
        \hspace{1.0cm}
% sum of products / product of sum
   \begin{tikzpicture}[-, thick, node distance=0.85cm]
   \node (bctop) {};
   \node [multiply] (bc) [below of=bctop, shift={(0,-0.59)}] {\(\scriptstyle{b+c}\)};
   \node (bcbottom) [below of=bc, shift={(0,-0.59)}] {};

   \draw (bctop) -- (bc) -- (bcbottom);

   \node (eq) [right of=bc, shift={(0.15,0)}] {\(=\)};

   \node [multiply] (b) [right of=eq, shift={(0,0.1)}] {\(\scriptstyle{b}\)};
   \node [delta] (dupe) [above right of=b, shift={(-0.2,0.1)}] {};
   \node (top) [above of=dupe, shift={(0,-0.2)}] {};
   \node [multiply] (c) [below right of=dupe, shift={(-0.2,-0.1)}] {\(\scriptstyle{c}\)};
   \node [plus] (adder) [below right of=b, shift={(-0.2,-0.3)}] {};
   \node (out) [below of=adder, shift={(0,0.2)}] {};

   \draw
   (dupe.left out) .. controls +(240:0.15) and +(90:0.15) .. (b.90)
   (dupe.right out) .. controls +(300:0.15) and +(90:0.15) .. (c.90)
   (top) -- (dupe.io)
   (adder.io) -- (out)
   (adder.left in) .. controls +(120:0.15) and +(270:0.15) .. (b.io)
   (adder.right in) .. controls +(60:0.15) and +(270:0.15) .. (c.io);
   \end{tikzpicture}
        \hspace{1.0cm}
% multiplication by 1
\raisebox{1.6em}{
   \begin{tikzpicture}[-, thick, node distance=0.85cm]
   \node (top) {};
   \node [multiply] (one) [below of=top] {1};
   \node (bottom) [below of=one] {};

   \draw (top) -- (one) -- (bottom);

   \node (eq) [right of=one] {\(=\)};
   \node (topid) [right of=top, shift={(0.6,0)}] {};
   \node (botid) [right of=bottom, shift={(0.6,0)}] {};

   \draw (topid) -- (botid);
   \end{tikzpicture}
}
        \hspace{1.0cm}
% multiplication by 0
\raisebox{1.6em}{
   \begin{tikzpicture}[-, thick, node distance=0.85cm]
   \node [multiply] (prod) {\(0\)};
   \node (in0) [above of=prod] {};
   \node (out0) [below of=prod] {};
   \node (eq) [right of=prod] {\(=\)};
   \node [bang] (del) [right of=eq, shift={(-0.2,0.2)}] {};
   \node [zero] (ins) [right of=eq, shift={(-0.2,-0.2)}] {};
   \node (in1) [above of=del, shift={(0,-0.2)}] {};
   \node (out1) [below of=ins, shift={(0,0.2)}] {};

   \draw (in0) -- (prod) -- (out0);
   \draw (in1) -- (del);
   \draw (ins) -- (out1);
   \end{tikzpicture}
}
    }
\end{center}

\vskip 1em \noindent
\textbf{(15)--(16)}  Scaling commutes with addition and zero:
\begin{invisiblelabel}
\label{eqn1516}
\end{invisiblelabel}

\begin{center}
    \scalebox{0.80}{
% Distributivity
   \begin{tikzpicture}[-, thick, node distance=0.85cm]
   \node [plus] (adder) {};
   \node (out) [below of=adder, shift={(0,0.2)}] {};
   \node [multiply] (L) [above left of=adder, shift={(0.15,0.45)}] {\(c\)};
   \node [multiply] (R) [above right of=adder, shift={(-0.15,0.45)}] {\(c\)};
   \node [coordinate] (RIn) [above of=R] {};
   \node [coordinate] (LIn) [above of=L] {};

   \draw (adder.right in) .. controls +(60:0.2) and +(270:0.2) .. (R.io) (R) -- (RIn);
   \draw (adder.left in) .. controls +(120:0.2) and +(270:0.2) .. (L.io) (L) -- (LIn);
   \draw (out) -- (adder);
   \end{tikzpicture}
        \hspace{0.1cm}
   \begin{tikzpicture}[node distance=1.15cm]
   \node (eq){\(=\)};
   \node [below of=eq] {};
   \end{tikzpicture}
   \begin{tikzpicture}[-, thick, node distance=0.85cm]
   \node (out) {};
   \node [multiply] (c) [above of=out] {\(c\)};
   \node [plus] (adder) [above of=c] {};
   \node [coordinate] (L) [above left of=adder, shift={(0.15,0.25)}] {};
   \node [coordinate] (R) [above right of=adder, shift={(-0.15,0.25)}] {};

   \draw (R) .. controls +(270:0.2) and +(60:0.4) .. (adder.right in) (adder) -- (c) -- (out);
   \draw (L) .. controls +(270:0.2) and +(120:0.4) .. (adder.left in);
   \end{tikzpicture}
        \hspace{0.7cm}
% multiplied insert
   \begin{tikzpicture}[-, thick, node distance=0.85cm]
   \node [multiply] (prod) {\(c\)};
   \node (out0) [below of=prod] {};
   \node [zero] (ins0) [above of=prod] {};
   \node (eq) [right of=prod] {\(=\)};
   \node (out1) [below right of=eq] {};
   \node [zero] (ins1) [above of=out1, shift={(0,0.2)}] {};

   \draw (out0) -- (prod) -- (ins0);
   \draw (out1) -- (ins1);
   \end{tikzpicture}
       \hspace{0.7cm}
    }
\end{center}

\vskip 1em \noindent
\textbf{(17)--(18)}  Scaling commutes with duplication and deletion:
\begin{invisiblelabel}
\label{eqn1718}
\end{invisiblelabel}

\begin{center}
    \scalebox{0.80}{
% Co-distributivity
   \begin{tikzpicture}[-, thick, node distance=0.85cm]
   \node (top) {};
   \node [delta] (dupe) [below of=top, shift={(0,0.2)}] {};
   \node [multiply] (L) [below left of=dupe, shift={(0.15,-0.45)}] {\(c\)};
   \node [multiply] (R) [below right of=dupe, shift={(-0.15,-0.45)}] {\(c\)};
   \node [coordinate] (ROut) [below of=R] {};
   \node [coordinate] (LOut) [below of=L] {};

   \draw (dupe.left out) .. controls +(240:0.2) and +(90:0.2) .. (L.90) (L) -- (LOut);
   \draw (dupe.right out) .. controls +(300:0.2) and +(90:0.2) .. (R.90) (R) -- (ROut);
   \draw (top) -- (dupe);
   \end{tikzpicture}
        \hspace{0.1cm}
   \begin{tikzpicture}[node distance=1.15cm]
   \node (eq){\(=\)};
   \node [below of=eq] {};
   \end{tikzpicture}
   \begin{tikzpicture}[-, thick, node distance=0.85cm]
   \node (top) {};
   \node [multiply] (c) [below of=top] {\(c\)};
   \node [delta] (dupe) [below of=c] {};
   \node [coordinate] (L) [below left of=dupe, shift={(0.15,-0.25)}] {};
   \node [coordinate] (R) [below right of=dupe, shift={(-0.15,-0.25)}] {};

   \draw (dupe.left out) .. controls +(240:0.2) and +(90:0.2) .. (L);
   \draw (dupe.right out) .. controls +(300:0.2) and +(90:0.2) .. (R);
   \draw (top) -- (c) -- (dupe);
   \end{tikzpicture}
       \hspace{0.7cm}
% deleted product
   \begin{tikzpicture}[-, thick, node distance=0.85cm]
   \node [multiply] (prod) {\(c\)};
   \node (in0) [above of=prod] {};
   \node [bang] (del0) [below of=prod] {};
   \node (eq) [right of=prod] {\(=\)};
   \node (in1) [above right of=eq] {};
   \node [bang] (del1) [below of=in1, shift={(0,-0.2)}] {};
   \node [below of=del1, shift={(0,-0.2)}] {};

   \draw (in0) -- (prod) -- (del0);
   \draw (in1) -- (del1);
   \end{tikzpicture}
    }.
  \end{center}

 %%%%%%%%%%%%%%%%%%%%%%%%%%%%%%%%%%%%%%%%
%%                                      %%
%      End of FinVect equation list      %
%%                                      %%
 %%%%%%%%%%%%%%%%%%%%%%%%%%%%%%%%%%%%%%%%

In fact, these equations are enough: any two ways of drawing a linear map as a signal-flow diagram
can be connected using these equations.  That is, together with the generators, they give a
presentation of \(\vectk\) as a PROP.  More precisely, the generating morphisms are the
\(\F\)-images of a signature \(\Sigma_{\vectk}\), these 18 equations are the \(2|k|^2 + 4|k| + 12\)
pairs\footnote{Equations {\bf (11)} and {\bf (12)} listed above are each \(|k|^2\) pairs of
\(\Sigma\)-terms, and equations {\bf (15)--(18)} are each \(|k|\) pairs of \(\Sigma\)-terms.} of
\(\Sigma\)-terms \(E_{\vectk}\), and \(\vectk\) is the coequalizer of \(\F E_{\vectk}
\rightrightarrows \F\Sigma_{\vectk}\).  Thus \(\vectk\) is the PROP presented by the symmetric
monoidal theory \((\Sigma_{\vectk},E_{\vectk})\).

\begin{theorem} \label{presvk} The PROP \(\vectk\) is presented by the generators given in
Lemma~\ref{gensvk}, and equations {\bf (1)--(18)} as listed above.
\end{theorem}

\begin{proof}
To prove this, we find a standard form for morphisms in \(\vectk\) and use Theorem~\ref{normform}.
That is, it suffices to show that any string diagram built from generating morphisms and the
braiding can be put into a standard form using topological equivalences and equations
\textbf{(1)--(18)}.

A qualitative description of this standard form will be helpful for understanding how an arbitrary
string diagram can be rewritten in this form.  By way of example, consider the linear
transformation \(T \maps \R^3 \to \R^2\) given by 
\[ (x_1, x_2, x_3) \mapsto (y_1, y_2) = (3x_1 + 7x_2 +
2x_3, 9x_1 + x_2).\]
Its standard form looks like this:
  \begin{center}
% Example of the construction of the standard form for FinVect_k
    \scalebox{0.80}
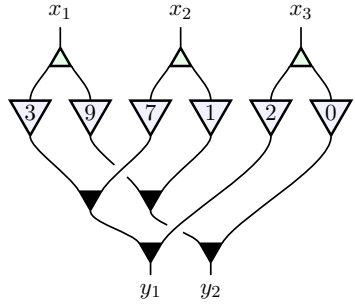
{
   \begin{tikzpicture}[thick]
   \node (i1) at (-2.5,3.4) {\(x_1\)};
   \node (i2) at (-0.5,3.4) {\(x_2\)};
   \node (i3) at (1.5,3.4) {\(x_3\)};
   \node[multiply] (a11) at (-3,1.732) {\(3\)};
   \node[multiply] (a21) at (-2,1.732) {\(9\)};
   \node[multiply] (a12) at (-1,1.732) {\(7\)};
   \node[multiply] (a22) at (0,1.732) {\(1\)};
   \node[multiply] (a13) at (1,1.732) {\(2\)};
   \node[multiply] (a23) at (2,1.732) {\(0\)};
   \node[delta] (D1) at (-2.5,2.6) {};
   \node[delta] (D2) at (-0.5,2.6) {};
   \node[delta] (D3) at (1.5,2.6) {};
   \node[plus] (S1T) at (-2,0.31) {};
   \node[plus] (S1B) at (-1,-0.556) {};
   \node[plus] (S2T) at (-1,0.31) {};
   \node[plus] (S2B) at (0,-0.556) {};
   \node (o1) at (-1,-1.24) {\(y_1\)};
   \node (o2) at (0,-1.24) {\(y_2\)};

   \draw
   (a11.io) .. controls +(270:0.2) and +(120:0.32) .. (S1T.left in)
   (a21.io) .. controls +(270:0.2) and +(120:0.32) .. (S2T.left in)
   (a22.io) .. controls +(270:0.2) and +(60:0.32) .. (S2T.right in)
   (S1T.io) .. controls +(270:0.2) and +(120:0.2) .. (S1B.left in)
   (i1) -- (D1) (i2) -- (D2) (i3) -- (D3)
   (S1B) -- (o1) (S2B) -- (o2)
   (S2T.io) .. controls +(270:0.2) and +(120:0.2) .. (S2B.left in)
   (a23.io) .. controls +(270:0.5) and +(60:0.5) .. (S2B.right in);
% make holes for the crossings
   \node [hole] (cross2) at (-0.59,-0.2) {};
   \node [hole] (cross1) at (-1.5,0.75) {};
   \draw
   (a12.io) .. controls +(270:0.2) and +(60:0.32) .. (S1T.right in)
   (a13.io) .. controls +(270:0.5) and +(60:0.5) .. (S1B.right in)
   (D1.left out) .. controls +(240:0.2) and +(90:0.2) .. (a11.90)
   (D1.right out) .. controls +(300:0.2) and +(90:0.2) .. (a21.90)
   (D2.left out) .. controls +(240:0.2) and +(90:0.2) .. (a12.90)
   (D2.right out) .. controls +(300:0.2) and +(90:0.2) .. (a22.90)
   (D3.left out) .. controls +(240:0.2) and +(90:0.2) .. (a13.90)
   (D3.right out) .. controls +(300:0.2) and +(90:0.2) .. (a23.90);
   \end{tikzpicture}
    }.
\end{center}
This is a string diagram picture of the following equation:
   \[ T x =
   \left[ \begin{array}{ccc}
   3 & 7 & 2 \\
   9 & 1 & 0 \end{array} \right]
   \left[ \begin{array}{c} x_1 \\ x_2 \\ x_3 \end{array} \right] =
   \left[ \begin{array}{c} y_1 \\ y_2 \end{array} \right].
  \]

In general, given a \(k\)-linear transformation \(T \maps k^m \to k^n\), we can describe it using an
\(n \times m\) matrix with entries in \(k\).  The case where \(m\) and/or \(n\) is zero gives a
matrix with no entries, so their standard form will be treated separately.  For positive values of
\(m\) and \(n\), the standard form has three distinct layers.  The top layer consists of \(m\)
clusters of \(n-1\) instances of \(\Delta\).  The middle layer is \(mn\) scalings.  The \(n\)
outputs of the \(j\)th cluster connect to the inputs of the scalings by \(\{a_{1j}, \dotsc,
a_{nj}\}\), where \(a_{ij}\) is the \(ij\) entry of \(A\), the matrix for \(T\).  The bottom layer
consists of \(n\) clusters of \(m-1\) instances of \(+\).  There will generally be braiding in this
layer as well, but since the category is being generated as symmetric monoidal, the locations of the
braidings doesn't matter so long as the topology of the string diagram is preserved.  The topology
of the sum layer is that the \(i\)th sum cluster gets its \(m\) inputs from the outputs of the
scalings by \(\{a_{i1}, \dotsc, a_{im}\}\).  The arrangement of the instances of \(\Delta\) and
\(+\) within their respective clusters does not matter, due to the associativity of \(+\) via
equation \textbf{(2)} and coassociativity of \(\Delta\) via equation \textbf{(5)}.  For the sake of
making the standard form explicit with respect to these equations, we may assume the right output of
a \(\Delta\) is always connected to a scaling input, and the right input of a \(+\) is always
connected to a scaling output.  This gives a prescription for drawing the standard form of a string
diagram with a corresponding matrix \(A\).

The standard form for \(T \maps k^0 \to k^n\) is \(n\) zeros (\(0 \oplus \dotsb \oplus 0\)), and the
standard form for \(T \maps k^m \to k^0\) is \(m\) deletions (\(! \oplus \dotsb \oplus\, !\)).

Each of the generating morphisms can easily be put into standard form: the string diagrams for zero,
deletion, and scaling are already in standard form.  The string diagram for duplication
(\emph{resp.} addition) can be put into standard form by attaching the scaling \(s_1\), equation
\textbf{(13)}, to each of the outputs (\emph{resp.} inputs).
  \begin{center}
% Standard FinVect form for generating morphisms
    \scalebox{0.80}
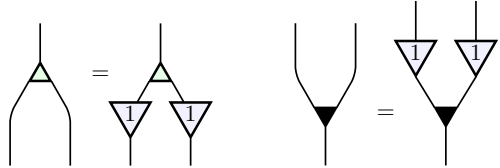
{
   \begin{tikzpicture}[thick]
   \node[delta] (dupe) {};
   \node[multiply] (o1) at (-0.5,-0.65) {\(1\)};
   \node[multiply] (o2) at (0.5,-0.65) {\(1\)};
   \node (in) [above of=dupe] {};
   \node (posto1) [below of=o1] {};
   \node (posto2) [below of=o2] {};

   \draw[rounded corners] (posto1) -- (o1) -- (dupe.left out);
   \draw[rounded corners] (posto2) -- (o2) -- (dupe.right out);
   \draw (in) -- (dupe);

   \node (equals) at (-1,0) {\(=\)};
   \node[delta] (dub) at (-2,0) {};
   \node (dubin) [above of=dub] {};
   \node[coordinate] (L) at (-2.5,-0.65) {};
   \node[coordinate] (R) at (-1.5,-0.65) {};
   \node (Ld) [below of=L] {};
   \node (Rd) [below of=R] {};

   \draw[rounded corners] (Ld) -- (L) -- (dub.left out) (Rd) -- (R) -- (dub.right out) (dubin) -- (dub);
   \end{tikzpicture}
\hspace{1cm}
   \begin{tikzpicture}[thick]
   \node[plus] (summer) {};
   \node[multiply] (o1) at (-0.5,1.02) {\(1\)};
   \node[multiply] (o2) at (0.5,1.02) {\(1\)};
   \node (out) [below of=summer] {};
   \node (posto1) [above of=o1] {};
   \node (posto2) [above of=o2] {};

   \draw (posto1) -- (o1) -- (o1.io) -- (summer.left in);
   \draw (posto2) -- (o2) -- (o2.io) -- (summer.right in);
   \draw (out) -- (summer);

   \node (equals) at (-1,0) {\(=\)};
   \node[plus] (adder) at (-2,0) {};
   \node (addout) [below of=adder] {};
   \node[coordinate] (L) at (-2.5,0.65) {};
   \node[coordinate] (R) at (-1.5,0.65) {};
   \node (Lu) [above of=L] {};
   \node (Ru) [above of=R] {};

   \draw[rounded corners] (Lu) -- (L) -- (adder.left in) (Ru) -- (R) -- (adder.right in) (addout) -- (adder);
   \end{tikzpicture}
    }
  \end{center}
The braiding morphism is just as basic to our argument as the generating morphisms, so we will need
to write the string diagram for \(B\) in standard form as well.  The matrix corresponding to
braiding is
\[\left[ \begin{array}{cc}
   0 & 1 \\
   1 & 0 \end{array} \right],\]
so its standard form is as follows:
  \begin{center}
% Standard form for braiding in FinVect_k
    \scalebox{0.80}
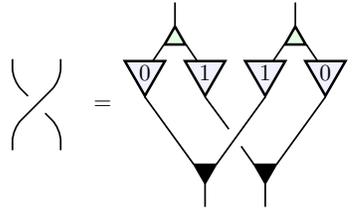
{
   \begin{tikzpicture}[thick]
   \node (UpUpLeft) at (-3.7,2.2) {};
   \node [coordinate] (UpLeft) at (-3.7,1.7) {};
   \node (mid) at (-3.3,1.3) {};
   \node [coordinate] (DownRight) at (-2.9,0.9) {};
   \node (DownDownRight) at (-2.9,0.4) {};
   \node [coordinate] (UpRight) at (-2.9,1.7) {};
   \node (UpUpRight) at (-2.9,2.2) {};
   \node [coordinate] (DownLeft) at (-3.7,0.9) {};
   \node (DownDownLeft) at (-3.7,0.4) {};

   \draw [rounded corners=2mm] (UpUpLeft) -- (UpLeft) -- (mid) --
   (DownRight) -- (DownDownRight) (UpUpRight) -- (UpRight) -- (DownLeft) -- (DownDownLeft);

   \node (eq) at (-2.2,1.3) {\(=\)};
   \node[coordinate] (i1) at (-1,3) {};
   \node[coordinate] (i2) at (1,3) {};
   \node[multiply] (a11) at (-1.5,1.832) {\(0\)};
   \node[multiply] (a21) at (-0.5,1.832) {\(1\)};
   \node[multiply] (a12) at (0.5,1.832) {\(1\)};
   \node[multiply] (a22) at (1.5,1.832) {\(0\)};
   \node[delta] (D1) at (-1,2.4) {};
   \node[delta] (D2) at (1,2.4) {};
   \node (cross) at (0,0.75) {};
   \node[plus] (S1) at (-0.5,0.2) {};
   \node[plus] (S2) at (0.5,0.2) {};
   \node[coordinate] (o1) at (-0.5,-0.4) {};
   \node[coordinate] (o2) at (0.5,-0.4) {};

   \draw
   (i1) -- (D1) (i2) -- (D2)
   (S1.io) -- (o1) (S2.io) -- (o2)
   (a11.io) -- (S1.left in) (a22.io) -- (S2.right in)
   (a21.io) -- (cross) -- (S2.left in) (a12.io) -- (S1.right in)
   (D1.left out) -- (a11) (D1.right out) -- (a21)
   (D2.left out) -- (a12) (D2.right out) -- (a22);
   \end{tikzpicture}
    }.
\end{center}

For \(n > 1\), any morphism built from \(n\) copies of the \Define{basic morphisms}---that is,
generating morphisms and the braiding---can be built up from a morphism built from \(n-1\) copies by
composing or tensoring with one more basic morphism.  Thus, to prove that any string diagram built
from basic morphisms can be put into its standard form, we can proceed by induction on the number of
basic morphisms.

Furthermore, because strings can be extended using the identity morphism, equation \textbf{(13)} can
be used to show tensoring with any generating morphism is equivalent to tensoring with \(1\),
followed by a composition:  \(\Delta = \Delta \of 1\), \(+ = 1 \of +\), \(c = 1 \of c\), \(! =\, !
\of 1\), \(0 = 1 \of 0\).  In the case of braiding, the step of tensoring with \(1\) is repeated
once before making the composition:  \(B = (1 \oplus 1) \of B\).
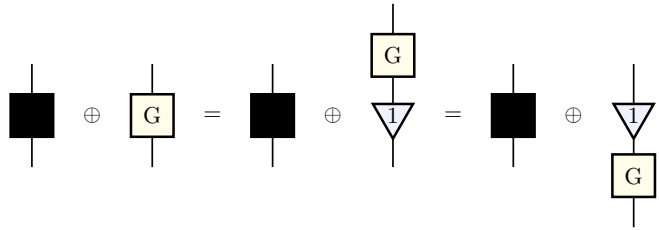
  \begin{center}
    \scalebox{0.80}{
   \begin{tikzpicture}[thick]
   \node[blackbox] (bb1) {};
   \node (tensor1) [right of=bb1] {\(\oplus\)};
   \node[sqnode] (gen1) [right of=tensor1] {G};
   \node (eq1) [right of=gen1] {\(=\)};
   \node[blackbox] (bb2) [right of=eq1] {};
   \node (tensor2) [right of=bb2] {\(\oplus\)};
   \node[sqnode] (gen2) [right of=tensor2, shift={(0,1)}] {G};
   \node (eq2) [right of=gen2, shift={(0,-1)}] {\(=\)};
   \node[blackbox] (bb3) [right of=eq2] {};
   \node (tensor3) [right of=bb3] {\(\oplus\)};
   \node[sqnode] (gen3) [right of=tensor3, shift={(0,-1)}] {G};

   \node[multiply] (times2) [below of=gen2] {\(1\)};
   \node[multiply] (times3) [above of=gen3] {\(1\)};

   \node (bb1in) [above of=bb1] {};
   \node (bb1out) [below of=bb1] {};
   \node (bb2in) [above of=bb2] {};
   \node (bb2out) [below of=bb2] {};
   \node (bb3in) [above of=bb3] {};
   \node (bb3out) [below of=bb3] {};
   \node (gen1in) [above of=gen1] {};
   \node (gen1out) [below of=gen1] {};
   \node (gen2in) [above of=gen2] {};
   \node (gen2out) [below of=times2] {};
   \node (gen3in) [above of=times3] {};
   \node (gen3out) [below of=gen3] {};

   \draw (bb1in) -- (bb1) -- (bb1out) (gen1in) -- (gen1) -- (gen1out)
   (bb2in) -- (bb2) -- (bb2out) (gen2in) -- (gen2) -- (times2) -- (gen2out)
   (bb3in) -- (bb3) -- (bb3out) (gen3in) -- (times3) -- (gen3) -- (gen3out);
   \end{tikzpicture}
    }
  \end{center}
Thus there are 11 cases to consider for this induction:  \(\oplus 1\), \(+ \of\), \(\of \Delta\),
\(\Delta \of\), \(\of +\), \(\of c\), \(c \of\), \(\of 0\), \(! \of\), \(B \of\), \(\of B\).
Without loss of generality, the string diagram \(S\) to which a generating morphism is added will be
assumed to be in standard form already.  Labels \(ij\) on diagrams illustrating these cases
correspond to strings incident to the scalings by \(a_{ij}\).
     \begin{itemize}[leftmargin=1em]
% tensor by \(1\)
\item \Define{\(\oplus 1\)}\\
When tensoring morphisms together, the matrix corresponding to \(C \oplus D\) is the block diagonal
matrix
\[\left[ \begin{array}{cc}
   C & 0 \\
   0 & D \end{array} \right],\]
where, by abuse of notation, the block \(C\) is the matrix corresponding to morphism \(C\), and
respectively \(D\) with \(D\).  Thus, when tensoring \(S\) by \(1\), we write the matrix for \(S\)
with one extra row and one extra column.  Each of these new entries will be \(0\) with the exception
of a \(1\) at the bottom of the extra column.  The string diagram corresponding to the new matrix
can be drawn in standard form as prescribed above.  Using equations \textbf{(14)}, \textbf{(4)}, and
\textbf{(1)}, the standard form reduces to \(S \oplus 1\).  The process is reversible, so if the
string diagram \(S\) can be drawn in standard form, the string diagram \(S \oplus 1\) can be drawn
in standard form, too.  The diagrams below show the relevant strings before they are reduced.
  \begin{center}
    \scalebox{0.80}
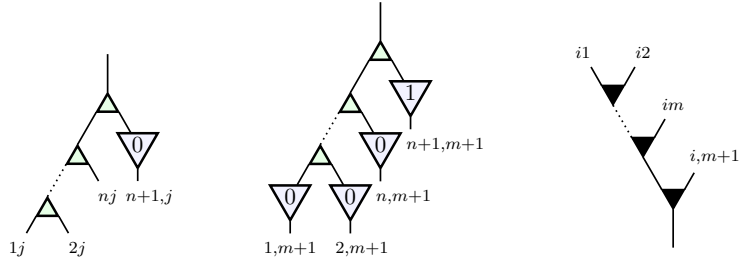
{
   \begin{tikzpicture}[thick]
   \node[delta] (D1) {};
   \node[delta] (D2) at (-0.5,-0.866) {};
   \node[delta] (D3) at (-1,-1.732) {};
   \node[multiply] (zero) at (0.5,-.65) {\(0\)};
   \node (i1) [above of=D1] {};
   \node (o1) at (0.5,-1.516) {\(\quad\scriptstyle{n+1,j}\)};
   \node (o2) at (0,-1.516) {\(\scriptstyle{nj}\)};
   \node (o3) at (-0.5,-2.382) {\(\scriptstyle{2j}\)};
   \node (o4) at (-1.5,-2.382) {\(\scriptstyle{1j}\)};

   \draw (i1) -- (D1) (D1.left out) -- (D2.io) (D1.right out) -- (zero)
   -- (o1) (D2.right out) -- (o2) (D3.right out) -- (o3) (D3.left out) -- (o4);
   \draw[dotted] (D2.left out) -- (D3.io);
   \end{tikzpicture}
       \hspace{1cm}
   \begin{tikzpicture}[thick]
   \node[delta] (D1) {};
   \node[delta] (D2) at (-0.5,-0.866) {};
   \node[delta] (D3) at (-1,-1.732) {};
   \node[multiply] (one) at (0.5,-.65) {\(1\)};
   \node (o1) at (0.5,-1.516) {\(\quad\qquad\scriptstyle{n+1,m+1}\)};
   \node (i1) [above of=D1] {};
   \node[multiply] (z2) at (0,-1.516) {\(0\)};
   \node (o2) at (0,-2.382) {\(\qquad\scriptstyle{n,m+1}\)};
   \node[multiply] (z3) at (-0.5,-2.382) {\(0\)};
   \node (o3) at (-0.5,-3.248) {\(\quad\scriptstyle{2,m+1}\)};
   \node[multiply] (z4) at (-1.5,-2.382) {\(0\)};
   \node (o4) at (-1.5,-3.248) {\(\scriptstyle{1,m+1}\)};

   \draw (i1) -- (D1) (D1.left out) -- (D2.io) (D1.right out) -- (one)
   -- (o1) (D2.right out) -- (z2) -- (o2) (D3.right out) -- (z3) --
   (o3) (D3.left out) -- (z4) -- (o4);
   \draw[dotted] (D2.left out) -- (D3.io);
   \end{tikzpicture}
       \hspace{1cm}
   \begin{tikzpicture}[thick]
   \node[plus] (P1) {};
   \node[plus] (P2) at (-0.5,0.866) {};
   \node[plus] (P3) at (-1,1.732) {};
   \node (zero) at (0.5,.65) {\(\quad\scriptstyle{i,m+1}\)};
   \node (o1) [below of=P1] {};
   \node (i2) at (0,1.516) {\(\scriptstyle{im}\)};
   \node (i3) at (-0.5,2.382) {\(\scriptstyle{i2}\)};
   \node (i4) at (-1.5,2.382) {\(\scriptstyle{i1}\)};

   \draw (o1) -- (P1) (P1.left in) -- (P2.io) (P1.right in) -- (zero)
   (P2.right in) -- (i2) (P3.right in) -- (i3) (P3.left in) -- (i4);
   \draw[dotted] (P2.left in) -- (P3.io);
   \end{tikzpicture}
    }
  \end{center}
Note that for \(i = n+1\), \(a_{i2} = \dotsb = a_{im} = 0\), so the scalings going to the sum
cluster will be \(s_0\), and \(a_{i,m+1} = 1\).  Otherwise \(a_{i,m+1} = 0\), and the rest of the
\(a_{ij}\) depend on the matrix corresponding to \(S\).  When \(S = (! \oplus \dotsb \oplus{} !)\),
the matrix corresponding to \(S \oplus 1\) has a single row, \((0 \cdots 0 \; 1)\), and the standard
form generated is just the middle diagram above.  When the same simplifications are applied, no sum
cluster exists to eliminate the zeros, so the standard form still simplifies to \(S \oplus 1\).
Dually, when \(S = (0 \oplus \dotsb \oplus 0)\), the matrix representation of \(S \oplus 1\) is a
column matrix.  No duplication cluster exists in the standard form for this matrix, so the same
simplifications again reduce to \(S \oplus 1\).
% Compositions - append \(+\)
\item \Define{\(+ \of\)}\\
If we compose the string diagram for addition with \(S\), first consider only the affected clusters
of additions: two clusters are combined into a larger cluster.  Without loss of generality we can
assume these are the first two clusters, or formally, \((+ \oplus 1^{n-2})(S)\).  We can rearrange
the sums using the associative law, equation \textbf{(2)}, and permute the inputs of this large
cluster using the commutative law, equation \textbf{(3)}.  After several iterations of these two
equations, the desired result is obtained:
  \begin{center}
% Plus trees
    \scalebox{0.80}
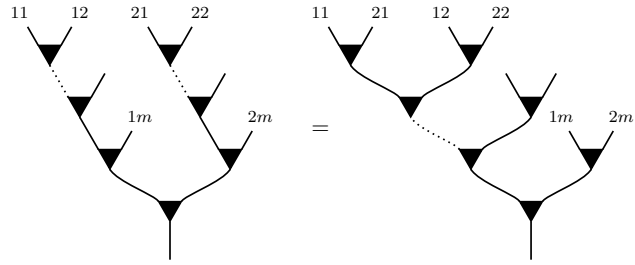
{
   \begin{tikzpicture}[thick]
   \node[plus] (base)                {};
   \node (out) [below of=base]       {};
   \node (Laim)     at (-0.25,0.217) {};
   \node (Lbase)    at (-0.87,0.433) {};
   \node[plus] (L1) at (-1.0,0.866)  {};
   \node[plus] (L2) at (-1.5,1.732)  {};
   \node[plus] (L3) at (-2.0,2.598)  {};
   \node (Raim)     at (0.25,0.217)  {};
   \node (Rbase)    at (0.87,0.433)  {};
   \node[plus] (R1) at (1.0,0.866)   {};
   \node[plus] (R2) at (0.5,1.732)   {};
   \node[plus] (R3) at (0.0,2.598)   {};
   \node (iL1)      at (-0.5,1.516)  {\(\scriptstyle{1m}\)};
   \node (iL2)      at (-1.0,2.382)  {};
   \node (iL3)      at (-1.5,3.248)  {\(\scriptstyle{12}\)};
   \node (iL4)      at (-2.5,3.248)  {\(\scriptstyle{11}\)};
   \node (iR1)      at (1.5,1.516)   {\(\scriptstyle{2m}\)};
   \node (iR2)      at (1.0,2.382)   {};
   \node (iR3)      at (0.5,3.248)   {\(\scriptstyle{22}\)};
   \node (iR4)      at (-0.5,3.248)  {\(\scriptstyle{21}\)};

   \draw (base.left in) .. controls (Laim) and (Lbase) .. (L1.io);
   \draw (base.right in) .. controls (Raim) and (Rbase) .. (R1.io);
   \draw (R1.left in) -- (R2.io) (L1.left in) -- (L2.io) (out) -- (base)
   (L1.right in) -- (iL1) (L2.right in) -- (iL2) (L3.right in) -- (iL3) (L3.left in) -- (iL4)
   (R1.right in) -- (iR1) (R2.right in) -- (iR2) (R3.right in) -- (iR3) (R3.left in) -- (iR4);
   \draw[dotted] (R2.left in) -- (R3.io) (L2.left in) -- (L3.io);

   \node (eq) at (2.5,1.3) {\(=\)};

   \node[plus] (base2) at (6,0)       {};
   \node (basement) [below of=base2]  {};
   \node[plus] (R)     at (7,0.866)   {};
   \node[plus] (L)     at (5,0.866)   {};
   \node[plus] (LR)    at (6,1.732)   {};
   \node[plus] (LL)    at (4,1.732)   {};
   \node[plus] (LLL)   at (3,2.598)   {};
   \node[plus] (LLR)   at (5,2.598)   {};
   \node (iRl)         at (6.5,1.516) {\(\scriptstyle{1m}\)};
   \node (iRr)         at (7.5,1.516) {\(\scriptstyle{2m}\)};
   \node (iLRl)        at (5.5,2.382) {};
   \node (iLRr)        at (6.5,2.382) {};
   \node (iLLRl)       at (4.5,3.248) {\(\scriptstyle{12}\)};
   \node (iLLRr)       at (5.5,3.248) {\(\scriptstyle{22}\)};
   \node (iLLLl)       at (2.5,3.248) {\(\scriptstyle{11}\)};
   \node (iLLLr)       at (3.5,3.248) {\(\scriptstyle{21}\)};

   \draw (base2.left in) .. controls (5.75,0.217) and (5.13,0.433) .. (L.io)
   (base2.right in) .. controls (6.25,0.217) and (6.87,0.433) .. (R.io)
   (base2) -- (basement) (R.right in) -- (iRr) (R.left in) -- (iRl)
   (LR.right in) -- (iLRr) (LR.left in) -- (iLRl)
   (LLR.right in) -- (iLLRr) (LLR.left in) -- (iLLRl)
   (LLL.right in) -- (iLLLr) (LLL.left in) -- (iLLLl)
   (L.right in) .. controls (5.25,1.083) and (5.87,1.299) .. (LR.io)
   (LL.left in) .. controls (3.75,1.949) and (3.13,2.165) .. (LLL.io)
   (LL.right in) .. controls (4.25,1.949) and (4.87,2.165) .. (LLR.io);
   \draw[dotted] (L.left in) .. controls (4.75,1.083) and (4.13,1.299) .. (LL.io);
   \end{tikzpicture}
    }.
  \end{center}
Now the right side of equation \textbf{(12)} appears in the diagram \(m\) times with \(a_{1j}\) and
\(a_{2j}\) in place of \(b\) and \(c\).  Equation \textbf{(12)} can therefore be used to simplify
to the scalings of \(a_{1j}+a_{2j}\).
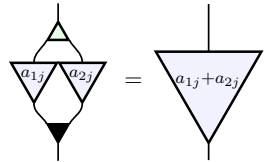
  \begin{center}
    \scalebox{0.80}{
   \begin{tikzpicture}[thick, node distance=0.85cm]
   \node (bctop) {};
   \node [multiply] (bc) [below of=bctop, shift={(0,-0.59)}] {\(\!\!\!{}^{a_{1j}+a_{2j}}\!\!\!\)};
   \node (bcbottom) [below of=bc, shift={(0,-0.59)}] {};

   \draw (bctop) -- (bc) -- (bcbottom);

   \node (eq) [left of=bc, shift={(-0.4,0)}] {\(=\)};

   \node [multiply] (b) [left of=eq, shift={(0,0.1)}] {\(\!\scriptstyle{a_{2j}}\!\)};
   \node [delta] (dupe) [above left of=b, shift={(0.2,0.1)}] {};
   \node (top) [above of=dupe, shift={(0,-0.2)}] {};
   \node [multiply] (c) [below left of=dupe, shift={(0.2,-0.1)}] {\(\!\scriptstyle{a_{1j}}\!\)};
   \node [plus] (adder) [below left of=b, shift={(0.2,-0.3)}] {};
   \node (out) [below of=adder, shift={(0,0.2)}] {};

   \draw
   (dupe.right out) .. controls +(300:0.15) and +(90:0.15) .. (b.90)
   (dupe.left out) .. controls +(240:0.15) and +(90:0.15) .. (c.90)
   (top) -- (dupe.io)
   (adder.io) -- (out)
   (adder.right in) .. controls +(60:0.15) and +(270:0.15) .. (b.io)
   (adder.left in) .. controls +(120:0.15) and +(270:0.15) .. (c.io);
   \end{tikzpicture}
    }
  \end{center}
The simplification removes one instance of \(\Delta\) from each of the \(m\) clusters of \(\Delta\)
and \(m\) instances of \(+\) from the large addition cluster.  There will remain \((m-1) + (m-1) +
(1) - (m) = m-1\) instances of \(+\), which is the correct number for the cluster.  \emph{I.e.}\ the
composition has been reduced to standard form.
\\
The argument is vastly simpler if \(S = (0 \oplus \dotsb \oplus 0)\).  In that case equation
\textbf{(1)} deletes the addition and one of the \(0\) morphisms, and \(S\) is still in the same
form.
  \begin{center}
    \scalebox{0.80}
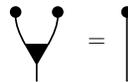
{
   \begin{tikzpicture}[thick]
   \node [plus] (sum) at (0,0) {};
   \node [zero] (zero1) at (-0.35,0.65) {};
   \node [zero] (zero2) at (0.35,0.65) {};
   \node (eq) at (1,0.15) {\(=\)};
   \node [zero] (zero3) at (1.5,0.65) {};
   \node (sumout) at (0,-0.65) {};
   \node (zeroout) at (1.5,-0.65) {};

   \draw (zero1) .. controls +(270:0.2) and +(120:0.2) .. (sum.left in)
         (sum.right in) .. controls +(60:0.2) and +(270:0.2) .. (zero2)
         (sum.io) -- (sumout) (zero3) -- (zeroout);
   \end{tikzpicture}
    }
  \end{center}
% prepend \(\Delta\)
\item \Define{\(\of \Delta\)}\\
The argument for \(S \of (\Delta \oplus 1^{m-2})\) is dual to the above argument, using the light
equations \textbf{(4)}, \textbf{(5)} and \textbf{(6)} instead of the dark equations \textbf{(1)},
\textbf{(2)} and \textbf{(3)}.
% append \(\Delta\)
\item \Define{\(\Delta \of\)}\\
For \((\Delta \oplus 1^{n-1}) \of S\), equation \textbf{(7)} can be used iteratively to `float'
the \(\Delta\) layer above each of the two \(+\) clusters formed by the first iteration.
  \begin{center}
    \scalebox{0.80}
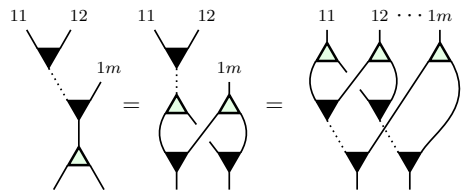
{
   \begin{tikzpicture}[thick]
   \node[plus] (base) {};
   \node[delta] (dub) at (0,-0.866) {};
   \node (oLa) at (-0.5,-1.516) {};
   \node (oRa) at (0.5,-1.516) {};
   \node (1ma) at (0.5,0.65) {\(\scriptstyle{1m}\)};
   \node[plus] (top) at (-0.5,0.866) {};
   \node (11a) at (-1,1.516) {\(\scriptstyle{11}\)};
   \node (12a) at (0,1.516) {\(\scriptstyle{12}\)};

   \draw[dotted] (base.left in) -- (top.io);
   \draw (oLa) -- (dub.left out) (oRa) -- (dub.right out) (dub) -- (base)
   (base.right in) -- (1ma) (top.right in) -- (12a) (top.left in) -- (11a);

   \node (eq1) at (0.875,0) {\(=\)};
   \node (oLb) at (1.625,-1.516) {};
   \node[plus] (baseLb) at (1.625,-0.866) {};
   \node[delta] (dubLb) at (1.625,0) {};
   \node[plus] (topb) at (1.625,0.866) {};
   \node (11b) at (1.125,1.516) {\(\scriptstyle{11}\)};
   \node (12b) at (2.125,1.516) {\(\scriptstyle{12}\)};
   \node (oRb) at (2.5,-1.516) {};
   \node[plus] (baseRb) at (2.5,-0.866) {};
   \node[delta] (dubRb) at (2.5,0) {};
   \node (1mb) at (2.5,0.65) {\(\scriptstyle{1m}\)};
   \node (crossb) at (2.0625,-0.433) {};

   \draw[dotted] (dubLb) -- (topb);
   \draw (11b) -- (topb.left in) (topb.right in) -- (12b) (dubRb) -- (1mb)
   (dubLb.right out) -- (crossb) -- (baseRb.left in) 
   (dubRb.left out) -- (baseLb.right in)
   (baseRb) -- (oRb) (baseLb) -- (oLb);
   \path (baseLb.left in) edge [bend left=30] (dubLb.left out)
   (baseRb.right in) edge [bend right=30] (dubRb.right out);

   \node (eq2) at (3.25,0) {\(=\)};
   \node (11c) at (4.125,1.516) {\(\scriptstyle{11}\)};
   \node[delta] (dubLc) at (4.125,0.866) {};
   \node[plus] (topLc) at (4.125,0) {};
   \node (ucross) at (4.5625,0.433) {};
   \node (12c) at (5,1.516) {\(\scriptstyle{12}\)};
   \node[delta] (dubRc) at (5,0.866) {};
   \node[plus] (topRc) at (5,0) {};
   \node[plus] (baseLc) at (4.625,-0.866) {};
   \node (oLc) at (4.625,-1.516) {};
   \node (1mc) at (6,1.516) {\(\scriptstyle{1m}\)};
   \node (dots) at (5.5,1.516) {\(\cdots\)};
   \node[delta] (dubmc) at (6,0.866) {};
   \node[plus] (baseRc) at (5.5,-0.866) {};
   \node (oRc) at (5.5,-1.516) {};

   \draw[dotted] (baseLc.left in) -- (topLc.io) (baseRc.left in) -- (topRc.io);
   \draw (oRc) -- (baseRc) (oLc) -- (baseLc) (dubmc) -- (1mc) (dubLc) -- (11c) (dubRc) -- (12c)
   (topLc.right in) -- (dubRc.left out) 
   (topRc.left in) -- (ucross) -- (dubLc.right out) 
   (baseLc.right in) -- (dubmc.left out)
   (baseRc.right in) .. controls (5.875,-0.433) and (6.625,0) .. (dubmc.right out);
   \path (topLc.left in) edge [bend left=30] (dubLc.left out)
   (topRc.right in) edge [bend right=30] (dubRc.right out);
   \end{tikzpicture}
    }
  \end{center}
Each of these instances of \(\Delta\) can pass through the scaling layer to \(\Delta\) clusters
using equation \textbf{(17)}.
\\
As before, we consider the subcase \(S = (0 \oplus \dotsb \oplus 0)\) separately.  Equation
\textbf{(8)} removes the duplication and creates a new zero, so \(S\) remains in the same form.
% prepend \(+\)
\item \Define{\(\of +\)}\\
For \(S(+ \oplus 1^{m-1})\), the argument is dual to the previous one: equation \textbf{(7)} is used
to `float' the additions down, equation \textbf{(15)} sends the additions through the
scalings, and equation \textbf{(9)} removes the addition and creates a new deletion in the
subcase \(S = (! \oplus \dotsb \oplus !)\).
% prepend \(c\)
\item \Define{\(\of s_c\)}\\
We can iterate equation \textbf{(17)} when a scaling is composed on top, as in \(S(s_c \oplus 1^{m-1})\).
  \begin{center}
    \scalebox{0.80}
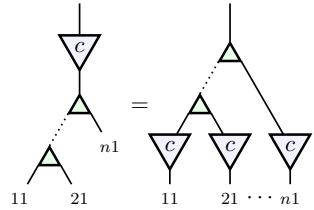
{

   \begin{tikzpicture}[thick]
   \node[multiply] (topc) {\(c\)};
   \node (top1) at (0,0.866) {};
   \node[delta] (dub1) [below of=topc] {};
   \node[delta] (dub2) at (-0.5,-1.866) {};
   \node (11a) at (-1,-2.516) {\(\scriptstyle{11}\)};
   \node (21a) at (0,-2.516) {\(\scriptstyle{21}\)};
   \node (n1a) at (0.5,-1.65) {\(\scriptstyle{n1}\)};

   \draw[dotted] (dub1.left out) -- (dub2.io);
   \draw (top1) -- (topc) -- (dub1) (dub1.right out) -- (n1a)
   (dub2.left out) -- (11a) (dub2.right out) -- (21a);

   \node (eq) at (1,-1) {\(=\)};
   \node (top2) at (2.5,0.866) {};
   \node[delta] (delt1) at (2.5,-0.134) {};
   \node[delta] (delt2) at (2,-1) {};
   \node[multiply] (c1) at (1.5,-1.65) {\(c\)};
   \node[multiply] (c2) at (2.5,-1.65) {\(c\)};
   \node[multiply] (c3) at (3.5,-1.65) {\(c\)};
   \node (11b) at (1.5,-2.516) {\(\scriptstyle{11}\)};
   \node (21b) at (2.5,-2.516) {\(\scriptstyle{21}\)};
   \node (dots) at (3,-2.516) {\(\cdots\)};
   \node (n1b) at (3.5,-2.516) {\(\scriptstyle{n1}\)};

   \draw[dotted] (delt1.left out) -- (delt2.io);
   \draw (delt2.left out) -- (c1) -- (11b) (delt2.right out) -- (c2) -- (21b)
   (delt1.right out) -- (c3) -- (n1b) (top2) -- (delt1);
   \end{tikzpicture}
    }
  \end{center}
The double scalings in the scaling layer reduce to a single scaling via equation \textbf{(11)}, \(s_c
\of a_{ij} = ca_{ij}\), which leaves the diagram in standard form.  The composition does nothing
when \(S = (! \oplus \dotsb \oplus{} !)\), due to equation \textbf{(18)}.
% append \(c\)
\item \Define{\(s_c \of\)}\\
A dual argument can be made for \((s_c \oplus 1^{n-1}) \of S\) using equations \textbf{(15)},
\textbf{(11)} and \textbf{(16)}.
% prepend \(0\)
\item \Define{\(\of 0\)}\\
For \(S(0 \oplus 1^{m-1})\), equations \textbf{(8)} and \textbf{(16)} eradicate the first \(\Delta\)
cluster and all the scalings incident to it, leaving behind \(n\) zeros.  Equation \textbf{(1)}
erases each of these zeros along with one addition per addition cluster, leaving a diagram that is
in standard form.
  \begin{center}
    \scalebox{0.80}
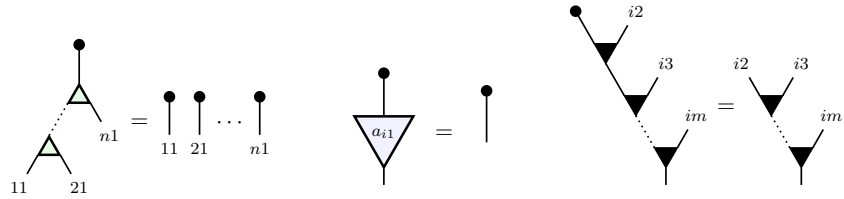
{
   \begin{tikzpicture}[thick]
   \node[zero] (top) {};
   \node[delta] (apex) at (0,-0.866) {};
   \node[delta] (dub) at (-0.5,-1.732) {};
   \node (11a) at (-1,-2.382) {\(\scriptstyle{11}\)};
   \node (21a) at (0,-2.382) {\(\scriptstyle{21}\)};
   \node (n1a) at (0.5,-1.516) {\(\scriptstyle{n1}\)};

   \draw[dotted] (apex.left out) -- (dub.io);
   \draw (top) -- (apex) (apex.right out) -- (n1a)
   (dub.left out) -- (11a) (dub.right out) -- (21a);

   \node (eq) at (1,-1.3) {\(=\)};
   \node[zero] (z1) at (1.5,-0.866) {};
   \node[zero] (z2) at (2,-0.866) {};
   \node[zero] (z3) at (3,-0.866) {};
   \node (dots) at (2.5,-1.3) {\(\cdots\)};
   \node (11b) at (1.5,-1.732) {\(\scriptstyle{11}\)};
   \node (21b) at (2,-1.732) {\(\scriptstyle{21}\)};
   \node (n1b) at (3,-1.732) {\(\scriptstyle{n1}\)};

   \draw (z1) -- (11b) (z2) -- (21b) (z3) -- (n1b);
   \end{tikzpicture}
       \hspace{1cm}
   \begin{tikzpicture}[thick]
   \node[zero] (zip) {};
   \node[multiply] (c) [below of=zip] {\(\scriptstyle{a_{i1}}\)};
   \node (out) [below of=c] {};
   \node (eq) [right of=c] {\(=\)};
   \node[zero] (nada) [above right of=eq] {};
   \node (bot) [below of=nada] {};

   \draw (zip) -- (c) -- (out) (nada) -- (bot);
   \end{tikzpicture}
       \hspace{1cm}
   \begin{tikzpicture}[thick]
   \node[zero] (zip) at (-0.5,0.65) {};
   \node[plus] (topa) {};
   \node[plus] (mida) at (0.5,-0.866) {};
   \node[plus] (bota) at (1,-1.732) {};
   \node (outa) at (1,-2.382) {};
   \node (i2a) at (0.5,0.65) {\(\scriptstyle{i2}\)};
   \node (i3a) at (1,-0.216) {\(\scriptstyle{i3}\)};
   \node (ima) at (1.5,-1.082) {\(\scriptstyle{im}\)};

   \draw[dotted] (mida.io) -- (bota.left in);
   \draw (bota) -- (outa) (bota.right in) -- (ima)
   (mida.right in) -- (i3a) (mida.left in) -- (topa.io)
   (topa.left in) -- (zip) (topa.right in) -- (i2a);

   \node (eq) at (2,-0.95) {\(=\)};
   \node[plus] (top) at (2.75,-0.866) {};
   \node[plus] (bot) at (3.25,-1.732) {};
   \node (out) at (3.25,-2.382) {};
   \node (i2) at (2.25,-0.216) {\(\scriptstyle{i2}\)};
   \node (i3) at (3.25,-0.216) {\(\scriptstyle{i3}\)};
   \node (im) at (3.75,-1.082) {\(\scriptstyle{im}\)};

   \draw[dotted] (top.io) -- (bot.left in);
   \draw (bot.right in) -- (im) (top.right in) -- (i3)
   (top.left in) -- (i2) (bot) -- (out);
   \end{tikzpicture}
    }
  \end{center}
When \(S = (! \oplus \dotsb \oplus{} !)\), the zero annihilates one of the deletions via equation
\textbf{(10)}.
% append \(!\)
\item \Define{\(! \of\)}\\
A dual argument erases the indicated output for the composition \((! \oplus 1^{n-1}) \of S\) using
equations \textbf{(9)}, \textbf{(18)}, and \textbf{(4)}.  Again, equation \textbf{(10)} annihilates
the deletion and one of the zeros if \(S = (0 \oplus \dotsb \oplus 0)\).
% append the braiding
\item \Define{\(B \of\)}\\
Since this category of string diagrams is symmetric monoidal, an appended braiding will naturally
commute with the addition cluster morphisms.  The principle that only the topology matters means the
composition \((B \oplus 1^{n-2}) \of S\) is in standard form.  Braiding will similarly commute with
deletion morphisms.
  \begin{center}
    \scalebox{0.80}{
   \begin{tikzpicture}[thick,node distance=0.5cm]
% Braiding is undone by cut ends
   \node [coordinate] (fstart) {};
   \node [coordinate] (ftop) [below of=fstart] {};
   \node (center) [below right of=ftop] {};
   \node [coordinate] (fout) [below right of=center] {};
   \node [bang] (fend) [below of=fout] {};
   \node [coordinate] (gtop) [above right of=center] {};
   \node [coordinate] (gstart) [above of=gtop] {};
   \node [coordinate] (gout) [below left of=center] {};
   \node [bang] (gend) [below of=gout] {};

   \draw [rounded corners=0.25cm] (fstart) -- (ftop) -- (center) --
   (fout) -- (fend) (gstart) -- (gtop) -- (gout) -- (gend);

   \node (eq1) [below right of=gtop, shift={(0.5,0)}] {\(=\)};
   \node [bang] (ftop1) [above right of=eq1, shift={(0.5,0)}] {};
   \node [coordinate] (fstart1) [above of=ftop1] {};
   \node (center1) [below right of=ftop1] {};
   \node [coordinate] (gtop1) [above right of=center1] {};
   \node [coordinate] (gstart1) [above of=gtop1] {};
   \node [coordinate] (gout1) [below left of=center1] {};
   \node [bang] (gend1) [below of=gout1] {};

   \draw [rounded corners=0.25cm] (fstart1) -- (ftop1)
   (gstart1) -- (gtop1) -- (gout1) -- (gend1);

   \node (eq2) [below right of=gtop1, shift={(0.35,0)}] {\(=\)};
   \node [bang] (big) [right of=eq2, shift={(0.2,-0.5)}] {};
   \node [bang] (bang) [right of=big] {};

   \draw (big) -- +(0,1) (bang) -- +(0,1);
   \end{tikzpicture}
    }
  \end{center}
% prepend the braiding
\item \Define{\(\of B\)}\\
Composing with \(B\) on the top, braiding commutes with duplication, scaling and zero, so \(S \of (B
\oplus 1^{m-2})\) almost trivially comes into standard form.
\end{itemize}
Having exhausted the ways basic morphisms can be attached to a given morphism, this completes the
induction.
\end{proof}

An interesting exercise is to use these equations to derive an equation that expresses the braiding
in terms of other basic morphisms.  One example of such a relation appeared in the introduction,
Section~\ref{sigflow}.  Here is another:
\begin{center}
 \scalebox{0.75}{
   \begin{tikzpicture}[-, thick]
   \node (Lin) {};
   \node [delta] (Ldub) [below of=Lin] {};
   \node [coordinate] (ULang) [below of=Ldub, shift={(-0.5,0)}] {};
   \node [multiply] (Lmin) [below of=ULang, shift={(0,0.65)}] {\(\scriptstyle{-1}\)};
   \node [coordinate] (BLang) [below of=Lmin, shift={(0,0.35)}] {};
   \node [plus] (Lsum) [below of=BLang, shift={(0.5,0)}] {};
   \node [plus] (Usum) [below of=Ldub, shift={(0.9,0)}] {};
   \node [delta] (Bdub) [below of=Usum] {};
   \node [delta] (Rdub) [above of=Usum, shift={(0.9,0)}] {};
   \node (Rin) [above of=Rdub] {};
   \node [coordinate] (URang) [below of=Rdub, shift={(0.5,0)}] {};
   \node [multiply] (Rmin) [below of=URang, shift={(0,0.65)}] {\(\scriptstyle{-1}\)};
   \node [coordinate] (BRang) [below of=Rmin, shift={(0,0.35)}] {};
   \node [plus] (Rsum) [below of=BRang, shift={(-0.5,0)}] {};
   \node (Lout) [below of=Lsum] {};
   \node (Rout) [below of=Rsum] {};

   \draw (Lin) -- (Ldub) (Ldub.right out) -- (Usum.left in) (Usum) --
   (Bdub) (Bdub.left out) -- (Lsum.right in) (Lsum) -- (Lout)
   (Rin) -- (Rdub) (Rdub.left out) -- (Usum.right in) (Bdub.right out)
   -- (Rsum.left in) (Rsum) -- (Rout);
   \draw (Ldub.left out) .. controls +(240:0.5) and +(90:0.3) .. (Lmin.90)
   (Lmin.io) .. controls +(270:0.3) and +(120:0.5) .. (Lsum.left in)
   (Rdub.right out) .. controls +(300:0.5) and +(90:0.3) .. (Rmin.90)
   (Rmin.270) .. controls +(270:0.3) and +(60:0.5) .. (Rsum.right in);

   \node (eq) [left of=Lmin, shift={(-0.3,0)}] {\(=\)};

   \node (center) [left of=eq, shift={(-0.2,0)}] {};
   \node [coordinate] (ftop) [above left of=center, shift={(0.35,-0.35)}] {};
   \node (fstart) [above of=ftop, shift={(0,-0.5)}] {};
   \node [coordinate] (fout) [below right of=center, shift={(-0.35,0.35)}] {};
   \node (fend) [below of=fout, shift={(0,0.5)}] {};
   \node [coordinate] (gtop) [above right of=center, shift={(-0.35,-0.35)}] {};
   \node (gstart) [above of=gtop, shift={(0,-0.5)}] {};
   \node [coordinate] (gout) [below left of=center, shift={(0.35,0.35)}] {};
   \node (gend) [below of=gout, shift={(0,0.5)}] {};

   \draw [rounded corners=7pt] (fstart) -- (ftop) -- (center) --
   (fout) -- (fend) (gstart) -- (gtop) -- (gout) -- (gend);
   \end{tikzpicture}
}.
  \end{center}
With a few more equations, \(\Vectk\) can be presented as merely a monoidal category.  Lafont
\cite{Lafont} mentioned this fact, and gave a full proof in the special case where \(k\) is the
field with two elements.

\section{Presenting \Define{$\relk$}}
\label{presrksection}
\begin{lemma} \label{gensrk} 
For any field \(k\), the PROP \(\relk\) is generated by these morphisms:
\begin{itemize}
\item addition \(+ \maps k \oplus k \asrelto k\)
\item zero \(0 \maps \{0\} \asrelto k\)
\item duplication \(\Delta \maps k \asrelto k \oplus k\)
\item deletion \(! \maps k \asrelto \{0\}\)
\item scaling \(s_c \maps k \asrelto k\) for any \(c \in k\)
\item cup \(\cup \maps k \oplus k \asrelto \{0\} \)
\item cap \(\cap \maps \{0\} \asrelto k \oplus k \)
\end{itemize}
\end{lemma}

\begin{proof}
A morphism of \(\Relk\), \(R\maps k^m \asrelto k^n\) is a subspace of \(k^m \oplus k^n \iso
k^{m+n}\).  In \(\relk\) this isomorphism is an equality.  This subspace can be expressed as a
system of \(k\)-linear equations in \(k^{m+n}\).  Theorem~\ref{presvk} tells us any number of
arbitrary \(k\)-linear combinations of the inputs may be generated.  Any \(k\)-linear equation of
those inputs can be formed by setting such a \(k\)-linear combination equal to zero.  In particular,
if caps are placed on each of the outputs to make them inputs and all the \(k\)-linear combinations
are set equal to zero, any \(k\)-linear system of equations of the inputs and outputs can be
formed.  Expressed in terms of string diagrams,
  
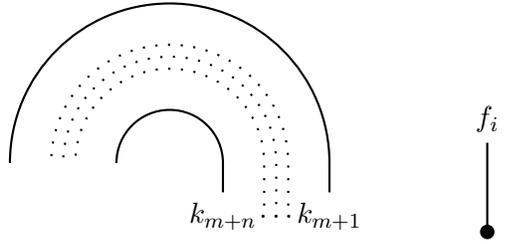
\begin{center}
% outputs become inputs
   \begin{tikzpicture}[-, thick, node distance=0.708cm]
   \node (mn) {\(k_{m+n}\)};
   \node (dots) [right of=mn] {\(\dots\)};
   \node (m1) [right of=dots] {\(k_{m+1}\)};
   \node [coordinate] (inbend1) [above of=m1] {};
   \node [coordinate] (inbend2) [above of=dots] {};
   \node [coordinate] (inbend3) [above of=mn] {};
   \node (blank) [left of=inbend3] {};
   \node [coordinate] (midbend3) [above of=blank] {};
   \node [coordinate] (midbend2) [above of=midbend3] {};
   \node [coordinate] (midbend1) [above of=midbend2] {};
   \node [coordinate] (outbend3) [left of=blank] {};
   \node [coordinate] (outbend2) [left of=outbend3] {};
   \node [coordinate] (outbend1) [left of=outbend2] {};

   \node (dota) [right of=mn, shift={(0.16,0)}] {};
   \node [coordinate] (inbenda2) [above of=dota] {};
   \node [coordinate] (midbenda2) [above of=midbend3, shift={(0,0.16)}] {};
   \node [coordinate] (outbenda2) [left of=outbend3, shift={(-0.16,0)}] {};
   \node (dotd) [right of=mn, shift={(-0.16,0)}] {};
   \node [coordinate] (inbendd2) [above of=dotd] {};
   \node [coordinate] (midbendd2) [above of=midbend3, shift={(0,-0.16)}] {};
   \node [coordinate] (outbendd2) [left of=outbend3, shift={(0.16,0)}] {};

   \draw[loosely dotted,out=-45,in=-135,relative]
   (dota) -- (inbenda2) to (midbenda2) to (outbenda2)
   (dotd) -- (inbendd2) to (midbendd2) to (outbendd2)
   (dots) -- (inbend2) to (midbend2) to (outbend2);

   \draw[out=-45,in=-135,relative]
   (m1) -- (inbend1) to (midbend1) to (outbend1)
   (mn) -- (inbend3) to (midbend3) to (outbend3);
   \end{tikzpicture}
\hspace{1cm}
   \begin{tikzpicture}[-, thick, node distance=1cm]
% relate ins and outs
   \node (f_i) at (0,1.5) {\(f_i\)};
   \node [zero] (zero) at (0,0) {};
   \draw (f_i) -- (zero);
   \end{tikzpicture}.
  \end{center}
The left diagram turns the \(n\) outputs into inputs by placing caps on all of them.  The morphism
zero gives the \(k\)-linear combination zero, so an arbitrary \(k\)-linear combination in
\(k^{m+n}\) is set equal to zero (\(f_i=0\)) via the cozero morphism.  These elements can be
combined with Theorem~\ref{presvk} to express any system of \(k\)-linear equations in \(k^{m+n}\).
\end{proof}

Putting these elements together, taking the \(\vectk\) portion as a black box and drawing a single
string for zero or more copies of \(k\), the picture is fairly simple:
  
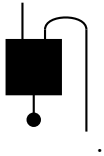
\begin{center}
   \begin{tikzpicture}[thick]
   \node [blackbox] (blackbox) {};
   \node [zero] (zilch) at (0,-0.7) {};
   \node (outs) at (0.7,-1) {};
   \node [coordinate] (capR) at (0.7,0.5) {};
   \node [coordinate] (capL) at (0.15,0.5) {};
   \node (ins) at (-0.15,1) {};

   \path (capL) edge[bend left=90] (capR);
   \draw (ins) -- (-0.15,0) (blackbox) -- (zilch) (capL) -- (0.15,0)
   (capR) -- (outs);
   \end{tikzpicture}.
  \end{center}

To obtain a presentation of \(\relk\) as a PROP, we need to find enough
equations obeyed by the generating morphisms listed in Lemma~\ref{gensrk}.  Equations
{\hyperref[eqn123]{\textbf{(1)--(18)}}} from Theorem~\ref{presvk} still apply, but we need more.  

For convenience, in the list below we draw the adjoint of any generating morphism by rotating it by
\(180^\circ\).  It will follow from equations \textbf{(19)} and \textbf{(20)} that the cap is the adjoint of the cup,
so this convenient trick is consistent even in that case, where \emph{a priori} there might have
been an ambiguity.

\vskip 1em \noindent
\textbf{(19)--(20)} \(\cap\) and \(\cup\) obey the zigzag equations, and thus give a
\(\dagger\)-compact category:
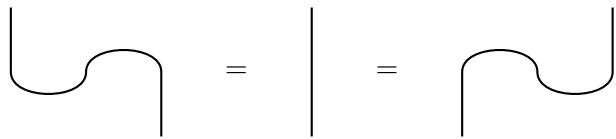
\begin{invisiblelabel}
\label{eqn1920}
\end{invisiblelabel}

\begin{center}
% Zigzag
   \begin{tikzpicture}[-, thick, node distance=1cm]
% Zig
   \node (zigtop) {};
   \node [coordinate] (zigincup) [below of=zigtop] {};
   \node [coordinate] (zigcupcap) [right of=zigincup] {};
   \node [coordinate] (zigoutcap) [right of=zigcupcap] {};
   \node (zigbot) [below of=zigoutcap] {};
   \node (equal) [right of=zigoutcap] {\(=\)};
% Vertical
   \node (mid) [right of=equal] {};
   \node (vtop) [above of=mid] {};
   \node (vbot) [below of=mid] {};
   \node (equals) [right of=mid] {\(=\)};
% Zag
   \node [coordinate] (zagoutcap) [right of=equals] {};
   \node (zagbot) [below of=zagoutcap] {};
   \node [coordinate] (zagcupcap) [right of=zagoutcap] {};
   \node [coordinate] (zagincup) [right of=zagcupcap] {};
   \node (zagtop) [above of=zagincup] {};
% Zigpath
   \path
   (zigincup) edge (zigtop) edge [bend right=90] (zigcupcap)
   (zigoutcap) edge (zigbot) edge [bend right=90] (zigcupcap)
% Verticalpath
   (vtop) edge (vbot)
% Zagpath
   (zagincup) edge (zagtop) edge [bend left=90] (zagcupcap)
   (zagoutcap) edge (zagbot) edge [bend left=90] (zagcupcap);
   \end{tikzpicture}
    \end{center}
   
\vskip 1em \noindent
\textbf{(21)--(22)} \( (k, +, 0, +^\dagger, 0^\dagger)\) is a Frobenius monoid:
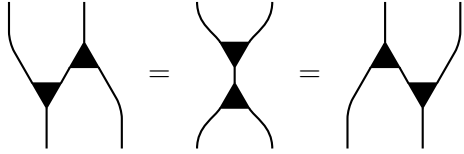
\begin{invisiblelabel}
\label{eqn2122}
\end{invisiblelabel}

\begin{center}
 \scalebox{1}{
% addition gives a Frobenius monoid
   \begin{tikzpicture}[thick]
   \node [plus] (sum1) at (0.5,-0.216) {};
   \node [coplus] (cosum1) at (1,0.216) {};
   \node [coordinate] (sum1corner) at (0,0.434) {};
   \node [coordinate] (cosum1corner) at (1.5,-0.434) {};
   \node [coordinate] (sum1out) at (0.5,-0.975) {};
   \node [coordinate] (cosum1in) at (1,0.975) {};
   \node [coordinate] (1cornerin) at (0,0.975) {};
   \node [coordinate] (1cornerout) at (1.5,-0.975) {};

   \draw[rounded corners] (1cornerin) -- (sum1corner) -- (sum1.left in)
   (1cornerout) -- (cosum1corner) -- (cosum1.right out);
   \draw (sum1.right in) -- (cosum1.left out)
   (sum1.io) -- (sum1out)
   (cosum1.io) -- (cosum1in);

   \node (eq1) at (2,0) {\(=\)};
   \node [plus] (sum2) at (3,0.325) {};
   \node [coplus] (cosum2) at (3,-0.325) {};
   \node [coordinate] (sum2inleft) at (2.5,0.975) {};
   \node [coordinate] (sum2inright) at (3.5,0.975) {};
   \node [coordinate] (cosum2outleft) at (2.5,-0.975) {};
   \node [coordinate] (cosum2outright) at (3.5,-0.975) {};

   \draw (sum2inleft) .. controls +(270:0.3) and +(120:0.15) .. (sum2.left in)
   (sum2inright) .. controls +(270:0.3) and +(60:0.15) .. (sum2.right in)
   (cosum2outleft) .. controls +(90:0.3) and +(240:0.15) .. (cosum2.left out)
   (cosum2outright) .. controls +(90:0.3) and +(300:0.15) .. (cosum2.right out)
   (sum2.io) -- (cosum2.io);

   \node (eq2) at (4,0) {\(=\)};
   \node [plus] (sum3) at (5.5,-0.216) {};
   \node [coplus] (cosum3) at (5,0.216) {};
   \node [coordinate] (sum3corner) at (6,0.434) {};
   \node [coordinate] (cosum3corner) at (4.5,-0.434) {};
   \node [coordinate] (sum3out) at (5.5,-0.975) {};
   \node [coordinate] (cosum3in) at (5,0.975) {};
   \node [coordinate] (3cornerin) at (6,0.975) {};
   \node [coordinate] (3cornerout) at (4.5,-0.975) {};

   \draw[rounded corners] (3cornerin) -- (sum3corner) -- (sum3.right in)
   (3cornerout) -- (cosum3corner) -- (cosum3.left out);
   \draw (sum3.left in) -- (cosum3.right out)
   (sum3.io) -- (sum3out)
   (cosum3.io) -- (cosum3in);
   \end{tikzpicture}
}
\end{center}

\vskip 1em \noindent
\textbf{(23)--(24)} \( (k, \Delta^\dagger, !^\dagger, \Delta, !) \) is a Frobenius monoid:
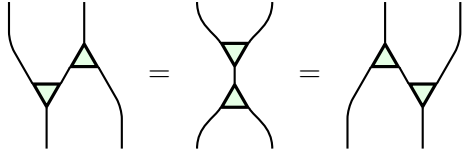
\begin{invisiblelabel}
\label{eqn2324}
\end{invisiblelabel}

\begin{center}
 \scalebox{1}{
% coduplication gives a Frobenius monoid
\begin{tikzpicture}[thick]
   \node [codelta] (sum1) at (0.5,-0.216) {};
   \node [delta] (cosum1) at (1,0.216) {};
   \node [coordinate] (sum1corner) at (0,0.434) {};
   \node [coordinate] (cosum1corner) at (1.5,-0.434) {};
   \node [coordinate] (sum1out) at (0.5,-0.975) {};
   \node [coordinate] (cosum1in) at (1,0.975) {};
   \node [coordinate] (1cornerin) at (0,0.975) {};
   \node [coordinate] (1cornerout) at (1.5,-0.975) {};

   \draw[rounded corners] (1cornerin) -- (sum1corner) -- (sum1.left in)
   (1cornerout) -- (cosum1corner) -- (cosum1.right out);
   \draw (sum1.right in) -- (cosum1.left out)
   (sum1.io) -- (sum1out)
   (cosum1.io) -- (cosum1in);

   \node (eq1) at (2,0) {\(=\)};
   \node [codelta] (sum2) at (3,0.325) {};
   \node [delta] (cosum2) at (3,-0.325) {};
   \node [coordinate] (sum2inleft) at (2.5,0.975) {};
   \node [coordinate] (sum2inright) at (3.5,0.975) {};
   \node [coordinate] (cosum2outleft) at (2.5,-0.975) {};
   \node [coordinate] (cosum2outright) at (3.5,-0.975) {};

   \draw (sum2inleft) .. controls +(270:0.3) and +(120:0.15) .. (sum2.left in)
   (sum2inright) .. controls +(270:0.3) and +(60:0.15) .. (sum2.right in)
   (cosum2outleft) .. controls +(90:0.3) and +(240:0.15) .. (cosum2.left out)
   (cosum2outright) .. controls +(90:0.3) and +(300:0.15) .. (cosum2.right out)
   (sum2.io) -- (cosum2.io);

   \node (eq2) at (4,0) {\(=\)};
   \node [codelta] (sum3) at (5.5,-0.216) {};
   \node [delta] (cosum3) at (5,0.216) {};
   \node [coordinate] (sum3corner) at (6,0.434) {};
   \node [coordinate] (cosum3corner) at (4.5,-0.434) {};
   \node [coordinate] (sum3out) at (5.5,-0.975) {};
   \node [coordinate] (cosum3in) at (5,0.975) {};
   \node [coordinate] (3cornerin) at (6,0.975) {};
   \node [coordinate] (3cornerout) at (4.5,-0.975) {};

   \draw[rounded corners] (3cornerin) -- (sum3corner) -- (sum3.right in)
   (3cornerout) -- (cosum3corner) -- (cosum3.left out);
   \draw (sum3.left in) -- (cosum3.right out)
   (sum3.io) -- (sum3out)
   (cosum3.io) -- (cosum3in);
\end{tikzpicture}
}
\end{center}

\vskip 1em \noindent
\textbf{(25)--(26)} The Frobenius monoid \( (k, +, 0, +^\dagger, 0^\dagger)\) is extra-special:
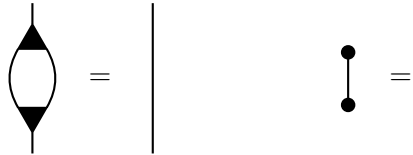
\begin{invisiblelabel}
\label{eqn2526}
\end{invisiblelabel}

\begin{center}
% The Frobenius monoid given by addition is special
\begin{tikzpicture}[thick]
   \node [plus] (sum) at (0.4,-0.5) {};
   \node [coplus] (cosum) at (0.4,0.5) {};
   \node [coordinate] (in) at (0.4,1) {};
   \node [coordinate] (out) at (0.4,-1) {};
   \node (eq) at (1.3,0) {\(=\)};
   \node [coordinate] (top) at (2,1) {};
   \node [coordinate] (bottom) at (2,-1) {};

   \path (sum.left in) edge[bend left=30] (cosum.left out)
   (sum.right in) edge[bend right=30] (cosum.right out);
   \draw (top) -- (bottom)
   (sum.io) -- (out)
   (cosum.io) -- (in);
\end{tikzpicture}
\qquad 
\qquad
\qquad
% The Frobenius monoid given by addition is extra-special
\begin{tikzpicture}[thick]
   \node [zero] (Bins) at (0,-0.35) {};
   \node [zero] (Tins) at (0,0.35) {};
   \node (eq) at (0.7,0) {\(=\)};
   \node [hole] at (0,-0.865) {};
   \draw (Tins) -- (Bins);
   \end{tikzpicture}
\end{center}

\vskip 1em \noindent
\textbf{(27)--(28)} The Frobenius monoid \( (k, \Delta^\dagger, !^\dagger, \Delta, !)\) is
extra-special:
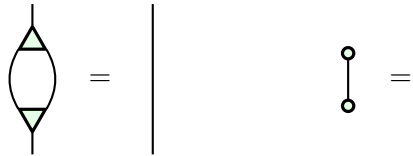
\begin{invisiblelabel}
\label{eqn2728}
\end{invisiblelabel}

\begin{center}
% The Frobenius monoid given by coduplication is special
\begin{tikzpicture}[thick]
   \node [codelta] (sum) at (0.4,-0.5) {};
   \node [delta] (cosum) at (0.4,0.5) {};
   \node [coordinate] (in) at (0.4,1) {};
   \node [coordinate] (out) at (0.4,-1) {};
   \node (eq) at (1.3,0) {\(=\)};
   \node [coordinate] (top) at (2,1) {};
   \node [coordinate] (bottom) at (2,-1) {};

   \path (sum.left in) edge[bend left=30] (cosum.left out)
   (sum.right in) edge[bend right=30] (cosum.right out);
   \draw (top) -- (bottom)
   (sum.io) -- (out)
   (cosum.io) -- (in);
\end{tikzpicture}
\qquad 
\qquad
\qquad
% The Frobenius monoid given by coduplication is extra-special
\begin{tikzpicture}[thick]
   \node [bang] (Bins) at (0,-0.35) {};
   \node [bang] (Tins) at (0,0.35) {};
   \node (eq) at (0.7,0) {\(=\)};
   \node [hole] at (0,-0.865) {};
   \draw (Tins) -- (Bins);
   \end{tikzpicture}
\end{center}

\vskip 1em \noindent
\textbf{(29)} \(\cup\) with a scaling of \(-1\) inserted can be expressed in terms of \(+\) and
\(0\):
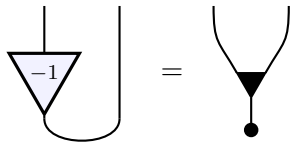
\begin{invisiblelabel}
\label{eqn29}
\end{invisiblelabel}

\begin{center}
\scalebox{1}{
   \begin{tikzpicture}[thick]
   \node [multiply] (neg) at (0,0.1) {\(\scriptstyle{-1}\)};
   \node [coordinate] (cupInLeft) at (0,1) {};
   \node [coordinate] (Lcup) at (0,-0.5) {};
   \node [coordinate] (Rcup) at (1,-0.5) {};
   \node [coordinate] (cupInRight) at (1,1) {};
   \node (eq) at (1.7,0.1) {\(=\)};
   \node [coordinate] (SumLeftIn) at (2.25,1) {};
   \node [coordinate] (SumRightIn) at (3.25,1) {};
   \node [plus] (Sum) at (2.75,0) {};
   \node [zero] (coZero) at (2.75,-0.65) {};

   \draw (SumRightIn) .. controls +(270:0.5) and +(60:0.5) .. (Sum.right in)
      (SumLeftIn) .. controls +(270:0.5) and +(120:0.5) .. (Sum.left in);
   \draw (cupInLeft) -- (neg) -- (Lcup)
      (Rcup) -- (cupInRight)
      (Sum) -- (coZero);
   \path (Lcup) edge[bend right=90] (Rcup);
   \end{tikzpicture}
}
\end{center}

\vskip 1em \noindent
\textbf{(30)} \(\cap\) can be expressed in terms of \(\Delta\) and \(!\):
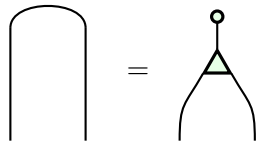
\begin{invisiblelabel}
\label{eqn30}
\end{invisiblelabel}

\begin{center}
\scalebox{1}{
% cap in terms of duplication and codeletion
   \begin{tikzpicture}[thick]
   \node (eq) at (0.2,-0.1) {\(=\)};
   \node [coordinate] (lcap) at (-1.5,0.5) {};
   \node [coordinate] (rcap) at (-0.5,0.5) {};
   \node [coordinate] (lcapbot) at (-1.5,-1) {};
   \node [coordinate] (rcapbot) at (-0.5,-1) {};
   \node [delta] (dub) at (1.25,0) {};
   \node [bang] (boom) at (1.25,0.65) {};
   \node [coordinate] (Leftout) at (0.75,-1) {};
   \node [coordinate] (Rightout) at (1.75,-1) {};

   \draw (dub.left out) .. controls +(240:0.5) and +(90:0.5) .. (Leftout)
      (dub.right out) .. controls +(300:0.5) and +(90:0.5) .. (Rightout);
   \draw (boom) -- (dub) (lcapbot) -- (lcap) (rcap) -- (rcapbot);
   \path (lcap) edge[bend left=90] (rcap);
   \end{tikzpicture}
}
\end{center}

\vskip 1em \noindent
\textbf{(31)} For any \(c \in k\) with \(c \ne 0\), scaling by \(c^{-1}\) is the adjoint of scaling
by \(c\):
\begin{invisiblelabel}
\label{eqn31}
\end{invisiblelabel}

\begin{center}
% Inverse scaling in terms of adjoints
   \begin{tikzpicture}[thick]
   \node[upmultiply] (c) {\(c\)};
   \node[coordinate] (in1) [above of=c] {};
   \node[coordinate] (out1) [below of=c] {};

   \draw (in1) -- (c) -- (out1);

   \node (eq) [right of=c] {\(=\)};
   \node[multiply] (mult) [right of=eq, shift={(0.5,0)}] {\(c^{-1}\!\!\)};
   \node[coordinate] (in) [above of=mult] {};
   \node[coordinate] (out) [below of=mult] {};

   \draw (in) -- (mult) -- (out);
   \end{tikzpicture}.
  \end{center}

\vskip 1em
Some curious identities can be derived from equations \textbf{(1)--(31)}, beyond those already
arising from \textbf{(1)--(18)}.  For example:

\vskip 1em \noindent
\textbf{(D1)--(D2)} Deletion and zero can be expressed in terms of other generating morphisms:
\begin{invisiblelabel}
\label{eqnD1D2}
\end{invisiblelabel}
  \begin{center}
   \begin{tikzpicture}[thick]
% "superfluous" deletion
   \node [bang] (bang) at (0.2,0) {};
   \node (eq1) at (1,1.15) {\(=\)};
   \node (rel1) at (1,0.8) {(27)};
   \node [bang] (buck) at (2,0) {};
   \node [codelta] (codub) at (2,0.65) {};
   \node [delta] (dub) at (2,1.65) {};
   \node (eq2) at (3,1.15) {\(=\)};
   \node (rel2) at (3,0.8) {(30)\({}^\dagger\)};
   \node [delta] (dupe) at (4,1) {};

   \path
   (codub.left in) edge[bend left=30] (dub.left out)
   (codub.right in) edge[bend right=30] (dub.right out);
   \draw
   (bang) -- +(up:2.3) (dub.io) -- (2,2.3) (dupe.io) -- (4,2.3)
   (codub.io) -- (buck)
   (dupe.left out) .. controls +(240:1.2) and +(300:1.2) .. (dupe.right out);
   \end{tikzpicture}
\hskip 2em
   \begin{tikzpicture}[thick]
% "superfluous" zero
   \node [zero] (Zero1) at (0,1) {};
   \node (eq1) at (0.8,-0.15) {\(=\)};
   \node (rel1) at (0.8,-0.5) {(28)};
   \node [bang] (cobang) at (1.6,1) {};
   \node [bang] (bang) at (1.6,0.2) {};
   \node [zero] (Zero2) at (1.6,-0.5) {};
   \node (eq2) at (2.4,-0.15) {\(=\)};
   \node (rel2) at (2.4,-0.5) {(14)};
   \node [multiply] (times) at (3.4,-0.15) {\(0\)};
   \node [bang] (cobuck) at (3.4,1) {};
   \node (eq3) at (4.4,-0.15) {\(=\)};
   \node (rel3) at (4.4,-0.5) {(D1)\({}^\dagger\)};
   \node [multiply] (oh) at (5.4,-0.5) {\(0\)};
   \node [codelta] (cod) at (5.4,0.2) {};

   \draw
   (Zero1) -- +(down:2.3) (cobang) -- (bang) (Zero2) -- +(down:0.8)
   (cobuck) -- (times) -- (3.4,-1.3) (cod.io) -- (oh) -- (5.4,-1.3)
   (cod.left in) .. controls +(120:1.2) and +(60:1.2) .. (cod.right in);
   \end{tikzpicture}.
  \end{center}
This does not diminish the role of deletion and zero.  Indeed, regarding these generating morphisms
as superfluous buries some of the structure of \(\relk\).

\vskip 1em \noindent
\textbf{(D3)} Addition can be expressed in terms of coaddition and scaling by \(-1\), and the cup:
\begin{invisiblelabel}
\label{eqnD3}
\end{invisiblelabel}
  \begin{center}
    \scalebox{1}{
   \begin{tikzpicture}[thick]
% Addition related to coaddition with a cup and a negative
   \node [multiply] (neg1) at (0,0) {\(\scriptstyle{-1}\)};
   \node [coplus] (cosum1) at (1,-0.21) {};

   \draw
   (neg1) -- (0,1) (cosum1.io) -- (1,1)
   (cosum1.right out) .. controls +(300:0.5) and +(90:0.5) .. +(0.3,-1.3)
   (neg1.io) .. controls +(270:1) and +(240:0.5) .. (cosum1.left out);
   % convert the cup
   \node (eq2) at (2,-0.216) {\(=\)};
   \node (rel2) at (2,-0.566) {(29)};
   \node [plus] (sum3) at (3,-0.432) {};
   \node [zero] (coz3) at (3,-1) {};
   \node [coplus] (cosum3) at (3.5,0) {};

   \draw
   (sum3.left in) .. controls +(120:0.5) and +(270:0.5) .. +(-0.3,1.3)
   (sum3) -- (coz3) (sum3.right in) -- (cosum3.left out)
   (cosum3.right out) .. controls +(300:0.5) and +(90:0.5) .. +(0.3,-1.3)
   (cosum3.io) -- +(0,0.76);
   % apply Frobenius
   \node (eq3) at (4.5,-0.216) {\(=\)};
   \node (rel3) at (4.5,-0.566) {(21)};
   \node [plus] (sum4) at (5.5,0.159) {};
   \node [coplus] (cosum4) at (5.5,-0.591) {};
   \node [zero] (coz4) at (5.1,-1.35) {};

   \draw
   (sum4.io) -- (cosum4.io)
   (cosum4.left out) .. controls +(240:0.2) and +(90:0.25) .. (coz4)
   (sum4.left in) .. controls +(120:0.5) and +(270:0.5) .. +(-0.3,1)
   (sum4.right in) .. controls +(60:0.5) and +(270:0.5) .. +(0.3,1)
   (cosum4.right out) .. controls +(300:0.5) and +(90:0.5) .. +(0.3,-1);
   % delete the cosum
   \node (eq4) at (6.5,-0.216) {\(=\)};
   \node (rel4) at (6.5,-0.566) {(1)\({}^\dagger\)};
   \node [plus] (sum5) at (7.5,-0.4) {};

   \draw
   (sum5.io) -- +(0,-0.5)
   (sum5.left in) .. controls +(120:0.5) and +(270:0.5) .. +(-0.3,1)
   (sum5.right in) .. controls +(60:0.5) and +(270:0.5) .. +(0.3,1);
   \end{tikzpicture}
    }.
  \end{center}

\vskip 1em \noindent
\textbf{(D4)} Duplication can be expressed in terms of coduplication and the cap:
\begin{invisiblelabel}
\label{eqnD4}
\end{invisiblelabel}
  \begin{center}
    \scalebox{1}{
   \begin{tikzpicture}[thick]
   \node [coordinate] (neg1) at (0,0) {};
   \node [codelta] (codub) at (1,0) {};

   \draw
   (neg1) -- (0,-0.91) (codub.io) -- (1,-0.91)
   (codub.right in) .. controls +(60:0.5) and +(270:0.5) .. +(0.3,1)
   (neg1) .. controls +(90:1) and +(120:0.5) .. (codub.left in);

   \node (eq) at (2,0) {\(=\)};
   \node [delta] (dub) at (3,0.2) {};

   \draw
   (dub.io) -- +(0,0.69)
   (dub.left out) .. controls +(240:0.5) and +(90:0.5) .. +(-0.3,-1)
   (dub.right out) .. controls +(300:0.5) and +(90:0.5) .. +(0.3,-1);
   \end{tikzpicture}
    },
  \end{center}
where the proof is similar to that of \textbf{(D3)}.

\vskip 1em
\vbox{
\noindent
\textbf{(D5)--(D7)} We can reformulate the bimonoid equations (7)--(9) using daggers:
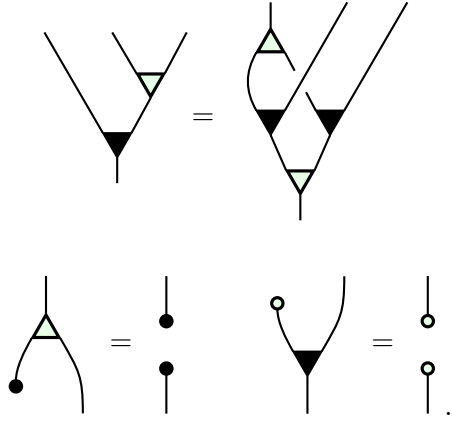
\begin{invisiblelabel}
\label{eqnD5D6D7}
\end{invisiblelabel}
  \begin{center}
    \scalebox{1}{
   \begin{tikzpicture}[thick, node distance=0.7cm]
% Addition + Coduplication
   \node (in1z) {};
   \node (in2z) [right of=in1z, shift={(0.2,0)}] {};
   \node (in3z) [right of=in2z, shift={(0.45,0)}] {};
   \node [codelta] (nabzip) [below right of=in2z, shift={(0.1,-0.3)}] {};
   \node [plus] (add) [below left of=nabzip, shift={(0.05,-0.3)}] {};
   \node (outz) [below of=add] {};
   \node (equal) [below right of=nabzip, shift={(0.2,0)}] {\(=\)};
   % Unzipped
   \node [plus] (addl) [right of=equal, shift={(0.2,0)}] {};
   \node (cross) [above right of=addl, shift={(-0.1,0)}] {};
   \node [delta] (delta) [above left of=cross, shift={(0.1,0)}] {};
   \node (in1u) [above of=delta] {};
   \node [plus] (addr) [below right of=cross, shift={(-0.1,0)}] {};
   \node [codelta] (nablunzip) [below left of=addr, shift={(0.1,-0.3)}] {};
   \node (outu) [below of=nablunzip] {};
   \node (in2u) [right of=in1u, shift={(0.4,0)}] {};
   \node (in3u) [right of=in2u, shift={(0.1,0)}] {};

   % Zipped path
   \draw (in1z) -- (add.left in) (add) -- (outz) (in2z) --
   (nabzip.left in) (in3z) -- (nabzip.right in) (nabzip.io) -- (add.right in);
   % Unzipped path
   \path
   (delta.left out) edge [bend right=30] (addl.left in);
   \draw (in1u) -- (delta) (delta.right out) -- (cross) -- (addr.left in);
   \draw (in2u) -- (addl.right in) (in3u) -- (addr.right in);
   \draw (addl.io) -- (nablunzip.left in) (addr.io) -- (nablunzip.right in)
   (nablunzip) -- (outu);
   \end{tikzpicture}
}
\end{center}
\begin{center}
\scalebox{1}{
   \begin{tikzpicture}[thick]
% Co0 + Duplication
   \node [zero] (coz1) at (0,0) {};
   \node [delta] (dub) at (0.4,0.75) {};

   \draw
   (coz1) .. controls +(90:0.25) and +(240:0.2) .. (dub.left out)
   (dub.right out) .. controls +(300:0.5) and +(90:0.5) .. +(0.3,-1)
   (dub.io) -- +(0,0.5);

   \node (eq) at (1.4,0.552) {\(=\)};
   \node [zero] (coz2) at (2,0.867) {};
   \node [zero] (zero) at (2,0.237) {};

   \draw (coz2) -- +(0,0.6) (zero) -- +(0,-0.6);
   \end{tikzpicture}
\hskip 3em
   \begin{tikzpicture}[thick]
% Co! + Addition
   \node [bang] (cobang1) at (0,0) {};
   \node [plus] (add) at (0.4,-0.75) {};

   \draw
   (cobang1) .. controls +(270:0.25) and +(120:0.2) .. (add.left in)
   (add.right in) .. controls +(60:0.5) and +(270:0.5) .. +(0.3,1)
   (add.io) -- +(0,-0.5);

   \node (eq) at (1.4,-0.552) {\(=\)};
   \node [bang] (cobang2) at (2,-0.867) {};
   \node [bang] (bang) at (2,-0.237) {};

   \draw (bang) -- +(0,0.6) (cobang2) -- +(0,-0.6);
   \end{tikzpicture}
    }.
  \end{center}}

\vskip 1em \noindent
\textbf{(D8)--(D9)} When \(c \ne 1\), we have:
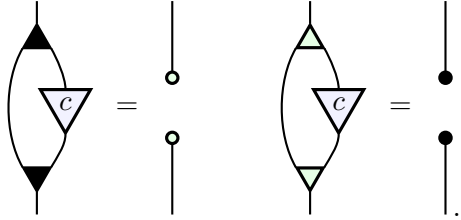
\begin{invisiblelabel}
\label{eqnD8D9}
\end{invisiblelabel}
  \begin{center}
   \begin{tikzpicture}[thick]
% Unspecializing the dark triangles
   \node [coplus] (cosum) at (0,1.8) {};
   \node [multiply] (times) at (0.38,0.95) {\(c\)};
   \node [plus] (sum) at (0,0) {};

   \draw
   (cosum.io) -- +(0,0.3) (sum.io) -- +(0,-0.3)
   (cosum.left out) .. controls +(240:0.5) and +(120:0.5) .. (sum.left in)
   (cosum.right out) .. controls +(300:0.15) and +(90:0.15) .. (times.90)
   (times.io) .. controls +(270:0.15) and +(60:0.15) .. (sum.right in);

   \node (eq) at (1.2,0.9) {\(=\)};
   \node [bang] (bang) at (1.8,1.3) {};
   \node [bang] (cobang) at (1.8,0.5) {};

   \draw (bang) -- +(0,1.02) (cobang) -- +(0,-1.02);
   \end{tikzpicture}
\hskip 3em
   \begin{tikzpicture}[thick]
% Unspecializing the light triangles
   \node [delta] (dub) at (0,1.8) {};
   \node [multiply] (times) at (0.38,0.95) {\(c\)};
   \node [codelta] (codub) at (0,0) {};

   \draw
   (dub.io) -- +(0,0.3) (sum.io) -- +(0,-0.3)
   (dub.left out) .. controls +(240:0.5) and +(120:0.5) .. (codub.left in)
   (dub.right out) .. controls +(300:0.15) and +(90:0.15) .. (times.90)
   (times.io) .. controls +(270:0.15) and +(60:0.15) .. (codub.right in);

   \node (eq) at (1.2,0.9) {\(=\)};
   \node [zero] (coz) at (1.8,1.3) {};
   \node [zero] (zero) at (1.8,0.5) {};

   \draw (coz) -- +(0,1.02) (zero) -- +(0,-1.02);
   \end{tikzpicture}.
  \end{center}
\noindent
Derived equations \textbf{(D5)--(D8)} are used below, and their proofs can be found in
Appendix~\ref{derivedeqns}.  While derived equation \textbf{(D9)} is not used below, it is dual to
equation \textbf{(D8)}.  With a different standard form on \(\relk\), equation \textbf{(D9)} would
be used in the proof of Theorem~\ref{presrk} below instead of equation \textbf{(D8)}.

Next we show that equations \textbf{(1)--(31)} are enough to give a presentation of \(\relk\) as a
PROP.  In terms of \smts{}, this means \(\relk\) is the coequalizer of
\(\F E_{\relk} \rightrightarrows \F\Sigma_{\relk}\), where the generating morphisms are the
\(\F\)-images of the signature \(\Sigma_{\relk}\), and the \(\Sigma\)-terms of \(E_{\relk}\) are
these 31 equations.  As before, we demonstrate the presentation by giving a standard form that any
\(\relk\) morphism can be written in and use induction to show that an arbitrary diagram can be
rewritten in its standard form using the given equations.

\begin{theorem} \label{presrk} The PROP \(\relk\) is presented by the morphisms given in
Lemma~\ref{gensrk}, and equations {\bf (1)--(31)} as listed above.
\end{theorem}

\begin{proof}
We prove this theorem by using the equations \textbf{(1)--(31)} to put any string diagram built from
the generating morphisms and braiding into a standard form, so that any two string diagrams
corresponding to the same morphism in \(\relk\) have the same standard form.

As before, we induct on the number of \Define{basic morphisms} involved in a string diagram, where
the basic morphisms are the generating morphisms together with the braiding.  If we let \(R \maps
k^m \asrelto k^n\) be a morphism in \(\relk\), we can build a string diagram \(S\) for \(R\) as in
Lemma~\ref{gensrk}.  Each output of \(S\) is capped, and, together with the inputs of \(S\), form
inputs for a \(\vectk\) block, \(T\).  For some \(r \leq m+n\), there are \(r\) outputs of
\(T\)--linear combinations of the \(m+n\) inputs--each set equal to zero via \((0^\dagger)^r\).
When \(T\) is in standard form for \(\vectk\), we say \(S\) is in \Define{prestandard form}, and can
be depicted as follows:
  \begin{center}
    \scalebox{1}
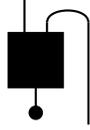
{
   \begin{tikzpicture}[thick]
   \node [blackbox] (blackbox) {};
   \node [zero] (zilch) at (0,-0.7) {};
   \node (outs) at (0.7,-1) {};
   \node [coordinate] (capR) at (0.7,0.5) {};
   \node [coordinate] (capL) at (0.15,0.5) {};
   \node (ins) at (-0.15,1) {};

   \path (capL) edge[bend left=90] (capR);
   \draw (ins) -- (-0.15,0) (blackbox) -- (zilch) (capL) -- (0.15,0)
   (capR) -- (outs);
   \end{tikzpicture}
    }.
  \end{center}
While the linear subspace of \(k^{m+n}\) defined by \(R\) is determined by a system of \(r\) linear
equations, the converse is not true, meaning there may be multiple prestandard string diagrams for a
single morphism \(R\).  The second stage of this proof collapses all the prestandard forms into a
standard form using some basic linear algebra.  The standard form will correspond to when the matrix
representation of \(T\) is written in row-reduced echelon form.  For this stage it will suffice to
show all the elementary row operations correspond to equations that hold between diagrams.  By
Theorem \ref{presvk}, an arbitrary \(\vectk\) block can be rewritten in its standard form, so the
\(\vectk\) blocks here need not be demonstrated in their standard form.
\\
There are eight base cases of a string diagram with one basic morphism to consider, one case per
basic morphism.  In each of these basic cases, the block of the diagram equivalent to a morphism in
\(\vectk\) is denoted by a dashed rectangle.  We first consider \(\cup\).

\vskip 1em \noindent
\textbf{(D10)}
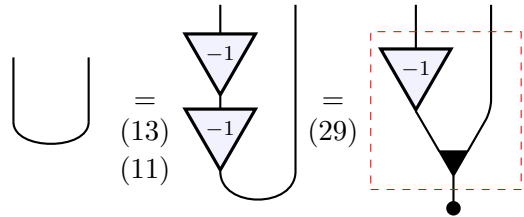
  \begin{center}
    \scalebox{1}{
% Demonstrating the case for cup
   \begin{tikzpicture}[thick]
   \node (Lin0) at (0,1) {};
   \node [coordinate] (Lcup) at (0,0) {};
   \node [coordinate] (Rcup) at (1,0) {};
   \node (Rin0) at (1,1) {};
   \node (eq1) at (1.75,0.25) {\(=\)};
   \node (rel1) at (1.75,-0.15) {(13)};
   \node (rel2) at (1.75,-0.65) {(11)};
   \node (Lin1) at (2.75,1.7) {};
   \node [multiply] (neg1a) at (2.75,0.9) {\(\scriptstyle{-1}\)};
   \node [multiply] (neg1b) at (2.75,-0.1) {\(\scriptstyle{-1}\)};
   \node (Rin1) at (3.75,1.7) {};
   \node (eq2) at (4.25,0.25) {\(=\)};
   \node (rel3) at (4.25,-0.15) {(29)};
   \node (Lin4) at (5.35,1.7) {};
   \node [multiply] (neg) at (5.35,0.7) {\(\scriptstyle{-1}\)};
   \node (Rin4) at (6.35,1.7) {};
   \node [coordinate] (Rcor4) at (6.35,0.15) {};
   \node [plus] (S4) at (5.85,-0.5) {};
   \node [zero] (coZ4) at (5.85,-1.15) {};

   \draw[rounded corners]
   (Rin4) -- (Rcor4) -- (S4.right in);
   \draw (Lin0) -- (Lcup) (Rcup) -- (Rin0)
   (Lin1) -- (neg1a) -- (neg1b)
   (neg1b.io) .. controls +(270:0.5) and +(270:0.5) .. +(1,0) -- (Rin1)
   (Lin4) -- (neg) (neg.io) -- (S4.left in) (S4) -- (coZ4);
   \draw [color=red, dashed, thin] (4.75,1.2) rectangle (6.75,-0.9);
   \path (Lcup) edge[bend right=90] (Rcup);
   \end{tikzpicture}
    }
  \end{center}
Capping each of the inputs turns this into the standard form of \(\cap\).  Aside from deletion, the
remaining generating morphisms can be formed by introducing a zigzag at each output and rewriting
the resulting cups as above.  The standard forms for \(0\) and \(!\) have simpler expressions.
  \begin{center}
    \scalebox{0.80}
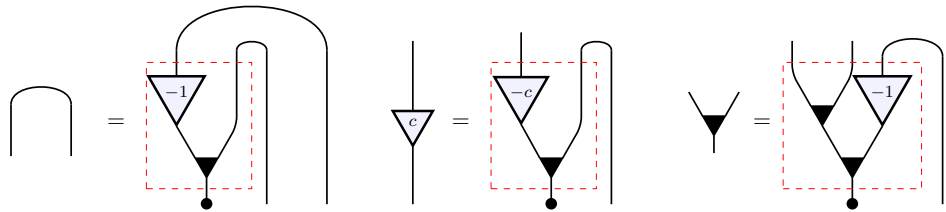
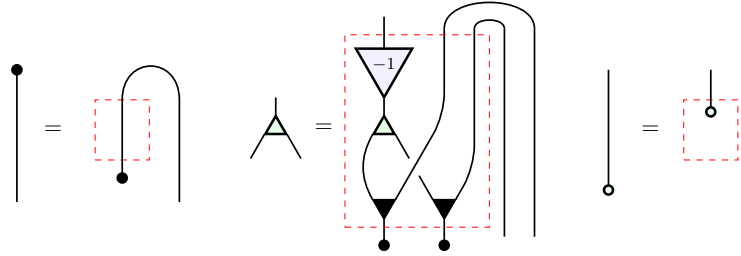
{
   \begin{tikzpicture}[thick]
% Cap equivalent
   \node (Lout0) at (-2.75,-0.5) {};
   \node [coordinate] (Lcap0) at (-2.75,0.5) {};
   \node [coordinate] (Rcap0) at (-1.75,0.5) {};
   \node (Rout0) at (-1.75,-0.5) {};
   \node (eq) at (-1,0.2) {\(=\)};
   \node [coordinate] (Lin) at (0,1.4) {};
   \node [multiply] (neg) at (0,0.7) {\(\scriptstyle{-1}\)};
   \node [coordinate] (Rin) at (1,1.4) {};
   \node [coordinate] (Rcor) at (1,0.15) {};
   \node [plus] (S) at (0.5,-0.5) {};
   \node [zero] (coZ) at (0.5,-1.15) {};
   \node [coordinate] (Rcap) at (2.5,1.4) {};
   \node [coordinate] (Lcap) at (1.5,1.4) {};
   \node (Rout) at (2.5,-1.3) {};
   \node (Lout) at (1.5,-1.3) {};

   \draw (Lout0) -- (Lcap0) (Rcap0) -- (Rout0) (Lin) -- (neg) (S) -- (coZ)
   (neg.io) -- (S.left in) (Rcap) -- (Rout) (Lcap) -- (Lout);
   \path (Lcap0) edge[bend left=90] (Rcap0) (Lin) edge[bend left=90]
   (Rcap) (Rin) edge[bend left=90] (Lcap);
   \draw[rounded corners] (Rin) -- (Rcor) -- (S.right in);
   \draw[color=red, dashed, thin] (-0.5,1.2) rectangle (1.25,-0.9);
   \end{tikzpicture}
       \hspace{0.69cm}
   \begin{tikzpicture}[thick]
% Multiplication by c equivalent
   \node (in) at (-1.8,1.7) {};
   \node [multiply] (times) at (-1.8,0.2) {\(\scriptstyle{c}\)};
   \node (out) at (-1.8,-1.3) {};
   \node (eq) at (-1,0.2) {\(=\)};
   \node [coordinate] (Lin) at (0,1.7) {};
   \node [multiply] (neg) at (0,0.7) {\(\scriptstyle{-c}\)};
   \node [coordinate] (Rin) at (1,1.4) {};
   \node [coordinate] (Rcor) at (1,0.15) {};
   \node [plus] (S) at (0.5,-0.5) {};
   \node [zero] (coZ) at (0.5,-1.15) {};
   \node [coordinate] (Lcap) at (1.5,1.4) {};
   \node (Lout) at (1.5,-1.3) {};

   \draw (in) -- (times) -- (out) (Lin) -- (neg) (S) -- (coZ)
   (neg.io) -- (S.left in) (Lcap) -- (Lout);
   \path (Rin) edge[bend left=90] (Lcap);
   \draw[rounded corners] (Rin) -- (Rcor) -- (S.right in);
   \draw[color=red, dashed, thin] (-0.5,1.2) rectangle (1.25,-0.9);
   \end{tikzpicture}
       \hspace{0.69cm}
   \begin{tikzpicture}[thick]
% Sum equivalent
   \node (in1) at (-2.3,0.85) {};
   \node (in2) at (-1.3,0.85) {};
   \node [plus] (S) at (-1.8,0.2) {};
   \node (out) at (-1.8,-0.45) {};
   \node (eq) at (-1,0.2) {\(=\)};
   \node [coordinate] (Lin) at (-0.5,1.016) {};
   \node (Linup) at (-0.5,1.7) {};
   \node [coordinate] (Rin) at (0.5,1.016) {};
   \node (Rinup) at (0.5,1.7) {};
   \node [multiply] (neg) at (1,0.7) {\(\scriptstyle{-1}\)};
   \node [coordinate] (Lcap) at (1,1.3) {};
   \node [plus] (Stop) at (0,0.366) {};
   \node [plus] (Sbot) at (0.5,-0.5) {};
   \node [zero] (coZ) at (0.5,-1.15) {};
   \node [coordinate] (Rcap) at (2,1.3) {};
   \node (Lout) at (2,-1.3) {};

   \draw[rounded corners] (Linup) -- (Lin) -- (Stop.left in)
   (Rinup) -- (Rin) -- (Stop.right in);
   \draw (in1) -- (S.left in) (S.right in) -- (in2) (S) -- (out)
   (Stop.io) -- (Sbot.left in)
   (Sbot) -- (coZ) (neg.io) -- (Sbot.right in) (neg) -- (Lcap)
   (Rcap) -- (Lout);
   \path (Lcap) edge[bend left=90] (Rcap);
   \draw[color=red, dashed, thin] (-0.65,1.2) rectangle (1.65,-0.9);
   \end{tikzpicture}
    }

    \scalebox{0.80}{
   \begin{tikzpicture}[thick]
% Zero equivalent
   \node [zero] (zero1) at (-0.2,1) {};
   \node (eq) at (0.4,0) {\(=\)};
   \node [zero] (coz) at (1.55,-0.8) {};
   \node [hole] (placeholder) at (0,-1.9) {};

   \draw[color=red, dashed, thin] (2,0.5) rectangle (1.1,-0.5);
   \draw (zero1) -- +(down:2.2) (coz) -- (1.55,0.5)
   .. controls +(90:0.75) and +(90:0.75) .. (2.5,0.5) -- (2.5,-1.2);
   \end{tikzpicture}
       \hspace{0.7cm}
   \begin{tikzpicture}[thick]
% Delta equivalent
   \node (out1) at (-2.3,0) {};
   \node (out2) at (-1.3,0) {};
   \node [delta] (D) at (-1.8,0.65) {};
   \node (in) at (-1.8,1.3) {};
   \node (eq) at (-1,0.65) {\(=\)};
   \node [coordinate] (in1) at (0,2.5) {};
   \node [multiply] (neg) at (0,1.7) {\(\scriptstyle{-1}\)};
   \node [delta] (D1) at (0,0.65) {};
   \node [plus] (S1) at (0,-0.65) {};
   \node [plus] (S2) at (1,-0.65) {};
   \node [zero] (coZ1) at (0,-1.3) {};
   \node [zero] (coZ2) at (1,-1.3) {};

   \node (cross) at (0.5,0) {};
   \node [coordinate] (Loutup) at (1,0.866) {};
   \node [coordinate] (Routup) at (1.5,0) {};
   \node [coordinate] (oLcap) at (1,2.3) {};
   \node [coordinate] (iLcap) at (1.5,2.3) {};
   \node [coordinate] (iRcap) at (2,2.3) {};
   \node [coordinate] (oRcap) at (2.5,2.3) {};
   \node (Lout) at (2,-1.3) {};
   \node (Rout) at (2.5,-1.3) {};

   \draw[rounded corners=10pt] (S1.right in) -- (Loutup) -- (oLcap)
   (iLcap) -- (Routup) -- (S2.right in);
   \draw (out1) -- (D.left out) (D.right out) -- (out2) (D) -- (in)
   (in1) -- (neg) -- (D1) (D1.right out) -- (cross) -- (S2.left in)
   (S2) -- (coZ2) (S1) -- (coZ1) (iRcap) -- (Lout) (oRcap) -- (Rout);
   \path (oLcap) edge[bend left=90] (oRcap) (iLcap) edge[bend left=90]
   (iRcap) (D1.left out) edge[bend right=30] (S1.left in);
   \draw[color=red, dashed, thin] (-0.65,2.2) rectangle (1.75,-1);
   \end{tikzpicture}
       \hspace{0.7cm}
   \begin{tikzpicture}[thick]
% Deletion equivalent
   \node [bang] (del1) at (0,-1) {};
   \node (eq) at (0.7,0) {\(=\)};
   \node [bang] (del2) at (1.7,0.3) {};
   \node [hole] (placeholder) at (0,-1.9) {};

   \draw[color=red, dashed, thin] (1.25,-0.5) rectangle (2.15,0.5);
   \draw (del1) -- +(up:2) (del2) -- +(up:0.7)
   ;
   \end{tikzpicture}
    }
  \end{center}
Braiding is two copies of \(s_1\) (scaling by 1) that have been braided together.
  \begin{center}
    \scalebox{0.80}
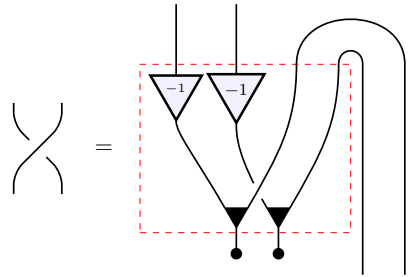
{
% Braiding equivalent
   \begin{tikzpicture}[thick]
   \node (UpUpLeft) at (-3.7,2) {};
   \node [coordinate] (UpLeft) at (-3.7,1.5) {};
   \node (mid) at (-3.3,1.1) {};
   \node [coordinate] (DownRight) at (-2.9,0.7) {};
   \node (DownDownRight) at (-2.9,0.2) {};
   \node [coordinate] (UpRight) at (-2.9,1.5) {};
   \node (UpUpRight) at (-2.9,2) {};
   \node [coordinate] (DownLeft) at (-3.7,0.7) {};
   \node (DownDownLeft) at (-3.7,0.2) {};

   \draw [rounded corners=2mm] (UpUpLeft) -- (UpLeft) -- (mid) --
   (DownRight) -- (DownDownRight) (UpUpRight) -- (UpRight) -- (DownLeft) -- (DownDownLeft);

   \node (eq) at (-2.2,1.1) {\(=\)};
   \node [plus] (sum1) at (0,0) {};
   \node [plus] (sum2) at (0.7,0) {};
   \node [zero] (coz1) at (0,-0.65) {};
   \node [zero] (coz2) at (0.7,-0.65) {};
   \node [multiply] (neg1) at (-1,2.066) {\(\scriptstyle^{-1}\)};
   \node [multiply] (neg2) at (0,2.066) {\(\scriptstyle{-1}\)};

   \draw[color=red, dashed, thin] (-1.6,-0.3) rectangle (1.9,2.5);
   \draw (sum2.left in) .. controls +(120:1) and +(270:0.2) .. (neg2.io);
   \node [hole] (hole) at (0.35,0.39) {};
   \draw (sum1.right in) .. controls +(60:1) and +(270:1) .. (1,2.5)
   .. controls +(90:1) and +(90:1) .. (2.8,2.5) -- (2.8,-1)
   (sum2.right in) .. controls +(60:1) and +(270:1) .. (1.7,2.5)
   .. controls +(90:0.3) and +(90:0.3) .. (2.1,2.5) -- (2.1,-1)
   (sum1.io) -- (coz1) (sum2.io) -- (coz2)
   (sum1.left in) .. controls +(120:1) and +(270:0.2) .. (neg1.io)
   (neg1) -- (-1,3.5) (neg2) -- (0,3.5);
   \end{tikzpicture}
    }
  \end{center}
Assuming any string diagram with \(j\) basic morphisms can be written in prestandard form, we show
an arbitrary diagram with \(j+1\) basic morphisms can be written in prestandard form as well.  Let
\(S\) be a string diagram on \(j\) basic morphisms, rewritten into prestandard form, with a
maximal \(\vectk\) subdiagram \(T\).  Several cases are considered: those putting a basic morphism
above \(S\), beside \(S\), and below \(S\).
     \begin{itemize}[leftmargin=1em]
     \item \Define{\(S \of G\) for a basic morphism \(G \neq \cap\)}\\
If a diagram \(G\) is composed above \(S\), \(G\) can combine with \(T\) to make a larger \(\vectk\)
subdiagram if \(G\) is \(c\), \(\Delta\), \(+\), \(B\), or \(0\), as these are morphisms in
\(\vectk\).  The generating morphisms \(\cap\), \(\cup\) and \(!\) are not on this list, though a
composition with \(\cup\) (\emph{resp.} \(!\)) would be equivalent to tensoring by \(\cup\) (\emph{resp.} \(!\)).
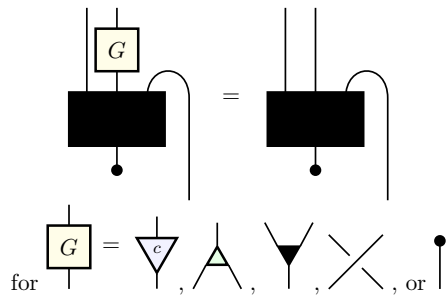
  \begin{center}
    \scalebox{0.80}{
   \begin{tikzpicture}[thick]
% FinVect_k black box absorbs FinVect_k generating morphisms.
   \filldraw[fill=black,draw=black] (-0.1,0.1) rectangle (1.5,-0.8);
   \node (eq) at (2.6,0) {\(=\)};
   \node [sqnode] (G) at (0.7,0.8) {\(G\)};
   \filldraw[fill=black,draw=black] (3.2,0.1) rectangle (4.8,-0.8);
   \node [zero] (Z) at (0.7,-1.2) {};
   \node [zero] (ero) at (4,-1.2) {};
   \draw (0.2,0) -- (0.2,1.5) (0.7,1.5) -- (G) -- (0.7,0) (0.7,-0.7) -- (Z)
   (ero) -- (4,-0.7) (3.5,0) -- (3.5,1.5) (4,0) -- (4,1.5);
   \draw (4.5,0) .. controls +(90:0.6) and +(90:0.6) .. (5.2,0) -- (5.2,-1.7)
         (1.2,0) .. controls +(90:0.6) and +(90:0.6) .. (1.9,0) -- (1.9,-1.7);
   \end{tikzpicture}
    }\\
    \scalebox{0.80}{
for   \begin{tikzpicture}[thick]
   \node [sqnode] (G) at (0,0) {\(G\)};
   \draw (0,0.7) -- (G) -- (0,-0.7);
   \node (eq) at (0.7,0) {\(=\)};
   \end{tikzpicture}
   \begin{tikzpicture}[thick]
   \node [multiply] (c) at (0,0) {\(\scriptstyle{c}\)};
   \draw (0,0.55) -- (c) -- (0,-0.65);
   \end{tikzpicture}, 
   \begin{tikzpicture}[thick]
   \node [delta] (dub) at (0,0) {};
   \draw (0,0.6) -- (dub) (dub.left out) -- (-0.4,-0.5) (dub.right out) -- (0.4,-0.5);
   \end{tikzpicture} , 
   \begin{tikzpicture}[thick]
   \node [plus] (sum) at (0,0) {};
   \draw (0,-0.6) -- (sum) (sum.left in) -- (-0.4,0.5) (sum.right in) -- (0.4,0.5);
   \end{tikzpicture}, 
   \begin{tikzpicture}[thick]
   \draw (0.45,-0.45) -- (-0.45,0.45);
   \node [hole] (hole) at (0,0) {};
   \draw (0.45,0.45) -- (-0.45,-0.45);
   \end{tikzpicture} , or
   \begin{tikzpicture}[thick]
   \node [zero] (zero) at (0,0.4) {};
   \draw (0,-0.4) -- (zero);
   \end{tikzpicture}
    }
  \end{center}
Putting these morphisms on top of \(S\) reduces to performing those compositions on \(T\).  The
maximal \(\vectk\) subdiagram now includes \(T\) and \(G\), with \(S\) unchanged outside the
\(\vectk\) block.
     \item \Define{\(B \of S\)}\\
\(B\) commutes with caps because the category is symmetric monoidal, so capping the braiding is
equivalent to putting the braiding on top of \(T\).  \(B\) is `absorbed' into \(T\), just as in
the \(S \of G\) case.
     \item \Define{\(S \oplus G\) for any basic morphism \(G\)}\\
If any two prestandard string diagrams \(S\) and \(S'\) are tensored together, the result combines
into one prestandard diagram.  This is evident because the category of string diagrams is symmetric
monoidal, and the \(\vectk\) blocks can be placed next to each other as the tensor of two \(\vectk\)
blocks.  These combine into a single \(\vectk\) block, and absorbing all the braidings into this
block as above brings the diagram into prestandard form.  Since each basic morphism can be written
as a prestandard diagram, the tensor \(S \oplus G\) is a special case of this.
  \begin{center}
    \scalebox{0.80}
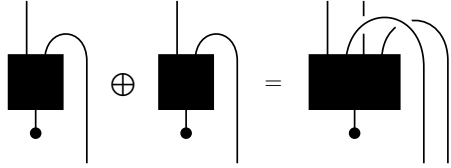
{
   \begin{tikzpicture}[thick]
% Two black boxes combine to make a bigger black box
   \filldraw[fill=black,draw=black] (-0.1,0.1) rectangle (0.8,-0.8);
   \node (oplus) at (1.8,-0.4) {\(\bigoplus\)};
   \filldraw[fill=black,draw=black] (2.4,0.1) rectangle (3.3,-0.8);
   \node (eq) at (4.3,-0.4) {\(=\)};
   \filldraw[fill=black,draw=black] (4.9,0.1) rectangle (6.4,-0.8);
   \node [zero] (Z) at (0.35,-1.2) {};
   \node [zero] (e) at (2.85,-1.2) {};
   \node [zero] (ro) at (5.65,-1.2) {};
   \draw (0.2,0) -- (0.2,1) (0.35,-0.7) -- (Z) (ro) -- (5.65,-0.7)
   (e) -- (2.85,-0.7) (2.7,0) -- (2.7,1) (5.2,0) -- (5.2,1)
   (5.8,0) -- (5.8,1);
   \draw (3,0) .. controls +(90:0.6) and +(90:0.6) .. (3.7,0) -- (3.7,-1.7)
         (0.5,0) .. controls +(90:0.6) and +(90:0.6) .. (1.2,0) -- (1.2,-1.7)
         (6.1,0) .. controls +(90:0.9) and +(90:0.9) .. (7.2,0) -- (7.2,-1.7);
   \node [hole] (ho) at (5.8,0.6) {};
   \node [hole] (le) at (6.46,0.63) {};
   \draw (5.5,-0.1) .. controls +(90:1.1) and +(90:1.1) .. (6.8,-0.1) -- (6.8,-1.7);
   \end{tikzpicture}
    }
  \end{center}
     \item \Define{\(s_c \of S\) for \(c \neq 0\)}\\
Because the outputs of \(S\) are capped, putting any morphism on the bottom of \(S\) is equivalent
(via equations \textbf{(19)} and \textbf{(20)}) to putting its adjoint on top of \(T\).  Putting \(c
\neq 0\) below \(S\) reduces to putting \(c^{-1}\) on top of \(T\) by equation \textbf{(31)}.  The
case of \(s_0\) will be considered below.  The other cases of adjoints of generating morphisms that
need to be considered more carefully are the ones that put \(\Delta^\dagger\), \(+^\dagger\) and
\(\cap = \cup^\dagger\) on top of \(T\).
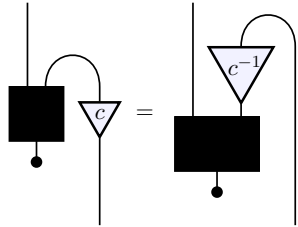
  \begin{center}
    \scalebox{0.80}{
   \begin{tikzpicture}[thick]
% capped multiplication inverse
   \filldraw[fill=black,draw=black] (-0.45,0.45) rectangle (0.45,-0.45);
   \node[multiply] (c) at (1.05,0) {\(c\)};
   \node[zero] (coz1) at (0,-0.8) {};

   \draw[out=-90,in=-90,relative,looseness=2]
   (c.io) -- (1.05,-1.85) (c.90) -- (1.05,0.45) to (0.15,0.45)
   (coz1) -- +(0,0.5) (-0.15,0.35) -- +(0,1.5);

   \node (eq) at (1.8,0) {\(=\)};
   \filldraw[fill=black,draw=black] (2.3,-0.05) rectangle (3.7,-0.95);
   \node[zero] (coz2) at (3,-1.3) {};
   \node[multiply] (cinv) at (3.4,0.8) {\(c^{-1}\!\!\)};

   \draw[out=90,in=90,relative,looseness=2]
   (2.6,-0.15) -- +(0,2) (cinv.io) -- +(0,-0.5) (coz2) -- +(0,0.5)
   (cinv.90) to +(0.9,0) -- (4.3,-1.85);
   \end{tikzpicture}
    }
  \end{center}
     \item \Define{\(\Delta \of S\)}\\
When putting \(\Delta^\dagger\) on top of \(T\), the idea is to make it `trickle down.'  If there
is a nonzero scaling incident to the \(\Delta\) cluster, \(\Delta^\dagger\) can slide through the
\(\Delta\)s using equation \textbf{(23)} to the first nonzero scaling, switching to equation
\textbf{(24)}.  When it encounters this \(s_c\), equation \textbf{(31)} turns \(c\) into
\((c^{-1})^\dagger\), equation \textbf{(17)}\({}^\dagger\) allows \(\Delta^\dagger\) to pass through
\((c^{-1})^\dagger\).  Both copies of \((c^{-1})^\dagger\) can return to being \(c\) by another
application of equation \textbf{(31)}, and the \(\Delta^\dagger\) moves on to the next layer.
  \begin{center}
    \scalebox{0.80}
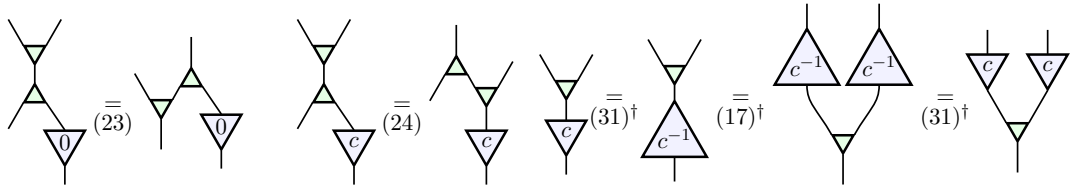
{
   \begin{tikzpicture}[thick]
% cipherphobic Frobenius delta slide
   \node[codelta] (cod1) at (0,0) {};
   \node[delta] (dub1) at (0,-0.734) {};
   \node[multiply] (zip1) at (0.5,-1.5) {\(0\)};

   \draw (cod1.left in) -- +(120:0.5) (cod1.right in) -- + (60:0.5) (cod1.io) -- (dub1.io)
   (dub1.right out) -- (zip1.90) (zip1.io) -- +(0,-0.3) (dub1.left out) -- +(240:0.5);

   \node (eq) at (1.3,-0.9) {\(=\)};
   \node (rel) at (1.3,-1.2) {(23)};
   \node[delta] (dub2) at (2.6,-0.468) {};
   \node[codelta] (cod2) at (2.1,-0.9) {};
   \node[multiply] (zip2) at (3.1,-1.234) {\(0\)};

   \draw (dub2.io) -- +(90:0.5) (dub2.right out) -- (zip2.90) (zip2.io) -- +(0,-0.3)
   (dub2.left out) -- (cod2.right in) (cod2.left in) -- +(120:0.5) (cod2.io) -- +(0,-0.5);
   \end{tikzpicture}
\qquad
   \begin{tikzpicture}[thick]
% quantaphilic Frobenius delta slide
   \node[codelta] (cod1) at (0,0) {};
   \node[delta] (dub1) at (0,-0.734) {};
   \node[multiply] (zip1) at (0.5,-1.5) {\(c\)};

   \draw (cod1.left in) -- +(120:0.5) (cod1.right in) -- + (60:0.5) (cod1.io) -- (dub1.io)
   (dub1.right out) -- (zip1.90) (zip1.io) -- +(0,-0.3) (dub1.left out) -- +(240:0.5);

   \node (eq) at (1.3,-0.9) {\(=\)};
   \node (rel) at (1.3,-1.2) {(24)};
   \node[delta] (dub2) at (2.2,-0.268) {};
   \node[codelta] (cod2) at (2.7,-0.7) {};
   \node[multiply] (zip2) at (2.7,-1.5) {\(c\)};

   \draw (dub2.io) -- +(90:0.5) (dub2.left out) -- +(240:0.5) (cod2.io) -- (zip2.90)
   (zip2.io) -- +(0,-0.3) (dub2.right out) -- (cod2.left in) (cod2.right in) -- +(60:0.5);
   \end{tikzpicture}
    }
    \scalebox{0.80}{
   \begin{tikzpicture}[thick]
% nonzero scalar multiplication commutes with codelta
   \node[multiply] (c) at (0,0) {\(c\)};
   \node[codelta] (cod1) at (0,0.75) {};
   \node (eq1) at (0.8,0.55) {\(=\)};
   \node (rel1) at (0.8,0.25) {(31)\({}^\dagger\)};

   \draw (c.90) -- (cod1.io) (cod1.left in) -- +(120:0.5) (cod1.right in) -- +(60:0.5)
   (c.io) -- +(0,-0.3);

   \node[upmultiply] (cinv2) at (1.8,-0.1) {\(\!c^{-1}\!\)};
   \node[codelta] (cod2) at (1.8,1) {};
   \node (eq2) at (2.9,0.55) {\(=\)};
   \node (rel2) at (2.9,0.25) {(17)\({}^\dagger\)};

   \draw (cinv2.io) -- (cod2.io) (cod2.left in) -- +(120:0.5) (cod2.right in) -- +(60:0.5)
   (cinv2.270) -- +(0,-0.4);

   \node[upmultiply] (cinv3a) at (4,1.1) {\(\!c^{-1}\!\)};
   \node[upmultiply] (cinv3b) at (5.2,1.1) {\(\!c^{-1}\!\)};
   \node[codelta] (cod3) at (4.6,-0.15) {};
   \node (eq3) at (6.3,0.55) {\(=\)};
   \node (rel3) at (6.3,0.25) {(31)\({}^\dagger\)};

   \draw (cinv3a.io) -- +(0,0.4) (cinv3b.io) -- +(0,0.4) (cod3.io) -- +(0,-0.5)
   (cod3.left in) .. controls +(120:0.5) and +(270:0.3) .. (cinv3a.270)
   (cod3.right in) .. controls +(60:0.5) and +(270:0.3) .. (cinv3b.270);

   \node[multiply] (c4a) at (7,1.1) {\(c\)};
   \node[multiply] (c4b) at (8,1.1) {\(c\)};
   \node[codelta] (cod4) at (7.5,0.05) {};

   \draw (c4a.90) -- +(0,0.4) (c4b.90) -- +(0,0.4) (cod4.io) -- +(0,-0.5)
   (cod4.left in) -- (c4a.270) (cod4.right in) -- (c4b.270);
   \end{tikzpicture}
    }
  \end{center}
When the codelta gets to a \(+\) cluster, derived equation \textbf{(D5)} has a net effect of
bringing it to the bottom of the subdiagram, as the other morphisms involved all belong to
\(\vectk\).  This allows the process to be repeated on the next addition until \(\Delta^\dagger\)
reaches the bottom of the \(+\) cluster.  Once there, codelta interacts with the cozero layer below
\(T\); equation \textbf{(8)}\({}^\dagger\) reduces it to a pair of cozeros.
  \begin{center}
% Zipper Codelta
    \scalebox{0.80}
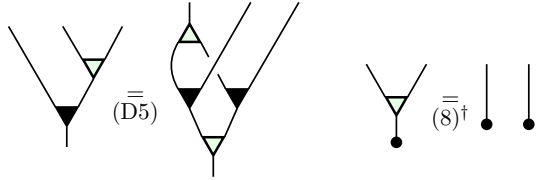
{
   \begin{tikzpicture}[-, thick, node distance=0.7cm]
% Zipped
   \node (in1z) {};
   \node (in2z) [right of=in1z, shift={(0.2,0)}] {};
   \node (in3z) [right of=in2z, shift={(0.45,0)}] {};
   \node [codelta] (nabzip) [below right of=in2z, shift={(0.1,-0.3)}] {};
   \node [plus] (add) [below left of=nabzip, shift={(0.05,-0.3)}] {};
   \node (outz) [below of=add] {};
   \node (equal) [below right of=nabzip, shift={(0.2,0)}] {\(=\)};
   \node (rel) [below of=equal, shift={(0,0.4)}] {(D5)};
% Unzipped
   \node [plus] (addl) [right of=equal, shift={(0.2,0)}] {};
   \node (cross) [above right of=addl, shift={(-0.1,0)}] {};
   \node [delta] (delta) [above left of=cross, shift={(0.1,0)}] {};
   \node (in1u) [above of=delta] {};
   \node [plus] (addr) [below right of=cross, shift={(-0.1,0)}] {};
   \node [codelta] (nablunzip) [below left of=addr, shift={(0.1,-0.3)}] {};
   \node (outu) [below of=nablunzip] {};
   \node (in2u) [right of=in1u, shift={(0.4,0)}] {};
   \node (in3u) [right of=in2u, shift={(0.1,0)}] {};

% Zipped path
   \draw (in1z) -- (add.left in) (add) -- (outz) (in2z) --
   (nabzip.left in) (in3z) -- (nabzip.right in) (nabzip.io) -- (add.right in);
% Unzipped path
   \path
   (delta.left out) edge [bend right=30] (addl.left in);
   \draw (in1u) -- (delta) (delta.right out) -- (cross) -- (addr.left in);
   \draw (in2u) -- (addl.right in) (in3u) -- (addr.right in);
   \draw (addl.io) -- (nablunzip.left in) (addr.io) -- (nablunzip.right in)
   (nablunzip) -- (outu);
   \end{tikzpicture}
   \qquad
   \begin{tikzpicture}[thick]
   \node [codelta] (nabla) at (0,-0.65) {};
   \node [zero] (Z1) at (0,-1.3) {};
   \node [coordinate] (il1) at (-0.5,0) {};
   \node [coordinate] (ir1) at (0.5,0) {};
   \node (eq) at (0.9,-0.6) {\(=\)};
   \node (rel) at (0.9,-0.9) {(8)\({}^\dagger\)};
   \node [coordinate] (il2) at (1.5,0) {};
   \node [coordinate] (ir2) at (2.2,0) {};
   \node [zero] (ZL) [below of=il2] {};
   \node [zero] (ZR) [below of=ir2] {};
   \node [hole] (space) at (0,-2) {};

   \draw (il1) -- (nabla.left in) (nabla) -- (Z1) (nabla.right in) -- (ir1)
   (il2) -- (ZL) (ir2) -- (ZR);
   \end{tikzpicture}
    }
  \end{center}
If all the scalings incident to the \(\Delta\) cluster are by \(0\), rather than trickling down,
\(\Delta^\dagger\) composes with \(!\) (due to equation \textbf{(14)}), which gives \(\cup\) by
equation \textbf{(30)}\({}^\dagger\).  By the zigzag identities, this cup becomes a cap that is
tensored with a subdiagram of \(S\) that is in prestandard form.
  \begin{center}
    \scalebox{0.80}
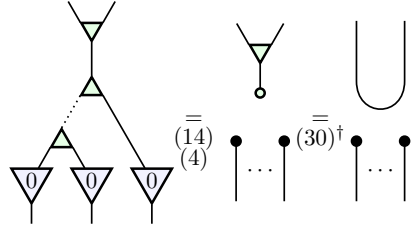
{
   \begin{tikzpicture}[thick]
% codelta on top with a bunch of multiplications by zero
   \node[codelta] (top2) at (1,0.866) {};
   \node[delta] (delt1) at (1,-0.134) {};
   \node[delta] (delt2) at (0.5,-1) {};
   \node[multiply] (c1) at (0,-1.65) {\(0\)};
   \node[multiply] (c2) at (1,-1.65) {\(0\)};
   \node[multiply] (c3) at (2,-1.65) {\(0\)};
   \node (11b) at (0,-2.516) {};
   \node (21b) at (1,-2.516) {};
   \node (n1b) at (2,-2.516) {};

   \draw[dotted] (delt1.left out) -- (delt2.io);
   \draw (delt2.left out) -- (c1) -- (11b) (delt2.right out) -- (c2) -- (21b)
   (delt1.right out) -- (c3) -- (n1b) (top2.io) -- (delt1.io)
   (top2.left in) -- +(120:0.4) (top2.right in) -- +(60:0.4);

   \node (eq) at (2.7,-0.65) {\(=\)};
   \node (rel1) at (2.7,-0.95) {(14)};
   \node (rel2) at (2.7,-1.3) {(4)};

   \node[codelta] (cod) at (3.8,0.5) {};
   \node[bang] (bang) at (3.8,-0.2) {};
   \node[zero] (ze) at (3.4,-1) {};
   \node[zero] (ro) at (4.2,-1) {};
   \node (dots) at (3.8,-1.5) {\(\cdots\)};

   \draw (cod.io) -- (bang) (ze) -- +(0,-1) (ro) -- +(0,-1)
   (cod.left in) -- +(120:0.4) (cod.right in) -- +(60:0.4);

   \node (eq2) at (4.8,-0.65) {\(=\)};
   \node (rel3) at (4.8,-0.95) {(30)\({}^\dagger\)};
   \node[zero] (ze1) at (5.4,-1) {};
   \node[zero] (ro1) at (6.2,-1) {};
   \node (dots1) at (5.8,-1.5) {\(\cdots\)};

   \draw[out=-90,in=-90,relative,looseness=2]
   (5.4,1) -- ++(0,-1) to +(0.8,0) -- +(0,1)
   (ze1) -- +(0,-1) (ro1) -- +(0,-1);
   \end{tikzpicture}
    }
  \end{center}
     \item \Define{\(+ \of S\)}\\
There is a similar trickle down argument for \(+^\dagger\).  First rewriting all instances of
\(s_0\) via equation \textbf{(14)}, the two \(\Delta\) clusters incident to the coaddition can
either reduce to \(\Delta\) clusters that are incident only to nonzero scalings or reduce to a
single deletion, as above, if all incident scalings were \(s_0\).  There are three cases of what
can happen from here.
        \begin{itemize}[leftmargin=1em]
        \item \Define{Both \(\Delta\) clusters were incident to only \(s_0\)}\\
In the first case, as above, the \(\Delta\) clusters will reduce to \(!\) incident to the outputs of
\(+^\dagger\).  Equations \textbf{(D7)} and \textbf{(28)} delete the coaddition.
        \item \Define{One \(\Delta\) cluster was incident to only \(s_0\)}\\
Without loss of generality, the \(!\) incident to \(+^\dagger\) is on the left.  Equation
\textbf{(D7)} replaces \(!\) and \(+^\dagger\) with \(!^\dagger \of !\), and equation \textbf{(30)}
replaces \(\Delta\) and \(!^\dagger\) with a cap.  The \(\Delta\) was -- and the cap is -- incident
to some scaling \(s_c\), \(c \ne 0\).  Without loss of generality, \(s_c\) is incident to the bottom
addition in the cluster.  Equation \textbf{(29)} replaces the addition and cozero with a cup and
\(s_{-1}\), which combines with \(s_c\) by equation \textbf{(11)}.  The cup and cap turn \(s_{-c}\)
around to its adjoint, which is scaling by \(-c^{-1}\), by equation \textbf{(31)}.
  \begin{center}
    \scalebox{0.80}
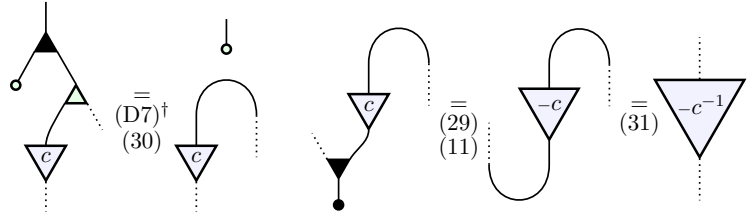
{
   \begin{tikzpicture}[thick]
% Cosum, one dub cluster all zeros
   \node[coplus] (cop) at (0.5,0.866) {};
   \node[delta] (dub) at (1,0) {};
   \node[bang] (cobig) at (0,0.2) {};
   \node[multiply] (c1) at (0.5,-1) {\(c\)};
   \node (eq) at (2.1,0) {\(=\)};
   \node (rel1a) at (2.1,-0.3) {(D7)\({}^\dagger\)};
   \node (rel1b) at (2.1,-0.75) {(30)};

   \draw (cop.io) -- +(0,0.5) (cop.left out) -- (cobig) (cop.right out) -- (dub.io)
   (dub.left out) .. controls +(240:0.2) and +(90:0.2) .. (c1.90);
   \draw[dotted] (c1.io) -- +(0,-0.5) (dub.right out) -- +(300:0.5);

   \node[bang] (cobang) at (3.5,0.8) {};
   \node[multiply] (c2) at (3,-1) {\(c\)};

   \draw[out=90,in=90,relative,looseness=2] (cobang) -- +(0,0.5) (c2.90) -- ++(0,0.5) to +(1,0);
   \draw[dotted] (c2.io) -- +(0,-0.5) (c2.90) ++(1,0.5) -- +(0,-0.7);
   \end{tikzpicture}
\qquad
   \begin{tikzpicture}[thick]
   \node[multiply] (c) at (1,-1) {\(c\)};
   \node[plus] (sum) at (0.5,-2) {};
   \node[zero] (coz) at (0.5,-2.65) {};
   \node (eq1) at (2.5,-1) {\(=\)};
   \node (rel1) at (2.5,-1.3) {(29)};
   \node (rel1b) at (2.5,-1.7) {(11)};

   \draw[out=90,in=90,relative,looseness=2] (c.90) -- ++(0,0.5) to +(1,0) (sum.io) -- (coz)
   (c.io) .. controls +(270:0.2) and +(60:0.2) .. (sum.right in);
   \draw[dotted] (sum.left in) -- +(120:0.5) (c.90) ++(1,0.5) -- +(0,-0.7);

   \node[multiply] (negc) at (4,-1) {\(\scriptstyle{-}\)\(c\)};
   \node (eq2) at (5.5,-1) {\(=\)};
   \node (rel2) at (5.5,-1.3) {(31)};

   \draw[out=90,in=90,relative,looseness=2] (negc.90) -- ++(0,0.4) to +(1,0)
   (negc.io) to (negc.io) -- ++(0,-0.4) to +(-1,0);
   \draw[dotted] (negc.90) ++(1,0.4) -- +(0,-0.7) (negc.io) ++(-1,-0.4) -- +(0,0.7);

   \node[multiply] (inv) at (6.5,-1) {\(\!\scriptstyle{-}\)\(c^{-1}\!\!\)};
   \draw[dotted] (inv.90) -- +(0,0.7) (inv.io) -- +(0,-0.7);
   \end{tikzpicture}
    }
  \end{center}
An addition cluster is above the \(-c^{-1}\) scaling and a duplication cluster is below, but because
those clusters are not otherwise connected to each other, there is a vertical arrangement of the
morphisms in the \(\vectk\) block of the string diagram such that no cups or caps are present.
        \item \Define{Both \(\Delta\) clusters are incident to at least one nonzero scaling}\\
Using equation \textbf{(D5)}\({}^\dagger\), a \(+^\dagger\) will pass through one \(\Delta\) at a
time.  A new \(\Delta^\dagger\) is created each time, but this can trickle down as before.
  \begin{center}
% Zipper Coplus
    \scalebox{0.80}
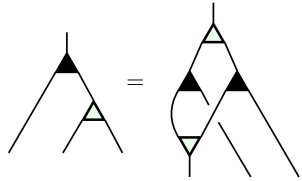
{
   \begin{tikzpicture}[-, thick, node distance=0.7cm]
% Zipped
   \node (out1z) {};
   \node (out2z) [right of=out1z, shift={(0.2,0)}] {};
   \node (out3z) [right of=out2z, shift={(0.45,0)}] {};
   \node [delta] (delzip) [above right of=out2z, shift={(0.1,0.3)}] {};
   \node [coplus] (coadd) [above left of=delzip, shift={(0.05,0.3)}] {};
   \node (inz) [above of=coadd] {};
   \node (equal) [above right of=delzip, shift={(0.2,0)}] {\(=\)};
% Unzipped
   \node [coplus] (coaddl) [right of=equal, shift={(0.2,0)}] {};
   \node (cross) [below right of=coaddl, shift={(-0.1,0)}] {};
   \node [codelta] (nabla) [below left of=cross, shift={(0.1,0)}] {};
   \node (out1u) [below of=nabla] {};
   \node [coplus] (coaddr) [above right of=cross, shift={(-0.1,0)}] {};
   \node [delta] (deltunzip) [above left of=coaddr, shift={(0.1,0.3)}] {};
   \node (inu) [above of=deltunzip] {};
   \node (out2u) [right of=out1u, shift={(0.4,0)}] {};
   \node (out3u) [right of=out2u, shift={(0.1,0)}] {};

% Zipped path
   \draw (out1z) -- (coadd.right in) (coadd) -- (inz) (out2z) --
   (delzip.left out) (out3z) -- (delzip.right out) (delzip.io) -- (coadd.left in);
% Unzipped path
   \path
   (nabla.left in) edge [bend left=30] (coaddl.right in);
   \draw (out1u) -- (nabla) (nabla.right in) -- (coaddr.right in);
   \draw (out2u) -- (cross)  -- (coaddl.left in) (out3u) -- (coaddr.left in);
   \draw (coaddl.io) -- (deltunzip.left out) (coaddr.io) -- (deltunzip.right out) (deltunzip) -- (inu);
   \end{tikzpicture}
    }
  \end{center}
Once the \(\Delta^\dagger\) trickles down, there are two possibilities for what is directly beneath
each \(+^\dagger\): either the same scenario will recur with a \(\Delta\) connected to one or both
outputs, which can only happen finitely many times, or two nonzero scalings will be below the
\(+^\dagger\).  A scaling by any unit in \(k\), \emph{i.e.}\ \(c \ne 0\), can move through a coaddition by
inserting \(c^{-1} c\) on the top branch and applying equation \textbf{(15)}\({}^\dagger\):
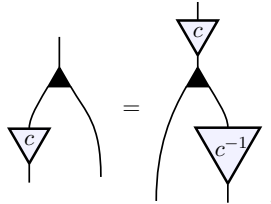
  \begin{center}
    \scalebox{0.80}{
   \begin{tikzpicture}[thick]
   \node[coplus] (cop1) at (0.5,1) {};
   \node[multiply] (c1) at (0,0) {\(c\)};
   \node (eq) at (1.7,0.5) {\(=\)};
   \node[multiply] (c2) at (2.8,1.8) {\(c\)};
   \node[coplus] (cop2) at (2.8,1) {};
   \node[multiply] (cinv) at (3.3,-0.1) {\(c^{-1}\!\!\)};

   \draw (cop1.io) -- +(0,0.5) (cop1.right out) .. controls +(300:0.5) and +(90:0.9) .. +(0.5,-1.5)
   (cop1.left out) .. controls +(240:0.2) and +(90:0.2) .. (c1.90) (c1.io) -- +(0,-0.3)
   (c2) -- +(0,0.5) (c2.io) -- (cop2.io) (cop2.left out) .. controls +(240:0.5) and +(90:0.9) .. +(-0.5,-1.9)
   (cop2.right out) .. controls +(300:0.2) and +(90:0.2) .. (cinv.90) (cinv.io) -- +(0,-0.3);
   \end{tikzpicture}
    }.
  \end{center}
This allows one of the outputs of the coaddition to connect directly to a \(+\) cluster.
         \begin{itemize}[leftmargin=1em]
         \item \Define{If both branches go to different \(+\) clusters}, Frobenius equations
\textbf{(21)--(22)} slide the \(+^\dagger\) down the \(+\) cluster on one side until it gets to the
end of that cluster.
  \begin{center}
    \scalebox{0.80}
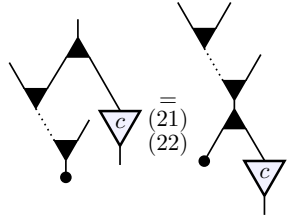
{
   \begin{tikzpicture}[thick]
   \node[plus] (P1a) {};
   \node[plus] (P1b) at (-0.5,0.866) {};
   \node[coplus] (cop1) at (0.2,1.61) {};
   \node[multiply] (c1) at (0.9,0.41) {\(c\)};
   \node[zero] (out1) at (0,-0.5) {};
   \node (eq) at (1.7,0.75) {\(=\)};
   \node (rel1) at (1.7,0.45) {(21)};
   \node (rel2) at (1.7,0.05) {(22)};

   \draw (P1a.io) -- (out1) (c1.io) -- +(0,-0.3) (cop1.io) -- +(0,0.3) (P1a.right in) -- +(60:0.4)
   (cop1.right out) -- (c1.90) (P1b.right in) -- (cop1.left out) (P1b.left in) -- +(120:0.5);
   \draw[dotted] (P1a.left in) -- (P1b.io);

   \node[plus] (P2a) at (2.8,1) {};
   \node[plus] (P2b) at (2.3,1.866) {};
   \node[coplus] (cop2) at (2.8,0.42) {};
   \node[multiply] (c2) at (3.3,-0.42) {\(c\)};
   \node[zero] (out2) at (2.3,-0.23) {};

   \draw (P2b.left in) -- +(120:0.5) (c2.io) -- +(0,-0.3) (P2a.io) -- (cop2.io) (cop2.left out) -- (out2)
   (cop2.right out) -- (c2.90) (P2b.right in) -- +(60:0.5) (P2a.right in) -- +(60:0.5);
   \draw[dotted] (P2a.left in) -- (P2b.io);
   \end{tikzpicture}
    }
  \end{center}
The only morphisms added to the \(\vectk\) block that are not from \(\vectk\) were the coaddition
and the cozero.  Since these reduce to an identity morphism string by equation
\textbf{(1)}\({}^\dagger\), the \(\vectk\) block is truly a \(\vectk\) block again.
         \item \Define{If both branches go to the same \(+\) cluster}, equation \textbf{(3)} and the
Frobenius equation \textbf{(21)} take both branches to the same addition.
  \begin{center}
    \scalebox{0.80}
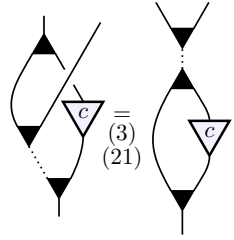
{
   \begin{tikzpicture}[thick]
   \node[coplus] (cop1) at (0,1)              {};
   \node[plus]  (add1a) at (-0.25,-0.433)     {};
   \node[plus]  (add1b) at (0.25,-1.3)        {};
   \node[multiply] (c1) at (0.65,-0.104) {\(c\)};
   \node           (eq) at (1.3,-0.15)   {\(=\)};
   \node         (rel1) at (1.3,-0.45)     {(3)};
   \node         (rel2) at (1.3,-0.85)    {(21)};

   \draw (add1a.left in) .. controls +(120:0.5) and +(240:0.5) .. (cop1.left out)
   (cop1.io) -- +(0,0.3) (cop1.right out) -- (c1.90) (add1b.io) -- +(0,-0.3)
   (c1.io) .. controls +(270:0.2) and +(60:0.2) .. (add1b.right in);
   \node[hole] (cross) at (0.42,0.5) {};
   \draw (add1a.right in) -- +(60:2);
   \draw[dotted] (add1a.io) -- (add1b.left in);

   \node[plus] (topadd) at (2.3,1.2) {};
   \node[coplus] (cop2) at (2.3,0.42) {};
   \node[multiply] (c2) at (2.75,-0.42) {\(c\)};
   \node[plus] (botadd) at (2.3,-1.5) {};

   \draw (topadd.left in) -- +(120:0.5) (topadd.right in) -- +(60:0.5) (botadd.io) -- +(0,-0.4)
   (cop2.right out) .. controls +(300:0.15) and +(90:0.15) .. (c2.90)
   (c2.io) .. controls +(270:0.2) and +(60:0.2) .. (botadd.right in)
   (botadd.left in) .. controls +(120:0.8) and +(240:0.8) .. (cop2.left out);
   \draw[dotted] (cop2.io) -- (topadd.io);
   \end{tikzpicture}
    }
  \end{center}
Depending on whether the remaining scaling is \(s_1\), either equation \textbf{(25)} reduces the
coaddition and the given addition to an identity string or equation \textbf{(D8)} applies.  In the
former case we are done, and in the latter case equations \textbf{(D7)} and
\textbf{(10)}\({}^\dagger\) remove the \(!^\dagger\) introduced by applying equation \textbf{(D8)}.
  \begin{center}
    \scalebox{0.80}
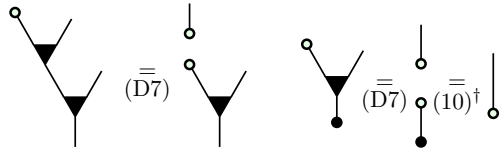
{
   \begin{tikzpicture}[thick]
   \node[bang] (cob1) at (0,1) {};
   \node[plus] (sum1a) at (0.5,0.35) {};
   \node[plus] (sum1b) at (1,-0.516) {};
   \node (eq) at (2.2,0) {\(=\)};
   \node (rel) at (2.2,-0.3) {(D7)};

   \draw (cob1) -- (sum1a.left in) (sum1a.right in) -- +(60:0.5) (sum1a.io) -- (sum1b.left in)
   (sum1b.right in) -- +(60:0.5) (sum1b.io) -- +(0,-0.5);

   \node[bang] (cob2) at (2.9,0.134) {};
   \node[bang] (bang) at (2.9,0.65) {};
   \node[plus] (sum2) at (3.4,-0.516) {};

   \draw (bang) -- +(0,0.5) (cob2) -- (sum2.left in) (sum2.right in) -- +(60:0.5)
   (sum2.io) -- +(0,-0.5);
   \end{tikzpicture}
\qquad
   \begin{tikzpicture}[thick]
   \node[bang] (cob1) at (0,0.65) {};
   \node[plus] (sum1) at (0.5,0) {};
   \node[zero] (coz1) at (0.5,-0.65) {};
   \node (eq1) at (1.3,0) {\(=\)};
   \node (rel1) at (1.3,-0.3) {(D7)};

   \draw (cob1) -- (sum1.left in) (sum1.right in) -- +(60:0.5) (sum1.io) -- (coz1);

   \node[bang] (cob2) at (1.9,-0.325) {};
   \node[bang] (bang) at (1.9,0.325) {};
   \node[zero] (coz2) at (1.9,-0.975) {};
   \node (eq2) at (2.5,0) {\(=\)};
   \node (rel2) at (2.5,-0.3) {(10)\({}^\dagger\)};

   \draw (bang) -- +(0,0.65) (cob2) -- (coz2);

   \node[bang] (cob3) at (3.1,-0.5) {};
   \draw (cob3) -- +(0,1);
   \end{tikzpicture}
    }
  \end{center}
          \end{itemize}
        \end{itemize}
     \item \Define{\(\cup \of S\) and \(S \of \cap\)}\\
Composing with a cup below \(S\) is equivalent to composing with cap above \(T\), since \(\cap =
\cup^\dagger\).  Using equation \textbf{(D10)}\({}^\dagger\), this cap can be replaced by
\(s_{-1}\), coaddition, and zero.  By the arguments above, \(s_{-1}\), \(+^\dagger\), and \(\zero\)
can each be absorbed into the \(\vectk\) block.
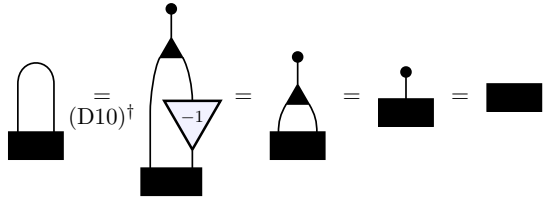
  \begin{center}
    \scalebox{0.80}{
   \begin{tikzpicture}[thick]
% Cap gets absorbed into the FinVect block
   \filldraw[fill=black,draw=black] (-0.4,-0.8) rectangle (0.5,-0.35);
   \draw[out=90,in=90,relative,looseness=2] (-0.25,-0.45) -- ++(0,0.9) to +(0.6,0) -- +(0,-0.9);
   \node (eq1) at (1.15,0.225) {\(=\)};
   \node (rel) at (1.15,-0.1) {(D10)\({}^\dagger\)};

   \filldraw[fill=black,draw=black] (1.8,-1.4) rectangle (2.8,-0.95);
   \node[multiply] (neg) at (2.65,-0.1) {\(\scriptstyle{-1}\)};
   \node[coplus] (cops) at (2.3,1) {};
   \node[zero] (zero1) at (2.3,1.7) {};
   \node (eq2) at (3.5,0.225) {\(=\)};

   \draw (zero1) -- (cops.io) (neg.io) -- (2.65,-1.05)
   (cops.right out) .. controls +(300:0.2) and +(90:0.2) .. (neg.90)
   (cops.left out) .. controls +(240:0.2) and +(90:0.2) .. (1.95,0) -- (1.95,-1.05);

   \filldraw[fill=black,draw=black] (3.95,-0.8) rectangle (4.85,-0.35);
   \node[coplus] (robbers) at (4.4,0.225) {};
   \node[zero] (zero2) at (4.4,0.9) {};
   \node (eq3) at (5.3,0.225) {\(=\)};

   \draw (zero2) -- (robbers.io)
   (robbers.left out) .. controls +(240:0.2) and +(90:0.2) .. +(255:0.5)
   (robbers.right out) .. controls +(300:0.2) and +(90:0.2) .. +(285:0.5);

   \filldraw[fill=black,draw=black] (5.75,0.2) rectangle (6.65,-0.25);
   \node[zero] (zero3) at (6.2,0.65) {};
   \draw (zero3) -- +(0,-0.5);
   \node (eq4) at (7.1,0.225) {\(=\)};
   \filldraw[fill=black,draw=black] (7.55,0.45) rectangle (8.45,0);
   \end{tikzpicture}
    }
  \end{center}
The compositions with zero and \(s_{-1}\) expand the \(\vectk\) block, thus have no effect on
whether the diagram can be written in prestandard form.
     \item \Define{\(! \of S\)}\\
When composing \(!^\dagger\) above \(T\), two possibilities arise, depending on whether there is a
layer of \(\Delta\)s in the \(\vectk\) block.  If there is such a layer, equation \textbf{(30)}
combines the \(!^\dagger\) with a \(\Delta\), making a cap on top of \(T\).  As we have just seen,
this can be rewritten in prestandard form.
  \begin{center}
    \scalebox{0.80}
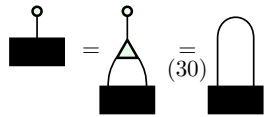
{
   \begin{tikzpicture}[thick]
   \filldraw[fill=black,draw=black] (-0.45,0.45) rectangle (0.45,0);
   \node[bang] (cob1) at (0,0.9) {};
   \node (eq1) at (0.9,0.225) {\(=\)};
   \node[delta] (dub) at (1.5,0.225) {};
   \node[bang] (cob2) at (1.5,0.9) {};
   \filldraw[fill=black,draw=black] (1.05,-0.8) rectangle (1.95,-0.35);
   \node (eq2) at (2.5,0.225) {\(=\)};
   \node (rel) at (2.5,-0.075) {(30)};
   \filldraw[fill=black,draw=black] (2.85,-0.8) rectangle (3.75,-0.35);

   \draw[out=90,in=90,relative,looseness=2] (3,-0.45) -- ++(0,0.9) to +(0.6,0) -- +(0,-0.9)
   (cob1) -- +(0,-0.5) (cob2) -- (dub.io)
   (dub.left out) .. controls +(240:0.2) and +(90:0.2) .. +(255:0.5)
   (dub.right out) .. controls +(300:0.2) and +(90:0.2) .. +(285:0.5);
   \end{tikzpicture}
    }
  \end{center}
If no layer of \(\Delta\)s exists, equations \textbf{(31)}\({}^\dagger\) and
\textbf{(18)}\({}^\dagger\) pass the codeletion through a nonzero scaling.  Then equations
\textbf{(D7)} and \textbf{(10)}\({}^\dagger\) can be used to remove \(!^\dagger\), as we have
already seen.  This leaves only the basic morphisms of \(\vectk\) within the \(\vectk\) block.
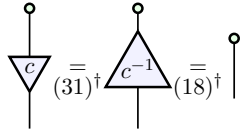
  \begin{center}
    \scalebox{0.80}{
   \begin{tikzpicture}[thick]
% codeletion kills all
   \node[bang] (cob1) at (0,1) {};
   \node[multiply] (c) at (0,0) {\(c\)};
   \node (eq1) at (0.8,0) {\(=\)};
   \node (rel1) at (0.8,-0.3) {(31)\({}^\dagger\)};

   \draw (cob1) -- (c) -- +(0,-1);

   \node[bang] (cob2) at (1.8,1) {};
   \node[upmultiply] (cinv) at (1.8,0) {\(\!c^{-1}\!\)};
   \node (eq2) at (2.8,0) {\(=\)};
   \node (rel2) at (2.8,-0.3) {(18)\({}^\dagger\)};

   \draw (cob2) -- (cinv) -- +(0,-1);

   \node[bang] (cob3) at (3.4,0.5) {};
   \draw (cob3) -- +(0,-1);
   \end{tikzpicture}
    }
  \end{center}
If the scaling is \(s_0\), equation \textbf{(14)} converts \(s_0\) to \(\zero \of !\), allowing
equation \textbf{(28)} to remove the \(!^\dagger\), with the same conclusion.
  \begin{center}
    \scalebox{0.80}{
   \begin{tikzpicture}[thick]
% codeletion is killed by multiplication by zero
   \node[bang] (cob1) at (0.3,1) {};
   \node[multiply] (c0) at (0.3,0) {\(0\)};
   \node (eq1) at (1,0) {\(=\)};
   \node (rel1) at (1,-0.3) {(14)};

   \draw (cob1) -- (c0) -- +(0,-1);

   \node[bang] (cob2) at (1.6,0.925) {};
   \node[bang] (bang) at (1.6,0.325) {};
   \node[zero] (zero1) at (1.6,-0.325) {};
   \node (eq2) at (2.2,0) {\(=\)};
   \node (rel2) at (2.2,-0.3) {(28)};

   \draw (cob2) -- (bang) (zero1) -- +(0,-0.65);

   \node[zero] (zero2) at (2.8,0.5) {};
   \draw (zero2) -- +(0,-1);
   \end{tikzpicture}
    }
  \end{center}
     \item \Define{\(s_0 \of S\)}\\
Composing with \(s_0\) below \(S\) is equivalent to composing with codeletion, followed by tensoring
with zero.  Codeletion is the \(! \of S\) case, and zero can be written in a prestandard form, so
this reduces to tensoring two diagrams that are in prestandard form.
  \begin{center}
    \scalebox{0.80}
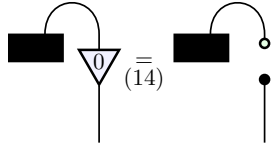
{
   \begin{tikzpicture}[thick]
% c=0
   \filldraw[fill=black,draw=black] (-0.45,0.45) rectangle (0.45,0);
   \node[multiply] (c) at (1.05,0) {\(0\)};

   \draw[out=-90,in=-90,relative,looseness=2]
   (c.io) -- (1.05,-1.35) (c.90) -- (1.05,0.45) to (0.15,0.45);

   \node (eq) at (1.8,0) {\(=\)};
   \node (rel) at (1.8,-0.3) {(14)};
   \filldraw[fill=black,draw=black] (2.3,0.45) rectangle (3.2,0);
   \node[zero] (ins2) at (3.8,-0.3) {};
   \node[bang] (del2) at (3.8,0.3) {};

   \draw[out=-90,in=-90,relative,looseness=2]
   (ins2) -- (3.8,-1.35) (del2) -- (3.8,0.45) to +(-0.9,0);
   \end{tikzpicture}
    }
  \end{center}
     \end{itemize}
Finally, we need to show the prestandard forms can be rewritten in standard form.  We need to show
what elementary row operations look like in terms of string diagrams.  We also need to show for an
arbitrary prestandard string diagram \(S\) with \(\vectk\) block \(T\) that if \(T\) is replaced
with \(T'\), the diagram where an elementary row operation has been performed on \(T\), the 
resulting diagram \(S'\) can be built from \(S\) using equations \textbf{(1)--(31)}.

Because the \(i\)th output of a \(\vectk\) diagram is a linear combinations of the inputs, with the
coefficients coming from the \(i\)th row of its matrix, rows of the matrix correspond to outputs of
the \(\vectk\) block.  Because of this, the row operation subdiagrams in \(S'\) will have
\(0^\dagger\)s immediately beneath them.  Showing \(S'\) can be built from \(S\) reduces to showing
composition of row operations with \(0^\dagger\)s builds the same number of \(0^\dagger\)s.
     \begin{itemize}[leftmargin=1em]
     \item Add a multiple \(c\) of one row to another row:\\
If we want to add a multiple of the \(\beta\) row to the \(\alpha\) row, we need a map \((y_\alpha,
y_\beta) \mapsto (y_\alpha + c y_\beta, y_\beta)\).  By the naturality of the braiding in a
\smc, we can ignore any intermediate outputs:
  \begin{center}
\scalebox{0.80}
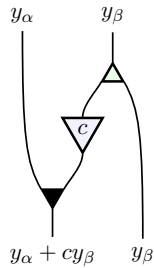
{
   \begin{tikzpicture}[thick]
% Add rows
   \node [plus] (plus) at (0,0) {};
   \node [multiply] (c) at (0.5,1.1) {\(c\)};
   \node [delta] (dub) at (1,2) {};
   \node (in1) at (-0.5,3) {\(y_\alpha\)};
   \node (in2) at (1,3) {\(y_\beta\)};
   \node (out1) at (0,-1) {\(y_\alpha + c y_\beta\)};
   \node (out2) at (1.5,-1) {\(y_\beta\)};

   \draw
   (in1) .. controls +(270:2) and +(120:0.5) .. (plus.left in)
   (plus.io) -- (out1) (dub.io) -- (in2)
   (out2) .. controls +(90:2) and +(300:0.5) .. (dub.right out)
   (dub.left out) .. controls +(240:0.2) and +(90:0.2) .. (c.90)
   (c.io) .. controls +(270:0.2) and +(60:0.2) .. (plus.right in);
   \end{tikzpicture}}.
  \end{center}
When two cozeros are composed on the bottom of this diagram, the result is two cozeros:
  \begin{center}
    \scalebox{0.80}
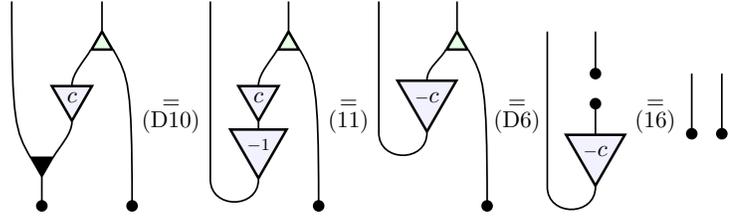
{
   \begin{tikzpicture}[thick]
% Add rows and undo with cozero
   \node [plus] (plus) at (-0.3,0) {};
   \node [multiply] (c) at (0.2,1.1) {\(c\)};
   \node [delta] (dub) at (0.7,2) {};
   \node [coordinate] (in1) at (-0.8,2.7) {};
   \node [coordinate] (in2) at (0.7,2.7) {};
   \node [zero] (out1) at (-0.3,-0.7) {};
   \node [zero] (out2) at (1.2,-0.7) {};

   \draw
   (in1) .. controls +(270:2) and +(120:0.5) .. (plus.left in)
   (plus.io) -- (out1) (dub.io) -- (in2)
   (out2) .. controls +(90:2) and +(300:0.5) .. (dub.right out)
   (dub.left out) .. controls +(240:0.2) and +(90:0.2) .. (c.90)
   (c.io) .. controls +(270:0.2) and +(60:0.2) .. (plus.right in);
   % Cup the plus
   \node (eq1) at (1.85,1) {\(=\)};
   \node (rel1) at (1.85,0.7) {(D10)};
   \node [multiply] (c2) at (3.3,1.1) {\(c\)};
   \node [multiply] (neg2) at (3.3,0.3) {\(\scriptstyle{-1}\)};
   \node [delta] (dub2) at (3.8,2) {};
   \node [zero] (coz2) at (4.3,-0.7) {};

   \draw
   (dub2) -- (3.8,2.7) (c2) -- (neg2)
   (dub2.left out) .. controls +(240:0.2) and +(90:0.2) .. (c2.90)
   (coz2) .. controls +(90:2) and +(300:0.5) .. (dub2.right out)
   (neg2.io) .. controls +(270:0.5) and +(270:0.5) .. +(-0.8,0) -- (2.5,2.7);
   % Combine the multiplications
   \node (eq2) at (4.8,1) {\(=\)};
   \node (rel2) at (4.8,0.7) {(11)};
   \node [multiply] (c3) at (6.1,1.1) {\(\scriptstyle{-}\)\(c\)};
   \node [delta] (dub3) at (6.6,2) {};
   \node [zero] (coz3) at (7.1,-0.7) {};

   \draw
   (dub3) -- (6.6,2.7)
   (dub3.left out) .. controls +(240:0.2) and +(90:0.2) .. (c3.90)
   (coz3) .. controls +(90:2) and +(300:0.5) .. (dub3.right out)
   (c3.io) .. controls +(270:0.5) and +(270:0.5) .. +(-0.8,0) -- (5.3,2.7);
   % Split the delta
   \node (eq3) at (7.6,1) {\(=\)};
   \node (rel3) at (7.6,0.7) {(D6)};
   \node [multiply] (c4) at (8.9,0.2) {\(\scriptstyle{-}\)\(c\)};
   \node [zero] (zero) at (8.9,1) {};
   \node [zero] (coz4) at (8.9,1.5) {};

   \draw
   (zero) -- (c4) (coz4) -- (8.9,2.2)
   (c4.io) .. controls +(270:0.5) and +(270:0.5) .. +(-0.8,0) -- (8.1,2.2);
   % Zero out the multiplication
   \node (eq4) at (9.9,1) {\(=\)};
   \node (rel4) at (9.9,0.7) {(16)};
   \node [zero] (coz5a) at (10.5,0.5) {};
   \node [zero] (coz5b) at (11,0.5) {};

   \draw (coz5a) -- (10.5,1.5) (coz5b) -- (11,1.5);
   \end{tikzpicture}
    }.
  \end{center}
     \item Swap rows:\\
If we want to swap the \(\beta\) row with the \(\alpha\) row, we need a map \((y_\alpha, y_\beta)
\mapsto (y_\beta, y_\alpha)\), which is the braiding of two outputs.  Again, intermediate outputs
may be ignored:
  
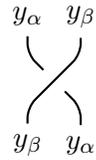
\begin{center}
   \begin{tikzpicture}[thick,node distance=0.5cm]
% Swap rows
   \node (fstart) {\(y_\alpha\)};
   \node [coordinate] (ftop) [below of=fstart] {};
   \node (center) [below right of=ftop] {};
   \node [coordinate] (fout) [below right of=center] {};
   \node (fend) [below of=fout] {\(y_\alpha\)};
   \node [coordinate] (gtop) [above right of=center] {};
   \node (gstart) [above of=gtop] {\(y_\beta\)};
   \node [coordinate] (gout) [below left of=center] {};
   \node (gend) [below of=gout] {\(y_\beta\)};

   \draw [rounded corners] (fstart) -- (ftop) -- (center) --
   (fout) -- (fend) (gstart) -- (gtop) -- (gout) -- (gend);
   \end{tikzpicture}.
  \end{center}
When two cozeros are composed at the bottom of this diagram, the cut strings untwist by the
naturality of the braiding:
  \begin{center}
    \scalebox{0.80}
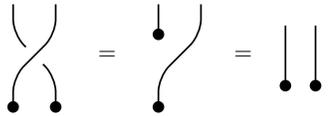
{
   \begin{tikzpicture}[thick,node distance=0.5cm]
% Swap rows and undo with cozero
   \node [coordinate] (fstart) {};
   \node [coordinate] (ftop) [below of=fstart] {};
   \node (center) [below right of=ftop] {};
   \node [coordinate] (fout) [below right of=center] {};
   \node [zero] (fend) [below of=fout] {};
   \node [coordinate] (gtop) [above right of=center] {};
   \node [coordinate] (gstart) [above of=gtop] {};
   \node [coordinate] (gout) [below left of=center] {};
   \node [zero] (gend) [below of=gout] {};

   \draw [rounded corners=0.25cm] (fstart) -- (ftop) -- (center) --
   (fout) -- (fend) (gstart) -- (gtop) -- (gout) -- (gend);

   \node (eq1) [below right of=gtop, shift={(0.5,0)}] {\(=\)};
   \node [zero] (ftop1) [above right of=eq1, shift={(0.5,0)}] {};
   \node [coordinate] (fstart1) [above of=ftop1] {};
   \node (center1) [below right of=ftop1] {};
   \node [coordinate] (gtop1) [above right of=center1] {};
   \node [coordinate] (gstart1) [above of=gtop1] {};
   \node [coordinate] (gout1) [below left of=center1] {};
   \node [zero] (gend1) [below of=gout1] {};

   \draw [rounded corners=0.25cm] (fstart1) -- (ftop1)
   (gstart1) -- (gtop1) -- (gout1) -- (gend1);

   \node (eq2) [below right of=gtop1, shift={(0.35,0)}] {\(=\)};
   \node [zero] (coz1) [right of=eq2, shift={(0.2,-0.5)}] {};
   \node [zero] (coz2) [right of=coz1] {};

   \draw (coz1) -- +(0,1) (coz2) -- +(0,1); 
   \end{tikzpicture}
    }.
  \end{center}
     \item Multiply a row by \(c \neq 0\):\\
The third row operation is multiplying an arbitrary row by a unit, but since \(k\) is a field, that
means any \(c \neq 0\).  This is just the scaling map on one of the outputs:
  
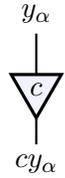
\begin{center}
   \begin{tikzpicture}[thick]
% Multiply a row by a scalar
   \node [multiply] (c) at (0,0) {\(c\)};
   \node (in) at (0,1) {\(y_\alpha\)};
   \node (out) at (0,-1) {\(c y_\alpha\)};

   \draw (in) -- (c) -- (out);
   \end{tikzpicture}.
  \end{center}
Because \(c\) is a unit, \(c^{-1} \in k\), so \(s_c\) can be replaced by the adjoint of scaling by
\(c^{-1}\).
  \begin{center}
    \scalebox{0.80}
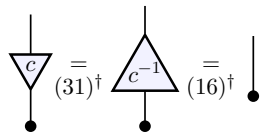
{
   \begin{tikzpicture}[thick]
% Multiply a row by a scalar, then undo by setting equal to zero
   \node [multiply] (c) at (0,0) {\(c\)};
   \node (in) at (0,1) {};
   \node [zero] (out) at (0,-1) {};

   \draw (in) -- (c) -- (out);

   \node (eq1) at (0.8,0) {\(=\)};
   \node (rel1) at (0.8,-0.35) {(31)\({}^\dagger\)};
   \node [zero] (coz1) at (1.9,-1) {};
   \node [upmultiply] (c1) at (1.9,-0.1) {\(\!c^{-1}\!\)};

   \draw (coz1) -- (c1) -- (1.9,1);

   \node (eq2) at (3,0) {\(=\)};
   \node (rel2) at (3,-0.35) {(16)\({}^\dagger\)};
   \node [zero] (coz2) at (3.7,-0.5) {};

   \draw (coz2) -- +(0,1);
   \end{tikzpicture}
    }
  \end{center}
     \end{itemize}
\hfill
\end{proof}
\pagebreak % This is an ugly hack to keep the qed symbol on the same page as the last line of the proof.

Given the PROP \(\relk\), it is natural to consider the free PROP \(\sigflow_k\), which is defined by
the same generators, but has no equations.  The morphisms in this PROP are signal-flow
diagrams\footnote{Because a free PROP must still respect the equations of the symmetry, our
signal-flow diagrams are isomorphism classes of string diagrams.}.  The general considerations of
Chapter~\ref{PROPs} give a functor from \(\sigflow_k\) to any other PROP that has the same
generators, which `imposes the equations' of the target PROP.  In particular, we get a \(\prop\)
morphism \(\bbox \maps \sigflow_k \to \relk\).  In the state-space context, control theorists are
only interested in a certain subcollection of signal-flow diagrams, which correspond to the
state-space equations:
\begin{align*}
\dot{x} &= Ax + Bu \\
y       &= Cx + Du,
\end{align*}
where \(A\), \(B\), \(C\) and \(D\) are linear maps; and \(u\), \(y\), and \(x\) are input, output,
and state vectors, respectively.  We formalize this correspondence in Chapter~\ref{goodflow}.
Intuitively, \(\sigflow\) is `too big' and \(\finrel\) is `too small', so we consider two ways to
get a Goldilocks PROP:  \(\st\) and \(\goodflow\).

The PROP \(\goodflow\) is the most coarse subPROP of \(\sigflow\) whose morphisms include the
subcollection of signal-flow diagrams that correspond to the state-space equations.  We show
\(\goodflow\) does not have any morphisms outside of this subcollection.  While, roughly speaking,
\(\goodflow\) is a way to `shrink' \(\sigflow\), \(\st\) is a way to `grow' \(\finrel\).  These two
approaches are explored in further detail in the following two chapters.

\section{An example}
\label{example}
  \begin{figure}[!h]
    \centering
    % Input file for shell.tex
% By Jason Erbele, based on the image on wikipedia commons: https://commons.wikimedia.org/wiki/File:Cart-pendulum.svg
% Mine looks better (:

\begin{tikzpicture}[>=stealth']
\draw (0,0) -- (7,0);
\foreach \x in {0.15,0.4,...,7}
  \draw (\x,0) -- +(245:0.3);
\fill[black!33] (2.5,0.4) circle (0.4);
\fill[black!33] (4.6,0.4) circle (0.4);
\draw[black!33] (3.55,1) -- +(90:5);
\fill[black!33] (3.55,1.7) circle (0.2);
\draw (3.55,1.7) circle (0.2);
\draw[->] (5,1.25) -- +(0:2);
\fill[black!33] (1.9,0.8) rectangle (5.2,1.7);
\draw (2.5,0.4) circle (0.4);
\draw (4.6,0.4) circle (0.4);
\draw (1.9,0.8) rectangle (5.2,1.7);
\filldraw[fill=black!33] (3.52,1.75) -- (1.52,5) -- (1.59,5.06) -- (3.59,1.81) -- (3.52,1.75);
\fill (1.555,5.03) circle (0.4);
\draw[<->] (3.55,4.2) arc (90:119.5:2.5);
\draw[<->] (0.5,3.3) -- (0.5,2.3) -- (1.5,2.3);
\node at (2.85,4.4) {\(\theta\)};
\node at (2.15,5.4) {\(m\)};
\node at (2.35,3.05) {\(\ell\)};
\node at (1.53,1.25) {\(M\)};
\node at (1.5,2.05) {\(x\)};
\node at (0.75,3.3) {\(y\)};
\node at (6,1.5) {\(F\)};
\end{tikzpicture}
    \caption[Schematic diagram of an inverted pendulum]{Schematic diagram of an inverted pendulum.\label{invpend}}
  \end{figure}
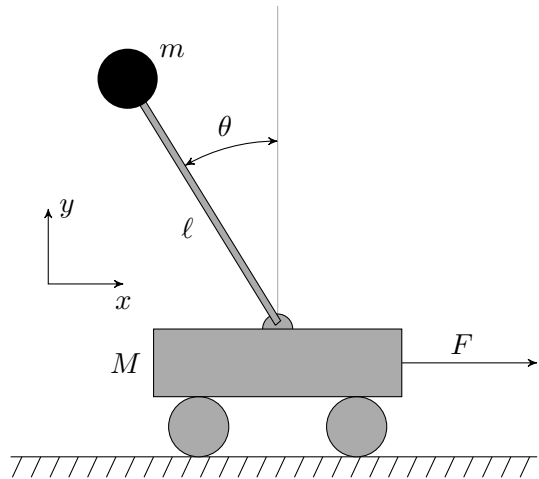
A famous example in control theory is the `inverted pendulum': an upside-down pendulum on a cart
\cite{Friedland}.  The pendulum naturally tends to fall over, but we can stabilize it by setting up
a feedback loop where we observe its position and move the cart back and forth in a suitable way
based on this observation.  Without introducing this feedback loop, let us see how signal-flow
diagrams can be used to describe the pendulum and the cart.  We shall see that the diagram for a
system made of parts is built from the diagrams for the parts, not merely by composing and
tensoring, but also with the help of duplication and coduplication, which give additional ways to
set variables equal to one another.

Suppose the cart has mass \(M\) and can only move back and forth in one direction, so its
position is described by a function \(x(t)\).  If it is acted on by a total force 
\(F_{\mathrm{net}}(t)\) then Newton's second law says 
\[
F_{\mathrm{net}}(t) = M \ddot{x}(t)  .
\]
We can thus write a signal-flow diagram with the force as input and the cart's position as output:
  
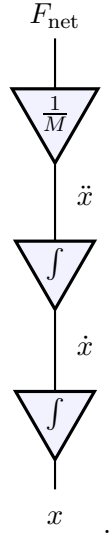
\begin{center}
   \begin{tikzpicture}[thick]
   \node[coordinate] (q) [label={[shift={(0,-0.6)}]\(x\)}] {};
   \node [integral] (diff) [above of=q] {\(\int\)};
   \node (v) [above of=diff, label={[shift={(0.4,-0.5)}]\(\dot x\)}] {};
   \node [integral] (dot) [above of=v] {\(\int\)};
   \node (a) [above of=dot, label={[shift={(0.4,-0.5)}]\(\ddot{x}\)}] {};
   \node [multiply] (M) [above of=a] {\(\frac{1}{M}\)};
   \node[coordinate] (F) [above of=M, label={[shift={(0,0)}]\(F_{\mathrm{net}}\)}] {};

   \draw (F) -- (M) -- (dot) -- (diff) -- (q);
   \end{tikzpicture}.
  \end{center}

The inverted pendulum is a rod of length \(\ell\) with a mass \(m\) at
its end, mounted on the cart and only able to swing back and forth in
one direction, parallel to the cart's movement.  If its angle from
vertical, \(\theta(t)\), is small, then its equation of motion is
approximately linear:
\[
   \ell \ddot{\theta}(t) = g \theta(t) - \ddot{x}(t) ,
\]
where \(g\) is the gravitational constant.  We can turn this equation
into a signal-flow diagram with \(\ddot{x}\) as input and \(\theta\) as output:
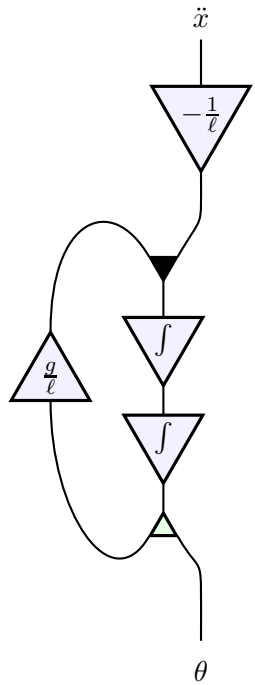
\begin{center}
\begin{tikzpicture}[thick]
   \node [multiply] (linverse) at (0,0) {\(-\frac{1}{\ell}\)};
   \node [plus] (adder) at (-0.5,-2) {};
   \node [integral] (int1) at (-0.5,-3) {\(\int\)};
   \node [upmultiply] (goverl) at (-2,-3.5) {\(\frac{g}{\ell}\)};
   \node [integral] (int2) at (-0.5,-4.3) {\(\int\)};
   \node [delta] (split) at (-0.5,-5.5) {};
   \node at (0 em,1.3) {\(\ddot{x}\)}; % this node is a label only
   \node at (0 em,-7.4) {\(\theta\)}; % this node is a label only

   \draw (linverse) -- (0,1)
         (linverse.io) .. controls +(270:0.8) and +(60:0.7) .. (adder.right in)
         (adder.io) -- (int1) -- (int2) -- (split)
         (goverl.io) .. controls +(90:1.5) and +(120:1) .. (adder.left in)
         (split.right out) .. controls +(300:0.7) and +(90:1.3) .. (0,-7)
         (split.left out) .. controls +(240:1) and +(270:1.5) .. (goverl.270);
\end{tikzpicture}.
\end{center}
Note that this already includes a kind of feedback loop, since the pendulum's angle affects
the force on the pendulum.  

Finally, there is an equation describing the total force on the cart:
\[ F_{\mathrm{net}}(t)  = F(t) - m g \theta(t) , \]
where \(F(t)\) is an externally applied force and \(-mg\theta(t)\) is the force due to the 
pendulum.  It will be useful to express this as follows:
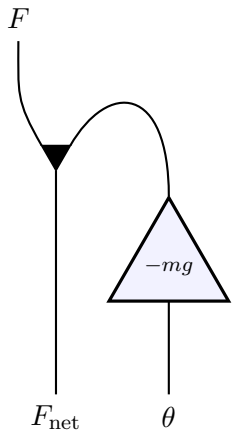
\begin{center}
\begin{tikzpicture}[thick]
   \node [plus] (differ) at (0,0) {};
   \node [upmultiply] (mg) at (1.5,-1.5) {\(\scriptstyle{-mg}\)};
   \node at (0,-3.5) {\(F_{\mathrm{net}}\)}; % this node is a label only
   \node at (1.5,-3.5) {\(\theta\)}; % this node is a label only
   \node at (-0.5,1.8) {\(F\)}; % this node is a label only

   \draw (mg.io) .. controls +(90:1.5) and +(60:1) .. (differ.right in)
         (differ.left in) .. controls +(120:.7) and +(270:.7) .. (-0.5,1.5)
         (differ.io) -- (0,-3.2) 
         (mg) -- (1.5,-3.2);
\end{tikzpicture}.
\end{center}
Here we are treating \(\theta\) as an output rather than an input, with the help of a cap.

The three signal-flow diagrams above describe the following linear relations:
\begin{eqnarray}
x &=& \int \int \frac{1}{M}  F_{\mathrm{net}}   \label{eq.1}  \\
\theta &=& \int \int 
\left(\frac{g}{\ell} \,\theta - \frac{1}{\ell}\, \ddot{x}\right)  \label{eq.2} \\
 F_{\mathrm{net}} + mg \theta &=&  F \label{eq.3},
\end{eqnarray}
where we treat \eqref{eq.1} as a linear relation with \(F_{\mathrm{net}}\) as input and
\(x\) as output, \eqref{eq.2} as a linear relation with \(\ddot{x}\) as input and \(\theta\)
as output, and \eqref{eq.3} as a linear relation with \(F\) as input and \((F_{\mathrm{net}}, 
\theta)\) as output.   

To understand how the external force affects the position of the cart and the angle of the pendulum,
we wish to combine all three diagrams to form a signal-flow diagram that has the external force
\(F\) as input and the pair \((x, \theta)\) as output.  This is not just a simple matter of
composing and tensoring the three diagrams.  We can take \(F_{\mathrm{net}}\), which is an output of
\eqref{eq.3}, and use it as an input for \eqref{eq.1}.  But we also need to duplicate \(\ddot{x}\),
which appears as an intermediate variable in \eqref{eq.1} since \(\ddot{x} = \frac{1}{M}
F_{\mathrm{net}}\), and use it as an input for \eqref{eq.2}.  Finally, we need to take the variable
\(\theta\), which appears as an output of both \eqref{eq.2} and \eqref{eq.3}, and identify the two
copies of this variable using coduplication.  To emphasize the relational nature of the component,
we shall write coduplication in terms of duplication and a cup, as follows:
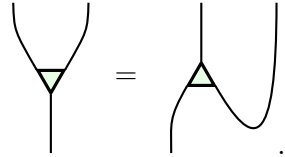
\begin{center}
\begin{tikzpicture}[thick]
\node [codelta] (split) at (-2,0) {};
\node [] at (-1,0) {\(=\)};
 \node [delta] (theta) at (0,0) {};
\draw     (theta.left out) .. controls +(240:0.5) and +(90:0.4) .. (-0.4,-1)
         (theta.right out) .. controls +(300:1) and +(270:2) .. (1,1)
         (theta.io) -- (0,1)
         (split.left in) .. controls +(120:0.6) .. (-2.5,1)
         (split.right in) .. controls +(60:0.6) .. (-1.5,1)
         (split.io) -- (-2,-1);
\end{tikzpicture}.
\end{center}

The result is this signal-flow diagram:
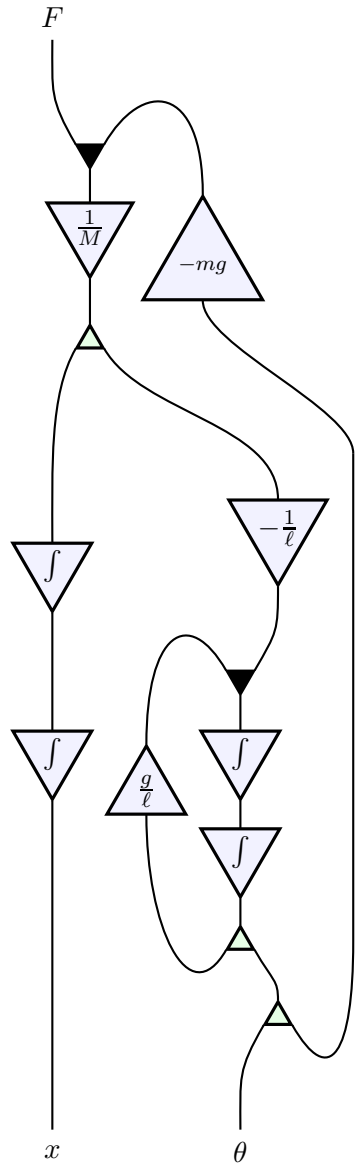
\begin{center}
\begin{tikzpicture}[thick]
   \node [plus] (differ) at (0,0) {};
   \node [upmultiply] (mg) at (1.5,-1.5) {\(\scriptstyle{-mg}\)};
   \node [delta] (split) at (0,-2.5) {};
   \node [multiply] (Minv) at (0,-1) {\(\frac{1}{M}\)};
   \node at (-0.5,1.8) {\(F\)}; % this node is a label only

   \draw (mg.io) .. controls +(90:1.5) and +(60:1) .. (differ.right in)
         (differ.left in) .. controls +(120:.7) and +(270:.7) .. (-0.5,1.5)
         (differ.io) -- (Minv.90)
         (Minv.io) -- (split.io)
         (mg.270) .. controls +(270:0.6) and +(90:0.6) .. (3.5,-4);

   \node [integral] (int1) at (-0.5,-5.5) {\(\int\)};
   \node [integral] (int2) at (-0.5,-8) {\(\int\)};
   \node at (-0.5 cm,-13.3) {\(x\)}; % this node is a label only

   \draw (split.left out) .. controls +(240:0.7) and +(90:1) .. (int1.90)
         (int1.io) -- (int2.90)
         (int2.io) -- (-0.5,-13);

   \node [multiply] (linverse) at (2.5,-5) {\(-\frac{1}{\ell}\)};
   \node [plus] (adder) at (2,-7) {};
   \node [integral] (int3) at (2,-8) {\(\int\)};
   \node [upmultiply] (goverl) at (0.75,-8.5) {\(\frac{g}{\ell}\)};
   \node [integral] (int4) at (2,-9.3) {\(\int\)};
   \node [delta] (split2) at (2,-10.5) {};
   \node [delta] (theta) at (2.5,-11.5) {};
   \node at (2 cm,-13.3) {\(\theta\)}; % this node is a label only

   \draw (split.right out) .. controls +(300:1) and +(90:1) .. (linverse.90) 
         (linverse.io) .. controls +(270:0.6) and +(60:0.6) .. (adder.right in)
         (adder.io) -- (int3) -- (int4) -- (split2)
         (goverl.io) .. controls +(90:1.5) and +(120:1) .. (adder.left in)
         (split2.right out) .. controls +(300:0.5) and +(90:0.5) .. (theta)
         (theta.left out) .. controls +(240:0.7) and +(90:0.7) .. (2,-13)
         (theta.right out) .. controls +(300:1) and +(270:1.5) .. (3.5,-10.5)
         (3.5,-10.5) -- (3.5,-4)
         (split2.left out) .. controls +(240:1) and +(270:1.5) .. (goverl.270);
\end{tikzpicture}.
\end{center}

This is not the signal-flow diagram for the inverted pendulum that one sees in
Friedland's textbook on control theory \cite{Friedland}.  We leave it as an exercise to the reader to 
rewrite the above diagram using the rules given in this paper, obtaining Friedland's diagram:

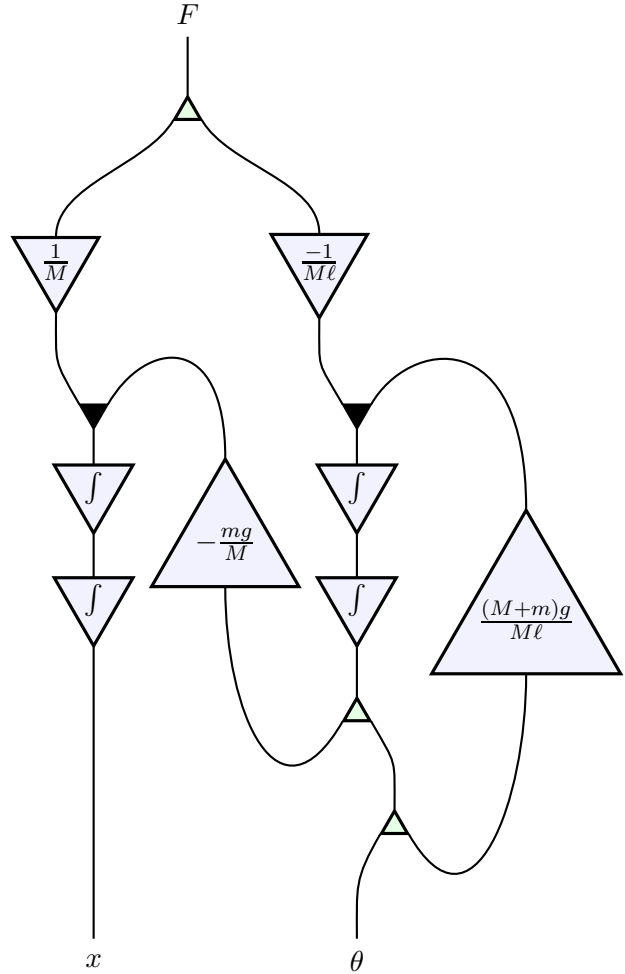
\begin{center}
\begin{tikzpicture}[thick]
   \node [delta] (F) at (0,10) {};
   \node at (0,11.3) {\(F\)}; % label node
   \node [multiply] (Minv) at (-1.75,8) {\(\frac{1}{M}\)};
   \node [multiply] (Mlinv) at (1.75,8) {\(\frac{-1}{M\ell}\)};
   \node [plus] (xsum) at (-1.25,6) {};
   \node [plus] (thsum) at (2.25,6) {};
   \node [integral] (i1) at (-1.25,5) {\(\int\)};
   \node [integral] (i2) at (-1.25,3.5) {\(\int\)};
   \node at (-1.25,-1.3) {\(x\)}; % label node
   \node [integral] (i3) at (2.25,5) {\(\int\)};
   \node [integral] (i4) at (2.25,3.5) {\(\int\)};
   \node [upmultiply] (mgM) at (0.5,4.25) {\(-\frac{mg}{M}\)};
   \node [upmultiply] (mess) at (4.5,3.25) {\(\!\!\frac{(M+m)g}{M\ell}\!\!\)};
   \node [delta] (theta1) at (2.25,2) {};
   \node [delta] (theta2) at (2.75,0.5) {};
   \node at (2.25,-1.3) {\(\theta\)}; % label node

   \draw (F) -- (0,11)
         (F.left out) .. controls +(240:0.7) and +(90:0.7) .. (Minv.90)
         (F.right out) .. controls +(300:0.7) and +(90:0.7) .. (Mlinv.90)
         (Minv.io) .. controls +(270:0.7) and +(120:0.7) .. (xsum.left in)
         (Mlinv.io) .. controls +(270:0.7) and +(120:0.7) .. (thsum.left in)
         (mgM.io) .. controls +(90:1.7) and +(60:1) .. (xsum.right in)
         (mess.io) .. controls +(90:2.5) and +(60:1) .. (thsum.right in)
         (xsum) -- (i1) -- (i2) -- (-1.25,-1)
         (thsum) -- (i3) -- (i4) -- (theta1)
         (mgM.270) .. controls +(270:1.7) and +(240:1.7) .. (theta1.left out)
         (mess.270) .. controls +(270:1.7) and +(300:1.7) .. (theta2.right out)
         (theta2.io) .. controls +(90:0.7) and +(300:0.7) .. (theta1.right out)
         (theta2.left out) .. controls +(240:0.7) and +(90:0.7) .. (2.25,-1)
;
\end{tikzpicture}.
\end{center}
As a start, one can use Theorem \ref{presrk} to prove that it is indeed possible to do this rewriting.  
To do this, simply check that both signal-flow diagrams define the same linear relation.  The proof of
the theorem gives a method to actually do the rewriting---but not necessarily the fastest method.

\section{Related work}
\label{vectrelated}
We conclude this chapter with some remarks aimed at setting it in context.  This chapter is heavily
based on \cite{BE}, so we would like to focus on comparisons with other papers published around the
same time.  On April 30th, 2014, after much of \cite{BE} was written, Soboci\'nski told Baez about
some closely related papers that he wrote with Bonchi and Zanasi \cite{BSZ1,BSZ2}.  These provide
interesting characterizations of symmetric monoidal categories equivalent to \(\Vectk\) and
\(\Relk\).  Later, while \cite{BE} was being refereed, Wadsley and Woods \cite{WW} generalized the
presentation of \(\vectk\) to the case where \(k\) is any commutative rig.  We discuss Wadsley and
Woods' work first, since doing so makes the exposition simpler.

What we have called \(\vectk\) here, Wadsley and Woods looked at from a slightly different
perspective, getting an isomorphic PROP \(\Mat(k)\), where a morphism \(f \maps m \to n\) is an \(n
\times m\) matrix with entries in \(k\), composition of morphisms is given by matrix multiplication,
and the tensor product of morphisms is the direct sum of matrices.  Wadsley and Woods gave an
elegant description of the algebras of \(\Mat(k)\).  Suppose \(\P\) is a PROP and \(\propQ\) is a
strict symmetric monoidal category.  Then the \Define{category of algebras} of \(\P\) in \(\propQ\)
is the category of strict symmetric monoidal functors \(F \maps \P \to \propQ\) and natural
transformations between these.  If for every choice of \(\propQ\) the category of algebras of \(\P\)
in \(\propQ\) is equivalent to the category of algebraic structures of some kind in \(\propQ\), we
say \(\P\) is the PROP for structures of that kind.

In this language, Wadsley and Woods proved that \(\Mat(k)\) is the PROP for `bicommutative bimonoids
over \(k\)'.  To understand this, first note that for any bicommutative bimonoid \(A\) in
\(\propQ\), the bimonoid endomorphisms of \(A\) can be added and composed, giving a rig \(\End(A)\).
A bicommutative bimonoid \Define{over \(k\)} in \(\propQ\) is one equipped with a rig homomorphism
\(\Phi_A \maps k \to \End(A)\).  Bicommutative bimonoids over \(k\) form a category where a
morphism \(f \maps A \to B\) is a bimonoid homomorphism compatible with this extra structure,
meaning that for each \(c \in k\) the square
\begin{center}
\begin{tikzpicture}[node distance=2cm,->]
% converted from xy to reduce clutter in the included packages
  \node (A1) at (0,0) {\(A\)};
  \node (B1) [below of=A1] {\(B\)};
  \node (A2) [right of=A1] {\(A\)};
  \node (B2) [below of=A2] {\(B\)};

  \draw (A1) to node[above] {\(\Phi_A(c)\)} (A2);
  \draw (B1) to node[below] {\(\Phi_B(c)\)} (B2);
  \draw (A1) to node[left] {\(f\)} (B1);
  \draw (A2) to node[right] {\(f\)} (B2);
\end{tikzpicture}
\end{center}
commutes.  Wadsley and Woods proved that this category is equivalent to the category of algebras of
\(\Mat(k)\) in \(\propQ\).

This result amounts to a succinct restatement of Theorem~\ref{presvk}, though technically the
result is a bit different, and the style of proof much more so.  The fact that an algebra of
\(\Mat(k)\) is a bicommutative bimonoid is equivalent to our equations
{\hyperref[eqn123]{\textbf{(1)--(10)}}}.  The fact that \(\Phi_A(c)\) is a bimonoid homomorphism for
all \(c \in k\) is equivalent to equations {\hyperref[eqn1516]{\textbf{(15)--(18)}}}, and the fact
that \(\Phi\) is a rig homomorphism is equivalent to equations
{\hyperref[eqn11121314]{\textbf{(11)--(14)}}}.  

Even better, Wadsley and Woods showed that \(\Mat(k)\) is the PROP for bicommutative bimonoids over
\(k\) whenever \(k\) is a commutative rig.  Subtraction and division are not required to define the
PROP \(\Mat(k)\), nor are they relevant to the definition of bicommutative bimonoids over \(k\).
Working with commutative rigs is not just generalization for the sake of generalization: it
clarifies some interesting facts.

For example, the commutative rig of natural numbers gives a PROP \(\Mat(\N)\).  This is equivalent
to the symmetric monoidal category where morphisms are isomorphism classes of spans of finite sets,
with disjoint union as the tensor product.  Lack \cite[Ex.\ 5.4]{Lack} had already shown that this
is the PROP for bicommutative bimonoids.  But this also follows from the result of Wadsley and
Woods, since every bicommutative bimonoid \(A\) is automatically equipped with a unique rig
homomorphism \(\Phi_A \maps \N \to \End(A)\).

Similarly, the commutative rig of booleans \(\B = \{F,T\}\), with `or' as addition and `and' as
multiplication, gives a PROP \(\Mat(\B)\).  This is equivalent to the symmetric monoidal category
where morphisms are relations between finite sets, with disjoint union as the tensor product.
Mimram \cite[Thm.\ 16]{Mimram} had already shown this is the PROP for \Define{special}
bicommutative bimonoids, meaning those where comultiplication followed by multiplication is the
identity:
\begin{center}
% special bicommutative bimonoids
   \begin{tikzpicture}[thick]
   \node [plus] (sum) at (0.4,-0.5) {};
   \node [delta] (cosum) at (0.4,0.5) {};
   \node [coordinate] (in) at (0.4,1) {};
   \node [coordinate] (out) at (0.4,-1) {};
   \node (eq) at (1.3,0) {\(=\)};
   \node [coordinate] (top) at (2,1) {};
   \node [coordinate] (bottom) at (2,-1) {};

   \path (sum.left in) edge[bend left=30] (cosum.left out)
   (sum.right in) edge[bend right=30] (cosum.right out);
   \draw (top) -- (bottom)
   (sum.io) -- (out)
   (cosum.io) -- (in);
   \end{tikzpicture}.
  \end{center}
But again, this follows from the general result of Wadsley and Woods.

Finally, taking the commutative ring of integers \(\Z\), Wadsley and Woods showed that \(\Mat(\Z)\)
is the PROP for bicommutative Hopf monoids.  The key here is that scaling by \(-1\) obeys the axioms
for an antipode, namely:
  \begin{center}
   \begin{tikzpicture}[thick]
% Hopf axioms
  \node [delta] (cosum) at (0,1.8) {};
   \node [multiply] (times) at (-0.34,0.97) {\tiny \(-1\)};
   \node [plus] (sum) at (0,0) {};

   \draw
   (cosum.io) -- +(0,0.3) (sum.io) -- +(0,-0.3)
   (cosum.right out) .. controls +(300:0.5) and +(60:0.5) .. (sum.right in)
   (cosum.left out) .. controls +(240:0.15) and +(90:0.15) .. (times.90)
   (times.io) .. controls +(270:0.15) and +(120:0.15) .. (sum.left in);

   \node (eq) at (1.2,0.9) {\(=\)};
   \node [bang] (bang) at (2.1,1.3) {};
   \node [zero] (cobang) at (2.1,0.5) {};

   \draw (bang) -- +(0,1.02) (cobang) -- +(0,-1.02);

   \node (eq) at (3,0.9) {\(=\)};
   \node [delta] (cosum) at (4.2,1.8) {};
   \node [multiply] (times) at (4.54,0.97) {\tiny \(-1\)};
   \node [plus] (sum) at (4.2,0) {};

   \draw
   (cosum.io) -- +(0,0.3) (sum.io) -- +(0,-0.3)
   (cosum.left out) .. controls +(240:0.5) and +(120:0.5) .. (sum.left in)
   (cosum.right out) .. controls +(300:0.15) and +(90:0.15) .. (times.90)
   (times.io) .. controls +(270:0.15) and +(60:0.15) .. (sum.right in);
  
   \end{tikzpicture}.
   \end{center}
More generally, whenever \(k\) is a commutative ring, the presence of \(-1 \in k\) guarantees that
a bimonoid over \(k\) is automatically a Hopf monoid over \(k\).  So, when \(k\) is a commutative
ring, Wadsley and Woods' result implies that \(\Mat(k)\) is the PROP for Hopf monoids over \(k\).  

Earlier, Bonchi, Soboci\'nski and Zanasi gave an elegant and very different proof that \(\Mat(R)\)
is the PROP for Hopf monoids over \(R\) when \(R\) is a principal ideal domain
\cite[Prop.\ 3.7]{BSZ1}.  The advantage of their argument is that they build up the PROP for Hopf
monoids over \(R\) from smaller pieces, using some ideas developed by Lack \cite{Lack}.

These authors also proved that \(\relk\) is a pushout in the category \(\prop\) of PROPs and
\(\prop\) morphisms:
\begin{center}
\begin{tikzpicture}[node distance=2cm]
% converted from xy to reduce clutter in the included packages
  \node (coproduct) at (-0.5,2) {\(\Mat(R) + \Mat(R)^{\mathrm{op}}\)};
  \node (cospan) [below of=coproduct] {\(\mathrm{Cospan}(\Mat(R))\)};
  \node (span) at (3.5,2) {\(\mathrm{Span}(\Mat(R))\)};
  \node (svk) [below of=span] {\(\relk\)};

  \draw[->] (coproduct) to (cospan);
  \draw[->] (coproduct) to (span);
  \draw[->] (cospan) to (svk);
  \draw[->] (span) to (svk);
  \draw (2.77,0.73) +(0,-0.57) -- +(0,0) -- +(0.57,0);
\end{tikzpicture}.
\end{center}
% xy original follows:
% \[ 
% \xymatrix{ \Mat(R) + \Mat(R)^{\mathrm{op}} \ar[dd] \ar[rr] && \mathrm{Span}(\Mat(R)) \ar[dd] \\  \\
% \mathrm{Cospan}(\Mat(R))  \ar[rr] && \relk
%  } 
% \] 

This pushout square requires a bit of explanation.
Here \(R\) is any principal ideal domain whose field of fractions is \(k\).  For example, we 
could take \(R = k\), though Bonchi, Soboci\'nski and Zanasi  are more interested in the example
where  \(R = \R[s]\) and \(k = \R(s)\).  A morphism in \(\mathrm{Span}(\Mat(R))\) is an isomorphism 
class of spans in \(\Mat(R)\).  There is a covariant functor 
\[      \begin{array}{ccc} \Mat(R) &\to& \mathrm{Span}(\Mat(R))   \\ 
                                    m \stackrel{f}{\to} n &\mapsto & m \stackrel{1}{\leftarrow} m \stackrel{f}{\to} n \end{array} \]
and also a contravariant functor
\[      \begin{array}{ccc} \Mat(R) &\to& \mathrm{Span}(\Mat(R))   \\ 
                                    m \stackrel{f}{\to} n &\mapsto & n \stackrel{f}{\leftarrow} m \stackrel{1}{\to} m. \end{array} \]
Putting these together we get the functor from \(\Mat(R) + \Mat(R)^{\mathrm{op}}\) to
\(\mathrm{Span}(\Mat(R))  \) that gives the top edge of the square.  Similarly, a morphism in
\(\mathrm{Cospan}(\Mat(R))\) is an isomorphism class of cospans in \(\Mat(R)\), and we have both a
covariant functor 
\[      \begin{array}{ccc} \Mat(R) &\to& \mathrm{Cospan}(\Mat(R))   \\
                                     m \stackrel{f}{\to} n &\mapsto & m \stackrel{f}{\rightarrow} n \stackrel{1}{\leftarrow} n \end{array} \]
and a contravariant functor
\[      \begin{array}{ccc} \Mat(R) &\to& \mathrm{Cospan}(\Mat(R)) \\ 
                                    m \stackrel{f}{\to} n &\mapsto & n \stackrel{1}{\rightarrow} n \stackrel{f}{\leftarrow} m. \end{array} \]
Putting these together we get the functor from \( \Mat(R) + \Mat(R)^{\mathrm{op}} \) to
\( \mathrm{Cospan}(\Mat(R)) \) that gives the left edge of the square. 

Bonchi, Soboci\'nski and Zanasi analyze this pushout square in detail, giving explicit
presentations for each of the PROPs involved, all based on their presentation of \(\Mat(R)\).  The
upshot is a presentation of \(\relk\) which is very similar to our presentation of \(\relk\).  Their
methods allow them to avoid many, though not all, of the lengthy arguments that involve putting
morphisms in standard form.

\chapter{The PROP $\st$}
\label{stateful}

\section{Constructing categories of state}
\label{statecats}
In the late 1950s and early 1960s, Kalman worked on the state-space approach to control theory:  In
1960 \cite{Kalman60} he introduced the concepts of \textit{controllability} and
\textit{observability} into control theory, showing in 1963 \cite{Kalman63} how these concepts can
be used to decompose an arbitrary linear control system into four parts.  He showed in the
time-invariant case, these four parts are actual subsystems.  Earlier work in control theory had
focused on \textit{transfer functions}, defined as the ratio of a transform (typically the Laplace
transform) of the output of a system by the transform of the input of the system.  Kalman showed
transfer functions only capture the part of the system that is both controllable and observable:
uncontrollable parts of the system do not depend on the input, and unobservable parts of the system
do not affect the output.

The morphisms of \(\vectk\) and \(\relk\) are completely determined by how the input relates to the
output, so while they are reasonable models for the frequency analysis approach and its transfer
functions, they are unsatisfactory as models for the state-space approach.  Our goal here is to
define new categories based on \(\vectk\) and \(\relk\) to address the shortcomings of these PROPs
in the state-space context for linear time-invariant systems.

Kalman's concepts of controllability and observability apply to systems of differential equations of the form
\[
\dot{x}(t) = A x(t) + B u(t)
\]
\[
y(t) = C x(t) + D u(t),
\]
where \(u(t)\) is the input vector, \(y(t)\) is the output vector, and \(x(t)\) is the state
vector.  Written as a signal-flow diagram, such a system looks like
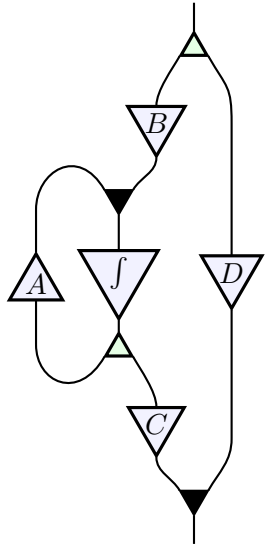
\begin{center}
% SFD for stateful morphisms
 \begin{tikzpicture}[thick]
% main nodes
   \node [delta] (usplit) at (-0.5,3) {};
   \node [upmultiply] (A) at (-2.6,-0.15) {\(A\)};
   \node [multiply] (B) at (-1,2) {\(B\)};
   \node [plus] (xdotsum) at (-1.5,1) {};
   \node [multiply] (int) at (-1.5,0) {\(\int\)};
   \node [delta] (xsplit) at (-1.5,-1) {};
   \node [multiply] (C) at (-1,-2) {\(C\)};
   \node [multiply] (D) at (0,0) {\(D\)};
   \node [plus] (ysum) at (-0.5,-3) {};

% auxiliary nodes
   \node [coordinate] (capend) [above of=A] {};
   \node [coordinate] (cupend) [below of=A, shift={(0,0.2)}] {};
   \node [coordinate] (ubend) [right of=B] {};
   \node [coordinate] (ybend) [right of=C, shift={(0,-0.2)}] {};

% wires
   \draw (usplit) -- +(0,0.6)
         (ysum) -- +(0,-0.6)
         (xsplit.left out) .. controls +(240:0.7) and +(270:0.5) .. (cupend)
         (xdotsum.left in) .. controls +(120:0.7) and +(90:0.5) .. (capend)
         (capend) -- (A) -- (cupend)
         (xsplit.right out) .. controls +(300:0.2) and +(90:0.2) .. (C.90)
         (C.270) .. controls +(270:0.2) and +(120:0.2) .. (ysum.left in)
         (ybend) .. controls +(270:0.3) and +(60:0.3) .. (ysum.right in)
         (ybend) -- (D) -- (ubend)
         (usplit.right out) .. controls +(300:0.5) and +(90:0.5) .. (ubend)
         (usplit.left out) .. controls +(240:0.2) and +(90:0.2) .. (B.90)
         (B.270) .. controls +(270:0.2) and +(60:0.2) .. (xdotsum.right in)
         (xdotsum) -- (int) -- (xsplit)
   ;
 \end{tikzpicture}
\end{center}
After taking Laplace transforms, we may also write the output \(y\) in terms of the input \(u\)
simply as
\[   y = (D+C(sI-A)^{-1}B) u .\]
Here the term \(Du\) gives the `direct' dependence of output on input, while the other term,
\(C(sI-A)^{-1}B) u\), gives its `indirect' dependence, mediated by the state \(x\).  We can
visualize this split into direct and indirect terms as a noncommuting square 
\begin{center}
\begin{tikzpicture}[->]
\node (A) at (0,0) {\(V_1\)};
\node (S) at (0,1.8) {\(S\)};
\node (T) at (2.2,1.8) {\(T\)};
\node (B) at (2.2,0) {\(V_2\)};

\path
(A) edge node[below] {\(D\)} (B)
edge node[left] {\(B\)} (S)
(S) edge node[above] {\((sI-A)^{-1}\)} (T)
(T) edge node[right] {\(C\)} (B)
;
\end{tikzpicture},
\end{center}
where \(D\) goes directly from the vector space \(V_1\) containing the input to the space \(V_2\) containing the output, while the other arrows compose to give the `indirect' map \(C(sI-A)^{-1}B\).  More abstractly, we can write such a square simply as
\begin{center}
\begin{tikzpicture}[->]
\node (A) at (0,0) {\(V_1\)};
\node (S) at (0,1.5) {\(S\)};
\node (T) at (1.5,1.5) {\(T\)};
\node (B) at (1.5,0) {\(V_2\)};

\path
(A) edge node[below] {\(d\)} (B)
edge node[left] {\(b\)} (S)
(S) edge node[above] {\(a\)} (T)
(T) edge node[right] {\(c\)} (B)
;
\end{tikzpicture},
\end{center}
Squares of this form serve as the morphisms in the PROPs we consider now.

\section{The Box construction}
\label{boxvsection}
In pursuit of this goal, we define a new construction that, for a suitable symmetric monoidal
category \(\catC\), forms a new symmetric monoidal category \(\boxful\).  The full details of how
this works can be found in Appendix~\ref{generalbox}.  For this section we will focus on the
particular case of the PROP \(\boxv\).

\begin{definition}
The category \(\boxv\) has
\begin{itemize}
  \item the same objects as \(\vectk\) (\emph{i.e.}\ vector spaces \(k^n\)),
  \item morphisms that are equivalence classes of
  \begin{center}
    \begin{tikzpicture}[->]
     \node (V)  at (0,0) {\(V_1\)};
     \node (VV) at (1.7,0) {\(V_1 \oplus V_1\)};
     \node (WS) at (4.4,0) {\(V_1 \oplus S\)};
     \node (WT) at (7.1,0) {\(V_1 \oplus T\)};
     \node (WW) at (9.4,0) {\(V_2 \oplus V_2\)};
     \node (W)  at (11.1,0) {\(V_2\)};

     \path
      (V)  edge node[above] {\(\Delta\)}                      (VV)
      (VV) edge node[above] {\(\mathrm{id}_{V_1} \oplus b\)} (WS)
      (WS) edge node[above] {\(\mathrm{id}_{V_1} \oplus a\)} (WT)
      (WT) edge node[above] {\(d \oplus c\)}                 (WW)
      (WW) edge node[above] {\(m\)}                           (W);
    \end{tikzpicture},
  \end{center}
  abbreviated \((d,c,a,b)\),
  \item composition given by
    \[(d',c',a',b') \of (d,c,a,b) = \left( d'd, [d'c \quad c'],
   \left[ \begin{array}{cc}
   a      & 0 \\
   a'b'ca & a'  \end{array} \right], 
   \left[ \begin{array}{c}
   b \\ b'd \end{array} \right] \right).\]
\end{itemize}
\end{definition}
The morphisms of \(\boxv\) can be depicted as non-commuting squares:
\begin{center}
    \begin{tikzpicture}[->]
     \node (A) at (0,0) {\(V_1\)};
     \node (S) at (0,1.5) {\(S\)};
     \node (T) at (1.5,1.5) {\(T\)};
     \node (B) at (1.5,0) {\(V_2\)};

     \path
      (A) edge node[below] {\(d\)} (B)
          edge node[left] {\(b\)} (S)
      (S) edge node[above] {\(a\)} (T)
      (T) edge node[right] {\(c\)} (B)
;
    \end{tikzpicture},
\end{center}
which explains the Box in the name of the category.  Two squares, \((d,c,a,b)\) and
\((d',c',a',b')\) are in the same equivalence class if there are isomorphisms \(\alpha \maps S \to
S'\) and \(\omega \maps T \to T'\) in \(\vectk\) such that the following diagram in \(\vectk\)
commutes:
\begin{center}
 \begin{tikzpicture}[->]
  \node (A) at (0,0) {\(V_1\)};
  \node (S) at (-1,1.5) {\(S\)};
  \node (T) at (2.5,1.5) {\(T\)};
  \node (B) at (1.5,0) {\(V_2\)};
  \node (S1) at (-1,-1.5) {\(S'\)};
  \node (T1) at (2.5,-1.5) {\(T'\)};

  \path
   (A) edge node[below right] {\(b'\)} (S1)
       edge node[above right] {\(b\)} (S)
   (S) edge node[above] {\(a\)} (T)
       edge node[left] {\(\alpha\)} (S1)
   (T) edge node[above left] {\(c\)} (B)
       edge node[right] {\(\omega\)} (T1)
   (S1) edge node[above] {\(a'\)} (T1)
   (T1) edge node[below left] {\(c'\)} (B)
;
 \end{tikzpicture}.
\end{center}
Since \(\boxv\) has the same objects as \(\vectk\), it is clear that \(\boxv\) is a skeletal
category.

We refer to the objects \(S\) and \(T\) as the \emph{prestate space} and \emph{state space},
respectively.  The formula for composition of morphisms in \(\boxv\) can be understood as coming
from the diagram:
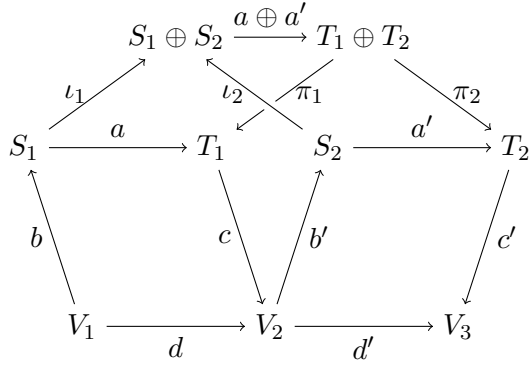
\begin{center}
% Almost pentagonal prism thingy
 \begin{tikzpicture}[->]
  \node (A) at (0,0) {\(V_1\)};
  \path (A) +(108:2.5) node (S1) {\(S_1\)};
  \path (S1) +(0:2.5) node (T1) {\(T_1\)};
  \path (A) +(0:2.5) node (B) {\(V_2\)};
  \path (B) +(0:2.5) node (C) {\(V_3\)};
  \path (B) +(72:2.5) node (S2) {\(S_2\)};
  \path (S2) +(0:2.5) node (T2) {\(T_2\)};
  \path (S1) +(36:2.5) node (S3) {\(S_1 \oplus S_2\)};
  \path (S3) +(0:2.5) node (T3) {\(T_1 \oplus T_2\)};

  \path
   (T1) edge node[left] {\(c\)} (B)
   (T2) edge node[right] {\(c'\)} (C)
   (T3) edge node[right] {\(\pi_1\)} (T1)
        edge node[right] {\(\pi_2\)} (T2)
;
  \draw[white,line width=5pt,-] (S2) -- (S3);

  \path
   (A) edge node[below] {\(d\)} (B)
       edge node[left] {\(b\)} (S1)
   (B) edge node[below] {\(d'\)} (C)
       edge node[right] {\(b'\)} (S2)
   (S1) edge node[above] {\(a\)} (T1)
        edge node[left] {\(\iota_1\)} (S3)
   (S2) edge node[above] {\(a'\)} (T2)
        edge node[left] {\(\iota_2\)} (S3)
   (S3) edge node[above] {\(a \oplus a'\)} (T3)
;
 \end{tikzpicture}.
\end{center}
The \(b\) and \(c\) sides in the composite come from the pentagons on the left and right,
respectively.  That is, \(b\) comes from summing the paths from \(V_1\) to \(S_1 \oplus S_2\), and
\(c\) comes from summing the paths from \(T_1 \oplus T_2\) to \(V_3\).  Similarly, the \(a\) side in
the composite comes from the paths from \(S_1 \oplus S_2\) to \(T_1 \oplus T_2\), which includes
the direct path \(a \oplus a'\) (taking \(S_1\) to \(T_1\) and \(S_2\) to \(T_2\)) and a looped
path that goes through \(a \oplus a'\) twice (taking \(S_1\) to \(T_2\)).  The \(d\) side in the
composite comes from the most direct path from \(V_1\) to \(V_3\).

\begin{theorem}
The category \(\boxv\) is a monoidal category with direct sum of vector spaces as the monoidal
product on objects of \(\boxv\), and \((d,c,a,b) \oplus (d',c',a',b') = (d \oplus d', c \oplus c', a
\oplus a', b \oplus b')\) as the monoidal product on morphisms.
\end{theorem}

\begin{proof}
To show \(\boxv\) is a category, we need to show the composition is well-defined and associative,
and the unit laws hold.  To show \(\boxv\) is a monoidal category, we also need to show the
associators and unitors exist and satisfy the pentagon and triangle equations.  We start by showing
composition in \(\boxv\) is well-defined.  Given two composable morphisms, \(f \maps V_1 \to V_2\)
and \(g \maps V_2 \to V_3\) with representatives \((d_1,c_1,a_1,b_1)\) and \((d_1',c_1',a_1',b_1')\)
for \(f\) and \((d_2,c_2,a_2,b_2)\) and \((d_2',c_2',a_2',b_2')\) for \(g\), we have the following
commutative diagrams:
\begin{center}
% Equivalence class thingy for f
 \begin{tikzpicture}[->]
  \node (A) at (0,0) {\(V_1\)};
  \node (S) at (-1,1.5) {\(S_1\)};
  \node (T) at (2.5,1.5) {\(T_1\)};
  \node (B) at (1.5,0) {\(V_2\)};
  \node (S1) at (-1,-1.5) {\(S_1'\)};
  \node (T1) at (2.5,-1.5) {\(T_1'\)};

  \path
   (A) edge node[below right] {\(b_1'\)} (S1)
       edge node[above right] {\(b_1\)} (S)
   (S) edge node[above] {\(a_1\)} (T)
       edge node[left] {\(\alpha_1\)} (S1)
   (T) edge node[above left] {\(c_1\)} (B)
       edge node[right] {\(\omega_1\)} (T1)
   (S1) edge node[above] {\(a_1'\)} (T1)
   (T1) edge node[below left] {\(c_1'\)} (B)
;
 \end{tikzpicture}
\end{center}
and
\begin{center}
% Equivalence class thingy for g
 \begin{tikzpicture}[->]
  \node (A) at (0,0) {\(V_2\)};
  \node (S) at (-1,1.5) {\(S_2\)};
  \node (T) at (2.5,1.5) {\(T_2\)};
  \node (B) at (1.5,0) {\(V_3\)};
  \node (S1) at (-1,-1.5) {\(S_2'\)};
  \node (T1) at (2.5,-1.5) {\(T_2'\)};

  \path
   (A) edge node[below right] {\(b_2'\)} (S1)
       edge node[above right] {\(b_2\)} (S)
   (S) edge node[above] {\(a_2\)} (T)
       edge node[left] {\(\alpha_2\)} (S1)
   (T) edge node[above left] {\(c_2\)} (B)
       edge node[right] {\(\omega_2\)} (T1)
   (S1) edge node[above] {\(a_2'\)} (T1)
   (T1) edge node[below left] {\(c_2'\)} (B)
;
 \end{tikzpicture}.
\end{center}
We leave it as an exercise to the reader to show \(\alpha_{12} = \alpha_1 \oplus \alpha_2 \maps S_1
\oplus S_2 \to S_1' \oplus S_2'\) and \(\omega_{12} = \omega_1 \oplus \omega_2 \maps T_1 \oplus T_2
\to T_1' \oplus T_2'\) are the isomorphisms required to make the corresponding diagram for \(g \of
f\) commute.

It is easy to verify that
\begin{center}
% Unit square thingy
 \begin{tikzpicture}[->]
  \node (A) at (0,0) {\(V\)};
  \node (S) at (0,1.5) {\(0\)};
  \node (T) at (1.5,1.5) {\(0\)};
  \node (B) at (1.5,0) {\(V\)};

  \path
   (A) edge node[below] {\(1_V\)} (B)
       edge node[left] {} (S)
   (S) edge node[above] {} (T)
   (T) edge node[right] {} (B)
;
 \end{tikzpicture}
\end{center}
is a left and right identity morphism in \(\boxv\).  The associators and unitors can be formed by
the same trick, but these (along with the pentagon and triangle equations) are trivial since
\(\vectk\) is a strict monoidal category.  It is also easy to see the monoidal product on morphisms
is compatible with composition, so it remains to show composition in \(\boxv\) is associative.

Aside from the source and target objects, morphisms in \(\boxv\) have six pieces of data: the
prestate space, the state space, and four linear maps, \(d\), \(c\), \(a\), and \(b\).  To check
associativity in \(\boxv\), we need to ensure both groupings of \(f_3 \of f_2 \of f_1\) of
composable morphisms in \(\boxv\) give the same results for all six of these pieces of data.  The
prestate space and the state space will be the same, thanks to the associativity of the monoidal
product in \(\vectk\).  Denoting the compositions \((d_j,c_j,a_j,b_j) \of (d_i,c_i,a_i,b_i)\) as
\((d_{ij},c_{ij},a_{ij},b_{ij})\), we get \((d_3,c_3,a_3,b_3) \of (d_{12},c_{12},a_{12},b_{12}) =
(d_{12,3},c_{12,3},a_{12,3},b_{12,3})\):
\begin{center}
% Associativity part 1
 \begin{tikzpicture}[->]
  \node (A) at (-0.3,0) {\(V_1\)};
  \node (B) at (2,0) {\(V_3\)};
  \node (C) at (3.5,0) {\(V_4\)};
  \node (S1) at (-1.3,1.5) {\(S_1 \oplus S_2\)};
  \node (T1) at (1,1.5) {\(T_1 \oplus T_2\)};
  \node (S2) at (3,1.5) {\(S_3\)};
  \node (T2) at (4.5,1.5) {\(T_3\)};

  \node (eq) at (5.15,0.75) {\(=\)};

  \node (A1) at (6.7,0) {\(V_1\)};
  \node (B1) at (10.0,0) {\(V_4\)};
  \node (S) at (6.7,1.5) {\((S_1 \oplus S_2) \oplus S_3\)};
  \node (T) at (10.0,1.5) {\((T_1 \oplus T_2) \oplus T_3\)};

  \path
   (A) edge node[below] {\(d_{12}\)} (B)
       edge node[left] {\(b_{12}\)} (S1)
   (B) edge node[below] {\(d_3\)} (C)
       edge node[right] {\(b_3\)} (S2)
   (S1) edge node[above] {\(a_{12}\)} (T1)
   (S2) edge node[above] {\(a_3\)} (T2)
   (T1) edge node[left] {\(c_{12}\)} (B)
   (T2) edge node[right] {\(c_3\)} (C)

   (A1) edge node[below] {\(d_{12,3}\)} (B1)
        edge node[left] {\(b_{12,3}\)} (S)
   (S) edge node[above] {\(a_{12,3}\)} (T)
   (T) edge node[right] {\(c_{12,3}\)} (B1)
;
 \end{tikzpicture}
\end{center}
and \((d_{23},c_{23},a_{23},b_{23}) \of (d_1,c_1,a_1,b_1) = (d_{1,23},c_{1,23},a_{1,23},b_{1,23})\):
\begin{center}
% Associativity part 2
 \begin{tikzpicture}[->]
  \node (A) at (-0.3,0) {\(X_1\)};
  \node (B) at (1.2,0) {\(X_2\)};
  \node (C) at (3.5,0) {\(X_4\)};
  \node (S1) at (-1.3,1.5) {\(S_1\)};
  \node (T1) at (0.2,1.5) {\(T_1\)};
  \node (S2) at (2.2,1.5) {\(S_2 \oplus S_3\)};
  \node (T2) at (4.5,1.5) {\(T_2 \oplus T_3\)};

  \node (eq) at (5.15,0.75) {\(=\)};

  \node (A1) at (6.7,0) {\(X_1\)};
  \node (B1) at (10.0,0) {\(X_4\)};
  \node (S) at (6.7,1.5) {\(S_1 \oplus (S_2 \oplus S_3)\)};
  \node (T) at (10.0,1.5) {\(T_1 \oplus (T_2 \oplus T_3)\)};

  \path
   (A) edge node[below] {\(d_1\)} (B)
       edge node[left] {\(b_1\)} (S1)
   (B) edge node[below] {\(d_{23}\)} (C)
       edge node[right] {\(b_{23}\)} (S2)
   (S1) edge node[above] {\(a_1\)} (T1)
   (S2) edge node[above] {\(a_{23}\)} (T2)
   (T1) edge node[left] {\(c_1\)} (B)
   (T2) edge node[right] {\(c_{23}\)} (C)

   (A1) edge node[below] {\(d_{1,23}\)} (B1)
        edge node[left] {\(b_{1,23}\)} (S)
   (S) edge node[above] {\(a_{1,23}\)} (T)
   (T) edge node[right] {\(c_{1,23}\)} (B1)
;
 \end{tikzpicture}.
\end{center}

For the linear map data, \(d\), \(c\), \(a\), and \(b\), the associativity of \(\boxv\) requires
\(d_{12,3} = d_{1,23}\), \(c_{12,3} = c_{1,23}\), \(a_{12,3} = a_{1,23}\), and \(b_{12,3} =
b_{1,23}\).  It is clear that the associativity requirement for \(d\) is met because composition of
linear maps is associative in \(\vectk\) --- \(d_{1,23} = d_1 (d_2 d_3) = (d_1 d_2) d_3 = d_{12,3}\).
% Top of the square is associative
We see the associativity requirement for \(a\) is met because \(a_{ij} = \left[
\begin{array}{cc}
a_i & 0 \\
a_j b_j c_i a_i & a_j
\end{array} \right],\)
which means
\[a_{12,3} = \left[
\begin{array}{cc}
a_{12} & 0 \\
a_3 b_3 c_{12} a_{12} & a_3
\end{array}
\right] = \left[
\begin{array}{ccc}
\left[\begin{array}{c}
a_1 \\ a_2 b_2 c_1 a_1
\end{array}\right. &
\left.\begin{array}{c}
0 \\ a_2
\end{array}\right] &
\begin{array}{c}
0 \\ 0
\end{array}
\\
\left[a_3 b_3 d_2 c_1 a_1 + a_3 b_3 c_2 a_2 b_2 c_1 a_1 \right. & \left. a_3 b_3 c_2 a_2 \right] & a_3
\end{array} \right],\]
since \(c_{12} = \left[ d_2 c_1 \quad c_2 \right]\).  A similar calculation gives the same matrix,
grouped slightly differently, for \(a_{1,23}\), 
\[a_{1,23} = \left[
\begin{array}{ccc}
a_1 & 0 & 0\\
\left[\begin{array}{c}
a_2 b_2 c_1 a_1\\
a_3 b_3 d_2 c_1 a_1 + a_3 b_3 c_2 a_2 b_2 c_1 a_1
\end{array}\right] &
\left[\begin{array}{c}
a_2 \\ a_3 b_3 c_2 a_2
\end{array}\right. &
\left.\begin{array}{c}
0 \\ a_3
\end{array}\right]
\end{array}\right].\]
% Yep.  Top of the square was associative

The proofs that \(c\) and \(b\) meet their respective associativity requirements are similar to each
other, transposed.  We present the argument for \(c\) and leave the argument for \(b\) to the
reader.  Since \(c_{ij} = [d_j c_i \quad c_j]\), we have
\begin{align*}
c_{12,3} & = [d_3 c_{12} \quad c_3]\\
         & = [d_3 [d_2 c_1 \quad c_2] \quad c_3]\\
         & = [[d_3 d_2 c_1 \quad d_3 c_2] \quad c_3]\\
         & = [d_3 d_2 c_1 \quad [d_3 c_2 \quad c_3]]\\
         & = [d_{23} c_1 \quad c_{23}] = c_{1,23}.
\end{align*}
So we see composition of morphisms in \(\boxv\) is associative.
\end{proof}

These kinds of manipulations of \(\boxv\) make more sense when \(\boxv\) is understood as a category
with an obvious evaluation functor \(\eval \maps \boxv \to \vectk\).  We can also find a
`feedthrough' functor \(\feed \maps \boxv \to \vectk\) and a functor in the reverse direction \(\G
\maps \vectk \to \boxv\).  The map of objects \(\eval_0 \maps \Obj(\boxv) \to \Obj(\vectk)\) is trivial,
\(\feed_0 = \eval_0\), and \(\G_0\) is its inverse.  The map of morphisms \(\eval_1 \maps
\Mor(\boxful) \to \Mor(\catC)\) is given by \(\eval_1(d,c,a,b) = d + c a b\), \(\feed_1 \maps
\Mor(\boxful) \to \Mor(\catC)\) is given by \(\feed_1(d,c,a,b) = d\), and \(\G_1 \maps \Mor(\catC)
\to \Mor(\boxful)\) is given by \(\G_1(d) = (d, !, 0, 0)\).  That is,

\begin{center}
\(\eval_1 \left(
 \begin{tikzpicture}[baseline=0.5cm,->]
  \node (A) at (0,0) {\(V_1\)};
  \node (C) at (1.4,0) {\(V_2\)};
  \node (B) at (0,1.4) {\(S\)};
  \node (D) at (1.4,1.4) {\(T\)};

  \path
   (A) edge node[below] {\(d\)} (C)
       edge node[left] {\(b\)} (B)
   (B) edge node[above] {\(a\)} (D)
   (D) edge node[right] {\(c\)} (C);
 \end{tikzpicture} \right) = d + c a b\),\quad
\(\feed_1 \left(
 \begin{tikzpicture}[baseline=0.5cm,->]
  \node (A) at (0,0) {\(V_1\)};
  \node (C) at (1.4,0) {\(V_2\)};
  \node (B) at (0,1.4) {\(S\)};
  \node (D) at (1.4,1.4) {\(T\)};

  \path
   (A) edge node[below] {\(d\)} (C)
       edge node[left] {\(b\)} (B)
   (B) edge node[above] {\(a\)} (D)
   (D) edge node[right] {\(c\)} (C);
 \end{tikzpicture} \right) = d\)
\\and
\(\G_1(d) =\)
 \begin{tikzpicture}[baseline=0.5cm,->]
  \node (A) at (0,0) {\(V_1\)};
  \node (C) at (1.4,0) {\(V_2\)};
  \node (B) at (0,1.4) {\(0\)};
  \node (D) at (1.4,1.4) {\(0\)};

  \path
   (A) edge node[below] {\(d\)} (C)
       edge (B)
   (B) edge (D)
   (D) edge (C);
 \end{tikzpicture}.
\end{center}

\begin{theorem}
\label{boxvectfunctors}
There are functors \(\eval\), \(\feed\), and \(\G\) as defined above which are \(\prop\) morphisms, and
\begin{itemize}
  \item \(\eval\) and \(\feed\) are full, but not faithful,
  \item \(\G\) is faithful, but not full,
  \item \(\G\) does not preserve limits or colimits.  In particular, \(\G\) has no adjoint.
\end{itemize}
\end{theorem}

\begin{proof}
It is easy to check that \(\eval\), \(\feed\), and \(\G\) all preserve identity maps.  Preservation
of composition is again easy to check for \(\feed\) and \(\G\), but \(\eval\) takes a little more
work.  On one hand, \(\eval(d,c,a,b) \of \eval(d',c',a',b') = (d+cab) \of (d'+c'a'b') =
dd'+dc'a'b'+cabd'+cabc'a'b'\).  On the other hand,
\begin{align*}
\eval((d,c,a,b) \of (d',c',a',b')) & = \eval\left(\left( dd', [dc' \quad c],
   \left[ \begin{array}{cc}
   a'      & 0 \\
   abc'a' & a  \end{array} \right], 
   \left[ \begin{array}{c}
   b' \\ bd' \end{array} \right] \right)\right)                               \\
                                   & = dd' + [dc' \quad c]
                                             \left[ \begin{array}{cc}
                                             a'      & 0 \\
                                              abc'a' & a  \end{array} \right] 
                                             \left[ \begin{array}{c}
                                             b' \\ bd' \end{array} \right]    \\
                                   & = dd'+dc'a'b'+cabc'a'b'+cabd'.
\end{align*}
Addition is commutative, so \(\eval\) preserves composition.

All three functors act as identities on objects, so it immediately follows they are essentially
surjective.  We note that \(\feed \of \G\) and \(\eval \of \G\) are both the identity functor on
\(\vectk\), which implies \(\feed\) and \(\eval\) are surjective on \emph{all} morphisms, hence
full.  This also implies \(\G\) is injective on morphisms, so \(\G\) is faithful.  On the other hand,
a morphism in \(\boxv\) between \(\G_0 (V_1)\) and \(\G_0 (V_2)\) where the prestate space or state
space are not isomorphic to the zero object is not the \(\G_1\)-image of any morphism in \(\vectk\),
so \(\G\) is not full.  Similarly, \(\feed\) and \(\eval\) cannot be faithful.

The object \(0\) in \(\vectk\) is both initial and terminal, so to show \(\G\) does not preserve
limits or colimits, it suffices to show \(\G_0 (0)\) is neither initial nor terminal.  For each
\(f \maps S \to T\) in \(\vectk\), there will be a morphism 
\(0_f \maps \G_0 (0) \to \G_0 (0)\) in \(\boxv\), given by
\begin{center}
\(0_f =\)
 \begin{tikzpicture}[baseline=0.5cm,->]
  \node (A) at (0,0) {\(0\)};
  \node (C) at (1.4,0) {\(0\)};
  \node (B) at (0,1.4) {\(S\)};
  \node (D) at (1.4,1.4) {\(T\)};

  \path
   (A) edge (C)
       edge (B)
   (B) edge node[above] {\(f\)} (D)
   (D) edge (C);
 \end{tikzpicture}.
\end{center}

Thus \(\G\) does not preserve initial or terminal objects, so it cannot preserve limits or colimits.
Because \(\G\) does not preserve limits, \(\G\) is not a right adjoint, and because \(\G\) does not
preserve colimits, \(\G\) is not a left adjoint.

  \begin{figure}[h]
    \centering
    \scalebox{0.85}{\begin{tikzpicture}
   \node (123) at (2.5,0)  {\(V_1 \oplus V_2 \oplus V_3\)};
   \node (231) at (5,2)    {\(V_2 \oplus V_3 \oplus V_1\)};
   \node (213) at (7.5,0)  {\(V_2 \oplus V_1 \oplus V_3\)};

   \draw [->] (123) to node [above left] {\((B_{V_1,V_2 \oplus V_3},!,0,0)\)} (231);
   \draw [->] (123) to node [below] {\((B_{V_1,V_2} \oplus \mathrm{Id},!,0,0)\)} (213);
   \draw [->] (213) to node [above right] {\((\mathrm{Id} \oplus B_{V_1,V_3},!,0,0)\)} (231);

   \draw [->,very thick] (-0.5,1) to node [above] {\(\G\)} (0.5,1);
   \draw [->,very thick] (0.5,0.5) to node [below] {\(\eval\)} (-0.5,0.5);

   \node (v123) at (-6,0)  {\(V_1 \oplus V_2 \oplus V_3\)};
   \node (v231) at (-4,2)  {\(V_2 \oplus V_3 \oplus V_1\)};
   \node (v213) at (-2,0)  {\(V_2 \oplus V_1 \oplus V_3\)};

   \draw [->] (v123) to node [above left] {\(B_{V_1,V_2 \oplus V_3}\)} (v231);
   \draw [->] (v123) to node [below] {\(B_{V_1,V_2} \oplus \mathrm{Id}\)} (v213);
   \draw [->] (v213) to node [above right] {\(\mathrm{Id} \oplus B_{V_1,V_3}\)} (v231);
\end{tikzpicture}}
    \caption[A coherence law preserved by $\G$]{A hexagon law is preserved by \(\eval\)
      and \(\G\).  Since \(\vectk\) and \(\boxv\) are strict monoidal categories, the associators
      are all identities, so three sides of the hexagon have been omitted.\label{stricthexagonlaw}}
  \end{figure}
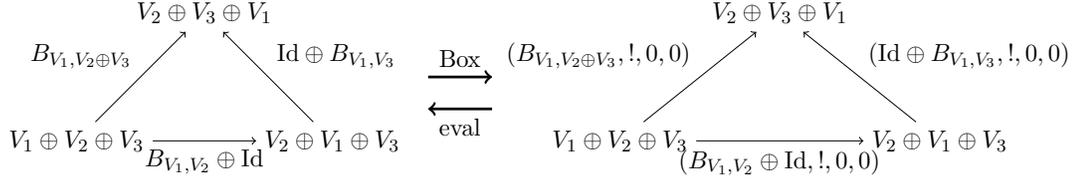

It remains to show these functors are \emph{symmetric} monoidal functors.  We already know
\(\vectk\) and \(\boxv\) are strict monoidal categories, so the associators and unitors are all
identity morphisms.  It is easy to check the symmetry isomorphisms in \(\boxv\) are the \(\G_1\)-images
of the symmetry isomorphisms in \(\vectk\).  Since \((d, !, 0, 0) \of (d', !, 0, 0) = (d \of d', !,
0, 0)\) in \(\boxv\), any coherence law in \(\vectk\) is preserved by \(\G\).  Similarly, \(\feed_1
(d,!,0,0) = \eval_1 (d,!,0,0) = d\), so \(\eval\) and \(\feed\) will also preserve the coherence
laws in \(\boxv\) that come from \(\vectk\).  See Figure~\ref{stricthexagonlaw} for an example of a
coherence law preserved by these three functors.

Thus we see all three functors are identity on objects \smf s between PROPs, hence \(\prop\)
morphisms.
\end{proof}

An alternate way to depict morphisms in \(\boxv\) is through string diagrams.  The morphism
\((d,c,a,b)\) is depicted:
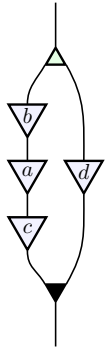
\begin{center}
\scalebox{0.75}{
  \begin{tikzpicture}[thick]
   \node [delta] (split) at (0,0) {};
   \node [coordinate] (above) [above of=split] {};
   \node [multiply] (d) at (0.5,-2) {\(d\)};
   \node [multiply] (a) at (-0.5,-2) {\(a\)};
   \node [multiply] (c) [below of=a] {\(c\)};
   \node [multiply] (b) [above of=a] {\(b\)};
   \node [plus] (join) at (0,-4.1) {};
   \node [coordinate] (below) [below of=join] {};
  
   \draw (above) -- (split.io)
         (split.left out) .. controls +(240:0.3) and +(90:0.3) .. (b.90)
         (c.io) .. controls +(270:0.3) and +(120:0.3) .. (join.left in)
         (join.io) -- (below)
         (split.right out) .. controls +(300:0.7) and +(90:0.7) .. (d.90)
         (b.io) -- (a.90) (a.io) -- (c.90)
         (d.io) .. controls +(270:0.7) and +(60:0.7) .. (join.right in);
  \end{tikzpicture}}.
\end{center}
This form will make it easier to see the connection between the state equations and the PROP
\(\st_k\), which we introduce in the next section.

\section{The PROP $\st$}
\label{stsection}
Recall the state equations, Equation~\ref{stateeq} and Equation~\ref{outputeq}, and the associated
convention that \(\mathrm{dim}(u) = m\), \(\mathrm{dim}(x) = n\), and \(\mathrm{dim}(y) = p\).
Note that the Laplace transform of these equations give the following: 
\begin{align*}
Xs - x(0) & = AX + BU \\
Y  & = CX + DU,
\end{align*}
where \(U, X,\) and \(Y\) are the Laplace transforms of \(u, x,\) and \(y\), respectively.  Under
the assumption that the initial state vector is zero, these equations can be expressed as a morphism in
\(\relks\) with the following signal-flow diagram: 
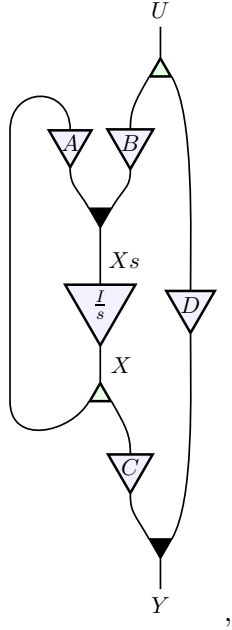
\begin{center}
\scalebox{0.8}{
\begin{tikzpicture}[thick]
  \node (u) at (0,0.4) {\(U\)};
  \node [delta] (d1) at (0,-0.6) {};
  \node [multiply] (b) at (-0.5,-1.8) {\(B\)};
  \node [multiply] (a) at (-1.5,-1.8) {\(A\)};
  \node [plus] (p1) at (-1,-3) {};
  \node [multiply] (s) at (-1,-4.5) {\(\frac{I}{s}\)};
  \node [multiply] (d) at (0.5,-4.5) {\(D\)};
  \node [delta] (d2) at (-1,-6) {};
  \node [multiply] (c) at (-0.5,-7.2) {\(C\)};
  \node [plus] (p2) at (0,-8.5) {};
  \node (y) at (0,-9.5) {\(Y\)};

  \node (dot) at (-0.6,-3.75) {\(Xs\)};
  \node (x) at (-0.65,-5.5) {\(X\)};

  \draw (u) -- (d1)
        (d1.left out) .. controls +(240:0.5) and +(90:0.3) .. (b.90)
        (d1.right out) .. controls +(-60:0.75) and +(90:1) .. (d.90)
        (b.io) .. controls +(270:0.3) and +(60:0.4) .. (p1.right in)
        (a.io) .. controls +(270:0.3) and +(120:0.4) .. (p1.left in)
        (p1.io) -- (s.90)
        (s.io) -- (d2.io)
        (d2.right out) .. controls +(-60:0.5) and +(90:0.3) .. (c.90)
        (c.io) .. controls +(270:0.3) and +(120:0.5) .. (p2.left in)
        (d.io) .. controls +(270:1) and +(60:0.75) .. (p2.right in)
        (p2.io) -- (y)
        (d2.left out) .. controls +(240:0.6) and +(270:0.75) .. (-2.5,-6.1)
        (-2.5,-6.1) -- (-2.5,-1.6)
        (-2.5,-1.6) .. controls +(90:0.75) and +(90:0.75) .. (a.90);
\end{tikzpicture}},
\end{center}
where the subdiagrams \(A, B, C,\) and \(D\) are the standard forms for their respective linear maps
in the state equations, and \(I\) is the identity on \(k^n\).  Note that each instance of addition
(\emph{resp.} duplication) in the diagram represents zero or more copies of the morphism, in
parallel.  This diagram can always be rewritten in a form without cups or caps using the reduction:
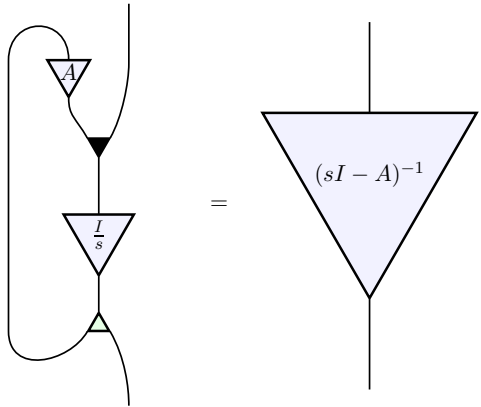
\begin{center}
\scalebox{0.8}{
\begin{tikzpicture}[thick]
  \node (b) at (-0.5,-1.8) {};
  \node [multiply] (a) at (-1.5,-1.8) {\(A\)};
  \node [plus] (p) at (-1,-3) {};
  \node [multiply] (s) at (-1,-4.5) {\(\frac{I}{s}\)};
  \node [multiply] (sia) at (3.5,-3.5) {\!\!\!\((sI-A)^{-1}\)\!\!\!};
  \node [delta] (d) at (-1,-6) {};
  \node (c) at (-0.5,-7.2) {};
  \node (eq) at (1,-4) {\(=\)};

  \draw (b.90) -- +(0,1)
        (b.90) .. controls +(270:0.3) and +(60:0.4) .. (p.right in)
        (a.io) .. controls +(270:0.3) and +(120:0.4) .. (p.left in)
        (p.io) -- (s.90)
        (s.io) -- (d.io)
        (sia.90) -- +(0,1.5)
        (sia.io) -- +(0,-1.5)
        (d.right out) .. controls +(-60:0.5) and +(90:0.3) .. (c.270)
        (d.left out) .. controls +(240:0.6) and +(270:0.75) .. (-2.5,-6.1)
        (-2.5,-6.1) -- (-2.5,-1.6)
        (-2.5,-1.6) .. controls +(90:0.75) and +(90:0.75) .. (a.90);
\end{tikzpicture}}.
\end{center}
Since \(A\) has no dependence on \(s\) in the time-independent case, \(sI-A\) will have an inverse
in \(\vectks\)\footnote{More generally, as long as the Laplace transform of \(A\) has no positive
powers of \(s\), \(sI-A\) can be guaranteed invertible.  Even in the case where \(A\) depends on
time, its Laplace transform will only ever contain positive powers of \(s\) if a distribution like
\(\frac{\mathrm{d}}{\mathrm{d}t}\delta(t)\) appears in \(A\).}.  Similarly, given \(a=(sI-A)^{-1}\),
we can find \(A = sI - a^{-1}\), so \(a\) and \(A\) both provide the same information.

Note that we have the right form for a morphism in \(\boxvs\), where \(a =
(sI-A)^{-1}, b=B, c=C, d=D\), but not all morphisms in \(\boxvs\) arise this way, \emph{e.g.}
\((s,!,0,0)\).  This motivates the definition of \(\st_k\) as the subPROP of \(\boxvs)\) where
the morphisms are of this form.

\begin{definition}
The category \(\st_k\) is the subPROP of \(\boxvs\) with
  \begin{itemize}
   \item the same objects as \(\vectks\)
   \item morphisms of the form \((D,C,(sI-A)^{-1},B)\) for some \(A,B,C,D \in \vectk\).
  \end{itemize}
\end{definition}
\begin{proposition}
\label{st_k:def}
The category \(\st_k\) is actually a PROP.
\end{proposition}
\begin{proof}
For this definition of \(\st_k\) to make sense, composition in \(\st_k\) must be closed:
\((D,C,(sI-A)^{-1},B) \of (D',C',(sI'-A')^{-1},B') = (D'',C'',(sI''-A'')^{-1},B'')\) for some
\(A'',B'',C'',D'' \in \vectk\).  We leave it to the reader to verify
\[
A'' = \left[ \begin{array}{cc}
A & 0 \\
B'C & A'
\end{array} \right],
\]
where \(I'' = I \oplus I'\).  Closure for the other three sides is immediate.
\end{proof}

In Appendix~\ref{generalbox} the Box construction is extended to allow, among other possibilities,
\(\Vectk\) to replace \(\vectk\).  While \(\Box(\Vectk)\) is no longer strict, it is still symmetric
monoidal.  Thus we can define \(\St_k\) as a symmetric monoidal subcategory of \(\Box(\Vectks)\)
similarly to how \(\st_k\) is defined here.  Indeed, \(\st_k\) is the skeletalization of \(\St_k\).

\section{Controllability and observability}
\label{cando}

Controllability and observability were introduced based on a physical interpretation of the
solutions to Equations \ref{stateeq} and \ref{outputeq}.  The results in this section have long been
well-known in the control theory literature \cite{Kalman60,Kalman63}.  When the matrices \(A(t)\)
and \(B(t)\) are continuous in \(t\), the general solution to Equation~\ref{stateeq} has the form
\[x(t) = \Phi(t,t_0) x(t_0) + \int_{t_0}^t \Phi(t, \tau) B(\tau) u(\tau) d\tau,\]
where \(\Phi(t, \tau)\) is a fundamental matrix of solutions satisfying \(\Phi(t,t) = I\) for all
\(t\).  A system is \Define{controllable} if for each state \(x\) and time \(t_0\) there is an input
function \(u(t)\) such that the state can be set to the equilibrium state in a finite amount of
time.  That is, there is a function \(u(t)\) such that \(x(t_1) = 0\) for some \(t_1 > t_0\).  An
equivalent characterization of controllability involves a symmetric matrix called the 
\emph{Controllability Gramian}, \(W_c(t_0, t_1)\).  A system is controllable if and only if
\[W_c(t_0, t_1) = \int_{t_0}^{t_1} \Phi(t_0, t) B(t) B^\top(t) \Phi^\top(t_0, t) dt\]
is positive definite for some \(t_1 > t_0\).  In the case of linear time-invariant systems, this
criterion is simplified to finding the row rank of the block matrix \(M_c = [B, AB, \dotsc,
A^{n-1}B]\).  This controllability matrix \(M_c\) is an \(n \times mn\) matrix, and a linear
time-invariant system is controllable when
\[\mathrm{rank}(M_c) = n.\]

While controllability of a system ignores the system's output, observability of a system ignores the
system's input, which can be accomplished by setting the input \(u(t) = 0\).  A system is
\Define{observable} if the state \(x(t_0)\) of the system can be determined at some later time
\(t_1\) by setting the input function \(u(t)\) to zero and measuring the output function \(y(t)\).
Kalman noticed a `duality principle' connecting controllability and observability.  If you
\begin{enumerate}
\item Reverse the direction of time,
\item Swap input and output constraints,
\item Replace \(\Phi\) with \(\Phi^\top\),
\end{enumerate}
then an observable system is transformed into a controllable system and vice versa.  Explicitly,
this duality transforms Equations \ref{stateeq} and \ref{outputeq} by making the following
replacements:
\begin{align*}
t - t_0 &= t_0 - t',\\
A(t-t_0) &\Leftrightarrow A^\top(t_0-t'),\\
B(t-t_0) &\Leftrightarrow C^\top(t_0-t'), \\
C(t-t_0) &\Leftrightarrow B^\top(t_0-t'),\\
D(t-t_0) &\Leftrightarrow D^\top(t_0-t').
\end{align*}
From this duality, we can see the \emph{Observability Gramian} will be
\[W_o(t_0, t_1) = \int_{t_0}^{t_1} \Phi^\top(t, t_1) C^\top(t) C(t) \Phi(t, t_1) dt,\]
which says a system is observable when \(W_o(t_0, t_1)\) is positive definite for some \(t_1 >
t_0\).  We also see the observability matrix for a linear time-invariant system will be \(M_o = [C,
CA, \dotsc, CA^{n-1}]^\top\).  This observability matrix \(M_o\) is an \(np \times n\) matrix, and
a linear time-invariant system is determined to be observable using the column rank:
\[\mathrm{rank}(M_o) = n.\]

If we view \(M_c\) as a linear transformation \(\R^{mn} \to \R^n\), a controllable system
has \(\mathrm{rank}(M_c) = n\), which means \(M_c\) is an epimorphism for a controllable system.
Similarly, when \(M_o \maps \R^n \to \R^{np}\) is viewed as a linear transformation, the system is
observable exactly when \(M_o\) is a monomorphism.

% Observation:  This duality principle is a time-varying version of Soboci\'nski's `bizarro' duality.

Diagrammatically, this can be expressed as saying that a stateful morphism \break % hack to prevent
                                % overfull hbox together with a line break in the middle of math
\((D,C,(sI-A)^{-1},B)\) is controllable if:
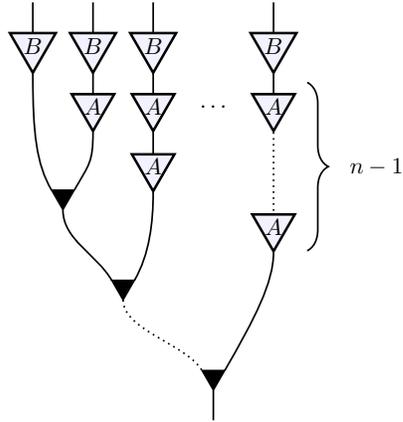
\begin{center}
\scalebox{0.8}{ \begin{tikzpicture}[thick]
   \node [multiply] (B1) {\(B\)};
   \node [multiply, right of=B1] (B2) {\(B\)};
   \node [multiply, right of=B2] (B3) {\(B\)};
   \node [multiply, below of=B2] (A21) {\(A\)};
   \node [multiply, below of=B3] (A31) {\(A\)};
   \node [multiply, below of=A31] (A32) {\(A\)};
   \node [right of=A31] (dots) {\(\dots\)};
   \node [multiply, right of=dots] (An1) {\(A\)};
   \node [multiply, above of=An1] (Bn) {\(B\)};
   \node [multiply, below of=An1, shift={(0,-1)}] (An) {\(A\)};

   \node [plus, below of=A21, shift={(-0.5,-0.5)}] (12sum) {};
   \node [plus, below of=A32, shift={(-0.5,-1)}] (123sum) {};
   \node [plus, below of=A32, shift={(1,-2.5)}] (12nsum) {};

   \draw (B1.90) -- +(90:0.5);
   \draw (B2.90) -- +(90:0.5);
   \draw (B3.90) -- +(90:0.5);
   \draw (Bn.90) -- +(90:0.5);
   \draw (12nsum.io) -- +(270:0.5);
   \draw (B2) -- (A21);
   \draw (B3) -- (A31) -- (A32);
   \draw (Bn) -- (An1);
   \draw[dotted] (An1) -- (An)
         (123sum.io) .. controls +(270:0.5) and +(120:0.5) .. (12nsum.left in);
   \draw (A21.270) .. controls +(270:0.5) and +(60:0.5) .. (12sum.right in)
         (A32.270) .. controls +(270:0.5) and +(60:0.5) .. (123sum.right in)
         (An.270) .. controls +(270:0.5) and +(60:1) .. (12nsum.right in)
         (B1.270) .. controls +(270:1) and +(120:0.5) .. (12sum.left in)
         (12sum.io) .. controls +(270:0.5) and +(120:0.5) .. (123sum.left in);

\draw [decorate,decoration={brace,amplitude=10pt},xshift=-4pt,yshift=0pt]
(4.7,-0.6) -- (4.7,-3.4)node [midway,xshift=3em] {\(n-1\)};

\end{tikzpicture} }
\end{center}
is an epimorphism in \(\vectk\), and it is observable if:
\begin{center}
\scalebox{0.8}{ \begin{tikzpicture}[thick]
   \node [multiply] (B1) {\(C\)};
   \node [multiply, right of=B1] (B2) {\(C\)};
   \node [multiply, right of=B2] (B3) {\(C\)};
   \node [multiply, above of=B2] (A21) {\(A\)};
   \node [multiply, above of=B3] (A31) {\(A\)};
   \node [multiply, above of=A31] (A32) {\(A\)};
   \node [right of=A31] (dots) {\(\dots\)};
   \node [multiply, right of=dots] (An1) {\(A\)};
   \node [multiply, below of=An1] (Bn) {\(C\)};
   \node [multiply, above of=An1, shift={(0,1)}] (An) {\(A\)};

   \node [delta, above of=A21, shift={(-0.5,0.5)}] (12sum) {};
   \node [delta, above of=A32, shift={(-0.5,1)}] (123sum) {};
   \node [delta, above of=A32, shift={(1,2.5)}] (12nsum) {};

   \draw (B1.270) -- +(270:0.5);
   \draw (B2.270) -- +(270:0.5);
   \draw (B3.270) -- +(270:0.5);
   \draw (Bn.270) -- +(270:0.5);
   \draw (12nsum.io) -- +(90:0.5);
   \draw (B2) -- (A21);
   \draw (B3) -- (A31) -- (A32);
   \draw (Bn) -- (An1);
   \draw[dotted] (An1) -- (An)
         (123sum.io) .. controls +(90:0.5) and +(240:0.5) .. (12nsum.left out);
   \draw (A21.90) .. controls +(90:0.5) and +(300:0.5) .. (12sum.right out)
         (A32.90) .. controls +(90:0.5) and +(300:0.5) .. (123sum.right out)
         (An.90) .. controls +(90:0.5) and +(300:1) .. (12nsum.right out)
         (B1.90) .. controls +(90:1) and +(240:0.5) .. (12sum.left out)
         (12sum.io) .. controls +(90:0.5) and +(240:0.5) .. (123sum.left out);

\draw [decorate,decoration={brace,amplitude=10pt},xshift=-4pt,yshift=0pt]
(4.7,3.4) -- (4.7,0.6)node [midway,xshift=3em] {\(n-1\)};

\end{tikzpicture} }
\end{center}
is a monomorphism in \(\vectk\).

Soboci\'nski noted \cite{GLA-epimono} there are purely diagrammatic tests for determining whether a
linear relation is an epimorphism and whether it is a monomorphism.  Given a linear map \(F \maps V
\to W\), it is a monomorphism if \(F^{\dagger} F = 1_V\) and an epimorphism if \(F F^{\dagger} =
1_W\).   This extends to linear relations.  Diagrammatically, given a linear relation \(F\),
depicted for convenience as
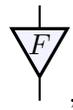
\begin{center}
 \begin{tikzpicture}[thick]
   \node[multiply] (F) {\(F\)};
   \draw (F.90) -- +(0,0.3) (F.270) -- +(0,-0.3);
 \end{tikzpicture},
\end{center}
\(F\) is a monomorphism if
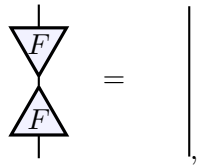
\begin{center}
% Sobo mono
 \begin{tikzpicture}[thick]
   \node[multiply] (F) {\(F\)};
   \node[upmultiply] (Fup) [below of=F] {\(F\)};
   \node (eq) [right of=F, shift={(0,-0.5)}] {\(=\)};
   \node[coordinate] (point) [right of=eq] {};
   \draw (point) -- +(0,1) (point) -- +(0,-1);
   \draw (F.90) -- +(0,0.3) (F) -- (Fup) (Fup.270) -- +(0,-0.3);
 \end{tikzpicture},
\end{center}
and \(F\) is an epimorphism if
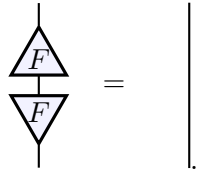
\begin{center}
% Sobo epi
 \begin{tikzpicture}[thick]
   \node[multiply] (F) {\(F\)};
   \node[upmultiply] (Fup) [above of=F, shift={(0,-0.3)}] {\(F\)};
   \node (eq) [right of=F, shift={(0,0.35)}] {\(=\)};
   \node[coordinate] (point) [right of=eq] {};
   \draw (point) -- +(0,1.1) (point) -- +(0,-1.1);
   \draw (Fup.90) -- +(0,0.3) (F) -- (Fup) (F.270) -- +(0,-0.3);
 \end{tikzpicture}.
\end{center}

Combining these results, we can say a stateful morphism \((D,C,(sI-A)^{-1},B)\) is controllable when
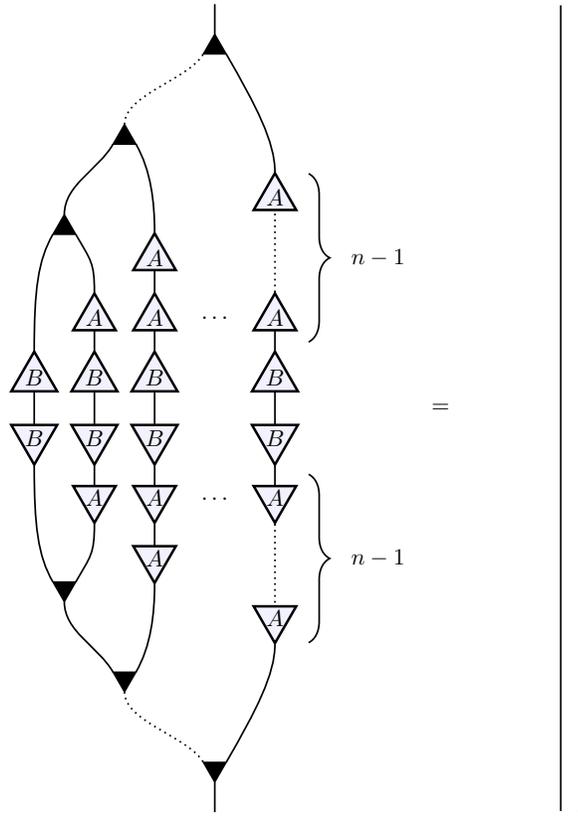
\begin{center}
\scalebox{0.80} { \begin{tikzpicture}[thick]
   \node [multiply] (B1) {\(B\)};
   \node [multiply, right of=B1] (B2) {\(B\)};
   \node [multiply, right of=B2] (B3) {\(B\)};
   \node [multiply, below of=B2] (A21) {\(A\)};
   \node [multiply, below of=B3] (A31) {\(A\)};
   \node [multiply, below of=A31] (A32) {\(A\)};
   \node [right of=A31] (dots) {\(\dots\)};
   \node [multiply, right of=dots] (An1) {\(A\)};
   \node [multiply, above of=An1] (Bn) {\(B\)};
   \node [multiply, below of=An1, shift={(0,-1)}] (An) {\(A\)};

   \node [plus, below of=A21, shift={(-0.5,-0.5)}] (12sum) {};
   \node [plus, below of=A32, shift={(-0.5,-1)}] (123sum) {};
   \node [plus, below of=A32, shift={(1,-2.5)}] (12nsum) {};

   \node [upmultiply, above of=B1] (B1a) {\(B\)};
   \node [upmultiply, right of=B1a] (B2a) {\(B\)};
   \node [upmultiply, right of=B2a] (B3a) {\(B\)};
   \node [upmultiply, above of=B2a] (A21a) {\(A\)};
   \node [upmultiply, above of=B3a] (A31a) {\(A\)};
   \node [upmultiply, above of=A31a] (A32a) {\(A\)};
   \node [right of=A31a] (dots) {\(\dots\)};
   \node [upmultiply, right of=dots] (An1a) {\(A\)};
   \node [upmultiply, below of=An1a] (Bna) {\(B\)};
   \node [upmultiply, above of=An1a, shift={(0,1)}] (Ana) {\(A\)};

   \node [coplus, above of=A21a, shift={(-0.5,0.5)}] (12suma) {};
   \node [coplus, above of=A32a, shift={(-0.5,1)}] (123suma) {};
   \node [coplus, above of=A32a, shift={(1,2.5)}] (12nsuma) {};

   \draw (B1) -- (B1a);
   \draw (B2) -- (B2a);
   \draw (B3) -- (B3a);
   \draw (Bn) -- (Bna);
   \draw (12nsum.io) -- +(270:0.5);
   \draw (B2) -- (A21);
   \draw (B3) -- (A31) -- (A32);
   \draw (Bn) -- (An1);
   \draw[dotted] (An1) -- (An)
         (123sum.io) .. controls +(270:0.5) and +(120:0.5) .. (12nsum.left in);
   \draw (A21.270) .. controls +(270:0.5) and +(60:0.5) .. (12sum.right in)
         (A32.270) .. controls +(270:0.5) and +(60:0.5) .. (123sum.right in)
         (An.270) .. controls +(270:0.5) and +(60:1) .. (12nsum.right in)
         (B1.270) .. controls +(270:1) and +(120:0.5) .. (12sum.left in)
         (12sum.io) .. controls +(270:0.5) and +(120:0.5) .. (123sum.left in);

\draw [decorate,decoration={brace,amplitude=10pt},xshift=-4pt,yshift=0pt]
(4.7,-0.6) -- (4.7,-3.4)node [midway,xshift=3em] {\(n-1\)};

   \draw (12nsuma.io) -- +(90:0.5);
   \draw (B2a) -- (A21a);
   \draw (B3a) -- (A31a) -- (A32a);
   \draw (Bna) -- (An1a);
   \draw[dotted] (An1a) -- (Ana)
         (123suma.io) .. controls +(90:0.5) and +(240:0.5) .. (12nsuma.left out);
   \draw (A21a.90) .. controls +(90:0.5) and +(300:0.5) .. (12suma.right out)
         (A32a.90) .. controls +(90:0.5) and +(300:0.5) .. (123suma.right out)
         (Ana.90) .. controls +(90:0.5) and +(300:1) .. (12nsuma.right out)
         (B1a.90) .. controls +(90:1) and +(240:0.5) .. (12suma.left out)
         (12suma.io) .. controls +(90:0.5) and +(240:0.5) .. (123suma.left out);

\draw [decorate,decoration={brace,amplitude=10pt},xshift=-4pt,yshift=0pt]
(4.7,4.4) -- (4.7,1.6)node [midway,xshift=3em] {\(n-1\)};

   \node (eq) [right of=Bn, shift={(1.75,0.5)}] {\(=\)};
   \node[coordinate, right of=eq, shift={(1,0)}] (point) {};

   \draw (point) -- +(0,6.7) (point) -- +(0,-6.7);
\end{tikzpicture} },
\end{center}
and observable when
\begin{center}
\scalebox{0.80} 
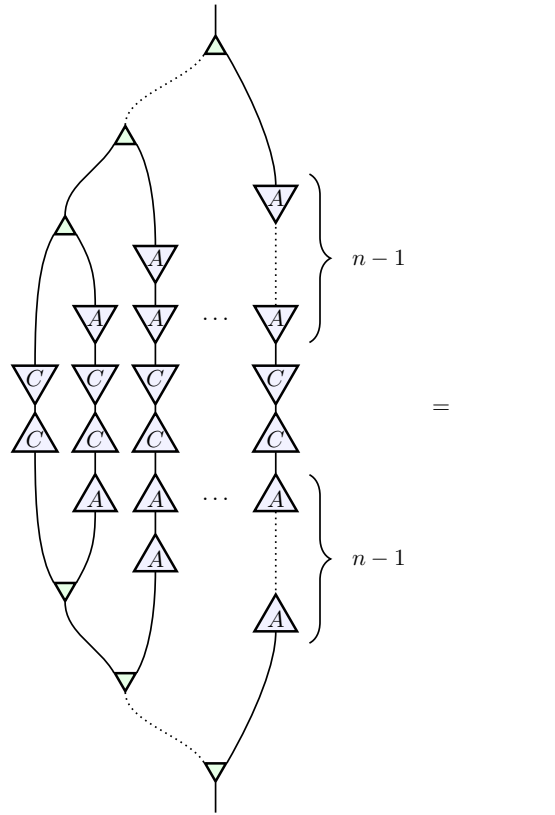
{ \begin{tikzpicture}[thick]
   \node [upmultiply] (B1) {\(C\)};
   \node [upmultiply, right of=B1] (B2) {\(C\)};
   \node [upmultiply, right of=B2] (B3) {\(C\)};
   \node [upmultiply, below of=B2] (A21) {\(A\)};
   \node [upmultiply, below of=B3] (A31) {\(A\)};
   \node [upmultiply, below of=A31] (A32) {\(A\)};
   \node [right of=A31] (dots) {\(\dots\)};
   \node [upmultiply, right of=dots] (An1) {\(A\)};
   \node [upmultiply, above of=An1] (Bn) {\(C\)};
   \node [upmultiply, below of=An1, shift={(0,-1)}] (An) {\(A\)};

   \node [codelta, below of=A21, shift={(-0.5,-0.5)}] (12sum) {};
   \node [codelta, below of=A32, shift={(-0.5,-1)}] (123sum) {};
   \node [codelta, below of=A32, shift={(1,-2.5)}] (12nsum) {};

   \node [multiply, above of=B1] (B1a) {\(C\)};
   \node [multiply, right of=B1a] (B2a) {\(C\)};
   \node [multiply, right of=B2a] (B3a) {\(C\)};
   \node [multiply, above of=B2a] (A21a) {\(A\)};
   \node [multiply, above of=B3a] (A31a) {\(A\)};
   \node [multiply, above of=A31a] (A32a) {\(A\)};
   \node [right of=A31a] (dots) {\(\dots\)};
   \node [multiply, right of=dots] (An1a) {\(A\)};
   \node [multiply, below of=An1a] (Bna) {\(C\)};
   \node [multiply, above of=An1a, shift={(0,1)}] (Ana) {\(A\)};

   \node [delta, above of=A21a, shift={(-0.5,0.5)}] (12suma) {};
   \node [delta, above of=A32a, shift={(-0.5,1)}] (123suma) {};
   \node [delta, above of=A32a, shift={(1,2.5)}] (12nsuma) {};

   \draw (B1) -- (B1a);
   \draw (B2) -- (B2a);
   \draw (B3) -- (B3a);
   \draw (Bn) -- (Bna);
   \draw (12nsum.io) -- +(270:0.5);
   \draw (B2) -- (A21);
   \draw (B3) -- (A31) -- (A32);
   \draw (Bn) -- (An1);
   \draw[dotted] (An1) -- (An)
         (123sum.io) .. controls +(270:0.5) and +(120:0.5) .. (12nsum.left in);
   \draw (A21.270) .. controls +(270:0.5) and +(60:0.5) .. (12sum.right in)
         (A32.270) .. controls +(270:0.5) and +(60:0.5) .. (123sum.right in)
         (An.270) .. controls +(270:0.5) and +(60:1) .. (12nsum.right in)
         (B1.270) .. controls +(270:1) and +(120:0.5) .. (12sum.left in)
         (12sum.io) .. controls +(270:0.5) and +(120:0.5) .. (123sum.left in);

\draw [decorate,decoration={brace,amplitude=10pt},xshift=-4pt,yshift=0pt]
(4.7,-0.6) -- (4.7,-3.4)node [midway,xshift=3em] {\(n-1\)};

   \draw (12nsuma.io) -- +(90:0.5);
   \draw (B2a) -- (A21a);
   \draw (B3a) -- (A31a) -- (A32a);
   \draw (Bna) -- (An1a);
   \draw[dotted] (An1a) -- (Ana)
         (123suma.io) .. controls +(90:0.5) and +(240:0.5) .. (12nsuma.left out);
   \draw (A21a.90) .. controls +(90:0.5) and +(300:0.5) .. (12suma.right out)
         (A32a.90) .. controls +(90:0.5) and +(300:0.5) .. (123suma.right out)
         (Ana.90) .. controls +(90:0.5) and +(300:1) .. (12nsuma.right out)
         (B1a.90) .. controls +(90:1) and +(240:0.5) .. (12suma.left out)
         (12suma.io) .. controls +(90:0.5) and +(240:0.5) .. (123suma.left out);

\draw [decorate,decoration={brace,amplitude=10pt},xshift=-4pt,yshift=0pt]
(4.7,4.4) -- (4.7,1.6)node [midway,xshift=3em] {\(n-1\)};

   \node (eq) [right of=Bn, shift={(1.75,0.5)}] {\(=\)};
   \node[coordinate, right of=eq, shift={(1,0)}] (point) {};

   \draw (point) -- +(0,6.7) (point) -- +(0,-6.7);
\end{tikzpicture} }.
\end{center}

While the controllability and observability criteria deal with linear maps (morphisms in
\(\vectk\)), the detour through \(\st_k\) is still necessary.  The linear maps \(A\), \(B\), and
\(C\) are defined in terms of stateful morphisms: a linear map in \(\vectks\) has no way of
`knowing' what \(A\), \(B\), and \(C\) are.  At best, a state space of minimum dimension can be
determined.  Through the rose-tinted glasses of \(\vectks\) alone, every morphism appears to be both
controllable and observable!

\chapter{The PROP $\goodflow$}
\label{goodflow}
In this chapter we will define the PROP \(\goodflow_k\) of `good' signal-flow diagrams over a field
\(k\).  Roughly speaking, a `good' signal-flow diagram is one for which we can describe
controllability and observability via stateful morphisms, using the results from
Section~\ref{cando}.  A sufficient condition on signal-flow diagrams for controllability and
observability to make sense, then, is for the signal-flow diagram to be a morphism in
\(\goodflow_k\).  A more formal description of what constitutes a `good' signal-flow diagram
involves the commutative square in Figure~\ref{fig:contflow}.
  \begin{figure}[!h]
    \centering
    \begin{tikzpicture}
   \node (SF) {\(\sigflow_{k,s}\)};
   \node (P) [above of=SF, shift={(0,1)}] {\(\goodflow_k\vphantom{\vectks}\)};
   \node (FR) [right of=SF, shift={(2,0)}] {\(\relks\)};
   \node (ST) [above of=FR, shift={(0,1)}] {\(\st_k\vphantom{\vectks}\)};

   \draw [->] (P) to node [above] {\(\dbox\)} (ST);
   \draw [right hook->] (P) to node [left] {\(j\)} (SF);
   \draw [->] (ST) to node [right] {\(i \of \eval\)} (FR);
   \draw [->] (SF) to node [below] {\(\bbox\)} (FR);
\end{tikzpicture}
    \caption[The $\prop$ morphism $\dbox$ makes this square commute]{This square in \(\prop\) commutes.\label{fig:contflow}}
  \end{figure}
The limitations of \(\st_k\) force limitations on the kinds of signal-flow diagrams that can be
considered `good' here.  The simplest signal-flow diagrams that fail to be `good' are the generators
\(\cup\) and \(\cap\).  Indeed, control theorists never write these explicitly, and only implicitly
use them as parts of larger signal-flow diagrams.

In order to formalize this discussion, we consider \(\sigflow_k\), the free PROP of all signal-flow
diagrams over a field \(k\), which we informally referred to at the end of
Section~\ref{presrksection}.  Recall, the signal-flow diagrams that generate this PROP are denoted
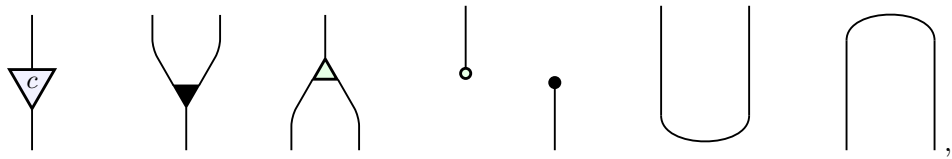
\begin{center}
\scalebox{0.9}{
 \begin{tikzpicture}[thick]
% scalar multiplication
   \node [coordinate] (in) at (0,2) {};
   \node [multiply] (mult) at (0,1) {\(c\)};
   \node [coordinate] (out) at (0,0) {};
   \draw (in) -- (mult) -- (out);
   \end{tikzpicture}
\hspace{3 em}
\begin{tikzpicture}[thick]
% addition
   \node [plus] (adder) at (0,0.85) {};
   \node [coordinate] (f) at (-0.5,1.5) {}; 
   \node [coordinate] (g) at (0.5,1.5) {};
   \node [coordinate] (out) at (0,0) {};
   \node [coordinate] (pref) at (-0.5,2) {};
   \node [coordinate] (preg) at (0.5,2) {};

   \draw [rounded corners] (pref) -- (f) -- (adder.left in);
   \draw [rounded corners] (preg) -- (g) -- (adder.right in);
   \draw (adder) -- (out);
   \end{tikzpicture} 
\hspace{2em}
  \begin{tikzpicture}[thick]
% duplication
   \node[delta] (dupe) at (0,1.15) {};
   \node[coordinate] (o1) at (-0.5,0.5) {};
   \node[coordinate] (o2) at (0.5,0.5) {};
   \node[coordinate] (in) at (0,2) {};
   \node [coordinate] (posto1) at (-0.5,0) {};
   \node [coordinate] (posto2) at (0.5,0) {};

   \draw[rounded corners] (posto1) -- (o1) -- (dupe.left out);
   \draw[rounded corners] (posto2) -- (o2) -- (dupe.right out);
   \draw (in) -- (dupe);
\end{tikzpicture}
\hspace{3em}
\begin{tikzpicture}[thick]
% deletion
   \node[coordinate] (in) at (0,2) {};
   \node [bang] (mult) at (0,1) {};
   \node [hole] (heightHolder) at (0,0) {};
   \draw (in) -- (mult);
   \end{tikzpicture}
\hspace{2em}
  \begin{tikzpicture}[thick]
% zero
   \node[hole] (heightHolder) at (0,2) {};
   \node [coordinate] (out) at (0,0) {};
   \node [zero] (del) at (0,1) {};
   \draw (del) -- (out);
\end{tikzpicture}
\hspace{3em}
 \begin{tikzpicture}[thick]
% cup
   \node [coordinate] (3) at (0,0.375) {};
   \node [coordinate] (4) at (1.3,0.375) {};
   \node [coordinate] (1) at (0,2) {};
   \node [coordinate] (2) at (1.3,2) {};
   \path
   (1) edge (3)
   (2) edge (4)
   (3) edge [-, bend right=90] (4);
   \end{tikzpicture}
        \hspace{3em}
   \begin{tikzpicture}[thick]
% cap 
   \node [coordinate] (3) at (0,1.625) {};
   \node [coordinate] (4) at (1.3,1.625) {};
   \node [coordinate] (1) at (0,0) {};
   \node [coordinate] (2) at (1.3,0) {};
   \path
   (3) edge (1)
   (4) edge (2)
   (3) edge [bend left=90] (4);
   \end{tikzpicture}
},
\end{center}%
where \(c \in k\) is arbitrary.  There is an obvious `black-box' functor \(\bbox \maps \sigflow_k
\to \relk\) which is also a \(\prop\) morphism, thanks to the machinery of Chapter~\ref{PROPs}.

More generally, suppose \(S\) is any subset of \(k\).  Then we can define \(\sigflow_S\) to be the
free PROP on the above generators, where \(c\) is restricted to be an element of \(S\).  This is a
subPROP of \(\sigflow_k\), so we can restrict the black-box functor \(\bbox \maps \sigflow_k \to
\relk\) and obtain a \(\prop\) morphism, which by abuse of notation we call \(\bbox \maps \sigflow_S
\to \relk\) and let the context indicate which \(\bbox\) is intended.  Note that when \(S\)
generates \(k\) as a field, the black-box functor from \(\sigflow_S\) is full.  This abuse will be
used primarily for \(k \cup \{\frac{1}{s}\}\) as a subset of \(k(s)\), with \(\sigflow_{k \cup
\{\frac{1}{s}\}}\) abbreviated as \(\sigflow_{k,s}\).  We call \(\sigflow_{k,s}\) the PROP of
signal-flow diagrams with integrators.  Integrators are treated separately from the other scalings
here because integration plays a special role in control theory.  Compared to \(\sigflow_k\),
\(\sigflow_{k,s}\) is a free PROP that has one extra generator to indicate integration:
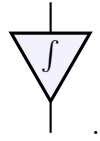
\begin{center}
\begin{tikzpicture}[thick]
 \node [integral] (int) {\(\int\)};
 \draw (int.90) -- +(0,0.4)
       (int.270) -- +(0,-0.4);
\end{tikzpicture}.
\end{center}

In Theorem~\ref{boxvectfunctors} we found some \(\prop\) morphisms from \(\st_k\) to \(\vectks\),
namely \(\eval\) and \(\feed\).  By composing \(\eval\) with the inclusion map \(i \maps \vectks \to
\relks\), we get a \(\prop\) morphism \(i \of \eval \maps \st_k \to \relks\) that will be
instrumental in determining which signal-flow diagrams are `good'.  Ultimately we will find a
\(\prop\) morphism \(\dbox \maps \goodflow_k \to \st_k\) and a commutative square in the category
\(\prop\) shown in Figure~\ref{fig:contflow}.  The desire to find the PROP \(\goodflow_k\) that
makes this square commute leads us directly to the definition of \(\goodflow_k\), stated in
Definition~\ref{gooddef}.

\section{Finding $\goodflow_k$}
\label{flowgood}
Since we wish to define \(\goodflow_k\) as a PROP of signal-flow diagrams with a certain property,
we first define PROPs of signal-flow diagrams and then narrow down to those with that property.
\begin{definition}
\label{sigflowdef}
The PROP \Define{\(\sigflow_k\)} is the free PROP on the generators of \(\relk\).  That is,
\(\sigflow_k = \F(\Sigma_{\relk})\), where
\begin{align*}\Sigma_{\relk} = {} & \{\sigma_+ \maps 2 \to 1, \sigma_0 \maps 0 \to 1, \sigma_\Delta \maps 1 \to 2,
\sigma_! \maps 1 \to 0, \sigma_\cup \maps 2 \to 0, \sigma_\cap \maps 0 \to 2\}
\\ & \cup \{\sigma_{s_c} \maps 1 \to 1 : c \in k\},\end{align*}
so that the \(\F\)-images of these formal symbols are the generators listed in Lemma~\ref{gensrk}:
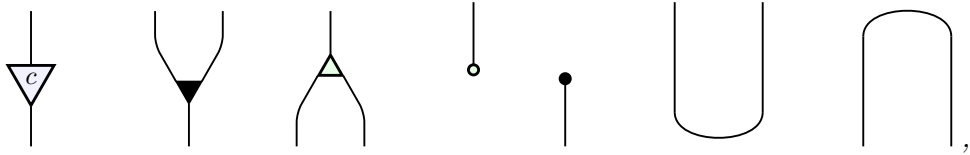
\begin{center}
\scalebox{0.9}{
 \begin{tikzpicture}[thick]
% scalar multiplication
   \node [coordinate] (in) at (0,2) {};
   \node [multiply] (mult) at (0,1) {\(c\)};
   \node [coordinate] (out) at (0,0) {};
   \draw (in) -- (mult) -- (out);
   \end{tikzpicture}
\hspace{3 em}
\begin{tikzpicture}[thick]
% addition
   \node [plus] (adder) at (0,0.85) {};
   \node [coordinate] (f) at (-0.5,1.5) {}; 
   \node [coordinate] (g) at (0.5,1.5) {};
   \node [coordinate] (out) at (0,0) {};
   \node [coordinate] (pref) at (-0.5,2) {};
   \node [coordinate] (preg) at (0.5,2) {};

   \draw [rounded corners] (pref) -- (f) -- (adder.left in);
   \draw [rounded corners] (preg) -- (g) -- (adder.right in);
   \draw (adder) -- (out);
   \end{tikzpicture} 
\hspace{2em}
  \begin{tikzpicture}[thick]
% duplication
   \node[delta] (dupe) at (0,1.15) {};
   \node[coordinate] (o1) at (-0.5,0.5) {};
   \node[coordinate] (o2) at (0.5,0.5) {};
   \node[coordinate] (in) at (0,2) {};
   \node [coordinate] (posto1) at (-0.5,0) {};
   \node [coordinate] (posto2) at (0.5,0) {};

   \draw[rounded corners] (posto1) -- (o1) -- (dupe.left out);
   \draw[rounded corners] (posto2) -- (o2) -- (dupe.right out);
   \draw (in) -- (dupe);
\end{tikzpicture}
\hspace{3em}
\begin{tikzpicture}[thick]
% deletion
   \node[coordinate] (in) at (0,2) {};
   \node [bang] (mult) at (0,1) {};
   \node [hole] (heightHolder) at (0,0) {};
   \draw (in) -- (mult);
   \end{tikzpicture}
\hspace{2em}
  \begin{tikzpicture}[thick]
% zero
   \node[hole] (heightHolder) at (0,2) {};
   \node [coordinate] (out) at (0,0) {};
   \node [zero] (del) at (0,1) {};
   \draw (del) -- (out);
\end{tikzpicture}
\hspace{3em}
 \begin{tikzpicture}[thick]
% cup
   \node [coordinate] (3) at (0,0.375) {};
   \node [coordinate] (4) at (1.3,0.375) {};
   \node [coordinate] (1) at (0,2) {};
   \node [coordinate] (2) at (1.3,2) {};
   \path
   (1) edge (3)
   (2) edge (4)
   (3) edge [-, bend right=90] (4);
   \end{tikzpicture}
        \hspace{3em}
   \begin{tikzpicture}[thick]
% cap 
   \node [coordinate] (3) at (0,1.625) {};
   \node [coordinate] (4) at (1.3,1.625) {};
   \node [coordinate] (1) at (0,0) {};
   \node [coordinate] (2) at (1.3,0) {};
   \path
   (3) edge (1)
   (4) edge (2)
   (3) edge [bend left=90] (4);
   \end{tikzpicture}
},
\end{center}%
where \(c \in k\) is arbitrary.
\end{definition}

For any subset \(S\) of the field \(k\), we also have the following definition.
\begin{definition}
\label{sigflowSdef}
The PROP \Define{\(\sigflow_S\)} is the free PROP on the generators of \(\relk\), with the scaling
morphisms \(s_c\) restricted to \(c \in S\).  That is, \(\sigflow_S = \F(\Sigma_{\relS})\), where
\begin{align*}\Sigma_{\relS} = {} & \{\sigma_+ \maps 2 \to 1, \sigma_0 \maps 0 \to 1, \sigma_\Delta \maps 1 \to 2,
\sigma_! \maps 1 \to 0, \sigma_\cup \maps 2 \to 0, \sigma_\cap \maps 0 \to 2\}
\\ & \cup \{\sigma_{s_c} \maps 1 \to 1 : c \in S\},\end{align*}
so that the \(\F\)-images of these formal symbols are the generators listed in
Definition~\ref{sigflowdef} above, but now with \(c \in S\).
\end{definition}

Since \(\sigflow_k = \F\Sigma_{\relk}\), there is an obvious \(\prop\) morphism \(\bbox \maps
\sigflow_k \to \relk\), namely the coequalizer of \(\F \sige_{\relk} \rightrightarrows
\F\Sigma_{\relk}\).  Following Lawvere's ideas on functorial semantics \cite{Lawvere}, we can treat
the PROP \(\sigflow_k\) as providing `syntax' and the PROP \(\relk\) as providing `semantics'.  In
other words, morphisms in \(\sigflow_k\) are a notation---signal-flow diagrams---while morphisms in
\(\relk\) are what this notation stands for, namely linear relations between inputs and outputs,
which we arrive at by imposing the equations of \(\relk\) on signal-flow diagrams.  Understood in
this light, the black-box functor \(\bbox \maps \sigflow_k \to \relk\) assigns to each signal-flow
diagram its meaning: the linear relation it stands for.

Because controllability and observability involve extending \(k\) to \(k(s)\), we will be concerned
with the PROPs \(\relks\) and \(\sigflow_{k,s}\), where \(\sigflow_{k,s}\) is the free PROP
\(\F(\Sigma_{\relk,s})\), and \(\Sigma_{\relk,s} = \Sigma_{\relk} \cup \{\sigma_{\!\int} \maps 1 \to
1\}\).  We take the \(\F\)-image of \(\sigma_{\!\int}\) to be a scaling by \(\frac{1}{s}\) in the
field extension \(k(s)\).  Then \(k \cup \frac{1}{s}\) is a subset of \(k(s)\), and we identify
\(\sigflow_{k,s}\) with \(\sigflow_{k \cup \frac{1}{s}}\).  Since \(k(s)\) is generated as a field
by \(k \cup \frac{1}{s}\), the restriction of \(\bbox \maps \sigflow_{k(s)} \to \relks\) to \(\bbox
\maps \sigflow{k,s} \to \relks\) is still full.  A factor of \(\frac{1}{s}\) comes from the Laplace
transform of \(\int_0^t f(\tau) d\tau\), so a `scale by \(\frac{1}{s}\)' morphism will be referred
to as an `integrator'.

In order to extend the controllability and observability results in \(\st_k\) to \(\sigflow_{k,s}\),
we need to find a subPROP \(\P\) of \(\sigflow_{k,s}\) that maps to \(\st_k\) such that there are
arrows making this diagram commute:
\begin{center}
\begin{tikzpicture}
   \node (SF) {\(\sigflow_{k,s}\)};
   \node (P) [above of=SF, shift={(0,1)}] {\(\P\vphantom{\vectks}\)};
   \node (ST) [right of=P, shift={(1,0)}] {\(\st_k\vphantom{\vectks}\)};
   \node (FV) [right of=ST, shift={(2.5,0)}] {\(\vectks\)};
   \node (FR) [below of=FV, shift={(0,-1)}] {\(\relks\)};

   \draw [->] (ST) to node [above] {\(\eval\)} (FV);
   \draw [right hook->] (P) to node [left] {\(j\)} (SF);
   \draw [right hook->] (FV) to node [right] {\(i\)} (FR);
   \draw [->] (P) to (ST);
   \draw [->] (SF) to node [below] {\(\bbox\)} (FR);
\end{tikzpicture},
\end{center}
where the evaluation map \(\eval \maps \st_k \to \vectks\) mentioned in
Theorem~\ref{boxvectfunctors} sends any stateful morphism to the linear map it describes.  Because
\(\P\) will be the PROP of signal-flow diagrams one might expect a control theorist to draw, we will
name this PROP \(\goodflow_k\).  Our goal, then, is to find this subPROP \(\goodflow_k\) of
\(\sigflow_{k,s}\) and a \(\prop\) morphism \(\dbox \maps \goodflow_k \to \st_k\) such that
\begin{center}
\begin{tikzpicture}
   \node (SF) {\(\sigflow_{k,s}\)};
   \node (P) [above of=SF, shift={(0,1)}] {\(\goodflow_k\vphantom{\vectks}\)};
   \node (FR) [right of=SF, shift={(2,0)}] {\(\relks\)};
   \node (ST) [above of=FR, shift={(0,1)}] {\(\st_k\vphantom{\vectks}\)};

   \draw [->] (P) to node [above] {\(\dbox\)} (ST);
   \draw [right hook->] (P) to node [left] {\(j\)} (SF);
   \draw [->] (ST) to node [right] {\(i \of \eval\)} (FR);
   \draw [->] (SF) to node [below] {\(\bbox\)} (FR);
\end{tikzpicture}
\end{center}
commutes.

This commutative square is not a pullback square, so \(\goodflow_k\) and \(\dbox\) cannot
be simply defined by a pullback.  A pullback square here would not give us a subPROP of
\(\sigflow_{k,s}\), since \(i \of \eval\) is not a monomorphism in \(\prop\).  To define
\(\goodflow_k\), we first need four processes that can be applied to any signal-flow diagram \(f
\maps m \to p\) in \(\sigflow_{k,s}\).  In Chapter~\ref{stateful} we saw that signal-flow diagrams
of the following form play a fundamental role in the state-space approach:
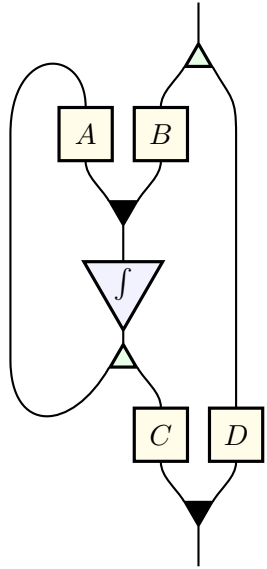
\begin{center}
\begin{tikzpicture}[thick]
% main nodes
   \node [delta] (usplit) at (2.5,3) {};
   \node [sqnode] (A) at (1,2) {\(A\)};
   \node [sqnode] (B) at (2,2) {\(B\)};
   \node [plus] (xdotsum) at (1.5,1) {};
   \node [multiply] (int) at (1.5,0) {\(\int\)};
   \node [delta] (xsplit) at (1.5,-1) {};
   \node [sqnode] (C) at (2,-2) {\(C\)};
   \node [sqnode] (D) at (3,-2) {\(D\)};
   \node [plus] (ysum) at (2.5,-3) {};

% auxiliary nodes
   \node [coordinate] (capend) [left of=A] {};
   \node [coordinate] (cupend) [below of=capend, shift={(0,-2)}] {};
   \node [coordinate] (ubend) [right of=B] {};

% wires
   \draw (usplit) -- +(0,0.75)
         (ysum) -- +(0,-0.75)
         (xsplit.left out) .. controls +(240:1) and +(270:1) .. (cupend)
         (A.90) .. controls +(90:0.8) and +(90:1.2) .. (capend)
         (capend) -- (cupend)
         (xsplit.right out) .. controls +(300:0.2) and +(90:0.2) .. (C.90)
         (C.270) .. controls +(270:0.2) and +(120:0.2) .. (ysum.left in)
         (D.270) .. controls +(270:0.2) and +(60:0.2) .. (ysum.right in)
         (D.90) -- (ubend)
         (usplit.right out) .. controls +(300:0.5) and +(90:0.5) .. (ubend)
         (usplit.left out) .. controls +(240:0.2) and +(90:0.2) .. (B.90)
         (A.270) .. controls +(270:0.2) and +(120:0.2) .. (xdotsum.left in)
         (B.270) .. controls +(270:0.2) and +(60:0.2) .. (xdotsum.right in)
         (xdotsum) -- (int) -- (xsplit)
;
\end{tikzpicture}.
\end{center}
When applied to a signal-flow diagram of this form, our four processes pick out the linear relations
\(A\), \(B\), \(C\), and \(D\) --- at least when they are linear \emph{maps}.  When all four of
these linear relations are linear maps, we decree that \(f\) is a morphism in \(\goodflow_k\) and
these linear maps form the \(\st_k\) morphism \(\dbox(f)\).

In what follows, we describe each of these four processes in generality and illustrate how they work
for signal-flow diagrams of the above form.  We choose the designations of \(m\), \(n\) and \(p\)
for the number of input wires, integrators, and output wires, respectively, on \(f\) in order to be
consistent with the convention established in Section~\ref{stsection}.

\begin{definition}
The linear relation \(A(f) \maps k^n \asrelto k^n\) is obtained from the signal-flow diagram \(f\) by
replacing the \(n\) wires leaving the integrators in \(!^p \of f \of \zero^m\) with inputs and the
\(n\) wires entering the integrators with outputs, then black-boxing the resulting signal-flow
diagram.
\end{definition}
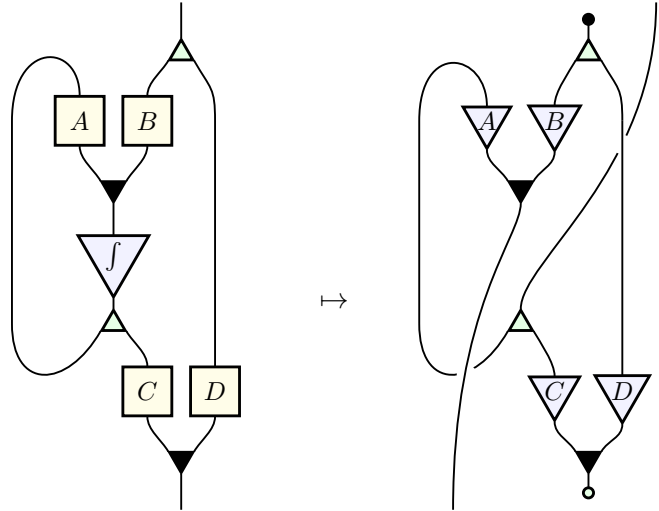
\vbox{\begin{example} \({}\)
\begin{center}
\scalebox{0.90}{\begin{tikzpicture}[thick]
% main nodes
   \node [delta] (usplit) at (2.5,3) {};
   \node [sqnode] (A) at (1,2) {\(A\)};
   \node [sqnode] (B) at (2,2) {\(B\)};
   \node [plus] (xdotsum) at (1.5,1) {};
   \node [multiply] (int) at (1.5,0) {\(\int\)};
   \node [delta] (xsplit) at (1.5,-1) {};
   \node [sqnode] (C) at (2,-2) {\(C\)};
   \node [sqnode] (D) at (3,-2) {\(D\)};
   \node [plus] (ysum) at (2.5,-3) {};

% auxiliary nodes
   \node [coordinate] (capend) [left of=A] {};
   \node [coordinate] (cupend) [below of=capend, shift={(0,-2)}] {};
   \node [coordinate] (ubend) [right of=B] {};

% wires
   \draw (usplit) -- +(0,0.75)
         (ysum) -- +(0,-0.75)
         (xsplit.left out) .. controls +(240:1) and +(270:1) .. (cupend)
         (A.90) .. controls +(90:0.8) and +(90:1.2) .. (capend)
         (capend) -- (cupend)
         (xsplit.right out) .. controls +(300:0.2) and +(90:0.2) .. (C.90)
         (C.270) .. controls +(270:0.2) and +(120:0.2) .. (ysum.left in)
         (D.270) .. controls +(270:0.2) and +(60:0.2) .. (ysum.right in)
         (D.90) -- (ubend)
         (usplit.right out) .. controls +(300:0.5) and +(90:0.5) .. (ubend)
         (usplit.left out) .. controls +(240:0.2) and +(90:0.2) .. (B.90)
         (A.270) .. controls +(270:0.2) and +(120:0.2) .. (xdotsum.left in)
         (B.270) .. controls +(270:0.2) and +(60:0.2) .. (xdotsum.right in)
         (xdotsum) -- (int) -- (xsplit)
;
\end{tikzpicture}} \qquad \raisebox{2.7cm}{\(\mapsto\)} \qquad \scalebox{0.90}{\begin{tikzpicture}[thick]
% main nodes
   \node [delta] (usplit) at (2.5,3) {};
   \node [multiply] (A) at (1,2) {\(A\)};
   \node [multiply] (B) at (2,2) {\(B\)};
   \node [plus] (xdotsum) at (1.5,1) {};
   \node [delta] (xsplit) at (1.5,-1) {};
   \node [multiply] (C) at (2,-2) {\(C\)};
   \node [multiply] (D) at (3,-2) {\(D\)};
   \node [plus] (ysum) at (2.5,-3) {};

% auxiliary nodes
   \node [coordinate] (capend) [left of=A] {};
   \node [coordinate] (cupend) [below of=capend, shift={(0,-2)}] {};
   \node [coordinate] (ubend) [right of=B] {};

   \node [bang] (bot) at (2.5,-3.5) {};
   \node [zero] (top) at (2.5,3.5) {};
   \node [coordinate] (u) at (3.5,3.75) {};
   \node [coordinate] (y) at (0.5,-3.75) {};

% wires
   \draw (usplit) -- (top)
         (ysum) -- (bot)
         (xsplit.left out) .. controls +(240:1) and +(270:1) .. (cupend)
         (A.90) .. controls +(90:0.8) and +(90:1.2) .. (capend)
         (capend) -- (cupend)
         (xsplit.right out) .. controls +(300:0.2) and +(90:0.2) .. (C.90)
         (C.270) .. controls +(270:0.2) and +(120:0.2) .. (ysum.left in)
         (D.270) .. controls +(270:0.2) and +(60:0.2) .. (ysum.right in)
         (xsplit.io) .. controls +(90:0.7) and +(270:2.5) .. (u)
         (usplit.right out) .. controls +(300:0.5) and +(90:0.5) .. (ubend)
         (usplit.left out) .. controls +(240:0.2) and +(90:0.2) .. (B.90)
         (A.270) .. controls +(270:0.2) and +(120:0.2) .. (xdotsum.left in)
         (B.270) .. controls +(270:0.2) and +(60:0.2) .. (xdotsum.right in);

   \node [hole] at (3,1.65) {};
   \node [hole] at (0.7,-1.75) {};

   \draw (xdotsum.io) .. controls +(270:0.5) and +(90:3) .. (y)
         (D.90) -- (ubend);
\end{tikzpicture}}
\end{center}
\end{example}}
\begin{definition}
The linear relation \(B(f) \maps k^m \asrelto k^n\) is obtained from \(f\) by replacing the \(n\)
wires entering the integrators in \(!^p \of f\) with outputs and replacing the \(n\) wires leaving
the integrators with \(\zero^n\), then black-boxing.
\end{definition}
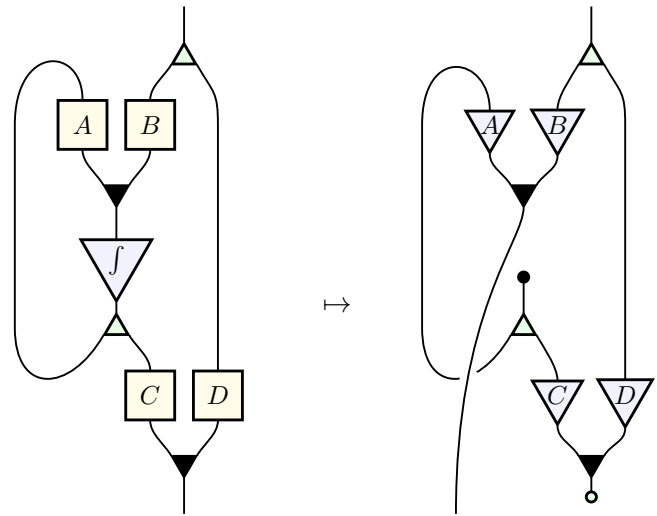
\vbox{\begin{example} \({}\)
\begin{center}
\scalebox{0.90}{\begin{tikzpicture}[thick]
% main nodes
   \node [delta] (usplit) at (2.5,3) {};
   \node [sqnode] (A) at (1,2) {\(A\)};
   \node [sqnode] (B) at (2,2) {\(B\)};
   \node [plus] (xdotsum) at (1.5,1) {};
   \node [multiply] (int) at (1.5,0) {\(\int\)};
   \node [delta] (xsplit) at (1.5,-1) {};
   \node [sqnode] (C) at (2,-2) {\(C\)};
   \node [sqnode] (D) at (3,-2) {\(D\)};
   \node [plus] (ysum) at (2.5,-3) {};

% auxiliary nodes
   \node [coordinate] (capend) [left of=A] {};
   \node [coordinate] (cupend) [below of=capend, shift={(0,-2)}] {};
   \node [coordinate] (ubend) [right of=B] {};

% wires
   \draw (usplit) -- +(0,0.75)
         (ysum) -- +(0,-0.75)
         (xsplit.left out) .. controls +(240:1) and +(270:1) .. (cupend)
         (A.90) .. controls +(90:0.8) and +(90:1.2) .. (capend)
         (capend) -- (cupend)
         (xsplit.right out) .. controls +(300:0.2) and +(90:0.2) .. (C.90)
         (C.270) .. controls +(270:0.2) and +(120:0.2) .. (ysum.left in)
         (D.270) .. controls +(270:0.2) and +(60:0.2) .. (ysum.right in)
         (D.90) -- (ubend)
         (usplit.right out) .. controls +(300:0.5) and +(90:0.5) .. (ubend)
         (usplit.left out) .. controls +(240:0.2) and +(90:0.2) .. (B.90)
         (A.270) .. controls +(270:0.2) and +(120:0.2) .. (xdotsum.left in)
         (B.270) .. controls +(270:0.2) and +(60:0.2) .. (xdotsum.right in)
         (xdotsum) -- (int) -- (xsplit)
;
\end{tikzpicture}} \qquad \raisebox{2.7cm}{\(\mapsto\)} \qquad \scalebox{0.90}{\begin{tikzpicture}[thick]
% main nodes
   \node [delta] (usplit) at (2.5,3) {};
   \node [multiply] (A) at (1,2) {\(A\)};
   \node [multiply] (B) at (2,2) {\(B\)};
   \node [plus] (xdotsum) at (1.5,1) {};
   \node [delta] (xsplit) at (1.5,-1) {};
   \node [multiply] (C) at (2,-2) {\(C\)};
   \node [multiply] (D) at (3,-2) {\(D\)};
   \node [plus] (ysum) at (2.5,-3) {};

% auxiliary nodes
   \node [coordinate] (capend) [left of=A] {};
   \node [coordinate] (cupend) [below of=capend, shift={(0,-2)}] {};
   \node [coordinate] (ubend) [right of=B] {};

   \node [bang] (bot) at (2.5,-3.5) {};
   \node [zero] (int) at (1.5,-0.25) {};
   \node [coordinate] (y) at (0.5,-3.75) {};

% wires
   \draw (usplit) -- +(0,0.75)
         (ysum) -- (bot)
         (xsplit.left out) .. controls +(240:1) and +(270:1) .. (cupend)
         (A.90) .. controls +(90:0.8) and +(90:1.2) .. (capend)
         (capend) -- (cupend)
         (xsplit.right out) .. controls +(300:0.2) and +(90:0.2) .. (C.90)
         (C.270) .. controls +(270:0.2) and +(120:0.2) .. (ysum.left in)
         (D.270) .. controls +(270:0.2) and +(60:0.2) .. (ysum.right in)
         (D.90) -- (ubend)
         (usplit.right out) .. controls +(300:0.5) and +(90:0.5) .. (ubend)
         (usplit.left out) .. controls +(240:0.2) and +(90:0.2) .. (B.90)
         (A.270) .. controls +(270:0.2) and +(120:0.2) .. (xdotsum.left in)
         (B.270) .. controls +(270:0.2) and +(60:0.2) .. (xdotsum.right in)
         (xsplit) -- (int);

   \node [hole] at (0.7,-1.75) {};

   \draw (xdotsum.io) .. controls +(270:0.5) and +(90:3) .. (y);
  \end{tikzpicture}}
\end{center}
\end{example}}
\begin{definition}
The linear relation \(C(f) \maps k^n \asrelto k^p\) is obtained from \(f\) dually, by replacing the
\(n\) wires leaving the integrators in \(f \of \zero^m\) with inputs and replacing the \(n\) wires
entering the integrators with \(!^n\), then black-boxing.
\end{definition}
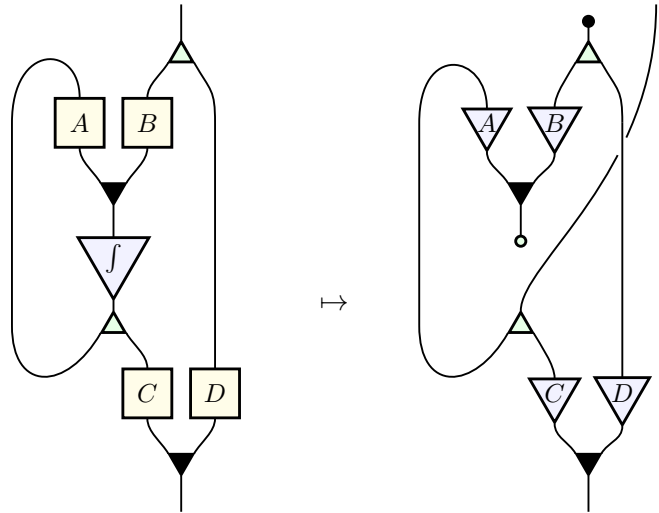
\vbox{\begin{example} \({}\)
\begin{center}
\scalebox{0.90}{\begin{tikzpicture}[thick]
% main nodes
   \node [delta] (usplit) at (2.5,3) {};
   \node [sqnode] (A) at (1,2) {\(A\)};
   \node [sqnode] (B) at (2,2) {\(B\)};
   \node [plus] (xdotsum) at (1.5,1) {};
   \node [multiply] (int) at (1.5,0) {\(\int\)};
   \node [delta] (xsplit) at (1.5,-1) {};
   \node [sqnode] (C) at (2,-2) {\(C\)};
   \node [sqnode] (D) at (3,-2) {\(D\)};
   \node [plus] (ysum) at (2.5,-3) {};

% auxiliary nodes
   \node [coordinate] (capend) [left of=A] {};
   \node [coordinate] (cupend) [below of=capend, shift={(0,-2)}] {};
   \node [coordinate] (ubend) [right of=B] {};

% wires
   \draw (usplit) -- +(0,0.75)
         (ysum) -- +(0,-0.75)
         (xsplit.left out) .. controls +(240:1) and +(270:1) .. (cupend)
         (A.90) .. controls +(90:0.8) and +(90:1.2) .. (capend)
         (capend) -- (cupend)
         (xsplit.right out) .. controls +(300:0.2) and +(90:0.2) .. (C.90)
         (C.270) .. controls +(270:0.2) and +(120:0.2) .. (ysum.left in)
         (D.270) .. controls +(270:0.2) and +(60:0.2) .. (ysum.right in)
         (D.90) -- (ubend)
         (usplit.right out) .. controls +(300:0.5) and +(90:0.5) .. (ubend)
         (usplit.left out) .. controls +(240:0.2) and +(90:0.2) .. (B.90)
         (A.270) .. controls +(270:0.2) and +(120:0.2) .. (xdotsum.left in)
         (B.270) .. controls +(270:0.2) and +(60:0.2) .. (xdotsum.right in)
         (xdotsum) -- (int) -- (xsplit)
;
\end{tikzpicture}} \qquad \raisebox{2.7cm}{\(\mapsto\)} \qquad \scalebox{0.90}{\begin{tikzpicture}[thick]
% main nodes
   \node [delta] (usplit) at (2.5,3) {};
   \node [multiply] (A) at (1,2) {\(A\)};
   \node [multiply] (B) at (2,2) {\(B\)};
   \node [plus] (xdotsum) at (1.5,1) {};
   \node [delta] (xsplit) at (1.5,-1) {};
   \node [multiply] (C) at (2,-2) {\(C\)};
   \node [multiply] (D) at (3,-2) {\(D\)};
   \node [plus] (ysum) at (2.5,-3) {};

% auxiliary nodes
   \node [coordinate] (capend) [left of=A] {};
   \node [coordinate] (cupend) [below of=capend, shift={(0,-2)}] {};
   \node [coordinate] (ubend) [right of=B] {};

   \node [zero] (top) at (2.5,3.5) {};
   \node [bang] (int) at (1.5,0.25) {};
   \node [coordinate] (u) at (3.5,3.75) {};

% wires
   \draw (usplit) -- (top)
         (ysum) -- +(0,-0.75)
         (xsplit.left out) .. controls +(240:1) and +(270:1) .. (cupend)
         (A.90) .. controls +(90:0.8) and +(90:1.2) .. (capend)
         (capend) -- (cupend)
         (xsplit.right out) .. controls +(300:0.2) and +(90:0.2) .. (C.90)
         (C.270) .. controls +(270:0.2) and +(120:0.2) .. (ysum.left in)
         (D.270) .. controls +(270:0.2) and +(60:0.2) .. (ysum.right in)
         (xsplit.io) .. controls +(90:0.7) and +(270:2.5) .. (u)
         (usplit.right out) .. controls +(300:0.5) and +(90:0.5) .. (ubend)
         (usplit.left out) .. controls +(240:0.2) and +(90:0.2) .. (B.90)
         (A.270) .. controls +(270:0.2) and +(120:0.2) .. (xdotsum.left in)
         (B.270) .. controls +(270:0.2) and +(60:0.2) .. (xdotsum.right in)
         (xdotsum) -- (int);

   \node [hole] at (3,1.65) {};

   \draw (D.90) -- (ubend);
  \end{tikzpicture}}
\end{center}
\end{example}}
\begin{definition}
The linear relation \(D(f) \maps k^m \asrelto k^p\) is obtained by replacing each integrator with
scaling by zero, then black-boxing.
\end{definition}
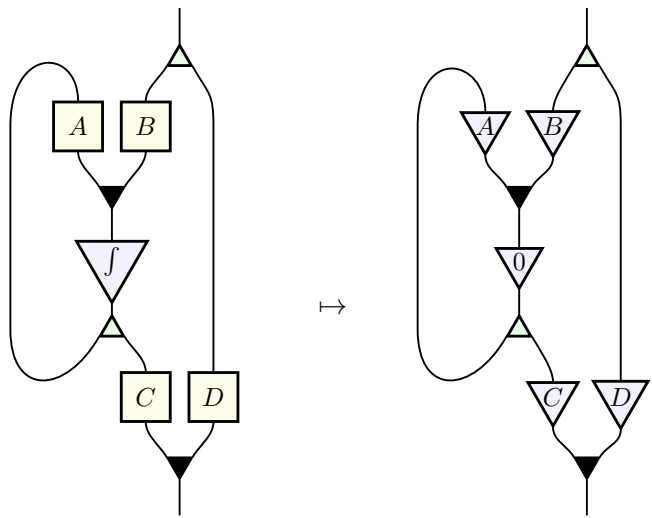
\vbox{\begin{example} \({}\)
\begin{center}
\scalebox{0.90}{\begin{tikzpicture}[thick]
% main nodes
   \node [delta] (usplit) at (2.5,3) {};
   \node [sqnode] (A) at (1,2) {\(A\)};
   \node [sqnode] (B) at (2,2) {\(B\)};
   \node [plus] (xdotsum) at (1.5,1) {};
   \node [multiply] (int) at (1.5,0) {\(\int\)};
   \node [delta] (xsplit) at (1.5,-1) {};
   \node [sqnode] (C) at (2,-2) {\(C\)};
   \node [sqnode] (D) at (3,-2) {\(D\)};
   \node [plus] (ysum) at (2.5,-3) {};

% auxiliary nodes
   \node [coordinate] (capend) [left of=A] {};
   \node [coordinate] (cupend) [below of=capend, shift={(0,-2)}] {};
   \node [coordinate] (ubend) [right of=B] {};

% wires
   \draw (usplit) -- +(0,0.75)
         (ysum) -- +(0,-0.75)
         (xsplit.left out) .. controls +(240:1) and +(270:1) .. (cupend)
         (A.90) .. controls +(90:0.8) and +(90:1.2) .. (capend)
         (capend) -- (cupend)
         (xsplit.right out) .. controls +(300:0.2) and +(90:0.2) .. (C.90)
         (C.270) .. controls +(270:0.2) and +(120:0.2) .. (ysum.left in)
         (D.270) .. controls +(270:0.2) and +(60:0.2) .. (ysum.right in)
         (D.90) -- (ubend)
         (usplit.right out) .. controls +(300:0.5) and +(90:0.5) .. (ubend)
         (usplit.left out) .. controls +(240:0.2) and +(90:0.2) .. (B.90)
         (A.270) .. controls +(270:0.2) and +(120:0.2) .. (xdotsum.left in)
         (B.270) .. controls +(270:0.2) and +(60:0.2) .. (xdotsum.right in)
         (xdotsum) -- (int) -- (xsplit)
;
\end{tikzpicture}} \qquad \raisebox{2.7cm}{\(\mapsto\)} \qquad \scalebox{0.90}{\begin{tikzpicture}[thick]
% main nodes
   \node [multiply] (D) at (3,-2) {\(D\)};
   \node [delta] (usplit) at (2.5,3) {};
   \node [multiply] (B) at (2,2) {\(B\)};
   \node [plus] (xdotsum) at (1.5,1) {};
   \node [multiply] (int) at (1.5,0) {\(0\)};
   \node [delta] (xsplit) at (1.5,-1) {};
   \node [plus] (ysum) at (2.5,-3) {};
   \node [multiply] (A) at (1,2) {\(A\)};
   \node [multiply] (C) at (2,-2) {\(C\)};

% auxiliary nodes
   \node [coordinate] (capend) [left of=A] {};
   \node [coordinate] (cupend) [below of=capend, shift={(0,-2)}] {};
   \node [coordinate] (ubend) [right of=B] {};

% wires
   \draw (usplit) -- +(0,0.75)
         (ysum) -- +(0,-0.75)
         (xsplit.left out) .. controls +(240:1) and +(270:1) .. (cupend)
         (A.90) .. controls +(90:0.8) and +(90:1.2) .. (capend)
         (capend) -- (cupend)
         (xsplit.right out) .. controls +(300:0.2) and +(90:0.2) .. (C.90)
         (C.270) .. controls +(270:0.2) and +(120:0.2) .. (ysum.left in)
         (D.270) .. controls +(270:0.2) and +(60:0.2) .. (ysum.right in)
         (D.90) -- (ubend)
         (usplit.right out) .. controls +(300:0.5) and +(90:0.5) .. (ubend)
         (usplit.left out) .. controls +(240:0.2) and +(90:0.2) .. (B.90)
         (A.270) .. controls +(270:0.2) and +(120:0.2) .. (xdotsum.left in)
         (B.270) .. controls +(270:0.2) and +(60:0.2) .. (xdotsum.right in)
         (xdotsum) -- (int) -- (xsplit)
;
  \end{tikzpicture}}.
\end{center}
\end{example}}
Note, in addition to its connection to the state equations, the signal-flow diagram used for the
examples is idealized to give a visual intuition of what the process does, with the design to
suggest \(A(f) = A\), \(B(f) = B\), \(C(f) = C\), and \(D(f) = D\).  As we shall see, this will be
the case when \(A(f)\), \(B(f)\), \(C(f)\), and \(D(f)\) are all linear maps, deepening the
connection between these processes and \(\st_k\), and hinting toward what makes a signal-flow
diagram `good'.

\begin{definition}
\label{gooddef}
The category \Define{\(\goodflow_k\)} is a subPROP of \(\sigflow_{k,s}\): a morphism \(f \maps
m \to p\) in \(\goodflow_k\) is a morphism \(f \maps m \to p\) in \(\sigflow_{k,s}\) such that the
linear relations \(A(f)\), \(B(f)\), \(C(f)\), and \(D(f)\) defined above are all linear maps.
\end{definition}

It is easy to see that all identity signal-flow diagrams are in \(\goodflow_k\).  It is also clear
that \(\goodflow_k\) is closed under direct sum.  It is not as obvious that all composites of
signal-flow diagrams in \(\goodflow_k\) are also in \(\goodflow_k\), but we show this in the proof
of Theorem~\ref{goodprop}, so \(\goodflow_k\) is indeed a PROP.

While this definition does not make it clear that the morphisms of \(\goodflow_k\) are closed under
composition, it has some advantages over other definitions that appear reasonable, such as one that
only insists `there is a morphism \(g \maps m \to p\) in \(\st_k\) with \(i(\eval(g)) = \bbox(f)\)'.
While closure under composition is then immediate, this alternative definition does not guarantee
the uniqueness of the morphism \(g\).  We can impose some extra conditions, for example requiring
the number of integrators in \(f\) to be equal to the dimension of the state space of \(g\), but
even this fails to make \(g\) unique.  For example, these morphisms in \(\st_k\) are different:
\begin{center}
\quad  \begin{tikzpicture}[baseline=0.5cm,->]
  \node (A) at (0,0) {\(k\)};
  \node (C) at (1.4,0) {\(k\)};
  \node (B) at (0,1.4) {\(k\)};
  \node (D) at (1.4,1.4) {\(k\)};

  \path
   (A) edge node[below] {\([1]\)} (C)
       edge node[left] {\([0]\)} (B)
   (B) edge node[above] {\([\frac{1}{s-1}]\)} (D)
   (D) edge node[right] {\([1]\)} (C);

   \node (neq) at (2.65,0.7) {\(\neq\)};

  \node (A) at (3.9,0) {\(k\)};
  \node (C) at (5.3,0) {\(k\)};
  \node (B) at (3.9,1.4) {\(k\)};
  \node (D) at (5.3,1.4) {\(k\)};

  \path
   (A) edge node[below] {\([1]\)} (C)
       edge node[left] {\([1]\)} (B)
   (B) edge node[above] {\([\frac{1}{s-1}]\)} (D)
   (D) edge node[right] {\([0]\)} (C);
 \end{tikzpicture},
\end{center}
since the former is observable and the latter is not observable\footnote{These two stateful morphisms
can also be distinguished based on controllabilty: the former is not controllable, while the latter
is controllable}.  However, both have state spaces of dimension 1, and both evaluate to the identity
relation, \(1\).  As we can see from this example, we need to be intelligent in how we translate
\(\goodflow_k\) signal-flow diagrams to \(\st_k\) morphisms if we want to have reasonable notions of
controllability and observability.  This will include knowing the number of integrators used, which
explains why the category of signal-flow diagrams from which we get \(\goodflow_k\) needs to be
\(\sigflow_{k,s}\) rather than \(\sigflow_{k(s)}\).  This is reflected in the processes \(A(f)\),
\(B(f)\), \(C(f)\), and \(D(f)\), since they cannot be defined on \(\sigflow_{k(s)}\).

While there is no reason \emph{a priori} to expect a signal-flow diagram to be in such a convenient
form as used in the examples above, the processes defined are unaffected by rewrites using the 
equations of \(\relk\), where the integrators are left to be free.  We therefore also consider the
PROP \(\porp = \P(\Sigma_{\relk,s},E_{\relk})\) and its associated `black-box' functor \(\dBox \maps
\sigflow_{k,s} \to \porp\).

\begin{lemma}
\label{goodrewrite}
If two signal-flow diagrams \(f_1\) and \(f_2\) are in \(\sigflow_{k,s}\) and \(\dBox f_1 =
\dBox f_2\), then \(A(f_1) = A(f_2)\), \(B(f_1) = B(f_2)\), \(C(f_1) = C(f_2)\), and \(D(f_1) =
D(f_2)\).
\end{lemma}
That is, rewrites from \(\relk\) (the ones that treat \(\frac{1}{s}\) as a free generator) have no
effect on \(A(f)\), \(B(f)\), \(C(f)\), and \(D(f)\).

\begin{proof}
We take advantage of the compositionality of signal-flow diagrams: a rewrite of \(f \in
\sigflow_{k,s}\) using one of equations {\hyperref[eqn123]{\textbf{(1)--(31)}}} can be localized to
a subdiagram of \(f\) that has no integrators in it.  Doing a rewrite on such a subdiagram and
then composing \(!^p \of f\) or \(f \of \zero^m\) yields the same signal-flow diagram as first
composing, then doing the same rewrite on that subdiagram.  Likewise, rewrites of such subdiagrams
and the changes to the integrators that occur in processes \(A(f), \dotsc, D(f)\) can be done in
either order with the result of the same signal-flow diagram either way.  Thus these rewrites
`commute' with the compositions and integrator replacements involved in these processes.  Since
black-boxing coequalizes signal-flow diagrams that differ in a rewrite from \(\relk\), a rewrite of
\(f\) using one of equations \textbf{(1)--(31)} will have no effect on \(A(f), \dotsc, D(f)\).

Since any rewrite of \(f\) from \(\relk\) is a series of rewrites using equations
\textbf{(1)--(31)}, no rewrite of \(f\) from \(\relk\) has an effect on \(A(f), \dotsc, D(f)\).
\end{proof}

Up to this point we have focused on breaking up signal-flow diagrams into four subdiagrams via the
PROP \(\porp\).  The following lemma shows that, in terms of \(\porp\), there are actually ten
subdiagrams than need to be considered in a fully general signal-flow diagram.  The processes
\(A(f), \dotsc, D(f)\) are each affected by several of these subdiagrams.  Our next goal is to show
that for a signal-flow diagram in \(\goodflow_k\), six of these subdiagrams will be trivial,
leaving only one non-trivial subdiagram each that can affect \(A(f), \dotsc, D(f)\).

\begin{lemma}
\label{porpnorm}
Any morphism in \(\porp\) can be rewritten in the following form:
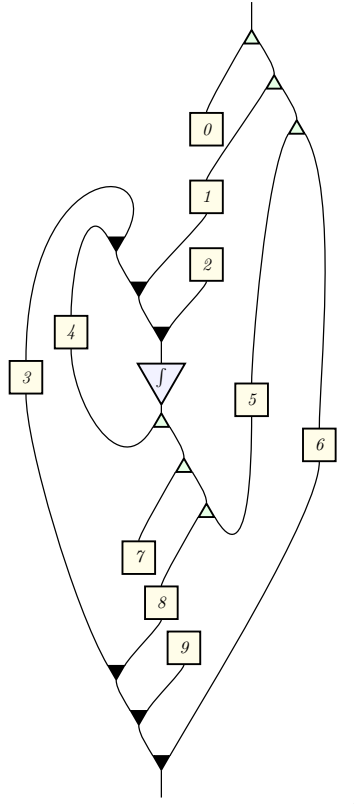
\begin{center}
\scalebox{0.60} {  \begin{tikzpicture}[thick]
% Top stuff
   \node[delta] (topD1) at (0,10) {};
   \node[delta, below of=topD1, shift={(0.5,0)}] (topD2) {};
   \node[delta, below of=topD2, shift={(0.5,0)}] (topD3) {};

% Integrator Zone stuff
   \node[plus] (topS1) at (-3,5.5) {};
   \node[plus, below of=topS1, shift={(0.5,0)}] (topS2) {};
   \node[plus, below of=topS2, shift={(0.5,0)}] (topS3) {};
   \node (int) [integral, below of=topS3] {\(\int\)};
   \node[delta, below of=int] (botD1) {};
   \node[delta, below of=botD1, shift={(0.5,0)}] (botD2) {};
   \node[delta, below of=botD2, shift={(0.5,0)}] (botD3) {};

% Bottom stuff
   \node[plus] (botS1) at (-3,-4) {};
   \node[plus, below of=botS1, shift={(0.5,0)}] (botS2) {};
   \node[plus, below of=botS2, shift={(0.5,0)}] (botS3) {};

% Boxes
   \node[sqnode] (box0) at (-1,8)      {0};
   \node[sqnode] (box1) at (-1,6.5)    {1};
   \node[sqnode] (box2) at (-1,5)      {2};
   \node[sqnode] (box3) at (-5,2.5)    {3};
   \node[sqnode] (box4) at (-4,3.5)    {4};
   \node[sqnode] (box5) at (0,2)       {5};
   \node[sqnode] (box6) at (1.5,1)     {6};
   \node[sqnode] (box7) at (-2.5,-1.5) {7};
   \node[sqnode] (box8) at (-2,-2.5)   {8};
   \node[sqnode] (box9) at (-1.5,-3.5) {9};

% wires
   \draw (topD1.io) -- +(0,0.6)
         (topD1.right out) .. controls +(300:0.2) and +(90:0.2) .. (topD2.io)
         (topD2.right out) .. controls +(300:0.2) and +(90:0.2) .. (topD3.io)
         (topS1.io) .. controls +(270:0.2) and +(120:0.2) .. (topS2.left in)
         (topS2.io) .. controls +(270:0.2) and +(120:0.2) .. (topS3.left in)
         (topS3) -- (int) -- (botD1)
         (botD1.right out) .. controls +(300:0.2) and +(90:0.2) .. (botD2.io)
         (botD2.right out) .. controls +(300:0.2) and +(90:0.2) .. (botD3.io)
         (botS1.io) .. controls +(270:0.2) and +(120:0.2) .. (botS2.left in)
         (botS2.io) .. controls +(270:0.2) and +(120:0.2) .. (botS3.left in)
         (botS3.io) -- +(0,-0.6);
% boxwires
   \draw (topD1.left out)  .. controls +(240:0.3) and +(90:0.5) .. (box0)
         (topD2.left out)  .. controls +(240:0.5) and +(90:0.5) .. (box1)
                      .. controls +(270:0.5) and +(60:0.3) .. (topS2.right in)
         (topS3.right in)  .. controls +(60:0.3) and +(270:0.5) .. (box2)
         (topS1.right in)  .. controls +(60:2) and +(90:5)      .. (box3)
                      .. controls +(270:1.5) and +(120:1.5) .. (botS1.left in)
         (topS1.left in)   .. controls +(120:1) and +(90:1.3)   .. (box4)
                     .. controls +(270:1.5) and +(240:1.5) .. (botD1.left out)
         (topD3.left out)  .. controls +(240:1) and +(90:1)     .. (box5)
                    .. controls +(270:1.5) and +(300:1.5) .. (botD3.right out)
         (topD3.right out) .. controls +(300:1.5) and +(90:1.5) .. (box6)
                      .. controls +(270:1.5) and +(60:1.5) .. (botS3.right in)
         (botD2.left out)  .. controls +(240:0.3) and +(90:0.5) .. (box7)
         (botD3.left out)  .. controls +(240:0.3) and +(90:0.5) .. (box8)
                      .. controls +(270:0.5) and +(60:0.3) .. (botS1.right in)
         (botS2.right in)  .. controls +(60:0.3) and +(270:0.5) .. (box9)
;
 \end{tikzpicture} },
\end{center}
where the numbered boxes are linear relations in \(\relk\).
\end{lemma}
We will refer to this form as \(\porp\)-normal form.

\begin{conjecture}
\label{conject}
We can further take boxes 1, 4\({}^{\dagger}\), 6, and 8 to be linear maps.
\end{conjecture}
This conjecture is not necessary to our argument, but it would simplify future work extending our
results if it is true.

We see there are two more ways to connect integrators, inputs, and outputs in \(\porp\) that are not
accounted for by \(A\), \(B\), \(C\), and \(D\) (boxes 3 and 5), along with four self-connections
(boxes 0, 2, 7, and 9).  The input of an integrator can connect to an output (box 3), and an input
can connect to the output of an integrator (box 5).  When either of these kinds of connection occurs
non-trivially in the signal-flow diagram \(f\), \(A(f)\) and \(D(f)\) will not be linear maps, so
these kinds of connection are trivial in any signal-flow diagram in \(\goodflow_k\).  When any of
the self-connections is non-trivial, either \(A(f)\) or \(D(f)\) will not be a linear map, so
self-connections are also trivial for any signal-flow diagram in \(\goodflow_k\).  Precisely what is
meant by `trivial' here is formalized in Lemma~\ref{goodstandard}.  Note: the status of \(B(f)\) and
\(C(f)\) as linear maps is also affected by these extra connections.

% \begin{proof}
% \Define{Completely left to the reader version:}\\
% To show any morphism in \(\porp\) can be written in this form, it suffices to check three things:
% generating morphisms, direct sum, and composition.  The first and second are left as straightforward
% exercises to the reader.  It is tedious but not difficult to show each generator can be written in
% this form, and it is easy to see a direct sum of morphisms in this form will again be in this form
% by placing the two diagrams on top of each other.  Checking that composition of two diagrams in this
% form will again be a diagram in this form is less straightforward, but is again left as an exercise
% to the reader.
% \end{proof}

\begin{proof}[Proof of Lemma~\ref{porpnorm}]
Since integrators are free in \(\porp\), their placement in a diagram is important, relative to how
their inputs and outputs connect to other parts of the diagram.  Considering all the ways inputs,
outputs, integrator inputs, and integrator outputs can connect, the diagram in Lemma~\ref{porpnorm}
has all such possible connections in it.  The only question, then, is whether there are other ways
to join the connections that cannot be rewritten in this form.  The joints here are all series of
parallel \(\Delta\) morphisms and series of parallel \(+\) morphisms,
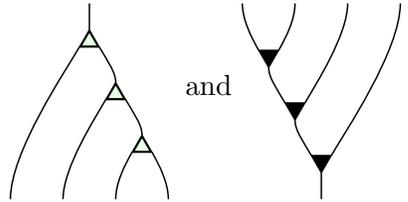
\begin{center}
\scalebox{0.7}{   \begin{tikzpicture}[thick]
   \node[delta] (delta) at (-0.5,1) {};
   \node[delta] (dupe) at (0,0)     {};
   \node[delta] (dub) at (0.5,-1)   {};

   \draw (0,-2)     .. controls +(90:0.3) and +(-120:0.3) .. (dub.left out)
         (-1,-2)    .. controls +(90:0.6) and +(-120:0.4) .. (dupe.left out)
         (-2,-2)    .. controls +(90:0.9) and +(-120:0.5) .. (delta.left out)
         (1,-2)     .. controls +(90:0.3) and +(-60:0.3)  .. (dub.right out)
         (dub.io)   .. controls +(90:0.2) and +(-60:0.3)  .. (dupe.right out)
         (dupe.io)  .. controls +(90:0.2) and +(-60:0.3)  .. (delta.right out)
         (delta.io) -- +(0,0.5)
;
   \end{tikzpicture}} \raisebox{1.4cm}{and} \scalebox{0.7}{   \begin{tikzpicture}[thick]
   \node[plus] (delta) at (0.5,-1) {};
   \node[plus] (dupe) at (0,0)     {};
   \node[plus] (dub) at (-0.5,1)   {};

   \draw (0,2)      .. controls +(-90:0.3) and +(60:0.3) .. (dub.left out)
         (1,2)      .. controls +(-90:0.6) and +(60:0.4) .. (dupe.left out)
         (2,2)      .. controls +(-90:0.9) and +(60:0.5) .. (delta.left out)
         (-1,2)     .. controls +(-90:0.3) and +(120:0.3)  .. (dub.right out)
         (dub.io)   .. controls +(-90:0.2) and +(120:0.3)  .. (dupe.right out)
         (dupe.io)  .. controls +(-90:0.2) and +(120:0.3)  .. (delta.right out)
         (delta.io) -- +(0,-0.5)
;
   \end{tikzpicture}}.
\end{center}
If a joint has a \(+^{\dagger}\) in it, the input string to the \(+^{\dagger}\) can be replaced with a
\(\Delta\)/\(!\) pair, as in equation~{\hyperref[eqn456]{\textbf{(4)}}}.
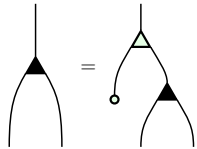
\begin{center}
\scalebox{0.7}{\begin{tikzpicture}[thick]
   \node[coplus] (plus1) at (0,0) {};
   \node[coplus] (plus2) at (2.5,-0.5) {};
   \node[delta] (dupe) at (2,0.5) {};
   \node[bang] (bang) at (1.5,-0.6) {};
   \node (eq) at (1,0) {\(=\)};

   \draw (-0.5,-1.5)  .. controls +(90:0.8) and +(-120:0.4) .. (plus1.left out)
         (0.5,-1.5)   .. controls +(90:0.8) and +(-60:0.4)  .. (plus1.right out)
         (plus1.io)   -- +(0,1)
         (2,-1.5)     .. controls +(90:0.4) and +(-120:0.4) .. (plus2.left out)
         (3,-1.5)     .. controls +(90:0.4) and +(-60:0.4)  .. (plus2.right out)
         (plus2.io)   .. controls +(90:0.2) and +(-60:0.3)  .. (dupe.right out)
         (bang)       .. controls +(90:0.5) and +(-120:0.4) .. (dupe.left out)
         (dupe.io)    -- +(0,0.5)
;
\end{tikzpicture}}
\end{center}
At this point, equation~{\hyperref[eqnD5D6D7]{\textbf{(D5)}}}\({}^{\dagger}\) can be applied as
often as necessary to allow the \(+^{\dagger}\)s to assimilate into one or more of the boxes.
Similarly, a joint with a \(\Delta^{\dagger}\) in it can be transformed to a joint with only \(+\)
morphisms in it.
\end{proof}

\vbox{The next lemma shows when \(f\) is a signal-flow diagram in \(\goodflow_k\), \(\dBox f\) can be
written in the form
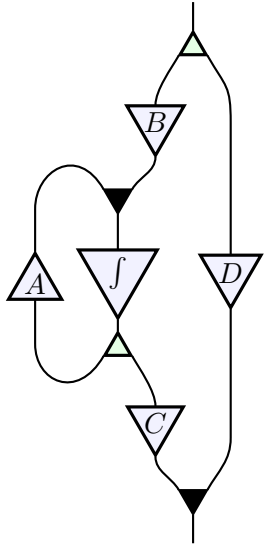
\begin{center}
 \begin{tikzpicture}[thick]
% main nodes
   \node [delta] (usplit) at (-0.5,3) {};
   \node [upmultiply] (A) at (-2.6,-0.15) {\(A\)};
   \node [multiply] (B) at (-1,2) {\(B\)};
   \node [plus] (xdotsum) at (-1.5,1) {};
   \node [multiply] (int) at (-1.5,0) {\(\int\)};
   \node [delta] (xsplit) at (-1.5,-1) {};
   \node [multiply] (C) at (-1,-2) {\(C\)};
   \node [multiply] (D) at (0,0) {\(D\)};
   \node [plus] (ysum) at (-0.5,-3) {};

% auxiliary nodes
   \node [coordinate] (capend) [above of=A] {};
   \node [coordinate] (cupend) [below of=A, shift={(0,0.2)}] {};
   \node [coordinate] (ubend) [right of=B] {};
   \node [coordinate] (ybend) [right of=C, shift={(0,-0.2)}] {};

% wires
   \draw (usplit) -- +(0,0.6)
         (ysum) -- +(0,-0.6)
         (xsplit.left out) .. controls +(240:0.7) and +(270:0.5) .. (cupend)
         (xdotsum.left in) .. controls +(120:0.7) and +(90:0.5) .. (capend)
         (capend) -- (A) -- (cupend)
         (xsplit.right out) .. controls +(300:0.2) and +(90:0.2) .. (C.90)
         (C.270) .. controls +(270:0.2) and +(120:0.2) .. (ysum.left in)
         (ybend) .. controls +(270:0.3) and +(60:0.3) .. (ysum.right in)
         (ybend) -- (D) -- (ubend)
         (usplit.right out) .. controls +(300:0.5) and +(90:0.5) .. (ubend)
         (usplit.left out) .. controls +(240:0.2) and +(90:0.2) .. (B.90)
         (B.270) .. controls +(270:0.2) and +(60:0.2) .. (xdotsum.right in)
         (xdotsum) -- (int) -- (xsplit)
   ;
 \end{tikzpicture}.
\end{center}}
This means when \(f \in \goodflow_k\), \(f\) can be transformed into the form of the signal-flow
diagram used to demonstrate the processes \(A(f), \dotsc, D(f)\) without affecting \(A(f), \dotsc,
D(f)\).

\begin{lemma}
\label{goodstandard}
If \(A(f), \dotsc, D(f)\) are all linear maps, the corresponding boxes in the \(\porp\)-normal form
are these maps, and the other boxes are trivial.  That is, box 0 is \(!^m\), box 1 is \(B(f)\), box
2 is \(\zero^n\), box 3 is \(\zero^p \of (\zero^{\dagger})^n\), box 4 is \(A(f)^{\dagger}\), box 5
is \((!^{\dagger})^n \of{} !^m\), box 6 is \(D(f)\), box 7 is \((!^{\dagger})^n\), box 8 is \(C(f)\),
and box 9 is \(\zero^p\).
\end{lemma}

\begin{proof}
Applying process \(D\) to the \(\porp\)-normal form, boxes 2, 4, and 7 become disconnected from the
inputs and outputs.  The portion of the diagram that remains can be rewritten in the form below,
noted as diagram for \(D\).

Similarly, the diagram for \(A\) is the result of applying process \(A\) to the \(\porp\)-normal
form and rewriting.  The application of process \(A\) disconnects boxes 0, 6, and 9 from the inputs
and outputs.  The diagram for \(A\) is of the same form as the diagram for \(D\), so we can apply
the same arguments to corresponding linear relations.

\begin{table}
\centering
\begin{tabular}{c c c}
 \begin{tikzpicture}[thick]
% Top stuff
   \node[delta] (D1) at (0,6.25) {};
   \node[delta, below of=D1, shift={(-0.75,0)}] (D2) {};
   \node[delta, below of=D2, shift={(-0.75,0)}] (D3) {};

% Bottom stuff
   \node[plus] (S1) at (-1.5,-0.25) {};
   \node[plus, below of=S1, shift={(0.75,0)}] (S2) {};
   \node[plus, below of=S2, shift={(0.75,0)}] (S3) {};

% Boxes
   \node[sqnode] (box1) at (-1,3.25) {1};
   \node[sqnode] (box8) at (-1,0.75) {8};
   \node[sqnode] (box0) at (-2,3.25) {0};
   \node[sqnode] (box3) at (-2,0.75) {3};
   \node[sqnode] (box5) at (0,3.25)  {5};
   \node[sqnode] (box6) at (1,2)     {6};
   \node[sqnode] (box9) at (0,0.75)  {9};

% Dots
   \node[zero, below of=box5] (z1) {};
   \node[bang, above of=box3] (b1) {};
   \node[zero, above of=box8] (z2) {};
   \node[bang, below of=box1] (b2) {};

% wires
   \draw (D1.io)        -- +(0,0.4)
         (D1.left out)  .. controls +(240:0.2) and +(90:0.2)  .. (D2.io)
         (D2.left out)  .. controls +(240:0.2) and +(90:0.2)  .. (D3.io)
         (D1.right out) .. controls +(300:1)   and +(90:1)    .. (box6.90)
         (D2.right out) .. controls +(300:0.5) and +(90:0.5)  .. (box5.90)
         (D3.right out) .. controls +(300:0.2) and +(90:0.2)  .. (box1.90)
         (D3.left out)  .. controls +(240:0.2) and +(90:0.2)  .. (box0.90)
         (box5.270)     -- (z1)
         (z2)           -- (box8.90)
         (b1)           -- (box3.90)
         (box1.270)     -- (b2)
         (box3.270)     .. controls +(270:0.2) and +(120:0.2) .. (S1.left in)
         (box8.270)     .. controls +(270:0.2) and +(60:0.2)  .. (S1.right in)
         (box9.270)     .. controls +(270:0.5) and +(60:0.5)  .. (S2.right in)
         (box6.270)     .. controls +(270:1)   and +(60:1)    .. (S3.right in)
         (S1.io)        .. controls +(270:0.2) and +(120:0.2) .. (S2.left in)
         (S2.io)        .. controls +(270:0.2) and +(120:0.2) .. (S3.left in)
         (S3.io)        -- +(0,-0.4);
 \end{tikzpicture} & \({}\) \qquad \({}\) &  \begin{tikzpicture}[thick]
% Top stuff
   \node[delta] (D1) at (0,6.25) {};
   \node[delta, below of=D1, shift={(-0.75,0)}] (D2) {};
   \node[delta, below of=D2, shift={(-0.75,0)}] (D3) {};

% Bottom stuff
   \node[plus] (S1) at (-1.5,-0.25) {};
   \node[plus, below of=S1, shift={(0.75,0)}] (S2) {};
   \node[plus, below of=S2, shift={(0.75,0)}] (S3) {};

% Boxes
   \node[sqnode] (box1) at (-1,3.25) {8};
   \node[sqnode] (box8) at (-1,0.75) {1};
   \node[sqnode] (box0) at (-2,3.25) {7};
   \node[sqnode] (box3) at (-2,0.75) {3\({}^{\dagger}\)};
   \node[sqnode] (box5) at (0,3.25)  {5\({}^{\dagger}\)};
   \node[sqnode] (box6) at (1,2)     {4\({}^{\dagger}\)};
   \node[sqnode] (box9) at (0,0.75)  {2};

% Dots
   \node[zero, below of=box5] (z1) {};
   \node[bang, above of=box3] (b1) {};
   \node[zero, above of=box8] (z2) {};
   \node[bang, below of=box1] (b2) {};

% wires
   \draw (D1.io)        -- +(0,0.4)
         (D1.left out)  .. controls +(240:0.2) and +(90:0.2)  .. (D2.io)
         (D2.left out)  .. controls +(240:0.2) and +(90:0.2)  .. (D3.io)
         (D1.right out) .. controls +(300:1)   and +(90:1)    .. (box6.90)
         (D2.right out) .. controls +(300:0.5) and +(90:0.5)  .. (box5.90)
         (D3.right out) .. controls +(300:0.2) and +(90:0.2)  .. (box1.90)
         (D3.left out)  .. controls +(240:0.2) and +(90:0.2)  .. (box0.90)
         (box5.270)     -- (z1)
         (z2)           -- (box8.90)
         (b1)           -- (box3.90)
         (box1.270)     -- (b2)
         (box3.270)     .. controls +(270:0.2) and +(120:0.2) .. (S1.left in)
         (box8.270)     .. controls +(270:0.2) and +(60:0.2)  .. (S1.right in)
         (box9.270)     .. controls +(270:0.5) and +(60:0.5)  .. (S2.right in)
         (box6.270)     .. controls +(270:1)   and +(60:1)    .. (S3.right in)
         (S1.io)        .. controls +(270:0.2) and +(120:0.2) .. (S2.left in)
         (S2.io)        .. controls +(270:0.2) and +(120:0.2) .. (S3.left in)
         (S3.io)        -- +(0,-0.4);
 \end{tikzpicture}\\
Diagram for \(D\) & & Diagram for \(A\)
\end{tabular}
\end{table}

Assuming \(A(f)\) and \(D(f)\) to be linear maps, we first argue that the linear relations in boxes
0 and 7 must be trivial.  Since the argument is the same for both, we focus on box 0.
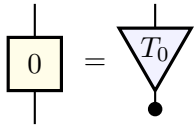
\begin{center}
\begin{tikzpicture}[thick]
  \node[sqnode] (oh) at (0,0) {0};
  \node (eq) at (0.8,0) {\(=\)};
  \node[multiply] (T) at (1.6,0.2) {\(T_0\)};
  \node[zero] (z) at (1.6,-0.6) {};

  \draw (oh.90) -- (0,0.8) (oh.270) -- (0,-0.8)
        (T.90) -- (1.6,0.8) (T.270) -- (z);
\end{tikzpicture}

\end{center}

Box 0 is a linear relation, which means it can be written in the standard form for linear relations
that was demonstrated in Chapter~\ref{vectrel}.  Box 0 also has no output, so it is a composite of a
linear map \(T_0\) and \((\zero^{\dagger})^j\) for some \(j \in \N\).  If \(j = 0\), \(T_0\) has no
outputs, meaning box 0 is \(!^m\).  If \(j > 0\), one of the inputs of box 0 is a linear combination
of the other inputs of box 0.  Since the input to the diagram for \(D\) is duplicated to give the
input to box 0, one of the inputs of \(D(f)\) is a linear combination of the other inputs of
\(D(f)\).  This means \(D(f)\) is not a linear map, a contradiction.  Thus the linear relation in
box 0 is \(!^m\), and the linear relation in box 7 must similarly be \(!^n\).

The argument for boxes 2 and 9 is \(*\)-dual to the argument for boxes 0 and 7.  The \(*\)-dual of
the standard form for linear relations is an alternate standard form for linear relations, so each
step of the above argument has a valid \(*\)-dual step.  This means the linear relations in boxes 2
and 9 must be \(\zero^n\) and \(\zero^p\), respectively.

Box 5 must have its inputs and outputs disconnected from each other, or else there is a non-trivial
linear combination of the inputs of \(D(f)\) that is equal to zero.  This contradicts the assumption
that \(D(f)\) is a linear map.
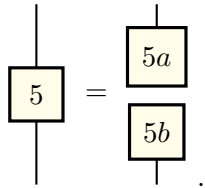
\begin{center}
\begin{tikzpicture}[thick]
  \node[sqnode] (5) at (0,0) {5};
  \node (eq) at (0.8,0) {\(=\)};
  \node[sqnode] (5a) at (1.6,0.5) {\(5a\)};
  \node[sqnode] (5b) at (1.6,-0.5) {\(5b\)};

  \draw (5.90) -- (0,1.2) (5.270) -- (0,-1.2)
        (5a.90) -- (1.6,1.2) (5b.270) -- (1.6,-1.2);
\end{tikzpicture}
.
\end{center}
Now the linear relation in box 5a is \(!^m\) for the same reason as box 0, and the linear relation
in box (5b)\({}^{\dagger}\) is \(!^n\) for the same reason as box 7.  Thus the linear relation in
box 5 must be \((!^{\dagger})^n \of{} !^m\).

Dually, the linear relation in box 3 must be \(\zero^p \of (\zero^{\dagger})^n\).

If we can show boxes 1 and 8 are linear maps, that will force the linear relation in box 6 to be the
only contribution to \(D(f)\) and the \(\dagger\)-dual of the linear relation in box 4 to be the
only contribution to \(A(f)\), making box 6 a linear map and box 4\({}^{\dagger}\) a linear map.
With boxes 1, 4\({}^{\dagger}\), 6, and 8 all linear maps and all other boxes trivial, the linear
maps \(B(f)\) and \(C(f)\) must be boxes 1 and 8, respectively.

Assume, by way of contradiction, that box 1 is not a linear map.  It is a linear relation, so one or
both of the following must happen:  some input to box 1 is a linear combination of the other inputs
to box 1 or some input to box 1\({}^*\) is a linear combination of the other inputs to box
1\({}^*\).  The latter possibility is equivalent to some output of box 1 can take on multiple values
given a fixed input to box 1.  Since the inputs to \(D(f)\) are duplicated to form the inputs to box
1, any linear dependence of the inputs to box 1 will translate into linear dependence of the same
inputs to \(D(f)\).  This contradicts the assumption that \(D(f)\) is a linear map.

Similarly, any linear dependence of the inputs to box 1\({}^*\) will translate into linear
dependence of the inputs to \(A(f)^*\).  Since \(A(f)\) is assumed to be a linear map, it follows
that \(A(f)^*\) must also be a linear map, specifically the transpose map.  Thus we have our
contradiction for the latter case.  This means the linear relation in box 1 must be a linear map.
Swapping the roles of \(A(f)\) and \(D(f)\), we see the same must be true of box 8.
\end{proof}

Aside from their not conforming to \(\goodflow_k\), there is a good control theory reason for why
boxes 3 and 5 should be trivial.  According to control theory folklore, a control system with
differentiators in it will not be causal: the present state depends on future states and inputs.  If
this were not enough, if box 3 or box 5 is allowed to be non-trivial in \(\goodflow_k\), the
\(\prop\) morphism analogous to \(\dbox\) in the next theorem no longer makes the commutative square
commute.

\begin{theorem}
\label{goodprop}
There is a functor from \(\goodflow_k\) to \(\st_k\) given by
\begin{align*}
\dbox \maps \goodflow_k & {} \to \st_k\\
f & {} \mapsto (D(f), C(f), (sI - A(f))^{-1}, B(f)).
\end{align*}
This functor is a \(\prop\) morphism that makes the following diagram in \(\prop\) commute:
\begin{center}
\begin{tikzpicture}
   \node (SF) {\(\sigflow_{k,s}\)};
   \node (P) [above of=SF, shift={(0,1)}] {\(\goodflow_k\vphantom{\vectks}\)};
   \node (FR) [right of=SF, shift={(2,0)}] {\(\relks\)};
   \node (ST) [above of=FR, shift={(0,1)}] {\(\st_k\vphantom{\vectks}\)};

   \draw [->] (P) to node [above] {\(\dbox\)} (ST);
   \draw [right hook->] (P) to node [left] {\(j\)} (SF);
   \draw [->] (ST) to node [right] {\(i \of \eval\)} (FR);
   \draw [->] (SF) to node [below] {\(\bbox\)} (FR);
\end{tikzpicture}.
\end{center}
\end{theorem}

This functor is the means by which we translate the controllability and observability results from
\(\st_k\) to \(\goodflow_k\).

\begin{proof}
To show \(\dbox\) is a functor, we need to check identities map to identities, and \(\dbox(f \of f')
= \dbox(f) \of \dbox(f')\).  The former is immediate, so we will focus on the latter.  To show
\(\dbox\) is also a \(\prop\) morphism, we further need the distinguished object to map to the
distinguished object, and \(\dbox(f \oplus f') = \dbox(f) \oplus \dbox(f')\).  These additional
criteria are straightforward to check, so we leave it to the reader to check them.  To show the
diagram in \(\prop\) commutes, we need for an arbitrary signal-flow diagram \(f \in \goodflow_k\) to
satisfy \(\bbox f = D(f) + C(f) (sI-A(f))^{-1} B(f)\).

\begin{itemize}
\item \(\bbox f = D(f) + C(f) (sI-A(f))^{-1} B(f)\).\\
By Lemma~\ref{goodstandard}, any signal-flow diagram \(f \in \goodflow_k\) can be rewritten using
the equations of \(\relk\) into the form
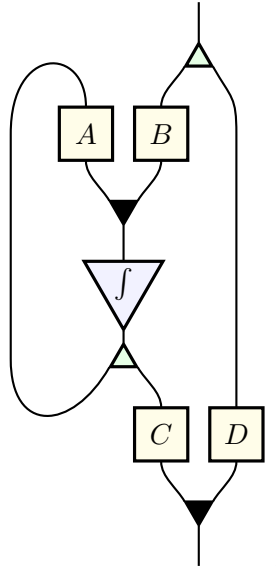
\begin{center}
\begin{tikzpicture}[thick]
% main nodes
   \node [delta] (usplit) at (2.5,3) {};
   \node [sqnode] (A) at (1,2) {\(A\)};
   \node [sqnode] (B) at (2,2) {\(B\)};
   \node [plus] (xdotsum) at (1.5,1) {};
   \node [multiply] (int) at (1.5,0) {\(\int\)};
   \node [delta] (xsplit) at (1.5,-1) {};
   \node [sqnode] (C) at (2,-2) {\(C\)};
   \node [sqnode] (D) at (3,-2) {\(D\)};
   \node [plus] (ysum) at (2.5,-3) {};

% auxiliary nodes
   \node [coordinate] (capend) [left of=A] {};
   \node [coordinate] (cupend) [below of=capend, shift={(0,-2)}] {};
   \node [coordinate] (ubend) [right of=B] {};

% wires
   \draw (usplit) -- +(0,0.75)
         (ysum) -- +(0,-0.75)
         (xsplit.left out) .. controls +(240:1) and +(270:1) .. (cupend)
         (A.90) .. controls +(90:0.8) and +(90:1.2) .. (capend)
         (capend) -- (cupend)
         (xsplit.right out) .. controls +(300:0.2) and +(90:0.2) .. (C.90)
         (C.270) .. controls +(270:0.2) and +(120:0.2) .. (ysum.left in)
         (D.270) .. controls +(270:0.2) and +(60:0.2) .. (ysum.right in)
         (D.90) -- (ubend)
         (usplit.right out) .. controls +(300:0.5) and +(90:0.5) .. (ubend)
         (usplit.left out) .. controls +(240:0.2) and +(90:0.2) .. (B.90)
         (A.270) .. controls +(270:0.2) and +(120:0.2) .. (xdotsum.left in)
         (B.270) .. controls +(270:0.2) and +(60:0.2) .. (xdotsum.right in)
         (xdotsum) -- (int) -- (xsplit)
;
\end{tikzpicture},
\end{center}
and this rewriting has no effect on \(A(f), \dotsc, D(f)\).  Since the equations of \(\relk\) are a
subset of the equations of \(\relks\), the \(\prop\) morphism \(\bbox\) factors through \(\dBox\).
This means the rewriting will also have no effect on \(\bbox f\), which imposes the equations of
\(\relks\) onto \(f\).  By imposing all of the equations of \(\relks\) on this diagram, we get
\(\bbox f = D + C (sI-A)^{-1} B\).  By applying the processes \(A, \dotsc, D\) to this diagram, we
see that \(A(f) = A\), \(B(f) = B\), \(C(f) = C\), and \(D(f) = D\).

\item \(\dbox(f \of f') = \dbox(f) \of \dbox(f')\).\\
Suppose \(f\) and \(f'\) are signal-flow diagrams in \(\goodflow_k\) which are composable in
\(\sigflow_{k,s}\), with \(\dbox(f) = (D,C,(sI-A)^{-1},B)\) and \(\dbox(f') =
(D',C',(sI'-A')^{-1},B')\).  In Chapter~\ref{stateful} we saw how to compose two stateful morphisms,
\(\dbox(f) \of \dbox(f')\), so we need to check that we get the same composite from \(\dbox(f \of
f') = (D'',C'',(sI''-A''),B'')\).  This would mean \(A'', \dotsc, D''\) are all linear maps, finally
justifying the assertion above that composition is closed in \(\goodflow_k\).  In matrix form, that
means we need to verify:
\begin{align*}
D'' &= DD' & C'' &= \left[\begin{array}{cc}D'C & C'\end{array}\right] \\
B'' &= \left[\begin{array}{c}B \\ B'D\end{array}\right] &
A'' &= \left[\begin{array}{cc} A & 0 \\ B'C & A'\end{array}\right].
\end{align*}
We will be performing several surgeries on the signal-flow diagram \(f \of f'\), so let's take a
good look at the `patient'.  By Lemma~\ref{goodstandard}, we need only consider a signal-flow
diagram of the form:
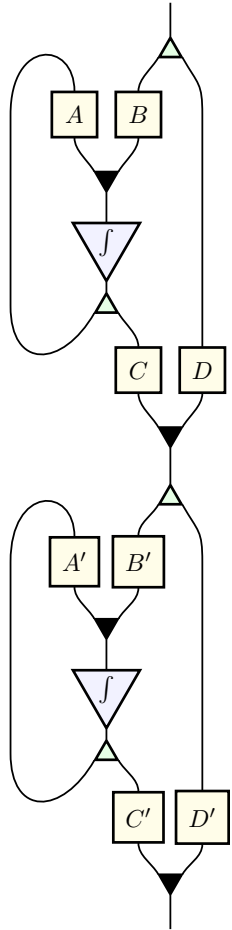
\begin{center}
\scalebox{0.85}{
\begin{tikzpicture}[thick]
% main nodes
   \node [delta] (usplit) at (2.5,3) {};
   \node [sqnode] (A) at (1,2) {\(A\)};
   \node [sqnode] (B) at (2,2) {\(B\)};
   \node [plus] (xdotsum) at (1.5,1) {};
   \node [multiply] (int) at (1.5,0) {\(\int\)};
   \node [delta] (xsplit) at (1.5,-1) {};
   \node [sqnode] (C) at (2,-2) {\(C\)};
   \node [sqnode] (D) at (3,-2) {\(D\)};
   \node [plus] (ysum) at (2.5,-3) {};

% auxiliary nodes
   \node [coordinate] (capend) [left of=A] {};
   \node [coordinate] (cupend) [below of=capend, shift={(0,-2)}] {};
   \node [coordinate] (ubend) [right of=B] {};

% wires
   \draw (usplit) -- +(0,0.75)
         (xsplit.left out) .. controls +(240:1) and +(270:1) .. (cupend)
         (A.90) .. controls +(90:0.8) and +(90:1.2) .. (capend)
         (capend) -- (cupend)
         (xsplit.right out) .. controls +(300:0.2) and +(90:0.2) .. (C.90)
         (C.270) .. controls +(270:0.2) and +(120:0.2) .. (ysum.left in)
         (D.270) .. controls +(270:0.2) and +(60:0.2) .. (ysum.right in)
         (D.90) -- (ubend)
         (usplit.right out) .. controls +(300:0.5) and +(90:0.5) .. (ubend)
         (usplit.left out) .. controls +(240:0.2) and +(90:0.2) .. (B.90)
         (A.270) .. controls +(270:0.2) and +(120:0.2) .. (xdotsum.left in)
         (B.270) .. controls +(270:0.2) and +(60:0.2) .. (xdotsum.right in)
         (xdotsum) -- (int) -- (xsplit)
;
% main nodes
   \node [delta] (usplitp) at (2.5,-4) {};
   \node [sqnode] (Ap) at (1,-5) {\(A'\)};
   \node [sqnode] (Bp) at (2,-5) {\(B'\)};
   \node [plus] (xdotsump) at (1.5,-6) {};
   \node [multiply] (intp) at (1.5,-7) {\(\int\)};
   \node [delta] (xsplitp) at (1.5,-8) {};
   \node [sqnode] (Cp) at (2,-9) {\(C'\)};
   \node [sqnode] (Dp) at (3,-9) {\(D'\)};
   \node [plus] (ysump) at (2.5,-10) {};

% auxiliary nodes
   \node [coordinate] (capendp) [left of=Ap] {};
   \node [coordinate] (cupendp) [below of=capendp, shift={(0,-2)}] {};
   \node [coordinate] (ubendp) [right of=Bp] {};

% wires
   \draw (usplitp) -- (ysum)
         (ysump) -- +(0,-0.75)
         (xsplitp.left out) .. controls +(240:1) and +(270:1) .. (cupendp)
         (Ap.90) .. controls +(90:0.8) and +(90:1.2) .. (capendp)
         (capendp) -- (cupendp)
         (xsplitp.right out) .. controls +(300:0.2) and +(90:0.2) .. (Cp.90)
         (Cp.270) .. controls +(270:0.2) and +(120:0.2) .. (ysump.left in)
         (Dp.270) .. controls +(270:0.2) and +(60:0.2) .. (ysump.right in)
         (Dp.90) -- (ubendp)
         (usplitp.right out) .. controls +(300:0.5) and +(90:0.5) .. (ubendp)
         (usplitp.left out) .. controls +(240:0.2) and +(90:0.2) .. (Bp.90)
         (Ap.270) .. controls +(270:0.2) and +(120:0.2) .. (xdotsump.left in)
         (Bp.270) .. controls +(270:0.2) and +(60:0.2) .. (xdotsump.right in)
         (xdotsump) -- (intp) -- (xsplitp)
;
\end{tikzpicture}
}.
\end{center}
Since the \(D(f)\) part of the process is functorial, \(D'' = D(f \of f') = D(f) \of D(f') = DD'\).
The \(B(f)\) and \(C(f)\) parts are dual to each other, so we show only the argument for \(B''\) and
leave \(C''\) as an easy exercise.

To find \(B'' = B(f \of f')\), we replace the \(n + n'\) wires entering the integrators in \(!^p \of
f \of f'\) with outputs and replace the \(n + n'\) wires leaving the integrators with
\(\zero^{n+n'}\).  Since the \(\zero\)s always meet linear maps, they destroy everything in their
paths until they reach a \(+\), at which point equation {\hyperref[eqn123]{\textbf{(1)}}} gives an
identity wire.  This gets rid of \(A\), \(A'\), \(C\), \(C'\), and all the \(+\)s.  Similarly, each
\(!\) will always meet linear maps, destroying everything in their paths until they reach a
\(\Delta\), at which point equation {\hyperref[eqn456]{\textbf{(4)}}} gives an identity wire.  This
gets rid of \(D'\) and would get rid of \(C'\) if it were not already gone.  From these
considerations, we get the following results on our `patient':
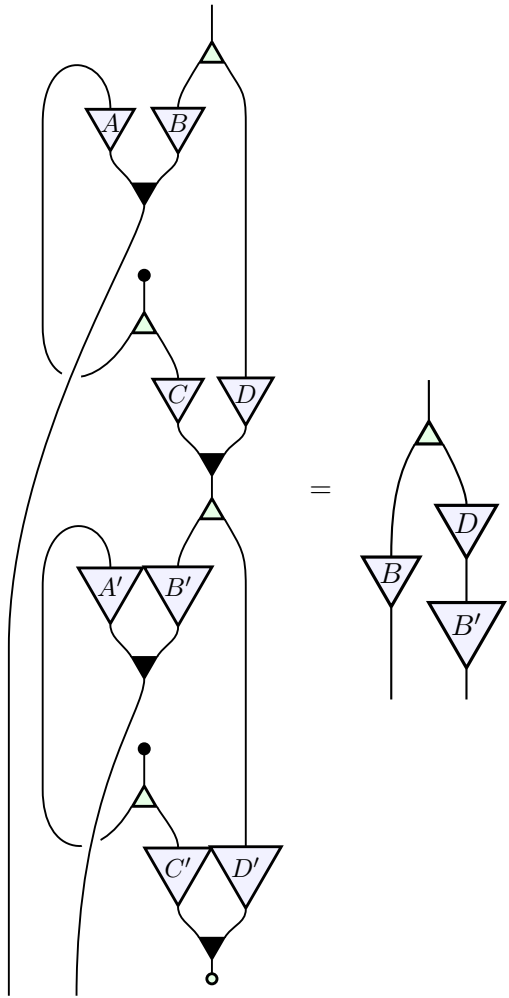
\begin{center}
\scalebox{0.9}{
\begin{tikzpicture}[thick]
% main nodes
   \node [delta] (usplit) at (2.5,3) {};
   \node [multiply] (A) at (1,2) {\(A\)};
   \node [multiply] (B) at (2,2) {\(B\)};
   \node [plus] (xdotsum) at (1.5,1) {};
   \node [delta] (xsplit) at (1.5,-1) {};
   \node [multiply] (C) at (2,-2) {\(C\)};
   \node [multiply] (D) at (3,-2) {\(D\)};
   \node [plus] (ysum) at (2.5,-3) {};

% auxiliary nodes
   \node [coordinate] (capend) [left of=A] {};
   \node [coordinate] (cupend) [below of=capend, shift={(0,-2)}] {};
   \node [coordinate] (ubend) [right of=B] {};

   \node [zero] (int) at (1.5,-0.25) {};
   \node [coordinate] (y) at (-0.5,-5.75) {};

% wires
   \draw (usplit) -- +(0,0.75)
         (xsplit.left out) .. controls +(240:1) and +(270:1) .. (cupend)
         (A.90) .. controls +(90:0.8) and +(90:1.2) .. (capend)
         (capend) -- (cupend)
         (xsplit.right out) .. controls +(300:0.2) and +(90:0.2) .. (C.90)
         (C.270) .. controls +(270:0.2) and +(120:0.2) .. (ysum.left in)
         (D.270) .. controls +(270:0.2) and +(60:0.2) .. (ysum.right in)
         (D.90) -- (ubend)
         (usplit.right out) .. controls +(300:0.5) and +(90:0.5) .. (ubend)
         (usplit.left out) .. controls +(240:0.2) and +(90:0.2) .. (B.90)
         (A.270) .. controls +(270:0.2) and +(120:0.2) .. (xdotsum.left in)
         (B.270) .. controls +(270:0.2) and +(60:0.2) .. (xdotsum.right in)
         (int) -- (xsplit)
;
   \node [hole] at (0.422,-1.75) {};

   \draw (xdotsum.io) .. controls +(270:0.5) and +(90:3) .. (y) -- (-0.5,-10.9);

% main nodes
   \node [delta] (usplitp) at (2.5,-3.75) {};
   \node [multiply] (Ap) at (1,-4.85) {\(A'\)};
   \node [multiply] (Bp) at (2,-4.85) {\(B'\)};
   \node [plus] (xdotsump) at (1.5,-6) {};
   \node [delta] (xsplitp) at (1.5,-8) {};
   \node [multiply] (Cp) at (2,-9) {\(C'\)};
   \node [multiply] (Dp) at (3,-9) {\(D'\)};
   \node [plus] (ysump) at (2.5,-10.15) {};

% auxiliary nodes
   \node [coordinate] (capendp) [left of=Ap] {};
   \node [coordinate] (cupendp) [below of=capendp, shift={(0,-2)}] {};
   \node [coordinate] (ubendp) [right of=Bp] {};

   \node [bang] (bot) at (2.5,-10.65) {};
   \node [zero] (intp) at (1.5,-7.25) {};
   \node [coordinate] (yp) at (0.5,-10.9) {};

% wires
   \draw (usplitp) -- (ysum)
         (ysump) -- (bot)
         (xsplitp.left out) .. controls +(240:1) and +(270:1) .. (cupendp)
         (Ap.90) .. controls +(90:0.8) and +(90:1.2) .. (capendp)
         (capendp) -- (cupendp)
         (xsplitp.right out) .. controls +(300:0.2) and +(90:0.2) .. (Cp.90)
         (Cp.270) .. controls +(270:0.2) and +(120:0.2) .. (ysump.left in)
         (Dp.270) .. controls +(270:0.2) and +(60:0.2) .. (ysump.right in)
         (Dp.90) -- (ubendp)
         (usplitp.right out) .. controls +(300:0.5) and +(90:0.5) .. (ubendp)
         (usplitp.left out) .. controls +(240:0.2) and +(90:0.2) .. (Bp.90)
         (Ap.270) .. controls +(270:0.2) and +(120:0.2) .. (xdotsump.left in)
         (Bp.270) .. controls +(270:0.2) and +(60:0.2) .. (xdotsump.right in)
         (intp) -- (xsplitp);

   \node [hole] at (0.72,-8.65) {};

   \draw (xdotsump.io) .. controls +(270:0.5) and +(90:3) .. (yp);
\end{tikzpicture}
}
\raisebox{6.64cm}{=}
\raisebox{3.94cm}{
\scalebox{1}{
\begin{tikzpicture}[thick]
% main nodes
   \node [delta] (usplit) at (2.5,3) {};
   \node [multiply] (B) at (2,1.175) {\(B\)};
   \node [multiply] (D) at (3,1.85) {\(D\)};
   \node [multiply] (Bp) at (3,0.5) {\(B'\)};
   \node [coordinate] (out1) at (2,-0.5) {};
   \node [coordinate] (out2) at (3,-0.5) {};

% wires
   \draw (usplit) -- +(0,0.75)
         (usplit.right out) .. controls +(300:0.2) and +(90:0.2) .. (D.90)
         (usplit.left out) .. controls +(240:0.4) and +(90:0.75) .. (B.90)
         (D.270) -- (Bp.90)  (Bp.270) -- (out2)
         (B.270) -- (out1)
;
\end{tikzpicture}
}}
.
\end{center}
That is, \(B'' = \left[\begin{array}{c}B \\ B'D\end{array}\right]\), which is exactly what is
required.  As noted above, the argument for \(C''\) is just a dual version of this one, reflected
about the \(x\)-axis and with colors swapped.

To find \(A'' = A(f \of f')\), we replace the \(n+n'\) wires leaving the integrators in \(!^p \of f
\of f' \of \zero^{m'}\) with inputs and the \(n+n'\) wires entering the integrators with outputs.
As with the \(\zero\)s in \(B''\), the \(\zero\)s here always meet linear maps, so they destroy
everything in their paths until they reach a \(+\), and likewise for the \(!\)s until they reach a
\(\Delta\).  This time the zeros get rid of \(B\) and \(D\), while the deletions get rid of \(C'\)
and \(D'\).  From there some zig-zags can be straightened out, and finally commutativity allows
\(A'\) and \(B'\) to be swapped.  Graphically, this proceeds as follows:
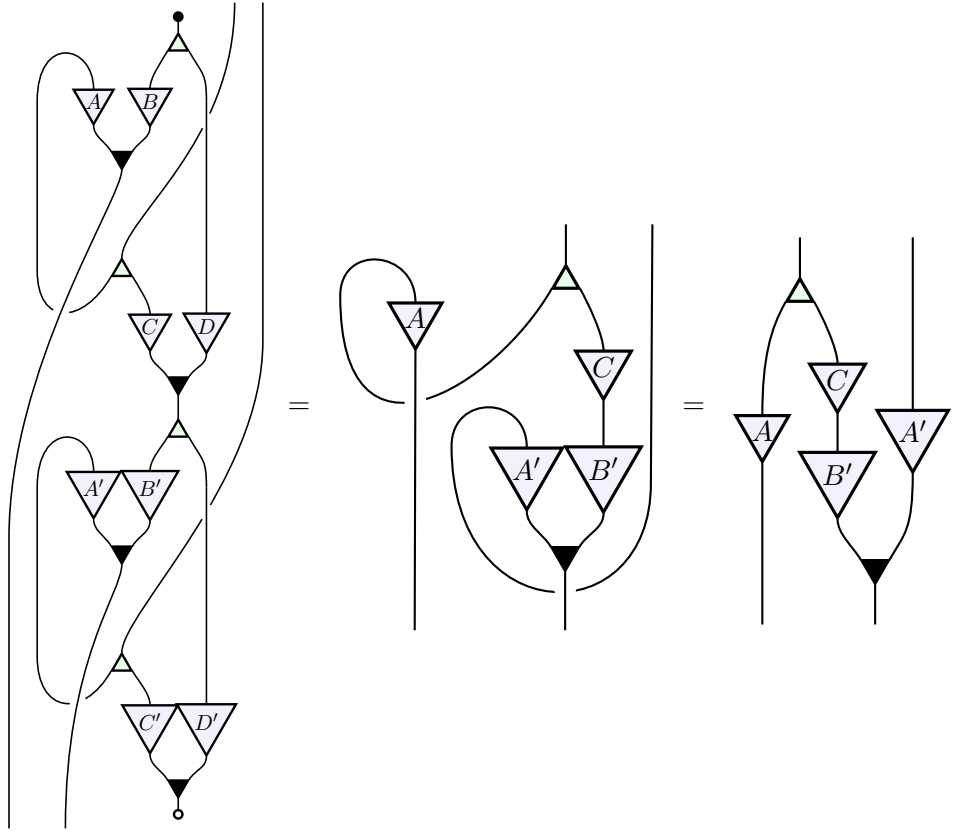
\begin{center}
\scalebox{0.75}{
\begin{tikzpicture}[thick]
% main nodes
   \node [delta] (usplit) at (2.5,3) {};
   \node [multiply] (A) at (1,2) {\(A\)};
   \node [multiply] (B) at (2,2) {\(B\)};
   \node [plus] (xdotsum) at (1.5,1) {};
   \node [delta] (xsplit) at (1.5,-1) {};
   \node [multiply] (C) at (2,-2) {\(C\)};
   \node [multiply] (D) at (3,-2) {\(D\)};
   \node [plus] (ysum) at (2.5,-3) {};

% auxiliary nodes
   \node [coordinate] (capend) [left of=A] {};
   \node [coordinate] (cupend) [below of=capend, shift={(0,-2)}] {};
   \node [coordinate] (ubend) [right of=B] {};

   \node [zero] (top) at (2.5,3.5) {};
   \node [coordinate] (u) at (3.5,3.75) {};
   \node [coordinate] (y) at (-0.5,-5.75) {};

% wires
   \draw (usplit) -- (top)

         (xsplit.left out) .. controls +(240:1) and +(270:1) .. (cupend)
         (A.90) .. controls +(90:0.8) and +(90:1.2) .. (capend)
         (capend) -- (cupend)
         (xsplit.right out) .. controls +(300:0.2) and +(90:0.2) .. (C.90)
         (C.270) .. controls +(270:0.2) and +(120:0.2) .. (ysum.left in)
         (D.270) .. controls +(270:0.2) and +(60:0.2) .. (ysum.right in)
         (xsplit.io) .. controls +(90:0.7) and +(270:2.5) .. (u)
         (usplit.right out) .. controls +(300:0.5) and +(90:0.5) .. (ubend)
         (usplit.left out) .. controls +(240:0.2) and +(90:0.2) .. (B.90)
         (A.270) .. controls +(270:0.2) and +(120:0.2) .. (xdotsum.left in)
         (B.270) .. controls +(270:0.2) and +(60:0.2) .. (xdotsum.right in);

   \node [hole] at (3,1.65) {};
   \node [hole] at (0.422,-1.75) {};

   \draw (xdotsum.io) .. controls +(270:0.5) and +(90:3) .. (y) -- +(0,-5.15)
         (D.90) -- (ubend);
% main nodes
   \node [delta] (usplitp) at (2.5,-3.85) {};
   \node [multiply] (Ap) at (1,-4.85) {\(A'\)};
   \node [multiply] (Bp) at (2,-4.85) {\(B'\)};
   \node [plus] (xdotsump) at (1.5,-6) {};
   \node [delta] (xsplitp) at (1.5,-8) {};
   \node [multiply] (Cp) at (2,-9) {\(C'\)};
   \node [multiply] (Dp) at (3,-9) {\(D'\)};
   \node [plus] (ysump) at (2.5,-10.15) {};

% auxiliary nodes
   \node [coordinate] (capendp) [left of=Ap] {};
   \node [coordinate] (cupendp) [below of=capendp, shift={(0,-2)}] {};
   \node [coordinate] (ubendp) [right of=Bp] {};

   \node [bang] (bot) at (2.5,-10.65) {};

   \node [coordinate] (up) at (4,-2.25) {};
   \node [coordinate] (yp) at (0.5,-10.9) {};

% wires
   \draw (usplitp) -- (ysum)
         (ysump) -- (bot)
         (xsplitp.left out) .. controls +(240:1) and +(270:1) .. (cupendp)
         (Ap.90) .. controls +(90:0.8) and +(90:1.2) .. (capendp)
         (capendp) -- (cupendp)
         (xsplitp.right out) .. controls +(300:0.2) and +(90:0.2) .. (Cp.90)
         (Cp.270) .. controls +(270:0.2) and +(120:0.2) .. (ysump.left in)
         (Dp.270) .. controls +(270:0.2) and +(60:0.2) .. (ysump.right in)
         (xsplitp.io) .. controls +(90:0.7) and +(270:2.5) .. (up) -- +(0,6)
         (usplitp.right out) .. controls +(300:0.5) and +(90:0.5) .. (ubendp)
         (usplitp.left out) .. controls +(240:0.2) and +(90:0.2) .. (Bp.90)
         (Ap.270) .. controls +(270:0.2) and +(120:0.2) .. (xdotsump.left in)
         (Bp.270) .. controls +(270:0.2) and +(60:0.2) .. (xdotsump.right in);

   \node [hole] at (3,-5.29) {};
   \node [hole] at (0.72,-8.65) {};

   \draw (xdotsump.io) .. controls +(270:0.5) and +(90:3) .. (yp)
         (Dp.90) -- (ubendp);
\end{tikzpicture}
}
\raisebox{5.53125cm}{=}
\raisebox{2.58125cm}{
\scalebox{1}{
\begin{tikzpicture}[thick]
% main nodes
   \node [delta] (usplit) at (2.5,3) {};
   \node [multiply] (A) at (0.5,2.5) {\(A\)};
   \node [multiply] (C) at (3,1.85) {\(C\)};
   \node [multiply] (Bp) at (3,0.5) {\(B'\)};
   \node [multiply] (Ap) at (1.98,0.5) {\(A'\)};
   \node [plus] (ysum) at (2.49,-0.65) {};

% auxiliary nodes
   \node [coordinate] (topA) [above of=A, shift={(0,-0.7)}] {};
   \node [coordinate] (topAp) [above of=Ap, shift={(0,-0.7)}] {};
   \node [coordinate] (leftAp) [left of=topAp] {};
   \node [coordinate] (rightAp) [right of=topAp, shift={(0.65,-0.5)}] {};
   \node [coordinate] (leftA) [left of=topA] {};

% wires
   \draw (usplit) -- +(0,0.75) +(1.15,0.75) -- (rightAp)
         (rightAp) .. controls +(270:2) and +(270:2.5) .. (leftAp)
         (leftAp) .. controls +(90:0.7) and +(90:0.7) .. (Ap.90)
         (usplit.right out) .. controls +(300:0.2) and +(90:0.2) .. (C.90)
         (usplit.left out) .. controls +(240:1.5) and +(270:2.5) .. (leftA)
         (leftA) .. controls +(90:0.7) and +(90:0.7) .. (A.90)
         (C.270) -- (Bp.90)
         (Bp.270) .. controls +(270:0.2) and +(60:0.2) .. (ysum.right in)
         (Ap.270) .. controls +(270:0.2) and +(120:0.2) .. (ysum.left in);

   \node [hole] at (2.49,-1.1) {};
   \node [hole] at (0.5,1.39) {};

   \draw (ysum) -- +(0,-1) +(-2,-1) -- (A.270);
\end{tikzpicture}
}}
\raisebox{5.53125cm}{=}
\raisebox{2.70625cm}{
\scalebox{1}{
\begin{tikzpicture}[thick]
% main nodes
   \node [delta] (usplit) at (2.5,3) {};
   \node [multiply] (A) at (2,1.175) {\(A\)};
   \node [multiply] (C) at (3,1.85) {\(C\)};
   \node [multiply] (Bp) at (3,0.6) {\(B'\)};
   \node [multiply] (Ap) at (4,1.175) {\(A'\)};
   \node [plus] (ysum) at (3.5,-0.65) {};

% wires
   \draw (usplit) -- +(0,0.75) +(1.5,0.75) -- (Ap.90)
         (usplit.right out) .. controls +(300:0.2) and +(90:0.2) .. (C.90)
         (usplit.left out) .. controls +(240:0.4) and +(90:0.75) .. (A.90)
         (C.270) -- (Bp.90)
         (Bp.270) .. controls +(270:0.2) and +(120:0.2) .. (ysum.left in)
         (Ap.270) .. controls +(270:0.75) and +(60:0.4) .. (ysum.right in)
         (ysum) -- +(0,-0.75) +(-1.5,-0.75) -- (A.270)
;
\end{tikzpicture}
}}
.
\end{center}
This final diagram for \(A''\), written in block matrix form is
\(\left[\begin{array}{cc}A & 0 \\ B'C & A'\end{array}\right]\).
\end{itemize}
This is exactly what was required.\end{proof}

\section{Duality properties of $\goodflow_k$}
\label{gooddual}
Recall that \(\sigflow_k\) has two different dagger structures, \(\dagger\) and \(*\).  As noted in
Section~\ref{cando}, controllability and observability are dual concepts, with the duality relating
to transposition.  Since \(-^*\) is also a duality related to transposition, this suggests a
connection between the controllable/observable duality and the \(-^*\) duality.

\begin{proposition}\label{bizarroco}
\(\goodflow_k\) is a dagger-category in only one of the two ways that \(\sigflow_{k,s}\) is.  Specifically,
the \(-^*\) dual of a morphism \(f\) in \(\goodflow_k\) is again a morphism in \(\goodflow_k\) such that
\begin{align*}
A(f^*) &= A(f)^* & B(f^*) &= C(f)^*\\
D(f^*) &= D(f)^* & C(f^*) &= B(f)^*.
\end{align*}
\end{proposition}
\begin{proof}
That the \(-^*\) duality behaves as described is an immediate consequence of \(\dBox(f^*) =
\dBox(f)^*\).  Thus we have
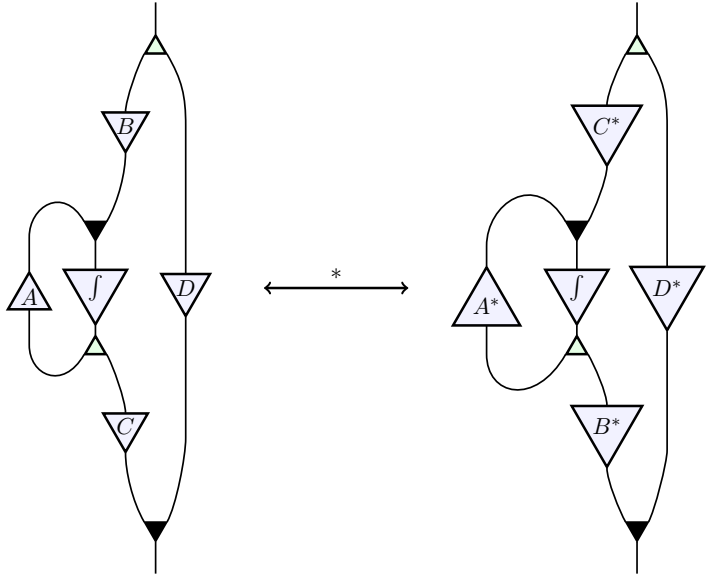
\begin{center}
\scalebox{0.80}{
 \begin{tikzpicture}[thick]
% main nodes
   \node [delta] (usplit) at (-0.5,4) {};
   \node [upmultiply] (A) at (-2.6,-0.15) {\(A\)};
   \node [multiply] (B) at (-1,2.7) {\(B\)};
   \node [plus] (xdotsum) at (-1.5,1) {};
   \node [multiply] (int) at (-1.5,0) {\(\int\)};
   \node [delta] (xsplit) at (-1.5,-1) {};
   \node [multiply] (C) at (-1,-2.3) {\(C\)};
   \node [multiply] (D) at (0,0) {\(D\)};
   \node [plus] (ysum) at (-0.5,-4) {};

% auxiliary nodes
   \node [coordinate] (capend) [above of=A] {};
   \node [coordinate] (cupend) [below of=A, shift={(0,0.2)}] {};
   \node [coordinate] (ubend) [right of=B] {};
   \node [coordinate] (ybend) [right of=C, shift={(0,-0.2)}] {};

% wires
   \draw (usplit) -- +(0,0.75)
         (ysum) -- +(0,-0.75)
         (xsplit.left out) .. controls +(240:0.7) and +(270:0.5) .. (cupend)
         (xdotsum.left in) .. controls +(120:0.7) and +(90:0.5) .. (capend)
         (capend) -- (A) -- (cupend)
         (xsplit.right out) .. controls +(300:0.2) and +(90:0.2) .. (C.90)
         (C.270) .. controls +(270:0.3) and +(120:0.4) .. (ysum.left in)
         (ybend) .. controls +(270:0.3) and +(60:0.3) .. (ysum.right in)
         (ybend) -- (D) -- (ubend)
         (usplit.right out) .. controls +(300:0.5) and +(90:0.5) .. (ubend)
         (usplit.left out) .. controls +(240:0.2) and +(90:0.2) .. (B.90)
         (B.270) .. controls +(270:0.3) and +(60:0.4) .. (xdotsum.right in)
         (xdotsum) -- (int) -- (xsplit)
   ;

% Dual main nodes
   \node [delta] (usplit) at (7.5,4) {};
   \node [upmultiply] (A) at (5,-0.3) {\(A^*\)};
   \node [multiply] (B) at (7,2.7) {\(C^*\)};
   \node [plus] (xdotsum) at (6.5,1) {};
   \node [multiply] (int) at (6.5,0) {\(\int\)};
   \node [delta] (xsplit) at (6.5,-1) {};
   \node [multiply] (C) at (7,-2.3) {\(B^*\)};
   \node [multiply] (D) at (8,0) {\(D^*\)};
   \node [plus] (ysum) at (7.5,-4) {};

% Dual auxiliary nodes
   \node [coordinate] (capend) [above of=A] {};
   \node [coordinate] (cupend) [below of=A, shift={(0,0.2)}] {};
   \node [coordinate] (ubend) [right of=B] {};
   \node [coordinate] (ybend) [right of=C, shift={(0,-0.4)}] {};

% Dual wires
   \draw (usplit) -- +(0,0.75)
         (ysum) -- +(0,-0.75)
         (xsplit.left out) .. controls +(240:1) and +(270:0.7) .. (cupend)
         (xdotsum.left in) .. controls +(120:1) and +(90:0.7) .. (capend)
         (capend) -- (A) -- (cupend)
         (xsplit.right out) .. controls +(300:0.2) and +(90:0.2) .. (C.90)
         (C.270) .. controls +(270:0.2) and +(120:0.2) .. (ysum.left in)
         (ybend) .. controls +(270:0.2) and +(60:0.2) .. (ysum.right in)
         (ybend) -- (D) -- (ubend)
         (usplit.right out) .. controls +(300:0.5) and +(90:0.5) .. (ubend)
         (usplit.left out) .. controls +(240:0.2) and +(90:0.2) .. (B.90)
         (B.270) .. controls +(270:0.2) and +(60:0.2) .. (xdotsum.right in)
         (xdotsum) -- (int) -- (xsplit)
   ;

   \draw[very thick,<->] (1.3,0) to node[above] {\(*\)} (3.7,0);
 \end{tikzpicture}
}.
\end{center}

To show \(\goodflow_k\) does not have the \(-^\dagger\) duality, it suffices to give a counterexample.
Let \(f\) be any signal-flow diagram such that \(D(f) = 0 \maps k \to k\).  For example take \(f\)
to be the signal-flow diagram:
\begin{center}
\begin{tikzpicture}[thick]
   \node[bang] (del) {};
   \node[zero] (zer) [below of=del, shift={(0,0.4)}] {};
   \draw (del) -- +(0,0.7)  (zer) -- +(0,-0.7);
\end{tikzpicture}.
\end{center}
We see \(D(f^\dagger)\) is the scaling by \(0^{-1}\).  While this is a linear relation, it is
clearly not a linear map.
\end{proof}

Our signal-flow diagrams have all been time-independent, so the time reversal part in Kalman's
duality is trivial for our diagrams.  Note also that the \(-^*\) duality is exactly transposition when
applied to linear maps, so we can rewrite the equations in Proposition~\ref{bizarroco}:
\begin{align*}
A(f^*) &= (A(f))^\top & B(f^*) &= (C(f))^\top \\
D(f^*) &= (D(f))^\top & C(f^*) &= (B(f))^\top.
\end{align*}
At the level of state-space equations, this is exactly the time-independent version of Kalman's
duality!  We predict the \(-^*\) duality can be rigorously extended to time-dependent signal-flow
diagrams in such a way that it exactly matches Kalman's duality.  In the present work we have taken
advantage of time-independence in the definition of \(\st_k\), restricting the appearance of Laplace
transform variable \(s\) to the \((sI-A)^{-1}\) part.  The time-dependent case will require a new
approach to \(\st_k\), hence a new approach to \(\goodflow_k\).

\chapter{Conclusions}
\label{conclusions}
While the story of control theory placed into the context of category theory is far from complete,
we have advanced the plot.  First, in Chapter~\ref{vectrel} we found a \smc{} that describes the
relation between the inputs and outputs of signal-flow diagrams and described it in terms of
generating morphisms and a set of equations between morphisms.  The strict version of this gave us
\(\relk\).  Bonchi, Soboci\'nski and Zanasi \cite{BSZ1,BSZ2} independently studied an equivalent
\smc{} around the same time and with the same generator-and-equations perspective, but from a very
different approach and with a slightly different set of generators.  On our way to \(\relk\), we
considered \(\vectk\), which is the strict \smc{} we would get if signal-flow diagrams had no
feedback.  Wadsley and Woods \cite{WW} considered \(\Mat(k)\), which is equivalent to our \(\vectk\)
when \(k\) is a field, but only insisted that \(k\) be a rig.  There are also many mysterious and
interesting connections between \(\relk\) and \smcs{} used in quantum mechanics 
\cite{AC,BS,CD,CP,CPV,CW,Kock,Kock2,RSW,Selinger,Vicary}, exposing similarities and subtle
differences between vector spaces, Hilbert spaces, and cobordisms.  Nevertheless, \(\relks\) is
helpless to describe certain important control theory concepts such as controllability and
observability.

To get a handle on these two concepts, we went back in Chapter~\ref{stateful} to the state-space
equations, Equations \ref{stateeq} and \ref{outputeq} upon which these concepts are founded.  By
encoding these equations in signal-flow diagrams, we found the category \(\st_k\), where the four
matrices in the state-space equations taken \emph{en masse} correspond to stateful morphisms.  While
for simplicity's sake we only considered the linear time-independent case where the four matrices
are constant in time, \(\st_k\) can easily be extended to the time-varying case by adding a very
mild condition on the matrices.  Controllability and observability only depend on the matrices in
the state-space equations in either case, so knowing a stateful morphism provides enough information
to determine controllability and observability.  The analysis simplifies greatly in the linear
time-independent case, where they can be determined in terms of epimorphisms and monomorphisms.

In Chapter~\ref{goodflow} we pushed the idea of controllability and observability a little further.
Stateful morphisms can only be determined if we already know what the matrices in the state-space
equations are, which leaves the issue of finding these matrices.  Given an arbitrary signal-flow
diagram, the values of these matrices may not be obvious, or worse, linear relations may be
involved, not just linear maps.  For this reason we limited the scope of signal-flow diagrams to
form a new category \(\goodflow_k\), where only the signal-flow diagrams that can be converted to
stateful morphisms are considered.  These signal-flow diagrams coincide with the ones most typically
drawn by control theorists.  By converting a signal-flow diagram in \(\goodflow_k\) to a stateful
morphism, we can determine controllability and observability for that signal-flow diagram.

While controllability and observability are important, there are many other concepts that are
important to control theorists, such as stability (which itself comes in several guises) and pole
placement.  On the other hand, there are several opportunities for the current work to be extended
to allow control theorists to draw more general signal-flow diagrams.  Our \(\st_k\) PROP, which
finally allowed us to describe controllability and observability in the category theory context,
should be extendable using the results of Appendix~\ref{generalbox}, allowing a larger collection of
signal-flow diagrams to be considered for controllability and observability.

Once we have extended \(\st_k\) in this way, we can also extend \(\goodflow_k\).  We have seen in
Chapter~\ref{goodflow} how the state-space equations \ref{stateeq} and \ref{outputeq} can be used to
form the PROP \(\goodflow_k\).  One of the features of \(\goodflow_k\) is the deterministic nature
of its morphisms: the current state and input uniquely determine the future state and output.  We
may not be able to induce the system to enter a given state if the system is not controllable, and
we may not be able to determine the state of the system if it is not observable, so the state of a
system can act as a `hidden variable'.  However, in some circumstances it may be useful to eschew
determinism.  We can do this by generalizing the state-space \emph{equations} to state-space
\emph{relations} in a way that does not sacrifice much of the convenience of dealing with linear
maps.

To get a flavor of this, we will still insist \(B\), \(C\) and \(D\) are linear maps, only allowing
\(A\) to be replaced with a linear relation.  In Appendix~\ref{generalbox} we see that we can
generalize a bit more than this, but the full generality at this point would only serve to weigh
down the exposition, obscuring what we wish to point out: a direction for extending
\(\goodflow_\R\).  The state-space relations will then appear as:
\begin{equation}\label{staterels}
\dot{x}(t) \in A(t) x(t) + B(t) u(t)
\end{equation}
\begin{equation}\label{outputrels}
y(t) = C(t) x(t) + D(t) u(t).
\end{equation}
Given two such systems: \(\dot{x}_1 \in A_1 x_1 + B_1 u_1\), \(y_1 = C_1 x_1 + D_1 u_1\) and
\(\dot{x}_2 \in A_2 x_2 + B_2 y_1\), \(y_2 = C_2 x_2 + D_2 y_1\); they can compose by writing all
the relations and eliminating the common \(y_1\).  Thus:
\begin{align*}
\dot{x}_1 &\in A_1 x_1 + B_1 u_1\\
\dot{x}_2 &\in A_2 x_2 + B_2 (C_1 x_1 + D_1 u_1)\\
y_2       &=   C_2 x_2 + D_2 (C_1 x_1 + D_1 u_1).
\end{align*}
Since each \(B\), \(C\) and \(D\) is a linear map, compositions of these linear relations still
distribute over addition, so the composite system can be written:

\begin{align*}
\dot{x}_1 &\in A_1 x_1 + B_1 u_1\\
\dot{x}_2 &\in (B_2 C_1) x_1 + A_2 x_2 + (B_2 D_1) u_1\\
y_2       &=   (D_2 C_1) x_1 + C_2 x_2 + (D_2 D_1) u_1.
\end{align*}
It is easy to show this is equivalent to
\begin{align*}
\left[\begin{array}{c} \dot{x}_1 \\ \dot{x}_2 \end{array}\right] &\in
A_3
\left[\begin{array}{c} x_1 \\ x_2 \end{array}\right] + 
\left[\begin{array}{c} B_1 \\ B_2 D_1 \end{array}\right] u_1 \\
y_2       &=
\left[\begin{array}{cc} D_2 C_1 & C_2 \end{array}\right]
\left[\begin{array}{c} x_1 \\ x_2 \end{array}\right] + (D_2 D_1) u_1,
\end{align*}
for some linear relation \(A_3\).  The fact that we get a system of state-space relations again
indicates that we should be able to form a PROP \(\looseflow_k\) from the state-space relations in
much the way that we formed \(\goodflow_k\) from the state-space equations.

A difficulty in dealing with \(\looseflow_k\) arises when trying to determine the linear relations
\(A(f)\), \(B(f)\), \(C(f)\), and \(D(f)\) associated with the signal-flow diagram \(f \in
\looseflow_k\).  The processes defined for finding linear maps from a signal-flow diagram \(f \in
\goodflow_k\) are not appropriate here.  When \(A(f)\) is a linear relation that is not a linear
map, these processes will give one or both of \(B(f)\) and \(C(f)\) as linear relations that are not
linear maps.  If Conjecture~\ref{conject} is true, it may be possible to use the \(\porp\)-normal
form of a signal-flow diagram to find the linear relations \(A(f)\), \(B(f)\), \(C(f)\), and
\(D(f)\).  This approach, if it works, would be more satisfying than the current \emph{ad hoc}
approach to finding these when they are linear maps.

Another direction for research, taken up by Baez, Coya and Rebro \cite{BCR}, is connecting this work
with the work of Baez and Fong \cite{BF} on passive linear networks.  Whereas we have primarily
focused on presenting PROPs in terms of generators and equations, Baez and Fong use a framework of
`decorated cospans'.  Using this framework, they find a black-box functor from \(\Ecirc\), the
category of open passive linear electric circuits, to \(\Lagr\), the category of symplectic vector
spaces over the field \(k(s)\) and Lagrangian relations.  These categories are equivalent to their
skeletons, so there is a black-box functor between their skeletons, \(\bbox \maps \ecirc \to
\lagr\).  Since there is a \smdf{} \(i \maps \lagr \to \relks\) that includes \(\lagr\) in
\(\relks\), composing with the black-box functor gives the \smdf
\[i \of \bbox \maps \ecirc \to \relks.\]
A key result here is to find a \smdf{} \(F \maps \ecirc \to \sigflow_{k(s)}\) such that the
black-box functor we defined from \(\sigflow_{k(s)}\) to \(\relks\) makes this functor diagram
commute (up to isomorphism):
\begin{center}
\begin{tikzpicture}
   \node (SF) {\(\sigflow_{k(s)}\)};
   \node (P) [above of=SF, shift={(0,2)}] {\(\ecirc \vphantom{\lagr}\)};
   \node (FR) [right of=SF, shift={(2,0)}] {\(\relks\)};
   \node (ST) [above of=FR, shift={(0,2)}] {\(\lagr\)};

   \draw [->] (P) to node [above] {\(\bbox\)} (ST);
   \draw [->] (P) to node [left] {\(F\)} (SF);
   \draw [right hook->] (ST) to node [right] {\(i\)} (FR);
   \draw [->] (SF) to node [below] {\(\bbox\)} (FR);
\end{tikzpicture}.
\end{center}
The main challenge in this endeavor is not in finding the functor \(F\), but in reconciling the
approaches sufficient to prove the square commutes.

\nocite{*}
% \singlespacing
% \bibliographystyle{alpha}
\bibliographystyle{plain}
\bibliography{bibfile}

\appendix
\chapter{Proofs of selected derived equations}
\label{derivedeqns}
In the proof of Theorem~\ref{presrk} we used several equations derived from the equations in our
presentation of \(\relk\).  While the more straightforward equations were demonstrated immediately,
some of the more useful derived equations are less straightforward and we demonstrate them here.
While these derived equations could simply be appended to our presentation, the elegance of only
using simple structures would be lost.  On the other hand, Heunen and Vicary \cite{HV} pointed out
only one Frobenius equation (per color) is necessary instead of two.  Several of the equations
necessary for the presentation of \(\vectk\) are also superfluous for the presentation of \(\relk\),
but it seems more elegant to build from simple structures than to minimize the number of equations
for the sake of minimization.

\section{(D5)}
\label{A:D5}
Derived equations {\hyperref[eqnD5D6D7]{\textbf{(D5)--(D7)}}} are variations on the bimonoid
equations {\hyperref[eqn78910]{\textbf{(7)--(9)}}}.  Derived equation \textbf{(D5)} can be proved as
follows:
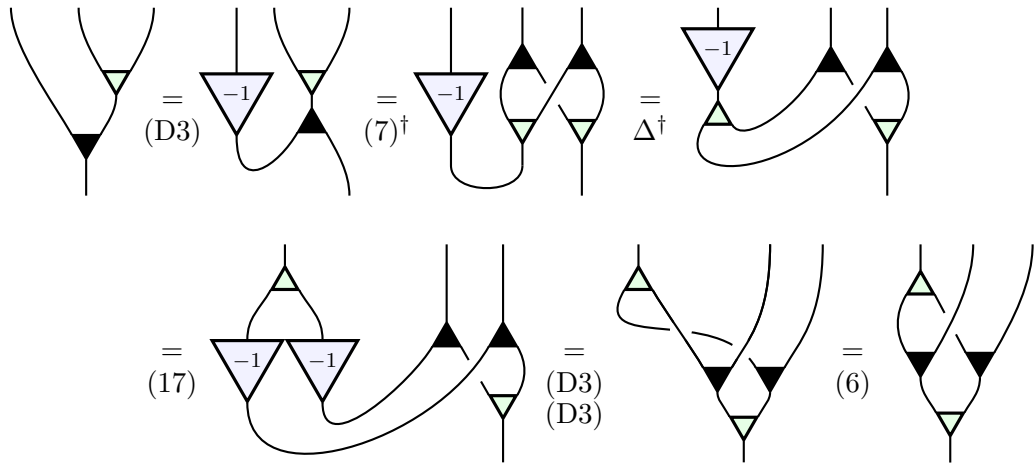
\begin{center}
   \begin{tikzpicture}[thick, node distance=0.7cm]
% 1
   \node [codelta] (nabzip) at (1.4,0.3) {};
   \node [plus] (add) at (1,-0.55) {};
   \node [coordinate] (outz) [below of=add] {};
   \node (equal) at (2.15,0) {\(=\)};
   \node [below=0.2em] at (equal) {(D3)};

   \draw (add.left in) .. controls +(120:0.5) and +(-90:0.5) .. (0,1.25)
         (add.io) -- (outz)
         (nabzip.left in) .. controls +(120:0.3) and +(-90:0.3) .. (0.9,1.25)
         (nabzip.right in) .. controls +(60:0.3) and +(-90:0.3) .. (1.9,1.25)
         (nabzip.io) .. controls +(-90:0.2) and +(60:0.2) .. (add.right in);
% 2
   \node [multiply] (m2) at (3,0.1) {\(\scriptstyle{-1}\)};
   \node [coplus] (plus2) at (4,-0.3) {};
   \node [codelta] (nab2) at (4,0.3) {};

   \draw (nab2.io) -- (plus2.io) (m2.90) -- (3,1.25)
         (nab2.left in) .. controls +(120:0.3) and +(-90:0.3) .. (3.5,1.25)
         (nab2.right in) .. controls +(60:0.3) and +(-90:0.3) .. (4.5,1.25)
         (plus2.left out) .. controls +(-120:0.5) and +(-90:0.8) .. (m2.io)
         (plus2.right out) .. controls +(-60:0.3) and +(90:0.3) .. (4.5,-1.25);

   \node (eq2) at (5,0) {\(=\)};
   \node [below=0.1em] at (eq2) {(7)\({}^{\dagger}\)};
% 3
   \node [multiply] (m3) at (5.85,0.1) {\(\scriptstyle{-1}\)};
   \node [coordinate] (m3aux) at (5.85,-0.85) {};
   \node [codelta] (addL) at (6.8,-0.35) {};
   \node (cross) [above right of=addL, shift={(-0.1,-0.0435)}] {};
   \node [codelta] (addR) [below right of=cross, shift={(-0.1,0.0435)}] {};
   \node [coplus] (dupeL) [above left of=cross, shift={(0.1,-0.0435)}] {};
   \node [coplus] (dupeR) [above right of=cross, shift={(-0.1,-0.0435)}] {};
   \node [coordinate] (f) [above of=dupeL] {};
   \node [coordinate] (g) [above of=dupeR] {};
   \node [coordinate] (sum1) [below of=addL, shift={(0,0.2)}] {};
   \node [coordinate] (sum2) [below of=addR, shift={(0,-0.2)}] {};
   \node [coordinate] (inL) [left of=sum1, shift={(0.2,0)}] {};
   \node (eq3) at (8.5,0) {\(=\)};
   \node [below=0.1em] at (eq3) {\(\Delta^{\dagger}\)};

   \path
   (addL) edge (sum1)
   (addL.right in) edge (dupeR.left out)
   (addL.left in) edge [bend left=30] (dupeL.left out)
   (addR) edge (sum2)
   (addR.left in) edge (cross)
   (addR.right in) edge [bend right=30] (dupeR.right out)
   (dupeL) edge (f)
   (dupeL.right out) edge (cross)
   (dupeR) edge (g);

   \draw (sum1) .. controls +(-90:0.4) and +(-90:0.4) .. (m3aux) -- (m3) -- +(0,1.15);
% 4
   \node [multiply] (m4) at (9.4,0.7) {\(\scriptstyle{-1}\)};
   \node [delta] (delt) at (9.4,-0.2) {};
   \node [codelta] (nab4) at (11.65,-0.35) {};
   \node [coplus] (coadd4L) at (10.9,0.513) {};
   \node (cross4) at (11.275,0.0815) {};
   \node [coplus] (coadd4R) at (11.65,0.513) {};

   \path (nab4.right in) edge [bend right=30] (coadd4R.right out);

   \draw (m4) -- +(0,0.55) (coadd4L) -- +(0,0.737) (coadd4R) -- +(0,0.737) (nab4) -- +(0,-0.9) (delt.io) -- (m4.io)
         (delt.right out) .. controls +(-60:0.3) and +(-120:0.6) .. (coadd4L.left out)
         (delt.left out) .. controls +(-120:1) and +(-120:1.7) .. (coadd4R.left out)
         (coadd4L.right out) -- (cross4) -- (nab4.left in);

   \node (spacer) at (14,0) {\({}\)};
   \end{tikzpicture}
\vskip 1.5em
\noindent
% 5
   \begin{tikzpicture}[thick, node distance=0.7cm]
   \node (spacer) at (-1,0) {\({}\)};
   \node (eq4) at (0.5,0) {\(=\)};
   \node [below=0.2em] at (eq4) {(17)};

   \node [delta] (delt5) at (2,0.95) {};
   \node [multiply] (m5L) at (1.5,-0.1) {\(\scriptstyle{-1}\)};
   \node [multiply] (m5R) at (2.5,-0.1) {\(\scriptstyle{-1}\)};
   \node [coplus] (coadd5L) at (4.15,0.2) {};
   \node [coplus] (coadd5R) at (4.9,0.2) {};
   \node (cross5) at (4.525,-0.2315) {};
   \node [codelta] (nab5) at (4.9,-0.663) {};
   \node (eq5) at (5.85,0) {\(=\)};
   \node [below=0.2em] at (eq5) {(D3)};
   \node [below=1.3em] at (eq5) {(D3)};

   \path (nab5.right in) edge [bend right=30] (coadd5R.right out);

   \draw (delt5) -- +(0,0.5) (coadd5L) -- +(0,1.25) (coadd5R) -- +(0,1.25) (nab5) -- +(0,-0.787)
         (delt5.right out) .. controls +(-60:0.3) and +(90:0.3) .. (m5R.90)
         (delt5.left out) .. controls +(-120:0.3) and +(90:0.3) .. (m5L.90)
         (coadd5L.left out) .. controls +(-120:0.5) and +(-90:0.8) .. (m5R.io)
         (coadd5R.left out) .. controls +(-120:1.5) and +(-90:1.3) .. (m5L.io)
         (coadd5L.right out) -- (cross5) -- (nab5.left in);
% 6
   \node [delta] (delt6) at (6.7,0.95) {};
   \node [plus] (add6L) at (7.75,-0.3) {};
   \node [codelta] (nab6) at (8.1,-0.95) {};
   \node [plus] (add6R) at (8.45,-0.3) {};
   \node (eq6) at (9.55,0) {\(=\)};
   \node [below=0.2em] at (eq6) {(6)};

   \draw (delt6) -- +(0,0.5) (nab6) -- +(0,-0.5)
         (add6R.left in) .. controls +(120:1) and +(-120:1) .. (delt6.left out)
         (add6L.left in) .. controls +(120:0.3) and +(-60:0.3) .. (delt6.right out)
         (add6L.right in) .. controls +(60:0.4) and +(-90:1) .. (8.45,1.45)
         (add6R.right in) .. controls +(60:0.4) and +(-90:1) .. (9.15,1.45)
         (nab6.right in) .. controls +(60:0.2) and +(270:0.1) .. (add6R.io)
         (nab6.left in) .. controls +(120:0.2) and +(270:0.1) .. (add6L.io);

   \node [hole] at (7.265,0.31) {};
   \node [hole] at (8.08,0.05) {};
   \draw (add6L.left in) .. controls +(120:0.3) and +(-60:0.3) .. (delt6.right out)
         (add6L.right in) .. controls +(60:0.4) and +(-90:1) .. (8.45,1.45);

% 7
   \node [plus] (addl) at (10.45,-0.1) {};
   \node (cross) at (10.845,0.395) {};
   \node [delta] (delta) at (10.45,0.89) {};
   \node [plus] (addr) [below right of=cross, shift={(-0.1,0)}] {};
   \node [codelta] (nablunzip) [below left of=addr, shift={(0.1,-0.3)}] {};
   \node (outu) [below of=nablunzip] {};

   \path
   (delta.left out) edge [bend right=30] (addl.left in);

   \draw (delta) -- +(0,0.56)
         (delta.right out) -- (cross) -- (addr.left in);
   \draw (addl.right in) .. controls +(60:0.5) and +(-90:0.5) .. (11.15,1.45)
         (addr.right in) .. controls +(60:0.5) and +(-90:0.5) .. (11.94,1.45)
         (addl.io) .. controls +(-90:0.2) and +(120:0.2) .. (nablunzip.left in)
         (addr.io) .. controls +(-90:0.2) and +(60:0.2) .. (nablunzip.right in);
   \draw (nablunzip) -- (outu);
   \end{tikzpicture}
\end{center}

\section{(D6)--(D7)}
\label{A:D6D7}
Derived equation {\hyperref[eqnD5D6D7]{\textbf{(D7)}}} can be proved as follows:

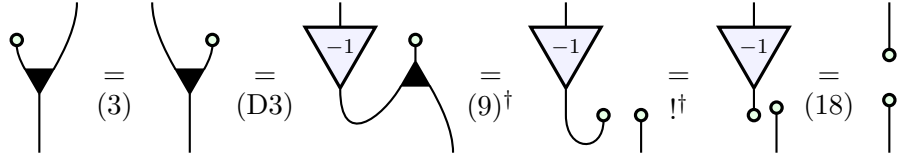
\begin{center}
\scalebox{1}{   \begin{tikzpicture}[thick]
   \node [bang] (bang1) at (-0.3,0.5) {};
   \node [plus] (plus1) at (0,0) {};
   \node (eq1) at (1,0) {\(=\)};
   \node at (1,-1em) {(3)};
   \node [plus] (plus2) at (2,0) {};
   \node [bang] (bang2) at (2.3,0.5) {};
   \node (eq2) at (3,0) {\(=\)};
   \node at (3,-1em) {(D3)};
   \node [multiply] (m3) at (4,0.4) {\(\scriptstyle{-1}\)};
   \node [coplus] (plus3) at (5,0) {};
   \node [bang] (bang3) at (5,0.5) {};
   \node (eq3) at (6,0) {\(=\)};
   \node at (6,-1em) {(9)\({}^{\dagger}\)};
   \node [multiply] (m4) at (7,0.4) {\(\scriptstyle{-1}\)};
   \node [bang] (bang4a) at (7.5,-0.5) {};
   \node [bang] (bang4b) at (8,-0.5) {};
   \node (eq4) at (8.5,0) {\(=\)};
   \node at (8.5,-1em) {\(!^{\dagger}\)};
   \node [multiply] (m5) at (9.5,0.4) {\(\scriptstyle{-1}\)};
   \node [bang] (bang5a) at (9.5,-0.5) {};
   \node [bang] (bang5b) at (9.8,-0.4) {};
   \node (eq5) at (10.5,0) {\(=\)};
   \node at (10.5,-1em) {(18)};
   \node [bang] (bang6a) at (11.3,0.3) {};
   \node [bang] (bang6b) at (11.3,-0.3) {};

   \draw (bang1) .. controls +(-90:0.15) and +(120:0.15) .. (plus1.left in)
         (bang2) .. controls +(-90:0.15) and +(60:0.15) .. (plus2.right in)
         (bang3) -- (plus3.io)
         (bang4a) .. controls +(-90:0.5) and +(-90:0.5) .. (7,-0.5) -- (m4.io)
         (bang4b) -- +(-90:0.5)
         (bang5a) -- (m5.io)
         (bang5b) -- +(-90:0.6)
         (bang6a) -- +(90:0.7)
         (bang6b) -- +(-90:0.7)
         (m3.90) -- (4,1)
         (m4.90) -- (7,1)
         (m5.90) -- (9.5,1)
         (plus1.io) -- (0,-1)
         (plus2.io) -- (2,-1)
         (plus1.right in) .. controls +(60:0.3) and +(-90:0.3) .. (0.5,1)
         (plus2.left in) .. controls +(120:0.3) and +(-90:0.3) .. (1.5,1)
         (plus3.left out) .. controls +(-120:0.5) and +(-90:0.8) .. (m3.io)
         (plus3.right out) .. controls +(-60:0.3) and +(90:0.3) .. (5.5,-1)
;
   \end{tikzpicture}}
\end{center}

The proof of derived equation \textbf{(D6)} is a vertically flipped and color-swapped version of the
proof of derived equation \textbf{(D7)} above, but without the scaling by \(-1\).  These two
equations are also proved in the \textit{Graphical Linear Algebra} blog. % \cite{GLA-keepcalm}.

\section{(D8)--(D9)}
\label{A:D8D9}
Derived equation {\hyperref[eqnD8D9]{\textbf{(D8)}}} is simply a statement that \(x+y\) and \(x+cy\)
are linearly independent whenever \(c \neq 1\).  This can be proved diagrammatically as follows:
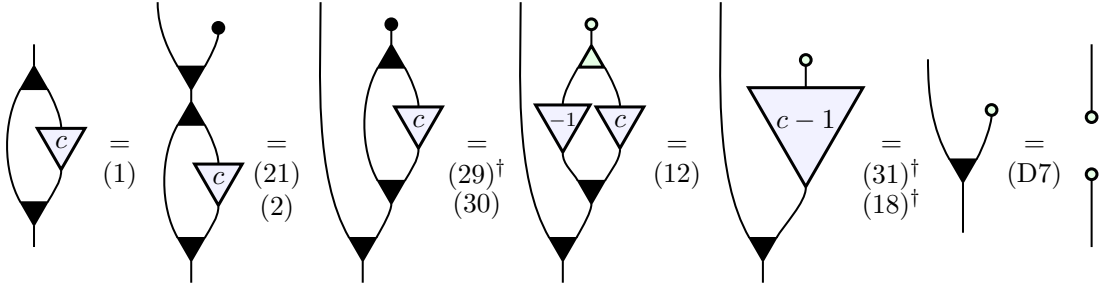
\begin{center}
\scalebox{0.95}{   \begin{tikzpicture}[thick]
   \node [coplus] (cosum) at (0,0.9) {};
   \node [multiply] (times) at (0.38,0.05) {\(c\)};
   \node [plus] (sum) at (0,-0.9) {};

   \draw
   (cosum.io) -- +(0,0.3) (sum.io) -- +(0,-0.3)
   (cosum.left out) .. controls +(240:0.5) and +(120:0.5) .. (sum.left in)
   (cosum.right out) .. controls +(300:0.15) and +(90:0.15) .. (times.90)
   (times.io) .. controls +(270:0.15) and +(60:0.15) .. (sum.right in);

   \node (eq) at (1.2,0) {\(=\)};
   \node [below=0.2em] at (eq) {(1)};
% 2
   \node [coplus] (cosum2) at (2.2,0.4) {};
   \node [multiply] (times2) at (2.58,-0.45) {\(c\)};
   \node [plus] (sum2) at (2.2,-1.4) {};
   \node [plus] (top2) at (2.2,1) {};
   \node [zero] (z2) at (2.58,1.65) {};

   \draw
   (cosum2.io) -- (top2.io) (sum2.io) -- +(0,-0.3)
   (cosum2.left out) .. controls +(240:0.5) and +(120:0.5) .. (sum2.left in)
   (cosum2.right out) .. controls +(300:0.15) and +(90:0.15) .. (times2.90)
   (times2.io) .. controls +(270:0.15) and +(60:0.15) .. (sum2.right in)
   (top2.right in) .. controls +(60:0.1) and +(270:0.2) .. (z2)
   (top2.left in) .. controls +(120:0.2) and +(270:0.4) .. +(-0.28,0.9);

   \node (eq2) at (3.4,0) {\(=\)};
   \node [below=0.2em] at (eq2) {(21)};
   \node [below=1.4em] at (eq2) {(2)};
% 3
   \node [coplus] (cosum3) at (5,1.2) {};
   \node [multiply] (times3) at (5.38,0.35) {\(c\)};
   \node [plus] (sum3) at (5,-0.6) {};
   \node [plus] (bot3) at (4.6,-1.4) {};
   \node [zero] (z3) at (5,1.7) {};

   \draw
   (cosum3.io) -- (z3) (bot3.io) -- +(0,-0.3)
   (cosum3.left out) .. controls +(240:0.5) and +(120:0.5) .. (sum3.left in)
   (cosum3.right out) .. controls +(300:0.15) and +(90:0.15) .. (times3.90)
   (times3.io) .. controls +(270:0.15) and +(60:0.15) .. (sum3.right in)
   (sum3.io) .. controls +(270:0.1) and +(60:0.2) .. (bot3.right in)
   (bot3.left in) .. controls +(120:1) and +(270:1) .. +(-0.4,3.3);

   \node (eq3) at (6.2,0) {\(=\)};
   \node [below=0.1em] at (eq3) {(29)\({}^{\dagger}\)};
   \node [below=1.3em] at (eq3) {(30)};
% 4
   \node [delta] (delta4) at (7.8,1.2) {};
   \node [multiply] (times4) at (8.2,0.35) {\(c\)};
   \node [multiply] (neg4) at (7.4,0.35) {\(\hspace{-1pt}\scriptstyle{-1}\hspace{-1pt}\)};
   \node [plus] (sum4) at (7.8,-0.6) {};
   \node [plus] (bot4) at (7.4,-1.4) {};
   \node [bang] (b4) at (7.8,1.7) {};

   \draw
   (delta4.io) -- (b4) (bot4.io) -- +(0,-0.3)
   (delta4.left out) .. controls +(240:0.15) and +(90:0.15) .. (neg4.90)
   (neg4.io) .. controls +(270:0.15) and +(120:0.15) .. (sum4.left in)
   (delta4.right out) .. controls +(300:0.15) and +(90:0.15) .. (times4.90)
   (times4.io) .. controls +(270:0.15) and +(60:0.15) .. (sum4.right in)
   (sum4.io) .. controls +(270:0.1) and +(60:0.2) .. (bot4.right in)
   (bot4.left in) .. controls +(120:1) and +(270:1) .. +(-0.4,3.3);

   \node (eq4) at (9,0) {\(=\)};
   \node [below=0.2em] at (eq4) {(12)};
% 5
   \node [multiply] (times5) at (10.8,0.35) {\(\hspace{-2.5pt}c-1\hspace{-2.5pt}\)};
   \node [plus] (sum5) at (10.2,-1.4) {};
   \node [bang] (b5) at (10.8,1.2) {};

   \draw
   (sum5.io) -- +(0,-0.3) (times5) -- (b5)
   (times5.io) .. controls +(270:0.2) and +(60:0.2) .. (sum5.right in)
   (sum5.left in) .. controls +(120:1) and +(270:1) .. +(-0.4,3.3);

   \node (eq5) at (12,0) {\(=\)};
   \node [below=0.1em] at (eq5) {(31)\({}^{\dagger}\)};
   \node [below=1.2em] at (eq5) {(18)\({}^{\dagger}\)};
% 6
   \node [plus] (sum6) at (13,-0.3) {};
   \node [bang] (b6) at (13.4,0.5) {};

   \draw
   (sum6.io) -- +(0,-0.7)
   (sum6.right in) .. controls +(60:0.2) and +(270:0.2) .. (b6)
   (sum6.left in) .. controls +(120:0.5) and +(270:0.5) .. +(-0.3,1.4);

   \node (eq6) at (14,0) {\(=\)};
   \node [below=0.2em] at (eq6) {(D7)};
% 7
   \node [bang] (bang) at (14.8,0.4) {};
   \node [bang] (cobang) at (14.8,-0.4) {};

   \draw (bang) -- +(0,1.02) (cobang) -- +(0,-1.02);
   \end{tikzpicture}}
\end{center}

Derived equation \textbf{(D9)} is the statement that when \(c \neq 1\), \(x = cx\) implies \(x=0\).
The proof of derived equation \textbf{(D9)} is a vertically flipped and color-swapped version of the
proof of derived equation \textbf{(D8)} above.

\section{Frobenius equations}
\label{A:frobenius}
In our presentation of \(\relk\), two of the Frobenius equations are superfluous.  For each pair of
Frobenius equations {\hyperref[eqn2122]{\textbf{(21)--(22)}}} and
{\hyperref[eqn2324]{\textbf{(23)--(24)}}} either equation can be derived from the other.
Furthermore, both equations of a pair can be derived from the `outer' equation.  Here we show how to
derive equation \textbf{(22)} from equation \textbf{(21)}, commutativity of \(+\), and
cocomutativity of \(+^{\dagger}\).

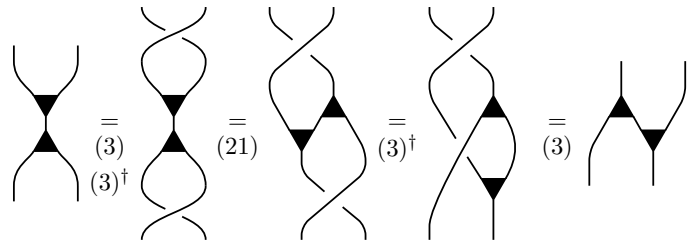
\begin{center}
\scalebox{0.85}{   \begin{tikzpicture}[thick]
   \node [plus] (sum1) at (1,0.325) {};
   \node [coplus] (cosum1) at (1,-0.325) {};
   \node [coordinate] (sum1inleft) at (0.5,0.975) {};
   \node [coordinate] (sum1inright) at (1.5,0.975) {};
   \node [coordinate] (cosum1outleft) at (0.5,-0.975) {};
   \node [coordinate] (cosum1outright) at (1.5,-0.975) {};
   \node [coordinate] (1in1) at (0.5,1.22) {};
   \node [coordinate] (1in2) at (1.5,1.22) {};
   \node [coordinate] (1out1) at (0.5,-1.22) {};
   \node [coordinate] (1out2) at (1.5,-1.22) {};

   \draw (sum1inleft) .. controls +(270:0.3) and +(120:0.15) .. (sum1.left in)
   (sum1inright) .. controls +(270:0.3) and +(60:0.15) .. (sum1.right in)
   (cosum1outleft) .. controls +(90:0.3) and +(240:0.15) .. (cosum1.left out)
   (cosum1outright) .. controls +(90:0.3) and +(300:0.15) .. (cosum1.right out)
   (sum1.io) -- (cosum1.io) (1in1) -- (sum1inleft) (1in2) -- (sum1inright)
   (1out1) -- (cosum1outleft) (1out2) -- (cosum1outright)
;

   \node (eq1) at (2,0) {\(=\)};
   \node at (2,-1em) {(3)};
   \node at (2,-2.4em) {(3)\({}^{\dagger}\)};

   \node [plus] (sum2) at (3,0.325) {};
   \node [coplus] (cosum2) at (3,-0.325) {};
   \node [coordinate] (sum2inleft) at (2.5,0.975) {};
   \node [coordinate] (sum2inright) at (3.5,0.975) {};
   \node [coordinate] (cosum2outleft) at (2.5,-0.975) {};
   \node [coordinate] (cosum2outright) at (3.5,-0.975) {};
   \node [coordinate] (2in1) at (2.5,1.822) {};
   \node [coordinate] (2in2) at (3.5,1.822) {};
   \node [coordinate] (2out1) at (2.5,-1.822) {};
   \node [coordinate] (2out2) at (3.5,-1.822) {};

   \draw (sum2inleft) .. controls +(270:0.3) and +(120:0.15) .. (sum2.left in)
   (sum2inright) .. controls +(270:0.3) and +(60:0.15) .. (sum2.right in)
   (cosum2outleft) .. controls +(90:0.3) and +(240:0.15) .. (cosum2.left out)
   (cosum2outright) .. controls +(90:0.3) and +(300:0.15) .. (cosum2.right out)
   (sum2.io) -- (cosum2.io)
   (sum2inright) .. controls +(90:0.4) and +(270:0.4) .. (2in1)
   (cosum2outleft) .. controls +(270:0.4) and +(90:0.4) .. (2out2)
;
   \node [hole] at (3,1.4) {};
   \node [hole] at (3,-1.4) {};

   \draw
   (sum2inleft) .. controls +(90:0.4) and +(270:0.4) .. (2in2)
   (cosum2outright) .. controls +(270:0.4) and +(90:0.4) .. (2out1)
;

   \node (eq2) at (4,0) {\(=\)};
   \node at (4,-1em) {(21)};

   \node [plus] (sum3) at (5,-0.216) {};
   \node [coplus] (cosum3) at (5.5,0.216) {};
   \node [coordinate] (sum3corner) at (4.5,0.434) {};
   \node [coordinate] (cosum3corner) at (6,-0.434) {};
   \node [coordinate] (sum3out) at (5,-0.757) {};
   \node [coordinate] (cosum3in) at (5.5,0.757) {};
   \node [coordinate] (3cornerin) at (4.5,0.757) {};
   \node [coordinate] (3cornerout) at (6,-0.757) {};
   \node (topcross3) at (5,1.191) {};
   \node (botcross3) at (5.5,-1.191) {};
   \node [coordinate] (3in1) at (4.5,1.605) {};
   \node [coordinate] (3in2) at (5.5,1.605) {};
   \node [coordinate] (3in1u) at (4.5,1.822) {};
   \node [coordinate] (3in2u) at (5.5,1.822) {};
   \node [coordinate] (3out1) at (5,-1.605) {};
   \node [coordinate] (3out2) at (6,-1.605) {};
   \node [coordinate] (3out1d) at (5,-1.822) {};
   \node [coordinate] (3out2d) at (6,-1.822) {};

   \draw[rounded corners] (3in2u) -- (3in2) -- (3cornerin) -- (sum3corner) -- (sum3.left in)
   (3out1d) -- (3out1) -- (3cornerout) -- (cosum3corner) -- (cosum3.right out)
   (cosum3.io) -- (cosum3in) -- (topcross3) -- (3in1) -- (3in1u)
   (sum3.io) -- (sum3out) -- (botcross3) -- (3out2) -- (3out2d)
;
   \draw (sum3.right in) -- (cosum3.left out)
;

   \node (eq3) at (6.5,0) {\(=\)};
   \node at (6.5,-1em) {(3)\({}^{\dagger}\)};

   \node [plus] (sum4) at (8,-0.975) {};
   \node [coplus] (cosum4) at (8,0.215) {};
   \node [coordinate] (4out2) at (8,-1.822) {};
   \node [coordinate] (cosum4corner) at (7,-1.605) {};
   \node [coordinate] (sum4corner) at (7,0.434) {};
   \node [coordinate] (sum4corn) at (7,0.757) {};
   \node [coordinate] (cosum4in) at (8,0.757) {};
   \node [coordinate] (4in2) at (8,1.605) {};
   \node [coordinate] (4in2u) at (8,1.822) {};
   \node [coordinate] (4in1) at (7,1.605) {};
   \node [coordinate] (4in1u) at (7,1.822) {};
   \node (cross4bot) at (7.5,-0.38) {};
   \node (cross4top) at (7.5,1.191) {};

   \draw (sum4.right in) .. controls +(60:0.4) and +(-60:0.4) .. (cosum4.right out)
   (sum4.io) -- (4out2)
   (cosum4.left out) .. controls +(-120:0.5) and +(90:0.3) .. (cosum4corner) -- +(0,-0.217)
;
   \draw[rounded corners] (sum4.left in) -- (cross4bot) -- (sum4corner) -- (sum4corn) -- (4in2) -- (4in2u)
   (cosum4.io) -- (cosum4in) -- (cross4top) -- (4in1) -- (4in1u)
;

   \node (eq4) at (9,0) {\(=\)};
   \node at (9,-1em) {(3)};

   \node [plus] (sum5) at (10.5,-0.216) {};
   \node [coplus] (cosum5) at (10,0.216) {};
   \node [coordinate] (sum5corner) at (11,0.434) {};
   \node [coordinate] (cosum5corner) at (9.5,-0.434) {};
   \node [coordinate] (sum5out) at (10.5,-0.975) {};
   \node [coordinate] (cosum5in) at (10,0.975) {};
   \node [coordinate] (5cornerin) at (11,0.975) {};
   \node [coordinate] (5cornerout) at (9.5,-0.975) {};

   \draw[rounded corners] (5cornerin) -- (sum5corner) -- (sum5.right in)
   (5cornerout) -- (cosum5corner) -- (cosum5.left out);
   \draw (sum5.left in) -- (cosum5.right out)
   (sum5.io) -- (sum5out)
   (cosum5.io) -- (cosum5in);
   \end{tikzpicture}}
\end{center}

Color swapping this argument gives an argument for deriving equation \textbf{(24)} from equation
\textbf{(23)}.  One way to arrive at the numbered Frobenius equations from an outer equation uses
counitality and coassociativity.  Putting the pieces together for this is left as an exercise to the
reader.

\section{Additional connections}
\label{A:connections}
Recall that a special commutative Frobenius monoid is:\\
\textbf{(F1)--(F3)} a commutative monoid:
\begin{invisiblelabel}
\label{eqnF1F2F3}
\end{invisiblelabel}
  \begin{center}
    \scalebox{0.80}{
% Unitality
   \begin{tikzpicture}[-, thick, node distance=0.74cm]
   \node [ba] (summer) {};
   \node [coordinate] (sum) [below of=summer] {};
   \node [coordinate] (Lsum) [above left of=summer] {};
   \node [oao] (insert) [above of=Lsum, shift={(0,-0.35)}] {};
   \node [coordinate] (Rsum) [above right of=summer] {};
   \node [coordinate] (sumin) [above of=Rsum] {};
   \node (equal) [right of=Rsum, shift={(0,-0.26)}] {\(=\)};
   \node [coordinate] (in) [right of=equal, shift={(0,1)}] {};
   \node [coordinate] (out) [right of=equal, shift={(0,-1)}] {};

   \draw (insert) .. controls +(270:0.3) and +(120:0.3) .. (summer.left in)
         (summer.right in) .. controls +(60:0.6) and +(270:0.6) .. (sumin)
         (summer) -- (sum)    (in) -- (out);
   \end{tikzpicture}
        \hspace{1.0cm}
% Associativity
   \begin{tikzpicture}[-, thick, node distance=0.7cm]
   \node [ba] (uradder) {};
   \node [ba] (adder) [below of=uradder, shift={(-0.35,0)}] {};
   \node [coordinate] (urm) [above of=uradder, shift={(-0.35,0)}] {};
   \node [coordinate] (urr) [above of=uradder, shift={(0.35,0)}] {};
   \node [coordinate] (left) [left of=urm] {};

   \draw (adder.right in) .. controls +(60:0.2) and +(270:0.1) .. (uradder.io)
         (uradder.right in) .. controls +(60:0.35) and +(270:0.3) .. (urr)
         (uradder.left in) .. controls +(120:0.35) and +(270:0.3) .. (urm)
         (adder.left in) .. controls +(120:0.75) and +(270:0.75) .. (left)
         (adder.io) -- +(270:0.5);

   \node (eq) [right of=uradder, shift={(0,-0.25)}] {\(=\)};

   \node [ba] (ulsummer) [right of=eq, shift={(0,0.25)}] {};
   \node [ba] (summer) [below of=ulsummer, shift={(0.35,0)}] {};
   \node [coordinate] (ulm) [above of=ulsummer, shift={(0.35,0)}] {};
   \node [coordinate] (ull) [above of=ulsummer, shift={(-0.35,0)}] {};
   \node [coordinate] (right) [right of=ulm] {};

   \draw (summer.left in) .. controls +(120:0.2) and +(270:0.1) .. (ulsummer.io)
         (ulsummer.left in) .. controls +(120:0.35) and +(270:0.3) .. (ull)
         (ulsummer.right in) .. controls +(60:0.35) and +(270:0.3) .. (ulm)
         (summer.right in) .. controls +(60:0.75) and +(270:0.75) .. (right)
         (summer.io) -- +(270:0.5);
   \end{tikzpicture}
        \hspace{1.0cm}
% Commutativity
   \begin{tikzpicture}[-, thick, node distance=0.7cm]
   \node [ba] (twadder) {};
   \node [coordinate] (twout) [below of=twadder] {};
   \node [coordinate] (twR) [above right of=twadder, shift={(-0.2,0)}] {};
   \node (cross) [above of=twadder] {};
   \node [coordinate] (twRIn) [above left of=cross, shift={(0,0.3)}] {};
   \node [coordinate] (twLIn) [above right of=cross, shift={(0,0.3)}] {};

   \draw (twadder.right in) .. controls +(60:0.35) and +(-45:0.25) .. (cross)
                            .. controls +(135:0.2) and +(270:0.4) .. (twRIn);
   \draw (twadder.left in) .. controls +(120:0.35) and +(-135:0.25) .. (cross.center)
                           .. controls +(45:0.2) and +(270:0.4) .. (twLIn);
   \draw (twout) -- (twadder);

   \node (eq) [right of=twR] {\(=\)};

   \node [coordinate] (L) [right of=eq] {};
   \node [ba] (adder) [below right of=L] {};
   \node [coordinate] (out) [below of=adder] {};
   \node [coordinate] (R) [above right of=adder] {};
   \node (cross) [above left of=R] {};
   \node [coordinate] (LIn) [above left of=cross] {};
   \node [coordinate] (RIn) [above right of=cross] {};

   \draw (adder.left in) .. controls +(120:0.7) and +(270:0.7) .. (LIn)
         (adder.right in) .. controls +(60:0.7) and +(270:0.7) .. (RIn)
         (out) -- (adder);
   \end{tikzpicture}
    }
\end{center}
\textbf{(F4)--(F6)} which is also a cocommutative comonoid:
\begin{invisiblelabel}
\label{eqnF4F5F6}
\end{invisiblelabel}
  \begin{center}
    \scalebox{0.80}{
     \reflectbox{
      \rotatebox[origin=c]{180}{
% Counitality
   \begin{tikzpicture}[-, thick, node distance=0.74cm]
   \node [ba] (summer) {};
   \node [coordinate] (sum) [below of=summer] {};
   \node [coordinate] (Lsum) [above left of=summer] {};
   \node [oao] (insert) [above of=Lsum, shift={(0,-0.35)}] {};
   \node [coordinate] (Rsum) [above right of=summer] {};
   \node [coordinate] (sumin) [above of=Rsum] {};
   \node (equal) [right of=Rsum, shift={(0,-0.26)}] {\(=\)};
   \node [coordinate] (in) [right of=equal, shift={(0,1)}] {};
   \node [coordinate] (out) [right of=equal, shift={(0,-1)}] {};
   \draw (insert) .. controls +(270:0.3) and +(120:0.3) .. (summer.left in)
         (summer.right in) .. controls +(60:0.6) and +(270:0.6) .. (sumin)
         (summer) -- (sum)    (in) -- (out);
   \end{tikzpicture}
        \hspace{1.0cm}
% Coassociativity
   \begin{tikzpicture}[-, thick, node distance=0.7cm]
   \node [ba] (uradder) {};
   \node [ba] (adder) [below of=uradder, shift={(-0.35,0)}] {};
   \node [coordinate] (urm) [above of=uradder, shift={(-0.35,0)}] {};
   \node [coordinate] (urr) [above of=uradder, shift={(0.35,0)}] {};
   \node [coordinate] (left) [left of=urm] {};
   \draw (adder.right in) .. controls +(60:0.2) and +(270:0.1) .. (uradder.io)
         (uradder.right in) .. controls +(60:0.35) and +(270:0.3) .. (urr)
         (uradder.left in) .. controls +(120:0.35) and +(270:0.3) .. (urm)
         (adder.left in) .. controls +(120:0.75) and +(270:0.75) .. (left)
         (adder.io) -- +(270:0.5);
   \node (eq) [right of=uradder, shift={(0,-0.25)}] {\(=\)};
   \node [ba] (ulsummer) [right of=eq, shift={(0,0.25)}] {};
   \node [ba] (summer) [below of=ulsummer, shift={(0.35,0)}] {};
   \node [coordinate] (ulm) [above of=ulsummer, shift={(0.35,0)}] {};
   \node [coordinate] (ull) [above of=ulsummer, shift={(-0.35,0)}] {};
   \node [coordinate] (right) [right of=ulm] {};
   \draw (summer.left in) .. controls +(120:0.2) and +(270:0.1) .. (ulsummer.io)
         (ulsummer.left in) .. controls +(120:0.35) and +(270:0.3) .. (ull)
         (ulsummer.right in) .. controls +(60:0.35) and +(270:0.3) .. (ulm)
         (summer.right in) .. controls +(60:0.75) and +(270:0.75) .. (right)
         (summer.io) -- +(270:0.5);
   \end{tikzpicture}
        \hspace{1.0cm}
% Cocommutativity
   \begin{tikzpicture}[-, thick, node distance=0.7cm]
   \node [ba] (twadder) {};
   \node [coordinate] (twout) [below of=twadder] {};
   \node [coordinate] (twR) [above right of=twadder, shift={(-0.2,0)}] {};
   \node (cross) [above of=twadder] {};
   \node [coordinate] (twRIn) [above left of=cross, shift={(0,0.3)}] {};
   \node [coordinate] (twLIn) [above right of=cross, shift={(0,0.3)}] {};
   \draw (twadder.right in) .. controls +(60:0.35) and +(-45:0.25) .. (cross.center)
                            .. controls +(135:0.2) and +(270:0.4) .. (twRIn);
   \draw (twadder.left in) .. controls +(120:0.35) and +(-135:0.25) .. (cross)
                           .. controls +(45:0.2) and +(270:0.4) .. (twLIn);
   \draw (twout) -- (twadder);
   \node (eq) [right of=twR] {\(=\)};
   \node [coordinate] (L) [right of=eq] {};
   \node [ba] (adder) [below right of=L] {};
   \node [coordinate] (out) [below of=adder] {};
   \node [coordinate] (R) [above right of=adder] {};
   \node (cross) [above left of=R] {};
   \node [coordinate] (LIn) [above left of=cross] {};
   \node [coordinate] (RIn) [above right of=cross] {};
   \draw (adder.left in) .. controls +(120:0.7) and +(270:0.7) .. (LIn)
         (adder.right in) .. controls +(60:0.7) and +(270:0.7) .. (RIn)
         (out) -- (adder);
   \end{tikzpicture}
    }}}
\end{center}
\textbf{(F7)--(F8)} that satisfies the Frobenius equations:
\begin{invisiblelabel}
\label{eqnF7F8}
\end{invisiblelabel}
\begin{center}
 \scalebox{0.80}{
% Frobenius monoid
   \begin{tikzpicture}[thick]
   \node [ba] (sum1) at (0.5,-0.216) {};
   \node [ab] (cosum1) at (1,0.216) {};
   \node [coordinate] (sum1corner) at (0,0.434) {};
   \node [coordinate] (cosum1corner) at (1.5,-0.434) {};
   \node [coordinate] (sum1out) at (0.5,-0.975) {};
   \node [coordinate] (cosum1in) at (1,0.975) {};
   \node [coordinate] (1cornerin) at (0,0.975) {};
   \node [coordinate] (1cornerout) at (1.5,-0.975) {};

   \draw[rounded corners] (1cornerin) -- (sum1corner) -- (sum1.left in)
   (1cornerout) -- (cosum1corner) -- (cosum1.right out);
   \draw (sum1.right in) -- (cosum1.left out)
   (sum1.io) -- (sum1out)
   (cosum1.io) -- (cosum1in);

   \node (eq1) at (2,0) {\(=\)};
   \node [ba] (sum2) at (3,0.325) {};
   \node [ab] (cosum2) at (3,-0.325) {};
   \node [coordinate] (sum2inleft) at (2.5,0.975) {};
   \node [coordinate] (sum2inright) at (3.5,0.975) {};
   \node [coordinate] (cosum2outleft) at (2.5,-0.975) {};
   \node [coordinate] (cosum2outright) at (3.5,-0.975) {};

   \draw (sum2inleft) .. controls +(270:0.3) and +(120:0.15) .. (sum2.left in)
   (sum2inright) .. controls +(270:0.3) and +(60:0.15) .. (sum2.right in)
   (cosum2outleft) .. controls +(90:0.3) and +(240:0.15) .. (cosum2.left out)
   (cosum2outright) .. controls +(90:0.3) and +(300:0.15) .. (cosum2.right out)
   (sum2.io) -- (cosum2.io);

   \node (eq2) at (4,0) {\(=\)};
   \node [ba] (sum3) at (5.5,-0.216) {};
   \node [ab] (cosum3) at (5,0.216) {};
   \node [coordinate] (sum3corner) at (6,0.434) {};
   \node [coordinate] (cosum3corner) at (4.5,-0.434) {};
   \node [coordinate] (sum3out) at (5.5,-0.975) {};
   \node [coordinate] (cosum3in) at (5,0.975) {};
   \node [coordinate] (3cornerin) at (6,0.975) {};
   \node [coordinate] (3cornerout) at (4.5,-0.975) {};

   \draw[rounded corners] (3cornerin) -- (sum3corner) -- (sum3.right in)
   (3cornerout) -- (cosum3corner) -- (cosum3.left out);
   \draw (sum3.left in) -- (cosum3.right out)
   (sum3.io) -- (sum3out)
   (cosum3.io) -- (cosum3in);
   \end{tikzpicture}
}
\end{center}
\textbf{(F9)} and the special equation:
\begin{invisiblelabel}
\label{eqnF9}
\end{invisiblelabel}
\begin{center}
 \scalebox{0.80}{
% The Frobenius monoid is special
\begin{tikzpicture}[thick]
   \node [ba] (sum) at (0.4,-0.5) {};
   \node [ab] (cosum) at (0.4,0.5) {};
   \node [coordinate] (in) at (0.4,1) {};
   \node [coordinate] (out) at (0.4,-1) {};
   \node (eq) at (1.3,0) {\(=\)};
   \node [coordinate] (top) at (2,1) {};
   \node [coordinate] (bottom) at (2,-1) {};

   \path (sum.left in) edge[bend left=30] (cosum.left out)
   (sum.right in) edge[bend right=30] (cosum.right out);
   \draw (top) -- (bottom)
   (sum.io) -- (out)
   (cosum.io) -- (in);
\end{tikzpicture}
}.
\end{center}
We have seen one of the bimonoid laws, equation {\hyperref[eqn78910]{\textbf{(10)}}}, is compatible
with the special commutative Frobenius monoid structure without making it trivial.  Indeed, we get
the extra-special commutative Frobenius monoid structure when we include this equation.  We can see
below that equation {\hyperref[eqn78910]{\textbf{(7)}}} is not only compatible with the special
commutative Frobenius monoid structure, it is a derived equation under the assumption of a special
commutative Frobenius monoid structure.  As noted by Heunen and Vicary \cite{HV}, if either of the
other two bimonoid equations ({\hyperref[eqn78910]{\textbf{(8)--(9)}}}) hold for a Frobenius monoid,
the monoid is trivial, so the characteristic difference between non-trivial bimonoids and
non-trivial Frobenius monoids is the way the unit and counit interact with the comultiplication and
multiplication, respectively.

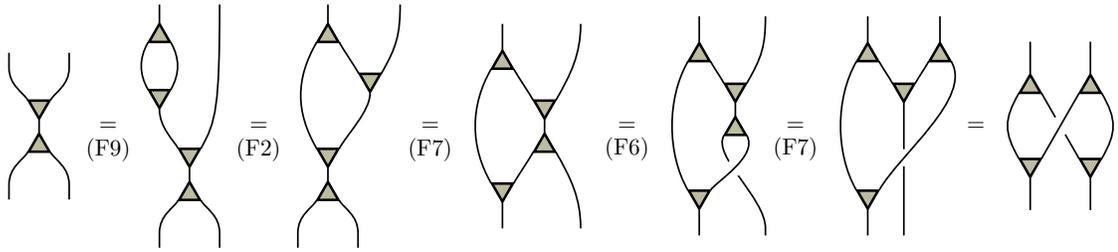
\begin{center}
\scalebox{0.80}{   \begin{tikzpicture}[thick]
   \node [ba] (sum1) at (0.5,0.325) {};
   \node [ab] (cosum1) at (0.5,-0.325) {};
   \node [coordinate] (sum1inleft) at (0,0.975) {};
   \node [coordinate] (sum1inright) at (1,0.975) {};
   \node [coordinate] (cosum1outleft) at (0,-0.975) {};
   \node [coordinate] (cosum1outright) at (1,-0.975) {};
   \node [coordinate] (1in1) at (0,1.22) {};
   \node [coordinate] (1in2) at (1,1.22) {};
   \node [coordinate] (1out1) at (0,-1.22) {};
   \node [coordinate] (1out2) at (1,-1.22) {};

   \draw (sum1inleft) .. controls +(270:0.3) and +(120:0.15) .. (sum1.left in)
   (sum1inright) .. controls +(270:0.3) and +(60:0.15) .. (sum1.right in)
   (cosum1outleft) .. controls +(90:0.3) and +(240:0.15) .. (cosum1.left out)
   (cosum1outright) .. controls +(90:0.3) and +(300:0.15) .. (cosum1.right out)
   (sum1.io) -- (cosum1.io) (1in1) -- (sum1inleft) (1in2) -- (sum1inright)
   (1out1) -- (cosum1outleft) (1out2) -- (cosum1outright)
;

   \node (eq1) at (1.65,0) {\(=\)};
   \node at (1.65,-1em) {(F9)};

   \node [ab] (ab2top) at (2.5,1.5) {};
   \node [ba] (ba2top) at (2.5,0.5) {};
   \node [ba] (ba2bot) at (3,-0.477) {};
   \node [ab] (ab2bot) at (3,-1.127) {};
   \node [coordinate] (2in1) at (2.5,2.022) {};
   \node [coordinate] (2in2) at (3.5,2.022) {};
   \node [coordinate] (ab2outleft) at (2.5,-1.777) {};
   \node [coordinate] (ab2outright) at (3.5,-1.777) {};
   \node [coordinate] (2out1) at (2.5,-2.022) {};
   \node [coordinate] (2out2) at (3.5,-2.022) {};

   \draw (ba2top.right in) .. controls +(60:0.3) and +(-60:0.3) .. (ab2top.right out)
   (ba2top.left in) .. controls +(120:0.3) and +(-120:0.3) .. (ab2top.left out)
   (ab2top.io) -- (2in1) (ba2bot.io) -- (ab2bot.io)
   (ba2top.io) .. controls +(-90:0.3) and +(120:0.15) .. (ba2bot.left in)
   (ba2bot.right in) .. controls +(60:0.5) and +(-90:1.7) .. (2in2)
   (ab2bot.left out) .. controls +(-120:0.15) and +(90:0.3) .. (ab2outleft) -- (2out1)
   (ab2bot.right out) .. controls +(-60:0.15) and +(90:0.3) .. (ab2outright) -- (2out2)
;

   \node (eq2) at (4.15,0) {\(=\)};
   \node at (4.15,-1em) {(F2)};

   \node [ab] (ab3top) at (5.3,1.5) {};
   \node [ab] (ab3bot) at (5.3,-1.127) {};
   \node [ba] (ba3top) at (6,0.77) {};
   \node [ba] (ba3bot) at (5.3,-0.477) {};
   \node [coordinate] (3in1) at (5.3,2.022) {};
   \node [coordinate] (3in2) at (6.5,2.022) {};
   \node [coordinate] (ab3outleft) at (4.8,-1.777) {};
   \node [coordinate] (ab3outright) at (5.8,-1.777) {};
   \node [coordinate] (3out1) at (4.8,-2.022) {};
   \node [coordinate] (3out2) at (5.8,-2.022) {};

   \draw (ab3top.left out) .. controls +(-120:0.7) and +(120:0.7) .. (ba3bot.left in)
   (ba3bot.io) -- (ab3bot.io) (ab3top.io) -- (3in1)
   (ab3top.right out) .. controls +(-60:0.15) and +(120:0.15) .. (ba3top.left in)
   (ba3top.io) .. controls +(-90:0.3) and +(60:0.3) .. (ba3bot.right in)
   (ba3top.right in) .. controls +(60:0.5) and +(-90:0.5) .. (3in2)
   (ab3bot.left out) .. controls +(-120:0.15) and +(90:0.3) .. (ab3outleft) -- (3out1)
   (ab3bot.right out) .. controls +(-60:0.15) and +(90:0.3) .. (ab3outright) -- (3out2)
;

   \node (eq3) at (7,0) {\(=\)};
   \node at (7,-1em) {(F7)};

   \node [ab] (ab4top) at (8.2,1.055) {};
   \node [ba] (ba4top) at (8.9,0.325) {};
   \node [ab] (ab4bot) at (8.9,-0.325) {};
   \node [ba] (ba4bot) at (8.2,-1.055) {};
   \node [coordinate] (4in1) at (8.2,1.7) {};
   \node [coordinate] (4in2) at (9.5,1.7) {};
   \node [coordinate] (4out1) at (8.2,-1.7) {};
   \node [coordinate] (4out2) at (9.5,-1.7) {};

   \draw (ab4top.left out) .. controls +(-120:0.7) and +(120:0.7) .. (ba4bot.left in)
   (ab4top.io) -- (4in1) (ba4bot.io) -- (4out1) (ab4bot.io) -- (ba4top.io)
   (ab4bot.left out) .. controls +(-120:0.15) and +(60:0.15) .. (ba4bot.right in)
   (ba4top.left in) .. controls +(120:0.15) and +(-60:0.15) .. (ab4top.right out)
   (ba4top.right in) .. controls +(60:0.5) and +(-90:0.5) .. (4in2)
   (ab4bot.right out) .. controls +(-60:0.5) and +(90:0.5) .. (4out2)
;
   \end{tikzpicture}    \begin{tikzpicture}[thick]
   \node (eq0) at (-0.3,0) {\(=\)};
   \node at (-0.3,-1em) {(F6)};

   \node [ab] (ab1top) at (0.9,1.177) {};
   \node [ba] (ba1top) at (1.5,0.6) {};
   \node [ab] (ab1bot) at (1.5,-0.05) {};
   \node [ba] (ba1bot) at (0.9,-1.177) {};
   \node [coordinate] (1in1) at (0.9,1.822) {};
   \node [coordinate] (1in2) at (2,1.822) {};
   \node [coordinate] (1out1) at (0.9,-1.822) {};
   \node [coordinate] (1out2) at (2,-1.822) {};
   \node [coordinate] (corner1) at (2,-1.7) {};

   \draw (ab1top.left out) .. controls +(-120:0.7) and +(120:0.7) .. (ba1bot.left in)
   (ab1top.io) -- (1in1) (ba1bot.io) -- (1out1) (ab1bot.io) -- (ba1top.io)
   (ba1top.left in) .. controls +(120:0.15) and +(-60:0.15) .. (ab1top.right out)
   (ba1top.right in) .. controls +(60:0.5) and +(-90:0.5) .. (1in2)
   (ab1bot.left out) .. controls +(-120:0.4) and +(90:0.5) .. (corner1) -- (1out2)
;
   \node [hole] at (1.5,-0.69) {};
   \draw (ab1bot.right out) .. controls +(-60:0.4) and +(60:0.15) .. (ba1bot.right in) 
;

   \node (eq1) at (2.5,0) {\(=\)};
   \node at (2.5,-1em) {(F7)};

   \node [ab] (ab2l) at (3.7,1.177) {};
   \node [ab] (ab2r) at (4.9,1.177) {};
   \node [ba] (ba2top) at (4.3,0.6) {};
   \node [ba] (ba2bot) at (3.7,-1.177) {};
   \node [coordinate] (2in1) at (3.7,1.822) {};
   \node [coordinate] (2in2) at (4.9,1.822) {};
   \node [coordinate] (2out1) at (3.7,-1.822) {};
   \node [coordinate] (2out2) at (4.3,-1.822) {};
   \node (cross2) at (4.3,-0.53) {};

   \draw (ab2l.left out) .. controls +(-120:0.7) and +(120:0.7) .. (ba2bot.left in)
   (ab2r.right out) .. controls +(-60:0.7) and +(60:0.7) .. (ba2bot.right in)
   (ab2l.io) -- (2in1) (ab2r.io) -- (2in2) (ba2bot.io) -- (2out1) (ba2top.io) -- (cross2) -- (2out2)
   (ab2l.right out) .. controls +(-60:0.15) and +(120:0.15) .. (ba2top.left in)
   (ab2r.left out) .. controls +(-120:0.15) and +(60:0.15) .. (ba2top.right in)
;

   \node (eq2) at (5.5,0) {\(=\)};

   \node [ab] (abl) at (6.4,0.65) {};
   \node [ab] (abr) at (7.4,0.65) {};
   \node [ba] (bal) at (6.4,-0.65) {};
   \node [ba] (bar) at (7.4,-0.65) {};
   \node [coordinate] (3in1) at (6.4,1.4) {};
   \node [coordinate] (3in2) at (7.4,1.4) {};
   \node [coordinate] (3out1) at (6.4,-1.4) {};
   \node [coordinate] (3out2) at (7.4,-1.4) {};
   \node (cross3) at (6.9,0) {};

   \draw (abl.io) -- (3in1) (abr.io) -- (3in2) (bal.io) -- (3out1) (bar.io) -- (3out2)
   (abl.left out) .. controls +(-120:0.5) and +(120:0.5) .. (bal.left in)
   (abr.right out) .. controls +(-60:0.5) and +(60:0.5) .. (bar.right in)
   (abr.left out) .. controls +(-120:0.15) and +(60:0.15) .. (bal.right in)
   (abl.right out) -- (cross3) -- (bar.left in)
;
   \node at (0,-1.912) {\({}\)};
   \end{tikzpicture}}
\end{center}
The last equality is due to the naturality of symmetry.

\chapter{Generalization of the Box construction}
\label{generalbox}
In Section~\ref{boxvsection} we described the Box construction in the particular case of \(\catC =
\vectk\).  The description does not change much when we allow \(\catC\) to be an arbitrary category
with biproducts.  A more substantial generalization is required in order to consider the case of
\(\catC = \relk\).  We build up to the PROP \(\boxr\) here, but leave its exploitation for future
work.

\begin{definition}
% Def. of the category of non-commuting squares, lite
Given a category \((\catC, m, 0, \Delta, !)\) with biproducts, we take \(\boxful\) to have
  \begin{itemize}
     \item the same objects as \(\catC\),
     \item the morphisms in \(\hom(X,Y)\) as equivalence classes of

   \begin{center}
    \begin{tikzpicture}[->]
     \node (V)  at (0,0) {\(X\)};
     \node (VV) at (1.7,0) {\(X \otimes X\)};
     \node (WS) at (4.4,0) {\(X \otimes S\)};
     \node (WT) at (7.1,0) {\(X \otimes T\)};
     \node (WW) at (9.4,0) {\(Y \otimes Y\)};
     \node (W)  at (11.1,0) {\(Y\)};

     \path
      (V)  edge node[above] {\(\Delta\)}                    (VV)
      (VV) edge node[above] {\(\mathrm{id}_{X} \otimes b\)} (WS)
      (WS) edge node[above] {\(\mathrm{id}_{X} \otimes a\)} (WT)
      (WT) edge node[above] {\(d \otimes c\)}               (WW)
      (WW) edge node[above] {\(m\)}                         (W);
    \end{tikzpicture},
   \end{center}
   abbreviated as \((d,c,a,b)\),
     \item composition given by \[(d,c,a,b) \of (d',c',a',b') = \left(d'd, [d'c \quad c'],
   \left[ \begin{array}{cc}
   a      & 0 \\
   a'b'ca & a'  \end{array} \right], 
   \left[ \begin{array}{c}
   b \\ b'd \end{array} \right] \right).\]
  \end{itemize}
\end{definition}
The morphisms in \(\hom(X,Y)\) are more convenient to work with when depicted as non-commuting
squares, as in
   \begin{center}
    \begin{tikzpicture}[->]
     \node (A) at (0,0) {\(X\)};
     \node (S) at (0,1.5) {\(S\)};
     \node (T) at (1.5,1.5) {\(T\)};
     \node (B) at (1.5,0) {\(Y\)};

     \path
      (A) edge node[below] {\(d\)} (B)
          edge node[left] {\(b\)} (S)
      (S) edge node[above] {\(a\)} (T)
      (T) edge node[right] {\(c\)} (B)
;
    \end{tikzpicture}.
   \end{center}
This form for morphisms explains the name \(\boxful\).

Two squares, \((d,c,a,b)\) and \((d,c',a',b')\) are in the same equivalence class if there are
isomorphisms \(\alpha \maps S \to S'\) and \(\omega \maps T \to T'\) such that the following diagram
in \(\catC\) commutes:

\begin{center}
 \begin{tikzpicture}[->]
  \node (A) at (0,0) {\(X\)};
  \node (S) at (-1,1.5) {\(S\)};
  \node (T) at (2.5,1.5) {\(T\)};
  \node (B) at (1.5,0) {\(Y\)};
  \node (S1) at (-1,-1.5) {\(S'\)};
  \node (T1) at (2.5,-1.5) {\(T'\)};

  \path
   (A) edge node[below right] {\(b'\)} (S1)
       edge node[above right] {\(b\)} (S)
   (S) edge node[above] {\(a\)} (T)
       edge node[left] {\(\alpha\)} (S1)
   (T) edge node[above left] {\(c\)} (B)
       edge node[right] {\(\omega\)} (T1)
   (S1) edge node[above] {\(a'\)} (T1)
   (T1) edge node[below left] {\(c'\)} (B)
;
 \end{tikzpicture}.
\end{center}
Note that when \(\catC\) is strict, \(\boxful\) will also be strict.

Because \(\catC\) has biproducts, each object in \(\catC\) is a bicommutative bimonoid, and
morphisms in \(\catC\) are all bimonoid homomorphisms.  This definition of \(\boxful\) generalizes
the construction we used to form the PROP \(\boxv\) from \(\vectk\), but it still does not
allow us to define \(\boxr\).  For that we need to drop the condition that \(\catC\) itself
has biproducts, but we still require some of the structure that came with having biproducts.  The
rough idea is to find a subcategory \(\catC'\) of \(\catC\) that has all the objects of \(\catC\)
such that \(\catC'\) has biproducts, define \(\boxful'\), and bootstrap up to \(\boxful\).
Generally, a subcategory `having all the objects' of a given category is not preserved by
equivalences of categories, but there is a good alternative.

\begin{definition}
An \Define{essentially wide} subcategory \(\catC'\) of \(\catC\) is a subcategory of \(\catC\) that
`essentially' contains all the objects of \(\catC\).  That is, the inclusion functor from \(\catC'\)
to \(\catC\) is essentially surjective on objects.
\end{definition}

In PROPs and other strict categories, where isomorphic objects are equal, the notion of an
essentially wide subcategory can be replaced with the notion of a \emph{wide} subcategory, also
referred to as a \emph{lluf} subcategory.  In this case `essentially surjective' in the above
definition is replaced with `bijective'.  That is, every object in \(\catC\) is also an object in
any wide subcategory of \(\catC\).

\begin{definition}
% Def. of the category of non-commuting squares
Given a category \((\catC, m, 0, \Delta, !)\) with an (essentially) wide subcategory \((\catC', m, 0,
\Delta, !)\) such that \(\catC'\) has biproducts, we take \(\boxful\) to have 
  \begin{itemize}
     \item the same objects as \(\catC\),
     \item the morphisms in \(\hom(X,Y)\) as equivalence classes of

   \begin{center}
    \begin{tikzpicture}[->]
     \node (V)  at (0,0) {\(X\)};
     \node (VV) at (1.7,0) {\(X \otimes X\)};
     \node (WS) at (4.4,0) {\(X \otimes S\)};
     \node (WT) at (7.1,0) {\(X \otimes T\)};
     \node (WW) at (9.4,0) {\(Y \otimes Y\)};
     \node (W)  at (11.1,0) {\(Y\)};

     \path
      (V)  edge node[above] {\(\Delta\)}                    (VV)
      (VV) edge node[above] {\(\mathrm{id}_{X} \otimes b\)} (WS)
      (WS) edge node[above] {\(\mathrm{id}_{X} \otimes a\)} (WT)
      (WT) edge node[above] {\(d \otimes c\)}               (WW)
      (WW) edge node[above] {\(m\)}                         (W);
    \end{tikzpicture},
   \end{center}
   abbreviated as \((d,c,a,b)\), such that \(\Delta \of c = \left[ \begin{array}{c} c \\
   c \end{array} \right] \of \Delta\), \(\Delta \of d = \left[ \begin{array}{c} d \\ d \end{array}
   \right] \of \Delta\), \(b \of m = m \of [b \quad b]\), and \(d \of m = m \of [d \quad d]\),
     \item composition given by \[(d,c,a,b) \of (d',c',a',b') = \left(d'd, [d'c \quad c'],
   \left[ \begin{array}{cc}
   a      & 0 \\
   a'b'ca & a'  \end{array} \right], 
   \left[ \begin{array}{c}
   b \\ b'd \end{array} \right] \right).\]
  \end{itemize}
\end{definition}

Again, the morphisms of \(\boxful\) can be depicted as non-commuting squares:
\begin{center}
 \begin{tikzpicture}[->]
  \node (A) at (0,0) {\(X\)};
  \node (S) at (0,1.5) {\(S\)};
  \node (T) at (1.5,1.5) {\(T\)};
  \node (B) at (1.5,0) {\(Y\)};

  \path
   (A) edge node[below] {\(d\)} (B)
       edge node[left] {\(b\)} (S)
   (S) edge node[above] {\(a\)} (T)
   (T) edge node[right] {\(c\)} (B)
;
 \end{tikzpicture}.
\end{center}
%
% Equivalence classes of morphisms
The technical conditions on \(b\), \(c\), and \(d\) introduced in this version of \(\boxful\) can be
summarized as \(b\) is a monoid homomorphism, \(c\) is a comonoid homomorphism, and \(d\) is a
bimonoid homomorphism.  These will be discussed further at the end of this section, together with a
depiction of the morphisms in \(\boxful\) as string diagrams.  The equivalence classes for morphisms 
in \(\boxful\) are similar to what they were above.  Two squares, \((d,c,a,b)\) and \((d,c',a',b')\)
are in the same equivalence class if there are isomorphisms \(\alpha \maps S \to S'\) and \(\omega
\maps T \to T'\) such that the following diagram in \(\catC\) commutes: 

\begin{center}
 \begin{tikzpicture}[->]
  \node (A) at (0,0) {\(X\)};
  \node (S) at (-1,1.5) {\(S\)};
  \node (T) at (2.5,1.5) {\(T\)};
  \node (B) at (1.5,0) {\(Y\)};
  \node (S1) at (-1,-1.5) {\(S'\)};
  \node (T1) at (2.5,-1.5) {\(T'\)};

  \path
   (A) edge node[below right] {\(b'\)} (S1)
       edge node[above right] {\(b\)} (S)
   (S) edge node[above] {\(a\)} (T)
       edge node[left] {\(\alpha\)} (S1)
   (T) edge node[above left] {\(c\)} (B)
       edge node[right] {\(\omega\)} (T1)
   (S1) edge node[above] {\(a'\)} (T1)
   (T1) edge node[below left] {\(c'\)} (B)
;
 \end{tikzpicture}.
\end{center}
In what follows, the term \(\boxful\) will refer to this more general version of \(\boxful\) unless
otherwise noted.

% Composition for noncommuting squares
We refer to the objects \(S\) and \(T\) as the \emph{prestate} and \emph{state}, respectively.  This
should recall their names when these objects were vector spaces.  In that case we referred to them
as \emph{prestate space} and \emph{state space}, respectively.  The formula for composition in
\(\boxful\) is easier to understand as coming from the diagram:
\begin{center}
 \begin{tikzpicture}[->]
  \node (A) at (0,0) {\(X_1\)};
  \path (A) +(108:2.5) node (S1) {\(S_1\)};
  \path (S1) +(0:2.5) node (T1) {\(T_1\)};
  \path (A) +(0:2.5) node (B) {\(X_2\)};
  \path (B) +(0:2.5) node (C) {\(X_3\)};
  \path (B) +(72:2.5) node (S2) {\(S_2\)};
  \path (S2) +(0:2.5) node (T2) {\(T_2\)};
  \path (S1) +(36:2.5) node (S3) {\(S_1 \otimes S_2\)};
  \path (S3) +(0:2.5) node (T3) {\(T_1 \otimes T_2\)};

  \path
   (T1) edge node[left] {\(c\)} (B)
   (T2) edge node[right] {\(c'\)} (C)
   (T3) edge node[right] {\(\pi_1\)} (T1)
        edge node[right] {\(\pi_2\)} (T2)
;
  \draw[white,line width=5pt,-] (S2) -- (S3);

  \path
   (A) edge node[below] {\(d\)} (B)
       edge node[left] {\(b\)} (S1)
   (B) edge node[below] {\(d'\)} (C)
       edge node[right] {\(b'\)} (S2)
   (S1) edge node[above] {\(a\)} (T1)
        edge node[left] {\(\iota_1\)} (S3)
   (S2) edge node[above] {\(a'\)} (T2)
        edge node[left] {\(\iota_2\)} (S3)
   (S3) edge node[above] {\(a \otimes a'\)} (T3)
;
 \end{tikzpicture}.
\end{center}
The \(b\) and \(c\) sides in the composite come from the pentagons on the left and right,
respectively.  That is, \(b\) comes from the ways to get from \(X_1\) to \(S_1 \otimes S_2\), and
\(c\) comes from the ways \(T_1 \otimes T_2\) can get to \(X_3\).  Similarly, the \(a\) side in the
composite comes from the ways \(S_1 \otimes S_2\) can get to \(T_1 \otimes T_2\), which includes a
direct path (from \(S_1\) to \(T_1\) and from \(S_2\) to \(T_2\)) and a looped path (from \(S_1\) to
\(T_2\)).  The \(d\) side in the composite comes from the most direct path from \(X_1\) to \(X_3\).

% Well-definedness of composition
Given two composable morphisms, \(f \maps X_1 \to X_2\) and \(g \maps X_2 \to X_3\) with
representatives \((d_1, c_1, a_1, b_1)\) and \((d_1', c_1', a_1', b_1')\) for \(f\) and \((d_2, c_2,
a_2, b_2)\) and \((d_2', c_2', a_2', b_2')\) for \(g\), we have the following commuting diagrams:
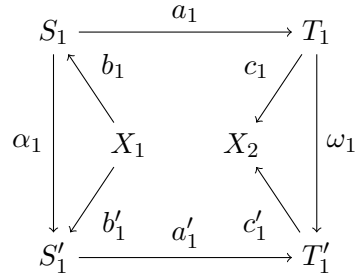
\begin{center}
 \begin{tikzpicture}[->]
  \node (A) at (0,0) {\(X_1\)};
  \node (S) at (-1,1.5) {\(S_1\)};
  \node (T) at (2.5,1.5) {\(T_1\)};
  \node (B) at (1.5,0) {\(X_2\)};
  \node (S1) at (-1,-1.5) {\(S_1'\)};
  \node (T1) at (2.5,-1.5) {\(T_1'\)};

  \path
   (A) edge node[below right] {\(b_1'\)} (S1)
       edge node[above right] {\(b_1\)} (S)
   (S) edge node[above] {\(a_1\)} (T)
       edge node[left] {\(\alpha_1\)} (S1)
   (T) edge node[above left] {\(c_1\)} (B)
       edge node[right] {\(\omega_1\)} (T1)
   (S1) edge node[above] {\(a_1'\)} (T1)
   (T1) edge node[below left] {\(c_1'\)} (B)
;
 \end{tikzpicture}
\end{center}
and
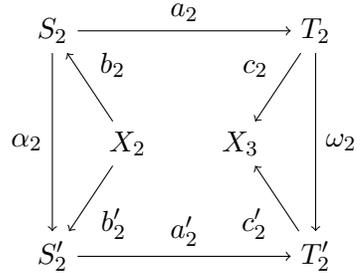
\begin{center}
 \begin{tikzpicture}[->]
  \node (A) at (0,0) {\(X_2\)};
  \node (S) at (-1,1.5) {\(S_2\)};
  \node (T) at (2.5,1.5) {\(T_2\)};
  \node (B) at (1.5,0) {\(X_3\)};
  \node (S1) at (-1,-1.5) {\(S_2'\)};
  \node (T1) at (2.5,-1.5) {\(T_2'\)};

  \path
   (A) edge node[below right] {\(b_2'\)} (S1)
       edge node[above right] {\(b_2\)} (S)
   (S) edge node[above] {\(a_2\)} (T)
       edge node[left] {\(\alpha_2\)} (S1)
   (T) edge node[above left] {\(c_2\)} (B)
       edge node[right] {\(\omega_2\)} (T1)
   (S1) edge node[above] {\(a_2'\)} (T1)
   (T1) edge node[below left] {\(c_2'\)} (B)
;
 \end{tikzpicture}.
\end{center}
We need isomorphisms \(\alpha_{12} \maps S_1 \tensor S_2 \to S_1' \tensor S_2'\) and \(\omega_{12}
\maps T_1 \tensor T_2 \to T_1' \tensor T_2'\) that make the corresponding diagram for \(g \of f\)
commute.  We leave it as an exercise to the reader to check that \(\alpha_{12} = \alpha_1 \tensor
\alpha_2\) and \(\omega_{12} = \omega_1 \tensor \omega_2\).

% What is multiplication and addition of morphisms?
For the matrix notation of composition to make sense, there need to be notions of addition and
multiplication.  Multiplication is simply composition in \(\catC\).  Addition is a generalization of
equation~{\hyperref[eqn11121314]{\textbf{(12)}}} from Section~\ref{finvect}, given in Definition~\ref{morphadd}.

\begin{definition}
\label{morphadd}
% Define addition of morphisms
Let \((\catC,m,0,\Delta,!)\) be a category with an (essentially) wide subcategory
\((\catC',m,0,\Delta,!)\) such that \(\catC'\) has biproducts, and \(A, B \in \Obj(\catC)\).  For
\(x,y \maps A \to B\) we define the operation \(x+y := m_B \of (x \otimes y) \of \Delta_A\)
\end{definition}

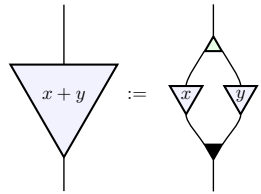
\begin{center}
\scalebox{0.65}{
\begin{tikzpicture}[thick]
% Addition of morphisms
   \node (bctop) {};
   \node [multiply] (bc) [below of=bctop, shift={(0,-1)}] {\(x+y\)};
   \node (bcbottom) [below of=bc, shift={(0,-1.1)}] {};

   \draw (bctop) -- (bc) -- (bcbottom);

   \node (eq) [right of=bc, shift={(0.5,0)}] {\(:=\)};

   \node [multiply] (b) [right of=eq] {\(x\)};
   \node [delta] (dupe) [above right of=b, shift={(-0.15,0.3)}] {};
   \node (top) [above of=dupe] {};
   \node [multiply] (c) [below right of=dupe, shift={(-0.15,-0.3)}] {\(y\)};
   \node [plus] (adder) [below right of=b, shift={(-0.15,-0.4)}] {};
   \node (out) [below of=adder] {};

   \draw
   (dupe.left out) .. controls +(240:0.15) and +(90:0.15) .. (b.90)
   (dupe.right out) .. controls +(300:0.15) and +(90:0.15) .. (c.90)
   (top) -- (dupe.io)
   (adder.io) -- (out)
   (adder.left in) .. controls +(120:0.15) and +(270:0.15) .. (b.io)
   (adder.right in) .. controls +(60:0.15) and +(270:0.15) .. (c.io);
\end{tikzpicture}
}.
\end{center}
It is clear that if \(x,y\) are endomorphisms of \(A\), \(x+y\) will also be an endomorphism of
\(A\).

\begin{theorem} \label{boxcat}
\(\boxful\) is a monoidal category with the same monoidal product on objects as \(\catC\), and
\((d,c,a,b) \otimes (d',c',a',b') = (d \otimes d', c \otimes c', a \otimes a', b \otimes b')\).
\end{theorem}
\begin{proof}
To show \(\boxful\) is a category, we need to show associativity and the unit laws hold.  To show
\(\boxful\) is a monoidal category, we also need to show the associator and unitors exist and satisfy
the pentagon and triangle equations.  It is clear that identity morphisms are formed when the
prestate and state are both the zero object and \(d\) is an identity morphism in \(\catC\):
\begin{center}
 \begin{tikzpicture}[->]
  \node (A) at (0,0) {\(X\)};
  \node (S) at (0,1.5) {\(0\)};
  \node (T) at (1.5,1.5) {\(0\)};
  \node (B) at (1.5,0) {\(X\)};

  \path
   (A) edge node[below] {\(1_X\)} (B)
       edge node[left] {} (S)
   (S) edge node[above] {} (T)
   (T) edge node[right] {} (B)
;
 \end{tikzpicture}.
\end{center}
The associator and unitors of \(\boxful\) are formed from those of \(\catC\) using the same
trick.  It's easy to see their pentagon and triangle equations follow directly from those in
\(\catC\).  It is also easy to see the monoidal product on morphisms is compatible with
composition.
Associativity of morphisms in \(\boxful\) reduces to associativity of the monoidal product for the
prestate and state, and an associativity requirement on the \(\catC\) morphisms \(d, c, a,\) and
\(b\).

Denoting the compositions \((d_j, c_j, a_j, b_j) \of (d_i, c_i, a_i, b_i)\) as \((d_{ij}, c_{ij},
a_{ij}, b_{ij})\), we get \((d_3, c_3, a_3, b_3) \of (d_{12}, c_{12}, a_{12}, b_{12}) = (d_{12,3},
c_{12,3}, a_{12,3}, b_{12,3})\):
\begin{center}
 \begin{tikzpicture}[->]
  \node (A) at (0,0) {\(X_1\)};
  \node (B) at (2,0) {\(X_3\)};
  \node (C) at (3.5,0) {\(X_4\)};
  \node (S1) at (-1,1.5) {\(S_1 \otimes S_2\)};
  \node (T1) at (1,1.5) {\(T_1 \otimes T_2\)};
  \node (S2) at (3,1.5) {\(S_3\)};
  \node (T2) at (4.5,1.5) {\(T_3\)};

  \node (eq) at (5.15,0.75) {\(=\)};

  \node (A1) at (6.7,0) {\(X_1\)};
  \node (B1) at (9.7,0) {\(X_4\)};
  \node (S) at (6.7,1.5) {\((S_1 \otimes S_2) \otimes S_3\)};
  \node (T) at (9.7,1.5) {\((T_1 \otimes T_2) \otimes T_3\)};

  \path
   (A) edge node[below] {\(d_{12}\)} (B)
       edge node[left] {\(b_{12}\)} (S1)
   (B) edge node[below] {\(d_3\)} (C)
       edge node[right] {\(b_3\)} (S2)
   (S1) edge node[above] {\(a_{12}\)} (T1)
   (S2) edge node[above] {\(a_3\)} (T2)
   (T1) edge node[left] {\(c_{12}\)} (B)
   (T2) edge node[right] {\(c_3\)} (C)

   (A1) edge node[below] {\(d_{12,3}\)} (B1)
        edge node[left] {\(b_{12,3}\)} (S)
   (S) edge node[above] {\(a_{12,3}\)} (T)
   (T) edge node[right] {\(c_{12,3}\)} (B1)
;
 \end{tikzpicture}
\end{center}
and \((d_{23}, c_{23}, a_{23}, b_{23}) \of (d_1, c_1, a_1, b_1) = (d_{1,23}, c_{1,23}, a_{1,23},
b_{1,23})\):
\begin{center}
 \begin{tikzpicture}[->]
  \node (A) at (0,0) {\(X_1\)};
  \node (B) at (1.5,0) {\(X_2\)};
  \node (C) at (3.5,0) {\(X_4\)};
  \node (S1) at (-1,1.5) {\(S_1\)};
  \node (T1) at (0.5,1.5) {\(T_1\)};
  \node (S2) at (2.5,1.5) {\(S_2 \otimes S_3\)};
  \node (T2) at (4.5,1.5) {\(T_2 \otimes T_3\)};

  \node (eq) at (5.15,0.75) {\(=\)};

  \node (A1) at (6.7,0) {\(X_1\)};
  \node (B1) at (9.7,0) {\(X_4\)};
  \node (S) at (6.7,1.5) {\(S_1 \otimes (S_2 \otimes S_3)\)};
  \node (T) at (9.7,1.5) {\(T_1 \otimes (T_2 \otimes T_3)\)};

  \path
   (A) edge node[below] {\(d_1\)} (B)
       edge node[left] {\(b_1\)} (S1)
   (B) edge node[below] {\(d_{23}\)} (C)
       edge node[right] {\(b_{23}\)} (S2)
   (S1) edge node[above] {\(a_1\)} (T1)
   (S2) edge node[above] {\(a_{23}\)} (T2)
   (T1) edge node[left] {\(c_1\)} (B)
   (T2) edge node[right] {\(c_{23}\)} (C)

   (A1) edge node[below] {\(d_{1,23}\)} (B1)
        edge node[left] {\(b_{1,23}\)} (S)
   (S) edge node[above] {\(a_{1,23}\)} (T)
   (T) edge node[right] {\(c_{1,23}\)} (B1)
;
 \end{tikzpicture}.
\end{center}
The associativity requirements for \(d, c, a,\) and \(b\) require there to be canonical isomorphisms
\(d_{12,3} \cong d_{1,23}\), \(c_{12,3} \cong c_{1,23}\), \(a_{12,3} \cong a_{1,23}\), and
\(b_{12,3} \cong b_{1,23}\).  The associativity requirement for \(d\) clearly holds because
composition is associative in \(\catC\) --- \(d_{1,23} = d_1 (d_2 d_3) \cong (d_1 d_2) d_3 =
d_{12,3}\).
% Top of the square is associative
We see the associativity requirement for \(a\) holds because \(a_{ij} = \left[
\begin{array}{cc}
a_i & 0 \\
a_j b_j c_i a_i & a_j
\end{array} \right],\)
which means
\[a_{12,3} = \left[
\begin{array}{cc}
a_{12} & 0 \\
a_3 b_3 c_{12} a_{12} & a_3
\end{array}
\right] = \left[
\begin{array}{ccc}
\left[\begin{array}{c}
a_1 \\ a_2 b_2 c_1 a_1
\end{array}\right. &
\left.\begin{array}{c}
0 \\ a_2
\end{array}\right] &
\begin{array}{c}
0 \\ 0
\end{array}
\\
\left[a_3 b_3 d_2 c_1 a_1 + a_3 b_3 c_2 a_2 b_2 c_1 a_1 \right. & \left. a_3 b_3 c_2 a_2 \right] & a_3
\end{array} \right],\]
since \(c_{12} = \left[ \begin{array}{cc} d_2 c_1 & c_2 \end{array} \right]\).  A similar
calculation gives a canonically isomorphic matrix for \(a_{1,23}\),
\[a_{1,23} = \left[
\begin{array}{ccc}
a_1 & 0 & 0\\
\left[\begin{array}{c}
a_2 b_2 c_1 a_1\\
a_3 b_3 d_2 c_1 a_1 + a_3 b_3 c_2 a_2 b_2 c_1 a_1
\end{array}\right] &
\left[\begin{array}{c}
a_2 \\ a_3 b_3 c_2 a_2
\end{array}\right. &
\left.\begin{array}{c}
0 \\ a_3
\end{array}\right]
\end{array}\right].\]
% Yep.  Top of the square was associative

The proofs of the associativity requirements for \(c\) and \(b\) are similar to each other,
transposed.  We present the argument for \(c\) and leave the argument for \(b\) to the reader.
Since \(c_{ij} = [d_j c_i \quad c_j]\), we have
\begin{align*}
c_{12,3} & = [d_3 c_{12} \quad c_3]\\
         & = [d_3 [d_2 c_1 \quad c_2] \quad c_3]\\
         & = [[d_3 d_2 c_1 \quad d_3 c_2] \quad c_3]\\
     & \cong [d_3 d_2 c_1 \quad [d_3 c_2 \quad c_3]]\\
         & = [d_{23} c_1 \quad c_{23}] = c_{1,23}.
\end{align*}
So we see composition of morphisms in \(\boxful\) is associative.
\end{proof}

The anatomy of \(\boxful\) makes more sense when it is understood as a category with an obvious
evaluation functor \(\eval \maps \boxful \to \catC\).  We can also find a `feedthrough' functor
\(\feed \maps \boxful \to \catC\) and functor in the reverse direction \(\G \maps \catC \to
\boxful\).  The maps of objects \(\eval_0 \maps \Obj(\boxful) \to \Obj(\catC)\) is a bijection,
\(\feed_0 = \eval_0\), and \(\G_0\) is its inverse.  The map of morphisms \(\eval_1 \maps
\Mor(\boxful) \to \Mor(\catC)\) is given by \(\eval_1(d,c,a,b) = d + c a b\), \(\feed_1 \maps
\Mor(\boxful) \to \Mor(\catC)\) is given by \(\feed_1(d,c,a,b) = d\), and \(\G_1 \maps \Mor(\catC)
\to \Mor(\boxful)\) is given by \(\G_1(d) = (d, !, 0, 0)\).  That is,

\begin{center}
\(\eval_1 \left(
 \begin{tikzpicture}[baseline=0.5cm,->]
  \node (A) at (0,0) {\(X_1\)};
  \node (C) at (1.4,0) {\(X_2\)};
  \node (B) at (0,1.4) {\(S\)};
  \node (D) at (1.4,1.4) {\(T\)};

  \path
   (A) edge node[below] {\(d\)} (C)
       edge node[left] {\(b\)} (B)
   (B) edge node[above] {\(a\)} (D)
   (D) edge node[right] {\(c\)} (C);
 \end{tikzpicture} \right) = d + c a b\),\quad
\(\feed_1 \left(
 \begin{tikzpicture}[baseline=0.5cm,->]
  \node (A) at (0,0) {\(X_1\)};
  \node (C) at (1.4,0) {\(X_2\)};
  \node (B) at (0,1.4) {\(S\)};
  \node (D) at (1.4,1.4) {\(T\)};

  \path
   (A) edge node[below] {\(d\)} (C)
       edge node[left] {\(b\)} (B)
   (B) edge node[above] {\(a\)} (D)
   (D) edge node[right] {\(c\)} (C);
 \end{tikzpicture} \right) = d\)
\\and
\(\G_1(d) =\)
 \begin{tikzpicture}[baseline=0.5cm,->]
  \node (A) at (0,0) {\(X_1\)};
  \node (C) at (1.4,0) {\(X_2\)};
  \node (B) at (0,1.4) {\(0\)};
  \node (D) at (1.4,1.4) {\(0\)};

  \path
   (A) edge node[below] {\(d\)} (C)
       edge (B)
   (B) edge (D)
   (D) edge (C);
 \end{tikzpicture}.
\end{center}

\begin{theorem} \label{boxjoint}
\(\eval\), \(\feed\) and \(\G\) are monoidal functors, and
\begin{itemize}
\item \(\eval\), \(\feed\) and \(\G\) are essentially surjective
\item \(\eval\) and \(\feed\) are full, but not faithful
\item \(\G\) is faithful, but not full
\item \(\G\) has no adjoint.
\item If \(\catC\) is a symmetric monoidal category (\emph{resp.} a braided monoidal category),
  \(\eval\), \(\feed\) and \(\G\) are symmetric (\emph{resp.} braided) monoidal functors.  In
  particular, \(\boxful\) will also be a symmetric (\emph{resp.} braided) monoidal category.
\end{itemize}
\end{theorem}

Note that since \(\boxful\) is strict whenever \(\catC\) is, this last item means \(\boxful\) is a
PROP whenever \(\catC\) is.

\begin{proof}
All three functors are bijective on objects, so it immediately follows they are essentially
surjective.  We note that \(\feed \of \G\) and \(\eval \of \G\) are both the identity functor on
\(\catC\), which implies \(\feed\) and \(\eval\) are surjective on \emph{all} morphisms, hence
full.  This also implies \(\G\) is injective on morphisms, so \(\G\) is faithful.  On the other
hand, a morphism in \(\boxful\) between \(\G_0 (V_1)\) and \(\G_0 (V_2)\) where the prestate or
state are not isomorphic to the zero object is not the \(\G_1\)-image of any morphism in \(\catC\),
so \(\G\) is not full.  Similarly, \(\feed\) and \(\eval\) cannot be faithful.

  \begin{figure}[h]
    \centering
    \begin{tikzpicture}
   \node (1203) at (180:3)  {\((V_1 \tensor V_2) \tensor V_3\)};
   \node (1023) at (135:2.75)  {\(V_1 \tensor (V_2 \tensor V_3)\)};
   \node (2301) at (45:2.75)   {\((V_2 \tensor V_3) \tensor V_1\)};
   \node (2031) at (0:3)    {\(V_2 \tensor (V_3 \tensor V_1)\)};
   \node (2013) at (315:2.75)  {\(V_2 \tensor (V_1 \tensor V_3)\)};
   \node (2103) at (225:2.75)  {\((V_2 \tensor V_1) \tensor V_3\)};

   \draw [->] (1203) to node [above, shift={(-1.25,0)}] {\((a_{V_1,V_2,V_3},!,0,0)\)} (1023);
   \draw [->] (1023) to node [above, shift={(0,0.2)}] {\((B_{V_1,V_2 \tensor V_3},!,0,0)\)} (2301);
   \draw [->] (2301) to node [above, shift={(1.25,0)}] {\((a_{V_2,V_3,V_1},!,0,0)\)} (2031);
   \draw [->] (1203) to node [below, shift={(-1.5,0)}] {\((B_{V_1,V_2} \tensor \mathrm{Id},!,0,0)\)} (2103);
   \draw [->] (2103) to node [below, shift={(0,-0.2)}] {\((a_{V_2,V_1,V_3},!,0,0)\)} (2013);
   \draw [->] (2013) to node [below, shift={(1.5,0)}] {\((\mathrm{Id} \tensor B_{V_1,V_3},!,0,0)\)} (2031);
\end{tikzpicture}
    \caption[A hexagon law inside $\boxful$]{A hexagon law inside \(\boxful\).  This diagram commutes
      in \(\boxful\) when the analogous diagram in \(\catC\), where the morphisms are the first
      coordinates of the morphisms here, commutes.\label{hexagonlaw}}
  \end{figure}

In Theorem~\ref{boxvectfunctors} we saw \(\G\) has no adjoint, taking advantage of the fact that
\(\vectk\) has an initial object and a terminal object, neither of which is preserved by \(\G\).  In
this more general setting, \(\catC\) may not have initial or terminal objects, so it is necessary to
show from the definitions that \(\G\) has no adjoint.  It is a straightforward exercise sketched out
below, which we leave to the reader to fill in the details.

Suppose \(R \maps \boxful \to \catC\), \(d \in \hom_{\catC}(A,B)\), and \(R(d,c,a,b) = f \in
\hom_{\catC}(A,B)\).  Clearly \(\G \of R(d,c,a,b) = (f,!,0,0)\).  Further suppose \(a \in
\hom_{\catC}(P,S)\), so that the prestate of \((d,c,a,b)\) is \(P\) and the state is \(S\).  If
\(R\) is a right adjoint to \(\G\), there would be a natural transformation \(\eta \maps 1_{\boxful}
\To \G \of R\) such that the square 
\begin{center}
% Natural transformation square
\begin{tikzpicture}[node distance=2cm,->]
  \node (A1) at (0,0) {\(A\)};
  \node (B1) [below of=A1] {\(A\)};
  \node (A2) [right of=A1] {\(B\)};
  \node (B2) [below of=A2] {\(B\)};

  \draw (A1) to node[above] {\((d,c,a,b)\)} (A2);
  \draw (B1) to node[below] {\((f,!,0,0)\)} (B2);
  \draw (A1) to node[left] {\(\eta_A\)} (B1);
  \draw (A2) to node[right] {\(\eta_B\)} (B2);
\end{tikzpicture}
\end{center}
commutes in \(\boxful\).  Taking \(P'\) and \(P''\) to be the prestates for \(\eta_A\) and
\(\eta_B\), respectively, this means \(P' \cong P'' \oplus P\).  However, \(P\) can vary without
affecting \(A\) or \(B\), while \(P'\) and \(P''\) are determined by \(A\) and \(B\).  This
contradiction means \(R\) cannot be a right adjoint to \(\G\).  Similarly \(\G\) has no left
adjoint, by considering the natural transformation \(\epsilon \maps \G \of L \To 1_{\boxful}\).

It is easy to check that the associator, unitor, and symmetry/braiding isomorphisms in \(\boxful\)
are the \(\G_1\)-images of the respective isomorphisms in \(\catC\).  In \(\boxful\) we have
\((d,!,0,0) \of (d',!,0,0) = (d \of d',!,0,0)\), so \(\G\) preserves the coherence laws that hold in
\(\catC\).  See Figure~\ref{hexagonlaw} for an example of one of these coherence laws in
\(\boxful\).  Thus \(\boxful\) is a symmetric (\emph{resp.} braided) monoidal category when
\(\catC\) is.  It is also easy to see \(\feed_1 (d,!,0,0) = \eval_1 (d,!,0,0) = d\), so \(\feed\)
and \(\eval\) will also preserve all the coherence laws.
\end{proof}

An alternate way to depict morphisms in \(\boxful\) is through string diagrams.  The morphism
\((d,c,a,b)\) is depicted:
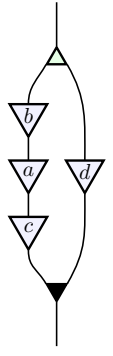
\begin{center}
\scalebox{0.75}{
  \begin{tikzpicture}[thick]
   \node [delta] (split) at (0,0) {};
   \node [coordinate] (above) [above of=split] {};
   \node [multiply] (d) at (0.5,-2) {\(d\)};
   \node [multiply] (a) at (-0.5,-2) {\(a\)};
   \node [multiply] (c) [below of=a] {\(c\)};
   \node [multiply] (b) [above of=a] {\(b\)};
   \node [plus] (join) at (0,-4.1) {};
   \node [coordinate] (below) [below of=join] {};
  
   \draw (above) -- (split.io)
         (split.left out) .. controls +(240:0.3) and +(90:0.3) .. (b.90)
         (c.io) .. controls +(270:0.3) and +(120:0.3) .. (join.left in)
         (join.io) -- (below)
         (split.right out) .. controls +(300:0.7) and +(90:0.7) .. (d.90)
         (b.io) -- (a.90) (a.io) -- (c.90)
         (d.io) .. controls +(270:0.7) and +(60:0.7) .. (join.right in);
  \end{tikzpicture}},
\end{center}
and it is easy to see the monoidal product \((d,c,a,b) \tensor (d',c',a',b')\) is what it is
supposed to be:
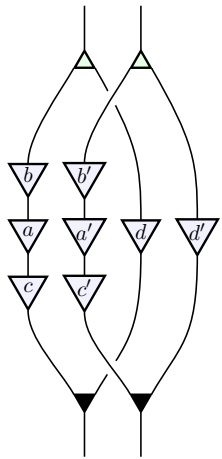
\begin{center}
\scalebox{0.75}{
  \begin{tikzpicture}[thick]
   \node [delta] (split) at (0,0) {};
   \node [coordinate] (above) [above of=split] {};
   \node [multiply] (d) at (1,-3) {\(d\)};
   \node [multiply] (a) at (-1,-3) {\(a\)};
   \node [multiply] (c) [below of=a] {\(c\)};
   \node [multiply] (b) [above of=a] {\(b\)};
   \node [plus] (join) at (0,-6) {};
   \node [coordinate] (below) [below of=join] {};
  
   \draw (above) -- (split.io)
         (split.left out) .. controls +(240:0.5) and +(90:0.5) .. (b.90)
         (c.io) .. controls +(270:0.5) and +(120:0.5) .. (join.left in)
         (join.io) -- (below)
         (split.right out) .. controls +(300:1.5) and +(90:1) .. (d.90)
         (b.io) -- (a.90) (a.io) -- (c.90)
         (d.io) .. controls +(270:1.5) and +(60:1) .. (join.right in);

   \node [hole] at (0.5,-0.63) {};
   \node [hole] at (0.5,-5.43) {};

   \node [delta] (split) at (1,0) {};
   \node [coordinate] (above) [above of=split] {};
   \node [multiply] (d) at (2,-3) {\(d'\)};
   \node [multiply] (a) at (0,-3) {\(a'\)};
   \node [multiply] (c) [below of=a] {\(c'\)};
   \node [multiply] (b) [above of=a] {\(b'\)};
   \node [plus] (join) at (1,-6) {};
   \node [coordinate] (below) [below of=join] {};
  
   \draw (above) -- (split.io)
         (split.left out) .. controls +(240:0.5) and +(90:0.5) .. (b.90)
         (c.io) .. controls +(270:0.5) and +(120:0.5) .. (join.left in)
         (join.io) -- (below)
         (split.right out) .. controls +(300:1.5) and +(90:1) .. (d.90)
         (b.io) -- (a.90) (a.io) -- (c.90)
         (d.io) .. controls +(270:1.5) and +(60:1) .. (join.right in);
  \end{tikzpicture}}.
\end{center}
Composition is still tedious with string diagrams, but the reason for the technical conditions on
\(b\), \(c\), and \(d\) is illuminated somewhat:
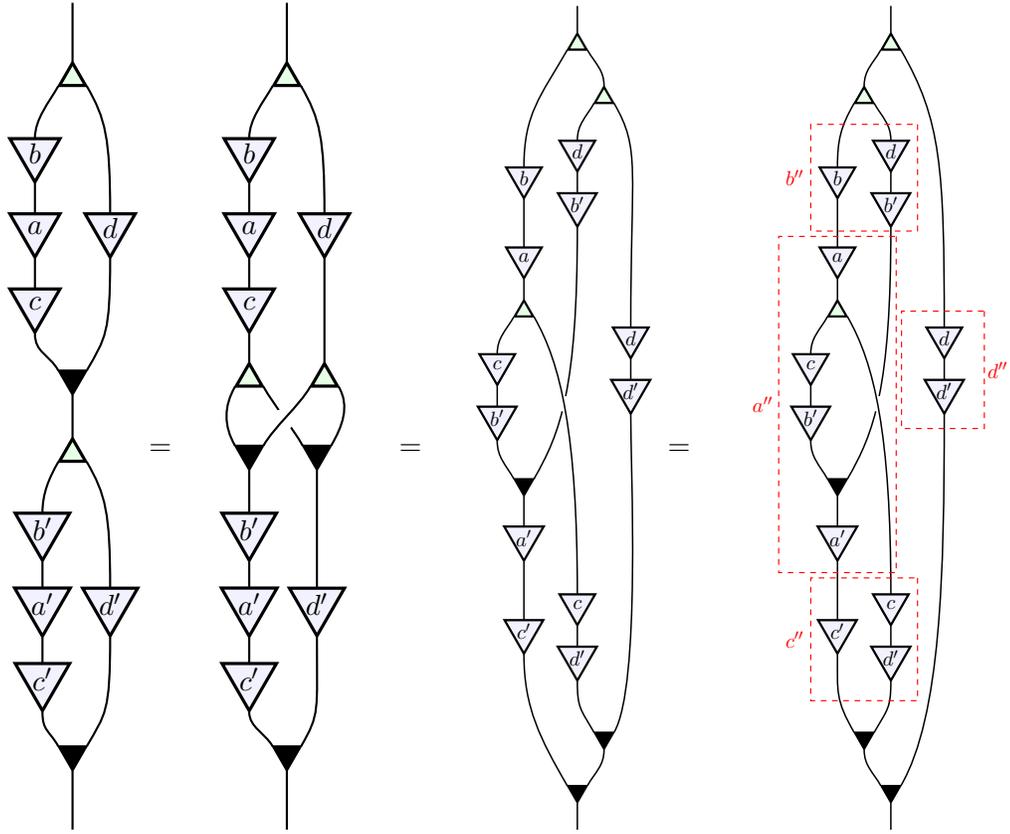
\begin{center}
  \begin{tikzpicture}[thick]
   \node [delta] (split) at (0,0) {};
   \node [coordinate] (above) [above of=split] {};
   \node [multiply] (d) at (0.5,-2) {\(d\)};
   \node [multiply] (a) at (-0.5,-2) {\(a\)};
   \node [multiply] (c) [below of=a] {\(c\)};
   \node [multiply] (b) [above of=a] {\(b\)};
   \node [plus] (join) at (0,-4) {};
   \node [delta] (below) [below of=join] {};

   \draw (above) -- (split.io)
         (split.left out) .. controls +(240:0.25) and +(90:0.25) .. (b.90)
         (c.io) .. controls +(270:0.25) and +(120:0.25) .. (join.left in)
         (join.io) -- (below)
         (split.right out) .. controls +(300:0.5) and +(90:1) .. (d.90)
         (b.io) -- (a.90) (a.io) -- (c.90)
         (d.io) .. controls +(270:1) and +(60:0.5) .. (join.right in);

   \node [multiply] (d) at (0.5,-7) {\(d'\)};
   \node [multiply] (a) at (-0.4,-7) {\(a'\)};
   \node [multiply] (c) [below of=a] {\(c'\)};
   \node [multiply] (b) [above of=a] {\(b'\)};
   \node [plus] (join) at (0,-9) {};

   \draw (below.left out) .. controls +(240:0.25) and +(90:0.25) .. (b.90)
         (c.io) .. controls +(270:0.25) and +(120:0.25) .. (join.left in)
         (below.right out) .. controls +(300:0.5) and +(90:1) .. (d.90)
         (b.io) -- (a.90) (a.io) -- (c.90)
         (d.io) .. controls +(270:1) and +(60:0.5) .. (join.right in);
   \node [coordinate] (below) [below of=join] {};
   \draw (join.io) -- (below);
  \end{tikzpicture}
  \raisebox{5cm}{=} 
\hspace{1em}
  \begin{tikzpicture}[thick]
   \node [delta] (split) at (0,0) {};
   \node [coordinate] (above) [above of=split] {};
   \node [multiply] (d) at (0.5,-2) {\(d\)};
   \node [multiply] (a) at (-0.5,-2) {\(a\)};
   \node [multiply] (c) [below of=a] {\(c\)};
   \node [multiply] (b) [above of=a] {\(b\)};
   \node [delta] (join) at (-0.5,-4) {};
   \node [delta] (below) [right of=join] {};

   \draw (above) -- (split.io)
         (split.left out) .. controls +(240:0.25) and +(90:0.25) .. (b.90)
         (d.io) -- (below.io)
         (split.right out) .. controls +(300:0.5) and +(90:1) .. (d.90)
         (b.io) -- (a.90) (a.io) -- (c.90) (c.io) -- (join.io);

   \node [plus] (split) [below of=join] {};
   \node [plus] (above) at (0.4,-5) {};

   \draw (join.left out) .. controls +(240:0.3) and +(120:0.3) .. (split.left in)
         (below.right out) .. controls +(300:0.3) and +(60:0.3) .. (above.right in)
         (join.right out) .. controls +(300:0.3) and +(120:0.3) .. (above.left in);
   \node [hole] (hole) at (-0.05,-4.55) {};
   \draw (below.left out) .. controls +(240:0.3) and +(60:0.3) .. (split.right in);

   \node [multiply] (d) at (0.4,-7) {\(d'\)};
   \node [multiply] (a) at (-0.5,-7) {\(a'\)};
   \node [multiply] (c) [below of=a] {\(c'\)};
   \node [multiply] (b) [above of=a] {\(b'\)};
   \node [plus] (join) at (0,-9) {};
   \node [coordinate] (below) [below of=join] {};

   \draw (split.io) -- (b.90)
         (c.io) .. controls +(270:0.2) and +(120:0.2) .. (join.left in)
         (join.io) -- (below)
         (above.io) .. controls +(270:0.25) and +(90:0.25) .. (d.90)
         (b.io) -- (a.90) (a.io) -- (c.90)
         (d.io) .. controls +(270:1) and +(60:0.5) .. (join.right in);
  \end{tikzpicture}
\hspace{1em}
  \raisebox{5cm}{=}
\hspace{1em}
\scalebox{0.71}{
  \begin{tikzpicture}[thick]
   \node [delta] (split1) at (0,0) {};
   \node [delta] (split2) at (0.5,-1) {};
   \node [multiply] (d1) at (0,-2) {\(d\)};
   \node [multiply] (bp1) [below of=d1] {\(b'\)};
   \node [multiply] (b1) at (-1,-2.5) {\(b\)};
   \node [multiply] (a) at (-1,-4) {\(a\)};
   \node [delta] (split3) [below of=a] {};
   \node [multiply] (c1) at (-1.5,-6) {\(c\)};
   \node [multiply] (bp2) [below of=c1] {\(b'\)};
   \node [plus] (join1) at (-1,-8.25) {};
   \node [multiply] (ap) [below of=join1] {\(a'\)};
   \node [multiply] (cp) at (-1,-11) {\(c'\)};
   \node [multiply] (c2) at (0,-10.5) {\(c\)};
   \node [multiply] (dp2) [below of=c2] {\(d'\)};
   \node [multiply] (d2) at (1,-5.5) {\(d\)};
   \node [multiply] (dp1) [below of=d2] {\(d'\)};
   \node [plus] (join2) at (0.5,-13) {};
   \node [plus] (join3) at (0,-14) {};

   \draw (d1.io) -- (bp1.90) (d2.io) -- (dp1.90) (b1.io) -- (a.90)
         (c1.io) -- (bp2.90) (c2.io) -- (dp2.90) (ap.io) -- (cp.90)
         (a.io) -- (split3.io) (join1.io) -- (ap.90)
         (split1.io) -- +(0,0.5) (join3.io) -- +(0,-0.5)
         (join1.left in) .. controls +(120:0.3) and +(270:0.3) .. (bp2.io)
         (join2.left in) .. controls +(120:0.5) and +(270:0.5) .. (dp2.io)
         (join3.left in) .. controls +(120:1) and +(270:1) .. (cp.io)
         (join1.right in) .. controls +(60:1) and +(270:3) .. (bp1.io)
         (join2.right in) .. controls +(60:1) and +(270:1) .. (dp1.io)
         (join3.right in) .. controls +(60:0.3) and +(270:0.3) .. (join2.io);

   \node [hole] at (-0.25,-6.735) {};

   \draw (split1.left out) .. controls +(240:1) and +(90:1) .. (b1.90)
         (split2.left out) .. controls +(240:0.3) and +(90:0.3) .. (d1.90)
         (split3.left out) .. controls +(240:0.3) and +(90:0.3) .. (c1.90)
         (split1.right out) .. controls +(300:0.3) and +(90:0.3) .. (split2.io)
         (split2.right out) .. controls +(300:1) and +(90:3) .. (d2.90)
         (split3.right out) .. controls +(300:1) and +(90:3) .. (c2.90);
  \end{tikzpicture}
}
  \raisebox{5cm}{=}
\hspace{1em}
\scalebox{0.71}{
  \begin{tikzpicture}[thick]
   \node [delta] (split1) at (0,0) {};
   \node [delta] (split2) at (-0.5,-1) {};
   \node [multiply] (d1) at (0,-2) {\(d\)};
   \node [multiply] (bp1) [below of=d1] {\(b'\)};
   \node [multiply] (b1) at (-1,-2.5) {\(b\)};
   \node [multiply] (a) at (-1,-4) {\(a\)};
   \node [delta] (split3) [below of=a] {};
   \node [multiply] (c1) at (-1.5,-6) {\(c\)};
   \node [multiply] (bp2) [below of=c1] {\(b'\)};
   \node [plus] (join1) at (-1,-8.25) {};
   \node [multiply] (ap) [below of=join1] {\(a'\)};
   \node [multiply] (cp) at (-1,-11) {\(c'\)};
   \node [multiply] (c2) at (0,-10.5) {\(c\)};
   \node [multiply] (dp2) [below of=c2] {\(d'\)};
   \node [multiply] (d2) at (1,-5.5) {\(d\)};
   \node [multiply] (dp1) [below of=d2] {\(d'\)};
   \node [plus] (join2) at (-0.5,-13) {};
   \node [plus] (join3) at (0,-14) {};

   \draw (d1.io) -- (bp1.90) (d2.io) -- (dp1.90) (b1.io) -- (a.90)
         (c1.io) -- (bp2.90) (c2.io) -- (dp2.90) (ap.io) -- (cp.90)
         (a.io) -- (split3.io) (join1.io) -- (ap.90)
         (split1.io) -- +(0,0.5) (join3.io) -- +(0,-0.5)
         (join1.left in) .. controls +(120:0.3) and +(270:0.3) .. (bp2.io)
         (join2.left in) .. controls +(120:0.7) and +(270:0.7) .. (cp.io)
         (join3.left in) .. controls +(120:0.3) and +(270:0.3) .. (join2.io)
         (join1.right in) .. controls +(60:1) and +(270:3) .. (bp1.io)
         (join2.right in) .. controls +(60:0.5) and +(270:0.5) .. (dp2.io)
         (join3.right in) .. controls +(60:1.5) and +(270:3) .. (dp1.io);

   \node [hole] at (-0.25,-6.735) {};

   \draw (split2.left out) .. controls +(240:0.5) and +(90:0.5) .. (b1.90)
         (split2.right out) .. controls +(300:0.3) and +(90:0.3) .. (d1.90)
         (split3.left out) .. controls +(240:0.3) and +(90:0.3) .. (c1.90)
         (split1.left out) .. controls +(240:0.3) and +(90:0.3) .. (split2.io)
         (split1.right out) .. controls +(300:1.5) and +(90:3) .. (d2.90)
         (split3.right out) .. controls +(300:1) and +(90:3) .. (c2.90);

   \draw [red,dashed,semithick] (1.75,-5) rectangle (0.2,-7.2)  % d
                                (0.5,-1.5) rectangle (-1.5,-3.5)  % b
                                (0.1,-3.6) rectangle (-2.1,-9.9)  % a
                                (0.5,-10) rectangle (-1.5,-12.3); % c

   \node [red] at (2,-6.1) {\(d''\)};
   \node [red] at (-1.8,-2.5) {\(b''\)};
   \node [red] at (-2.4,-6.75) {\(a''\)};
   \node [red] at (-1.8,-11.15) {\(c''\)};
  \end{tikzpicture}
}.
\end{center}
The first equality is because every object is a bimonoid.  The second equality is due to the
technical conditions and topology-preserving moves.  The last equality is due to every object being
bicommutative.  The dashed boxes indicate which portions of the string diagram correspond to the
components of the composite morphism: \((d',c',a',b') \of (d,c,a,b) = (d'',c'',a'',b'')\).

\end{document}